\documentclass[a4paper,12pt,reqno,makeidx]{amsart}
\usepackage{amsfonts,amssymb,amscd,amsmath,latexsym,amsbsy,enumerate,stmaryrd,a4wide}
\usepackage[hypertex]{hyperref}
\newtheorem{thm}{Theorem}[section]
\newtheorem{cor}[thm]{Corollary}
\newtheorem{lemma}[thm]{Lemma}
\newtheorem{prop}[thm]{Proposition}
\newtheorem{defn}[thm]{Definition}
\theoremstyle{definition}
\newtheorem{remark}[thm]{Remark}
\numberwithin{equation}{section}

\def\bolde{\boldsymbol{e}}
\def\al{\alpha}
\def\bal{\boldsymbol\al}
\def\be{\beta}
\def\ga{\gamma}
\def\bga{\boldsymbol\ga}
\def\de{\delta}
\def\ep{\varepsilon}
\def\bep{\boldsymbol\varepsilon}
\def\et{\eta}
\def\te{\theta}
\def\la{\lambda}
\def\ka{\kappa}
\def\si{\sigma}
\def\vp{\varphi}
\def\om{\omega}

\def\ze{\zeta}
\def\Om{\Omega}
\def\De{\Delta}
\def\bDe{\boldsymbol\Delta}
\def\La{\Lambda}

\def\Up{\Upsilon}
\def\Z{\mathbb{Z}}

\def\R{\mathbb{R}}
\def\C{\mathbb{C}}
\def\N{\mathbb{N}}
\def\T{\mathbb{T}}
\def\NN{\N_0}
\def\KA{\mathbf{K}}
\def\EA{\mathbf{E}}
\def\FA{\mathbf{F}}
\def\OmA{\mathbf{\Omega}}
\def\DeA{\mathbf{\Delta}}
\def\SA{\mathbf{S}}
\def\cH{\mathcal{H}}
\def\cA{\mathcal{A}}
\def\cF{\mathcal{F}}
\def\cK{\mathcal{K}}
\def\cL{\mathcal{L}}
\def\cM{\mathcal{M}}
\def\cN{\mathcal{N}}
\def\cU{\mathcal{U}}
\def\Span{\text{\rm Span}}
\def\sgn{\text{\rm sgn}}
\def\ot{\otimes}
\def\su{U_q({\mathfrak{su}}(1,1))}
\def\hf{\frac12}
\def\Id{\text{\rm{Id}}}
\def\Tr{\text{\rm{Tr}}}
\def\Aut{\text{\rm{Aut}}}
\def\bS{\boldsymbol S}

\newcommand{\Res}[1]{\underset{#1}{\mathrm{Res}}\,}
\newcommand{\rphis}[5]{\,_{#1}\vp_{#2} \left( \genfrac{.}{.}{0pt}{}{#3}{#4}
\ ;#5 \right)}
\newcommand{\rpsis}[5]{\,_{#1}\psi_{#2} \left( \genfrac{.}{.}{0pt}{}{#3}{#4}
\ ;#5 \right)}

\newcommand{\Psis}[3]{\Psi \left( \genfrac{.}{.}{0pt}{}{#1}{#2}
\ ;#3 \right)}

\makeindex

\begin{document}
\title[Dual quantum group]
{The dual quantum group for the
quantum group analogue of the normalizer of $SU(1,1)$ in
$SL(2,\C)$}

\author{Wolter Groenevelt, Erik Koelink and Johan Kustermans}
\address{Technische Universiteit Delft, DIAM, PO Box 5031,
2600 GA Delft, the Netherlands}
\email{w.g.m.groenevelt@tudelft.nl}
\address{Radboud Universiteit, IMAPP, FNWI, Heyendaalseweg 135, 6525 AJ Nijmegen,
the Netherlands}
\email{e.koelink@math.ru.nl}
\address{Blijde Inkomststraat 85 bus 22, B-3000 Leuven, Belgium}
\email{johan.kustermans@realdolmen.com}

\begin{abstract}
The quantum group analogue of the normalizer of $SU(1,1)$ in $SL(2,\C)$
is an important and non-trivial example of a non-compact quantum group. The general
theory of locally compact quantum groups in the operator
algebra setting implies the existence of the
dual quantum group. The first main goal of the paper is to give
an explicit description  of the dual quantum group for this example
involving the quantized enveloping algebra $\su$. It turns out that
$\su$ does not suffice to generate the dual quantum group. The
dual quantum group is graded with respect to commutation and anticommutation
with a suitable analogue of the Casimir operator characterized by an affiliation
relation to a von Neumann algebra.
This is used to
obtain an explicit set of generators. Having the dual quantum group the
left regular corepresentation of the  quantum group analogue of the normalizer
of $SU(1,1)$ in $SL(2,\C)$ is decomposed into irreducible corepresentations.
Upon restricting the irreducible corepresentations to $\su$-representation
one finds combinations of the positive and negative discrete series representations with
the strange series  representations as well as combinations of the principal unitary series
representations.
The detailed analysis of this example involves analysis of special functions of
basic hypergeometric type and, in particular, some results on these special functions
are obtained, which are stated separately.

The paper is split into two parts; the first part gives almost all of the statements
and the results, and the statements in the first part are independent of the
second part. The second part contains the proofs of all the statements.
\end{abstract}

\maketitle
\tableofcontents


\section*{Preamble}

The proofs of the statements in the paper are technical. To enhance
the readability of the paper, the paper is essentially split into
two parts. The first part contains all the statements and can
be read independently from the second part containing the proofs.
Moreover, Section \ref{sec:BHSresults} is independent of the
remainder of the paper, and in Section \ref{sec:BHSresults} we
state explicit results for special functions of basic hypergeometric
type. This section is meant for people interested in special functions.
For the convenience of the reader we have added an index, which includes
references to notations frequently used.

\section{Introduction}

On the one hand, the general theory of quantum groups has its roots
in approaches in the axiomatizations of generalizations of
groups such that the Pontryagin-van Kampen duality for locally compact abelian groups
extends to this
wider class. On the other hand, a large class of explicit and
interesting quantum groups arose from various cases, e.g. $R$-matrices as solutions
of the Yang-Baxter equation and the
RTF-formalism. For the quantum groups related to compact groups
arising in this way the duality is formulated on the level of
Hopf algebra duality between the quantized function algebra and the
quantized enveloping algebra.
For the historic development of the general theory for locally compact quantum
groups we refer to the papers --especially the introductions--
\cite{Kust}, \cite{KustVASENS}, \cite{MasuNW}, and books \cite{EnocS}, \cite{Timm}.
For the development of quantum groups involving the Yang-Baxter equation and
the RTF-formalism we refer to the books \cite{CharP}, \cite{EtinS}, \cite{Kass}.
It has turned out that many of the examples arising in this way fit into
the general theory of quantum groups, especially for the quantum group analogues
of compact groups. These quantum groups can usually be analyzed in an algebraic
way. For quantum group analogues of non-compact groups the situation is not
so clear.

As it turns out the Hopf algebra arising from the standard $R$-matrix for
$SL(2,\C)$ has three different $\ast$-structures \cite{MasuMNNSU}, and
we consider a $\ast$-structure making the Hopf algebra into a Hopf $\ast$-algebra as the
choice of an appropriate real form. The compact case, corresponding to the quantum
group analogue of $SU(2)$, has been studied extensively, see
\cite{CharP}, \cite{EtinS}, \cite{Kass} and references given there. This is also the basic
example of a quantum group having an
intimate link with special functions of
basic hypergeometric type \cite{GaspR}, see \cite{CharP}, \cite{EtinS}, \cite{Kass}, \cite{SchmK},
as well as \cite{KoelAAM}.
Then there is the non-compact case associated to the non-compact group
$SU(1,1)$ and a non-compact case associated to the group $SL(2,\R)$.
Although $SU(1,1)\cong SL(2,\R)$ as Lie groups the corresponding Hopf $\ast$-algebras
for the deformed case are different. The subject of this paper is the Hopf $\ast$-algebra associated
to the group $SU(1,1)$, in which the deformation parameter $q$ is real.
For the case of the Hopf $\ast$-algebra associated to $SL(2,\R)$ the deformation
parameter is on the unit circle, and the situation changes dramatically,
see \cite{CharP} and for recent progress on the level of associated special functions
see van de Bult \cite{vdBult}.

In this paper we focus on the Hopf $\ast$-algebra associated to $SU(1,1)$,
which is recalled in Section \ref{sec:quantumSU11description}. We also
recall that Woronowicz \cite{WoroCMP91} showed that there was no way to extend
the comultiplication of this Hopf $\ast$-algebra in analytic way, i.e.
to the level of operators on Hilbert spaces.
Based on work of Korogodsky \cite{Koro} and Woronowicz \cite{Woropp} it
is possible to show that there exists a quantum group analogue of the
normalizer of $SU(1,1)$ in $SL(2,\C)$ in the context of the definition
of Kustermans and Vaes \cite{KustVASENS}, \cite{KustVMS} (see also
\cite{Kust} and \cite{VanDpp} for an introduction) on the level of a von Neumann algebraic
quantum group. This has been shown in \cite{KoelKCMP}, where special
functions of basic hypergeometric type proved to be essential in the
construction. The purpose of this paper is to give an explicit description
of the dual quantum group and to decompose the left regular corepresentation
into irreducible corepresentations for this explicit quantum group.
In the decomposition of the left regular corepresentation we see
the analogy with the group case, since in the left regular
representation of the group $SU(1,1)\cong SL(2,\R)$ only discrete series representations and
principal unitary series occur. However, in the quantum group case the discrete series are no
longer split up into a positive discrete series and a negative
discrete series.

Some of these results have been announced in \cite{KoelK-LN}, and in this paper
we give full proofs of these statements. This paper can be read independently
from \cite{KoelK-LN}.
After browsing the
paper it should be clear to the casual reader that making the general
quantum group machinery work for this specific case is a very technical
business. However, we believe that this is worthwhile since
$SU(1,1)\cong SL(2,\R)$ is one of the most important
non-compact Lie groups \cite{vDijk}, \cite{HoweT}, \cite{Knap}, \cite{KoorEM}, \cite{Lang},
\cite{Sugi} and any reasonable quantum group theory has to have the
example of a quantum group analogue of $SU(1,1)$. Moreover, we hope that
understanding this example may also lead to other non-trivial examples
of non-compact quantum groups and related quantum homogeneous spaces, such
as quantum group analogues of $SU(n,1)$, and related homogeneous spaces
$SU(n,1)/S(U(n)\times U(1))$.
Moreover, in the operator algebra context $K$-theory is available, and the
first step in this direction is taken \cite{Diep}. In particular, one
can ask for a $K$-theoretic approach to discrete series representations
in this setting.
We expect that the
link with special functions can lead to new and deep results in the
theory of special functions, and we have included some
highly non-trivial examples
in Section \ref{sec:BHSresults}, but we expect that the relation is
deeper and not yet fully exploited. E.g. the link with twisted
primitive elements, suitable Cartan type decompositions and
(associated) spherical functions and corresponding transform as indicated in
\cite{KoelSRIMS} can be studied from an operator algebraic point of view,
see also \cite[Ch.~3]{DeSm} for a more general study of the Plancherel measure
in this context.
Having the decomposition of the left
regular representation available it is now also natural to consider other questions, e.g.
can we decompose tensor products, describe the intertwiners in terms
of special functions, etc? We are confident that the interpretation
of discrete series representations in this context gives a solution to
indeterminacy problems related to certain tensor product decompositions
of infinite dimensional representations of $\su$, see
cf. \cite{GroeAAM}, \cite{GroeRJ} for cases where the indeterminacy is absent.

We now describe the contents of the paper.
In Sections \ref{sec:vNalgquantumgroups}-\ref{sec:quantumSU11description}
we recall the necessary background on general locally compact quantum groups
in the von Neumann algebraic setting and the specific example that we study.
In Section \ref{sec:vNalgquantumgroups}
we recall the  Kustermans-Vaes approach to locally compact quantum groups on the
von Neumann algebraic level, which is the framework for this paper, and
Section \ref{sec:vNalgquantumgroups} is mainly based on \cite{KustVMS}.
Next in Section \ref{sec:quantumSU11description} we give a concise description of the
Hopf $\ast$-algebra and the quantum group analogue of the normalizer of
$SU(1,1)$ in $SL(2,\C)$ in the context of Section \ref{sec:vNalgquantumgroups}.
Section \ref{sec:quantumSU11description} is based on \cite{KoelKCMP}.
In Section \ref{sec:VonNeumanAlgDualQuantumGroup} we give an explicit description
of the dual quantum group in the case of the quantum group analogue of the normalizer of
$SU(1,1)$ in $SL(2,\C)$. In particular
we show that the generators of the quantized universal enveloping algebra $\su$
can be realized as unbounded operators affiliated to the von Neumann algebra
of the dual quantum group. We discuss how the (suitable extension of the) Casimir
operator can be used to find sufficiently many generators. It turns out that
the self-adjoint extension of the algebraically defined symmetric, but not essentially
self-adjoint, operator is characterized by affiliation to the von Neumann algebra
for the dual quantum group. We also show that comultiplication defined on
the von Neumann algebraic  group coincides with the
 comultiplication of the Hopf $\ast$-algebra $\su$.
In Section \ref{sec:decompleftregularcorep} the decomposition of the
left regular corepresentation is presented, it involves
analogues of the principal unitary series
representations and discrete series representations.
In Section \ref{sec:BHSresults}
we collect some interesting new (as far as we are aware) results for special
functions of basic hypergeometric type which are  byproducts of the
approach taken. In particular, the results discussed in Section \ref{sec:BHSresults} can be
read independently by someone only interested in special functions, but
the proofs are dependent on the rest of the paper.
Sections \ref{sec:VonNeumanAlgDualQuantumGroup}--\ref{sec:BHSresults} describe the
results of this paper in detail and form the core of the paper.
All the main results and its background can be obtained from
Sections \ref{sec:vNalgquantumgroups}--\ref{sec:BHSresults}.
The gist of the main results are obtained when reading only this part
of the paper, which can also be viewed as a very extended introduction.
The proofs of all statements in these sections are given in the remainder of the paper
consisting of Sections \ref{sec:extensionsgeneratorssu}--\ref{sec:specialfunctions}.
In Appendix \ref{app:notationterminology} we recall some notation and terminology
of von Neumann algebras, whereas we recall the necessary details of the special
functions involved in Appendix \ref{app:specialfunctions}. In Appendix
\ref{app:jacobi} we discuss a specific example of a Jacobi operator, whereas
Appendix \ref{app:proofsoflemmas} contains nitty-gritty proofs of some
intermediate lemmas.


\section{Von Neumann algebraic quantum groups}\label{sec:vNalgquantumgroups}

In this section we recall the definition of the von Neumann
algebraic quantum groups and related results. So we work
with a theory on the quantum group
analogue of locally compact groups in the realm
of operator algebras. We summarize the main features, and
we discuss the group case for a unimodular Lie group $G$.
The proofs of all statements
can be found in the papers \cite{KustVASENS}, \cite{KustVMS}
by Kustermans and Vaes. Introductory texts on this subject
are \cite{Kust}, \cite{VanDpp}, see also \cite{Timm}.
In Section  \ref{sec:quantumSU11description} we describe the example
we study, namely the von Neumann algebraic quantum group
associated to the normalizer of $SU(1,1)$ in $SL(2,\C)$, which
is essentially recalling the results of \cite{KoelKCMP}.

\begin{defn}\label{def:vNalgquantumgroup}
Consider a von Neumann algebra $M$ together with a unital
normal $\ast$-ho\-mo\-mor\-phism $\De\colon M\to M\otimes M$
(the comultiplication)\index{comultiplication} such that
$(\De\ot\Id)\De=(\Id\ot\De)\De$ (coassociativity).
Moreover, if there exist two normal semi-finite faithful weights
$\vp$, $\psi$ on $M$ such that
\[
\begin{split}
\vp\bigl((\om\ot\Id)\De(x)\bigr)&\, =\, \vp(x)\om(1), \qquad
\forall \ \om\in M^+_*,\, \forall\ x\in \cM^+_\vp
\qquad \text{(left invariance),}\\
\psi\bigl((\Id\ot\om)\De(x)\bigr)&\, =\, \psi(x)\om(1), \qquad
\forall \ \om\in M^+_*,\, \forall\ x\in \cM^+_\psi
\qquad \text{(right invariance),}
\end{split}
\]
then $(M,\De)$ is a von Neumann algebraic quantum group.\index{quantum group}
\index{von Neumann algebraic quantum group}
\end{defn}

Note that we suppress $\vp$ and $\psi$ from the notation
$(M,\De)$ for a von Neumann algebraic quantum group.

The notation in Definition \ref{def:vNalgquantumgroup}
follows the standard notation for weights,
tensor products and preduals, see e.g.
\cite{KadiR}, \cite{Take}, which are
briefly recalled in Appendix \ref{app:notationterminology}. We recall
here the basic constructions for weights, since
the related modular objects play an important role, see \cite{Take}.
In particular,
a weight\index{weight} is a map $\vp\colon M_+\to [0,\infty]$,
$M_+$\index{M@$M_+$} being the cone of positive elements in $M$,
such that $\vp(x+y)=\vp(x)+\vp(y)$ and
$\vp(\la x) = \la\vp(x)$ for $\la\geq 0$.
Then $\cM^+ =\{ x\in M_+ \mid \vp(x)<\infty\}$,
 $\cN= \{ x\in M\mid \vp(x^\ast x)<\infty\}$ is a left ideal
and $\cM$ is the linear span of $\cM^+$ in $M$.
Then $\cM =\cN^\ast\cN$, and $\vp$ extends
uniquely to $\cM$. The weight $\vp$ is faithful\index{weight!faithful} if
$\vp(x)\not=0$ for all non-zero $x\in M_+$. The weight
$\vp$ is semifinite\index{weight!semifinite} if $\cM$ is $\si$-strong-$\ast$
dense in $M_+$ or $(\cM)''=M$. The weight $\vp$
is normal\index{weight!normal}  if $\vp(\sup_{\la} x_\la)= \sup_{\la} \vp(x_\la)$ for
any bounded increasing net $\{ x_\la\}_{\la\in \Lambda}$ in $M_+$,
and this can be reformulated in various different ways.
Normal semifinite faithful weight is abbreviated to
nsf weight.\index{weight!nsf}

A GNS-construction\index{GNS-construction}\index{weight!GNS-construction} for a weight is similar to a GNS-construction
for a state. A GNS-construction for a weight $\vp$
is a triple $(\cH,\pi,\La)$  consisting of a Hilbert
space $\cH$, a $\ast$-ho\-mo\-mor\-phism $\pi\colon M\to B(\cH)$
and a linear map $\La\colon \cN\to \cH$ such that
\begin{enumerate}
\item $\La(\cN)$ is dense in $\cH$;
\item $\langle \La(a), \La(b) \rangle = \vp(b^\ast a)$ for all $a,b\in \cN$;
\item $\pi(x)\, \La(a) = \La (xa)$ for all $x\in M$, $a\in\cN$.
\end{enumerate}
In case $\vp$ is a nsf weight, the representation $\pi$
is injective, normal and nondegenerate, and $\La$ is
closed for the $\si$-strong-$\ast$ topology on $M$ and
the norm topology of $\cH$. In case we want to stress
the dependence on the weight $\vp$ we use the notation
$\cM^+_\vp$, $\cM_\vp$, $\cN_\vp$, $\cH_\vp$, $\pi_\vp$, $\La_\vp$ as
in Definition \ref{def:vNalgquantumgroup}.

The weight $\vp$, respectively $\psi$, in
Definition \ref{def:vNalgquantumgroup} is the
left, respectively right, Haar weight\index{Haar weight} for the
von Neumann algebraic quantum group $(M,\De)$.
It can be shown that the left and right Haar weights
are unique up to a constant.

In this paper, we mainly deal with the von Neumann
algebra $M$ and the corresponding von Neumann
algebra $\hat M$ for the dual von Neumann
algebraic quantum group, see Theorem \ref{thm:duallocallycompactquantumgroup},
and the weights do not
play a big role, but the associated modular operator,
modular conjugation and modular automorphism
group plays an important role. In order to
obtain the properties of these operators,
consider the GNS-representation for $\vp$ and the
antilinear map from $\La(\cN\cap \cN^\ast)\subset \cH$
to itself defined by $\La(x)\mapsto \La(x^\ast)$.
This map has polar decomposition $J\nabla^{1/2}$,
where $J\colon \cH\to \cH$ is an antilinear
isometry and $J^2=\Id$. $J$ is the modular
conjugation\index{modular conjugation} and the (generally unbounded) self-adjoint
operator $\nabla$ is the modular operator\index{modular operator} associated
with the weight $\vp$. Then
\begin{equation}\label{eq:JMJisMcommutant}
J \pi(M) J = \pi(M)', \qquad \nabla^{it} \pi(M) \nabla^{-it} = \pi(M),
\ \  t\in\R.
\end{equation}
Here $\pi(M)' = \{ x\in B(\cH) \mid xy=yx \ \forall\, y\in \pi(M)\}$ is
the commutant\index{commutant} of $\pi(M)$.
For a nsf weight $\pi$ is faithful, and then we identify
$\pi(M)$ with $M$, so that \eqref{eq:JMJisMcommutant}
gives $J\, M\, J = M'$, $\nabla^{it} M \nabla^{-it} = M$, $t\in\R$.
Then $\si_t(x)= \nabla^{it} x \nabla^{-it}$, $x\in M$, defines
a strongly continuous one-parameter group $\si$ of $\ast$-automorphisms
on $M$ for the nsf weight $\vp$. It is the
modular automorphism group\index{modular automorphism group} $\si=\si^\vp$ for the nsf weight $\vp$.

Having the GNS-construction for the left invariant
nsf weight $\vp$ we define
\[
W^\ast \bigl( \La(a)\ot \La(b) \bigr) =
\bigl( \La\ot\La\bigr) \bigl( \De(b)(a\ot 1)\bigr),
\]
then $W$ is a unitary operator on $\cH\ot\cH$, which
is known as the multiplicative unitary\index{multiplicative unitary} and is instrumental
in the development of locally compact quantum groups,
as pointed out initially in \cite{BaajS}.
Identifying
$M$ with $\pi(M)$, we obtain
$\De(x)= W^\ast(1\ot x)W$ for all $x\in M$, so that
the multiplicative unitary implements the
comultiplication.\index{comultiplication}

\begin{remark}\label{rmk:classicalcase1}
To see how groups are included in this definition take
a group $G$, which for convenience we assume to be
a unimodular Lie group. Then the von Neumann algebra $M\cong
L^\infty(G)$ is acting by multiplication operators
on the Hilbert space $L^2(G)$, defined with
respect to the left Haar measure $d_lg$.
So we consider $M$ as a subalgebra of $B(L^2(G))$.
Then
$\vp(f) = \int_G f(g)\, d_lg$ for $f\in L^\infty(G)\cap L^2(G)=\cM$,
and the corresponding GNS-construction
of $\vp$ is $(L^2(G), \Id, \La)$ where
$\La\colon \cN_\vp=L^2(G)\cap L^\infty(G) \to L^2(G)$, $x\mapsto x$.
In this case the predual is
$M_\ast= L^1(G) \subset M^\ast$ by considering
$L^1\ni f \mapsto (L^\infty(G)\ni x\mapsto \int_G f(g)x(g)\, d_l g)$
and $M \cong (M_\ast)^\ast$ and the $\si$-weak topology
is the $\si(M,M_\ast)$-topology. For $f, h\in L^2(G)$
a normal functional $\om_{f,h}$ is defined as
the matrix element $\om_{f,h}(x) =\langle xf,h\rangle
= \int_G x(g)f(g)\overline{h(g)}d_lg$.
In this case the multiplicative unitary $W$ is
\[
\begin{split}
&W \colon L^2(G) \ot L^2(G) \cong L^2(G\times G) \to  L^2(G) \ot L^2(G) \cong L^2(G\times G) \\
&\bigl(W^\ast f\bigr)(g,h) = f(g,gh), \qquad \bigl(W f\bigr)(g,h) = f(g,g^{-1}h).
\end{split}
\]
Particular to the unimodular Lie group case is that the antipode
$S\colon M\to M$, $(Sx)(g)=x(g^{-1})$ is bounded, but
in the general case it is not. To indicate how the antipode can be
obtained from the invariant weight in the general case note that
\[
\int_G x(h^{-1}g)\, y(g)\, d_lg = \int_G x(g)\, y(hg)\, d_lg
\]
so that
\[
S \colon (\Id\ot \vp)\bigl(\De(x)(1\ot y)\bigr) \mapsto
(\Id\ot \vp)\bigl((1\ot x)\De(y)\bigr).
\]
\end{remark}

This in particular gives the key to defining the antipode
$S$ on a von Neumann algebraic quantum group as an
unbounded operator. A basic result is a polar decomposition
of the antipode. To be precise, there exists a unique
$\ast$-anti-automorphism $R\colon M\to M$ and a unique
strongly continuous one-parameter group of $\ast$-automorphisms
$\tau\colon \R \to \Aut(M)$ satisfying
\begin{equation}\label{eq:Spolardecomposition}
S=R\tau_{-i/2}, \qquad R^2=\Id, \qquad \tau_t R= R \tau_t \ \   \forall t\in \R.
\end{equation}
$R$ is known as the unitary antipode,\index{unitary antipode} and
$\tau$ the scaling group.\index{scaling group}
One can show that $\vp R$ is a right invariant nsf weight, and
one can make the choice $\psi =\vp R$ for the right Haar weight,
which we assume from now on.

An interesting result in the theory of locally compact quantum
group is duality, see Theorems \ref{thm:duallocallycompactquantumgroup}
and \ref{thm:doubledualisback}.
For the dual locally compact
quantum group we have
\begin{equation}\label{eq:defdualhatM}\index{M@$\hat{M}$}
\hat{M} = \overline{\{ (\om\ot\Id)(W)\mid \om \in B(\cH)_\ast\}} \subset B(\cH),
\end{equation}
where the closure is with respect to the $\si$-strong-$\ast$ topology and
$\cH$ is the GNS-space for the left invariant weight $\vp$.

\begin{thm}[\cite{KustVMS}]\label{thm:duallocallycompactquantumgroup}
$\hat{M}$ is a von Neumann algebra acting on $\cH$, and there
exists a unique normal injective $\ast$-homomorphism
$\hat{\De} \colon \hat{M} \to \hat{M}\ot \hat{M}$,
$\hat{\De}(x) = \Sigma \, W  (x\ot 1) W^\ast\, \Sigma$ for
$x\in \hat{M}$. Moreover, $(\hat{M},\hat{\De})$ is a locally
compact quantum group; the Pontryagin dual of $(M,\De)$\index{Pontryagin dual}
or the dual locally compact quantum group.\index{dual locally compact quantum group}
\end{thm}

Here $\Sigma\colon\cH\ot\cH\to\cH\ot\cH$\index{S@$\Sigma$} denotes the
flip operator $\Sigma \colon a\ot b\mapsto b\ot a$.

In particular, the dual locally compact quantum group
comes with two nsf weights $\hat{\vp}$ and
$\hat{\psi}$. Let $\hat{J}$ and $\hat{\nabla}$ be the
modular conjugation and modular group for the left invariant
dual weight $\hat{\vp}$. Then, in the realization of $\hat{M}$
on the GNS-space $\cH$ we have for the unitary antipode\index{unitary antipode}
as in \eqref{eq:Spolardecomposition} the relations
\begin{equation}\label{eq:actionunitaryantipode}\index{unitary antipode}
R(x) = \hat{J}\, x^\ast \, \hat{J}, \qquad \forall\, x\in M,  \qquad
\hat{R}(x) = J\, x^\ast\, J, \qquad \forall\, x\in \hat{M}.
\end{equation}

It follows from Theorem \ref{thm:duallocallycompactquantumgroup}
that the multiplicative unitary for the dual von Neumann
algebraic quantum group is $\hat{W}= \Sigma W^\ast \Sigma$.
The multiplicative unitary\index{multiplicative unitary} $W\in M\ot \hat{M}$, where we
consider $M$ acting on the GNS-space $H$ for the left invariant
weight $\vp$. Moreover,
\begin{equation}\label{eq:hatJotJWhatJotJisWster}
\bigl( \hat{J}\ot J\bigr)\, W \, \bigl( \hat{J}\ot J\bigr) \,=\, W^\ast.
\end{equation}

A unitary corepresentation\index{corepresentation}\index{corepresentation!unitary}
$U$ of a von Neumann algebraic quantum
group on a Hilbert space $H$ is a unitary element $U\in M\ot B(H)$
such that $(\De\ot\Id)(U)=U_{13}U_{23}\in M\ot M\ot B(H)$, where
the standard leg-numbering is used in the right hand side. In particular,
it follows from the pentagonal identity\index{pentagonal identity}
$W_{12}W_{13}W_{23}=W_{23}W_{12}$
and $\De(x)=W^\ast(1\ot x)W$ that the multiplicative unitary
$W$ defines a unitary corepresentation of $M$ on the GNS-space.
This corepresentation is the analogue of the left regular representation
of a Lie group $G$ on the Hilbert space $L^2(G)$.
A closed subspace $L\subseteq H$ for the unitary corepresentation is an
invariant subspace\index{invariant subspace}
if $(\om\ot \Id)(U)$ preserves $L$ for all $\om\in M_\ast$.
In particular, it follows from Definition
\ref{thm:duallocallycompactquantumgroup} that an invariant subspace
is precisely the closed subspace invariant for the action of the
dual von Neumann algebra $\hat{M}$, since it is generated by
$(\om\ot \Id)(W)$, $\om\in M_\ast$. A unitary corepresentation $U$
in the Hilbert space $H$
is irreducible\index{corepresentation!irreducible} if there are only trivial (i.e. equal to
$\{0\}$ or the whole Hilbert space $H$)  invariant subspaces.
In particular, $\{ (\om\ot\Id)(U)\mid \om \in M_\ast\}''= B(H)$ implies
that $U$ is an irreducible unitary corepresentation.

The nice feature of the von Neumann algebraic quantum groups is the
following theorem, due to Kustermans and Vaes \cite{KustVASENS}, \cite{KustVMS},
which is a far-reaching generalization of the Pontryagin-van Kampen duality.

\begin{thm}\label{thm:doubledualisback}
$(\hat{\hat{M}},\hat{\hat{\De}})= (M,\De)$.
\end{thm}

\begin{remark}\label{rmk:classicalcase2} We finish by discussing some
of the above in the case of a unimodular Lie group $G$ continuing
Remark \ref{rmk:classicalcase1}.
Identify $\om\in M_\ast$
with a function $k\in L^1(G)$, then $(\om\ot\Id)(W)\in B(L^2(G))$
is the convolution operator $f\mapsto k\ast f$,
$(k\ast f)(g)= \int_G k(s) f(s^{-1}g)\, d_ls$. Then the product
in $\hat{M}$ corresponds to the convolution product, and the
dual left invariant weight on such a convolution
operator is evaluation of the kernel at the identity of the group $G$.
To see that the corepresentation associated to the multiplicative
unitary corresponds to the left regular representation, say $\la$, we
check
\[
\begin{split}
&\Bigl( \bigl( (\om_{f_1,f_2} \ot \Id )(W)\bigr)\, f_3\Bigr) (h) =
\langle W(f_1\ot f_3), f_2\rangle_1\,  (h)  =\,
\int_G \bigl( W(f_1\ot f_3)\bigr) (g,h) \bar f_2(g)\, d_lg  \\
=&\, \int_G f_1(g)\, \bar f_2(g) \, f_3(g^{-1}h)  \, d_lg = \bigl( \la(f_1 \bar f_2)\, f_3\bigr)(h).
\end{split}
\]
Since the normal functional $\om_{f_1,f_2}$ corresponds to $f_1\bar f_2\in L^1(G)$,
the required result follows.
\end{remark}


\section{The quantum group analogue
of the normalizer of $SU(1,1)$ in $SL(2,\C)$}\label{sec:quantumSU11description}

In this section we recall the von Neumann algebraic quantum group
for which we calculate the dual von Neumann algebraic quantum group,
and for which we decompose the left regular corepresentation.
Except for the last paragraph, all the results described are
taken from \cite{KoelKCMP}.

The Lie group $SU(1,1)\cong SL(2,\R)$ is one of the most important
non-compact Lie groups. On the level of Hopf algebras, a classification
of real forms of the quantized universal enveloping algebra
$U_q(\mathfrak{sl}(2,\C))$ results in three different real forms,
i.e. Hopf $\ast$-algebras;
the compact case $U_q(\mathfrak{su}(2))$ for $0<q<1$, which is extensively
studied \cite{CharP}, \cite{EtinS}, \cite{Kass}, \cite{Lusz};
the non-compact case $U_q(\mathfrak{sl}(2,\R))$ with $q$ on the unit
circle, see e.g. the previously mentioned books and \cite{vdBult};
and the non-compact case $\su$ for $0<q<1$. In these cases there
is a related dual Hopf $\ast$-algebra which is a deformation of the
algebra of polynomials on the related group. We refer to the
books  \cite{CharP}, \cite{EtinS}, \cite{Kass}, \cite{Lusz}, as well
as to \cite{BurbK}, \cite{MasuMNNSU}, \cite{KoelSRIMS} for more
information and references. However, as Woronowicz \cite{WoroCMP91}
proved, there is no C$^\ast$-algebra interpretation for the related
Hopf $\ast$-algebra with a well-defined comultiplication. Later,
Korogodsky \cite{Koro} indicates how the ill-defined comultiplication
could be avoided.
With the introduction of the theory of von Neumann algebraic quantum
\cite{KustVASENS}, \cite{KustVMS} it is natural to ask whether or not this
important example can be incorporated in the theory of von Neumann
algebraic quantum groups. As it turns out the answer is yes, and the key
to the solution is using special functions.

All statements of this section are proved in
\cite{KoelKCMP}, except \eqref{eq:expressionhatJonfmpt} for which
a direct proof is given.

Throughout the paper, we fix a number $0 < q < 1$. Define $\cA_q$\index{A@$\cA_q$} to be the unital
$\ast$-algebra generated by elements
$\bal$, $\bga$ and $\bolde$ and relations
\begin{equation} \label{eq:generatorsrelations}
\begin{split}
&\bal^\dag \bal - \bga^\dag \bga = \bolde \qquad  \bal \bal^\dag - q^2 \, \bga^\dag \bga = \bolde
\qquad \bga^\dag \bga = \bga\, \bga^\dag
\\  &\bal \,\bga = q \, \bga\, \bal \qquad  \bal \,\bga^\dag = q \, \bga^\dag \bal   \\
&\bolde^\dag = \bolde  \qquad  \bolde^2 = 1 \qquad   \bal \, \bolde = \bolde \,\bal \qquad \bga \, \bolde = \bolde \,\bga
\end{split}
\end{equation}
where $\dag$ denotes the $\ast$-operation on $\cA_q$ (in order to distinguish this kind of adjoint with the adjoints of possibly unbounded operators in Hilbert spaces).
In case we take $\bolde=1$ in \eqref{eq:generatorsrelations} we obtain the
$\ast$-algebra which is usually associated with the algebra of polynomials
on the quantum analogue of $SU(1,1)$, see \cite{MasuMNNSU},
\cite{KoroV}, \cite{KoelSRIMS}.
The additional generator $\bolde$ has been introduced by Korogodsky \cite{Koro}.

For completeness we give the Hopf $\ast$-algebra\index{Hopf $\ast$-algebra $\cA_q$} structure on
$\cA_q$. By $\cA_q \odot \cA_q$ we denote the algebraic tensor product.
There exists a unique unital $\ast$-homomorphism $\bDe : \cA_q \rightarrow \cA_q \odot \cA_q$ such that
\begin{equation}\label{eq:Hopfalgebracomultiplication}
\begin{split}
&\bDe(\bal)  =  \bal \ot \bal + q \, (\bolde \, \bga^\dag) \ot \bga \qquad
\bDe(\bga) =  \bga \ot \bal + (\bolde \bal^\dag) \ot \bga \qquad
\bDe(\bolde)  =  \bolde \ot \bolde
 \end{split}
\end{equation}
The counit $\bep\colon \cA_q\to \cA_q$ and antipode $\bS\colon \cA_q\to \cA_q$ are given by
\begin{equation}\label{eq:Hopfalgebracounitantipode}
\begin{split}
& \bS(\bal) = \bolde\,\bal^\dag \qquad  \bS(\bal^\dag) = \bolde\, \bal \qquad
\bS(\bga) = - q \, \bga \qquad \bS(\bga^\dag) = - \frac{1}{q}\,\bga^\dag
\qquad \bS(\bolde) = \bolde \\
& \bep(\bal) = 1 \qquad \bep(\bga) = 0 \qquad \bep(\bolde) = 1
\end{split}
\end{equation}
This makes $\cA_q$ into a Hopf $\ast$-algebra.

To see that for $q=1$ we obtain the Hopf $\ast$-algebra of polynomials
on the group $SU(1,1)$ (when restricting to the sub-Hopf $\ast$-algebra
$\cA_q^1$ given by $\bolde =1$) and on the normalizer $N_{SL(2,\C)}(SU(1,1))$ of
$SU(1,1)$ in $SL(2,\C)$ we recall
\[
\begin{split}
SU(1,1) \,= \bigg\{ g\in SL(2,\C) \mid g^\ast J g = J = \begin{pmatrix} 1 & 0\\ 0 &-1  \end{pmatrix} \bigg\}
\, = \bigg\{ \begin{pmatrix} a & c \\ \bar c & \bar a \end{pmatrix} \mid a,c \in\C, \
|a|^2-|c|^2= 1 \bigg\}
\end{split}
\]
\index{S@$SU(1,1)$}
and we let $\bal(g)=a$, $\bga(g)=c$. Similarly,
\[
\begin{split}
&\, N_{SL(2,\C)}(SU(1,1)) \, = \{ g\in SL(2,\C) \mid g^\ast J g = \pm J\} \\
&\, = \bigg\{ \begin{pmatrix} a & c \\ \ep \bar c & \ep \bar a \end{pmatrix} \mid a,c \in\C, \
\ep \in \{\pm 1\}, |a|^2-|c|^2= \ep \bigg\}
= SU(1,1) \cup  SU(1,1)\, \begin{pmatrix}0&-1 \\1&0\end{pmatrix}
\end{split}
\]
\index{N@$N_{SL(2,\C)}(SU(1,1))$}
and we put $\bal(g)=a$, $\bga(g)=c$, $\bolde(g)=\ep$.

The following result by Woronowicz \cite{WoroCMP91} states that one cannot
expect a suitable quantum group on an operator algebra level arising from
Hopf $\ast$-algebra $\cA_q^1$ (i.e. with $\bolde=1$ in \eqref{eq:generatorsrelations}).
In Theorem \ref{thm:Woronowicznogothm} a representation of $\cA_q^1$ consists of
two closed operators $\al$ and $\ga$ acting in a Hilbert space $H$ such that
the domains of $\al$, $\ga$, $\al^\ast$, $\ga^\ast$ are equal, say $D$, and
such that the relations in \eqref{eq:generatorsrelations} are represented in
a weak sense, e.g. $\bal \,\bga = q \, \bga\, \bal$ is translated by
$\langle \ga v,\al^\ast w\rangle = q \langle \al v, \ga^\ast w\rangle$
for all $v,w \in D$, etc.

\begin{thm}[Woronowicz \cite{WoroCMP91}]\label{thm:Woronowicznogothm}
For $(\al^1,\ga^1)$, resp. $(\al^2,\ga^2)$, closed operators on an
infinite dimensional Hilbert space $H^1$, resp. $H^2$, representing the relations, there exist
no closed operators $\al$, $\ga$ acting on $H^1\ot H^2$ representing the
relations and extending
$\al^1 \ot \al^2 + q \, (\ga^1)^\ast \ot \ga^2$,
$\ga^1 \ot \al^2 +  (\al^1)^\ast \ot \ga^2$, such that $\al^\ast$, $\ga^\ast$ extend
$(\al^1)^\ast \ot (\al^2)^\ast + q \, \ga^1 \ot (\ga^2)^\ast$,
$(\ga^1)^\ast \ot (\al^2)^\ast +  \al^1\ot (\ga^2)^\ast$.
\end{thm}

Theorem \ref{thm:Woronowicznogothm} is a negative result, but Korogodsky
\cite{Koro} pointed out how to proceed by adding the additional
generator $\bolde$.

It is not hard to represent the commutation relations
\eqref{eq:generatorsrelations} by unbounded operators acting on
the Hilbert space $H=L^2(\T)\oplus L^2(I_q)$, where
$I_q = -q^\N \cup q^\Z$\index{I@$I_q$} and equipped with the counting measure.
Here $\T=\{z\in\C \mid |z|=1\}$\index{T@$\T$}
denotes the unit circle, $\N=\{1,2,\cdots\}$\index{N@$\N$} and $\NN= \{0,1,2,\cdots\}$\index{N@$\NN$}.
If $p \in I_q$, we define $\de_p(x) = \de_{x,p}$ for all $x \in I_q$,
so the family $\{ \de_p \mid p \in I_q\}$ is the natural orthonormal basis of
$L^2(I_q)$. For $L^2(\T)$ we have the natural orthonormal basis $\{ \ze^m \mid m \in \Z\}$,
with $\ze$ the identity function on $\T$. Then $\{ \ze^m \ot \de_p \mid
m\in \Z,\, p\in I_q\}$ is an orthonormal basis for $H$.
Define linear operators $\al_0$, $\gamma_0$, $e_0$ on the space $E$ of
finite linear combinations of $\ze^m\ot \de_p$ by
\begin{equation}\label{eq:formalrepsal0ga0}
\begin{split}
\al_0(\ze^m \ot \de_p ) &=  \sqrt{\sgn(p)+p^{-2}}\,\,\ze^m \ot \de_{qp},  \\
\ga_0(\ze^m \ot \de_p)  &=   p^{-1}\,\,\ze^{m+1} \ot  \de_p,  \qquad
 e_0(\ze^m \ot \de_p)  =  \sgn(p) \,\, \ze^m \ot \de_p.
\end{split}
\end{equation}
for all $p \in I_q$, $m \in \Z$. The actions of $\al_0^\dag$ and
$\ga_0^\dag$ on $E$ can be given in a similar fashion by taking formal adjoints, and
these satisfy the relations \eqref{eq:generatorsrelations}, and give a faithful representation
of the algebra $\cA_q$.
Then \cite[\S 2]{KoelKCMP} the operators
$\al_0$, $\ga_0$ are closable with densely defined closed unbounded
operators $\al$, $\ga$ as their closure.
Moreover, the adjoints $\al^\ast$ and $\ga^\ast$
are the closures of $\al_0^\dag$, $\ga_0^\dag$. Let $e$ be the closure
of $e_0$, then $e$ is a bounded linear self-adjoint operator on $H$.
As discussed by Woronowicz \cite{Woropp} and in \cite{KoelKCMP}, it
is not sufficient to consider the von Neumann algebra generated
by $\al$, $\ga$ and $e$ in order to obtain a well-defined
comultiplication. Consider the linear map $T\colon \ze^m\ot \de_p
\mapsto \ze^m\ot \de_{-p}$, $T\in B(H)$, where we take $\de_p=0$
in case $p\notin I_q$, and let $u$ be its partial isometry.

\begin{defn}\label{def:vNalgMinBH}
$M$ is the von Neumann algebra in $B(H)$ generated by $\al$, $\ga$, $e$ and $u$.
\end{defn}

By definition, see Appendix \ref{ssecA:affiliationandgenerators}, $\al$ and $\ga$
are affiliated to $M$.

It can be shown \cite[Lemma~2.4 (3)]{KoelKCMP} that $M= L^\infty(\T)\ot B(L^2(I_q))$.
We define the operators
\begin{equation*}
\Phi(m,p,t) \colon \zeta^r\otimes \de_x \mapsto  \de_{xt} \, \zeta^{m+r}\otimes
\de_p, \qquad m,r\in \Z, \, p,t,x\in I_q.
\end{equation*}
A straightforward calculation gives
\begin{equation*}
\begin{split}
\Phi(m_1,p_1,t_1)\, \Phi(m_2,p_2,t_2) = \de_{p_2, t_1} \, \Phi(m_1+m_2, p_1, t_2),
\qquad \Phi(m,p,t)^\ast = \Phi(-m,p,t)
\end{split}
\end{equation*}
In particular the finite linear span of the operators
$\Phi(m,p,t)$ form a $\si$-weakly dense $\ast$-subalgebra in $M$.

In order to show that $M$ is the von Neumann algebra of a
von Neumann algebraic quantum group we need to define the
comultiplication $\De$ and the left and right invariant nsf weights $\vp$ and $\psi$
such that the requirements of Definition \ref{def:vNalgquantumgroup}
are met. We start with the construction of the left invariant
nsf weight by writing down its GNS-construction.
Define $\Tr = \Tr_{L^\infty(\T)}\ot \Tr_{B(L^2(I_q))}$ on
$M$, where $\Tr_{L^\infty(\T)}$ and $\Tr_{B(L^2(I_q))}$ are
the canonical traces on $L^\infty(\T)$, i.e.
$\Tr_{L^\infty(\T)}(f) = \int_\T f(\ze)\, d\ze$ with
normalization $\Tr_{L^\infty(\T)}(1)=1$, and on
$B(L^2(I_q))$, normalized by $\Tr_{B(L^2(I_q))}(P)=1$
for any rank one orthogonal projection. Note that $\Tr$ is
a tracial weight on $M$ so in particular its
modular group is trivial. For $\Tr$ we have the following
GNS-construction:
\begin{itemize}
\item a Hilbert space $\cK = H \ot L^2(I_q)
= L^2(\T) \ot L^2(I_q) \ot L^2(I_q)$\index{K@$\cK$} equipped with the
orthonormal basis $\{f_{mpt}\mid m\in \Z,\, p,t\in I_q\}$;\index{F@$f_{mpt}$}
\item a unital $\ast$-homomorphism $\pi\colon M \to B(\cK)$,
$\pi(a) = a \ot \Id_{L^2(I_q)}$ for $a\in M$;
\item $\La_{\Tr}\colon \cN_{\Tr} \to \cK$,
$a\mapsto \sum_{p\in I_q} (a\ot \Id_{L^2(I_q))}) f_{0,p,p}$.
\end{itemize}
We define the left invariant nsf weight $\vp$ formally
as $\vp(x) = \Tr(|\ga|\, x \, |\ga|)$ with the operator
$|\ga|$ affiliated to $M$.
We proceed by
defining the set $D$ as the set of elements of $x\in M$
such that $x|\ga|$ extends to a bounded operator on $H$,
denoted by $\overline{x|\ga|}$, and such that
$\overline{x |\ga|}\in \cN_{\Tr}$, and for $x\in D$ we
put $\La(x) = \La_{\Tr}(\overline{x |\ga|})$. The set
$D$ is then a core for the operator $\La$ which is
closable for the $\si$-strong-$\ast$--norm topology.

\begin{defn}\label{def:leftinvweightSU11}
The nsf weight $\vp$ on $M$ is defined by its
GNS-construction $(\cK,\pi,\La)$.
\end{defn}

\begin{remark}\label{rmk:propertiessiJnabla}
From the general theory of nsf weights as recalled
in Section \ref{sec:vNalgquantumgroups} we know that $\vp$ comes with a
modular automorphism group $\si$, a modular
conjugation $J$ and modular operator $\nabla$.
In particular, as established
in \cite[\S 4]{KoelKCMP}, we have:
\begin{itemize}
\item $\si_t(x) = |\ga|^{2it}x|\ga|^{-2it}$ for all $x\in M$, $t\in\R$;
\item $\Phi(m,p,t)\in \cN_\vp$ and $\La(\Phi(m,p,t)) = |t|^{-1} f_{mpt}$;
\item $\Phi(m,p,t)\in \cM_\vp$ and $\vp(\Phi(m,p,t)) = |t|^{-2} \de_{m,0}\de_{p,t}$;
\item $\Phi(m,p,t)$ is analytic 
for $\si$ and $\si_z(\Phi(m,p,t))= |p^{-1}t|^{2iz}
\Phi(m,p,t)$ for all $z\in \C$;
\item $J\, f_{mpt} = f_{-m,t,p}$;
\item $f_{mpt}$ in the domain of $\nabla$ and $\nabla f_{mpt} = |p^{-1}t|^2\, f_{mpt}$.
\end{itemize}
\end{remark}

\begin{remark}\label{rmk:identifyMinGNSspacevp}
Note that in particular we can use $\pi$ to identify
$M\subset B(H)$ with its image $\pi(M)\subset B(\cK)$.
From now on we use this identification, and we
work with $M$ realized as von Neumann algebra
in $B(\cK)$.
\end{remark}

In \cite[\S 4]{KoelKCMP} it is observed that the right invariant
weight
$\psi=\vp$, so it remains to construct the
comultiplication which we give using the
multiplicative unitary $W\in B(\cK\ot\cK)$.
We give an explicit expression for
$W^\ast\in B(\cK\ot \cK)$ in terms
of basic hypergeometric series $a_p$ in
\eqref{eq:Wexplicitinapxy}.
The functions $a_p(\cdot,\cdot)$ are recalled in Definition \ref{def:functionap},
and the unitarity of the multiplicative unitary $W$ is
closely related to orthogonality properties of these
functions $a_p$. Then the comultiplication is given
by, recall Remark \ref{rmk:identifyMinGNSspacevp} that
we view $M\subset B(\cK)$,
\begin{equation}\label{eq:defcomultiplicationthroughW}
\De(x)= W^\ast(1\otimes x)W, \qquad x\in M.
\end{equation}
In fact, this formula has led to the definition
of the multiplicative unitary in
\eqref{eq:Wexplicitinapxy}, since the functions $a_p$
are interpreted as Clebsch-Gordan coefficients for
the tensor product decomposition of the representations
considered in \eqref{eq:formalrepsal0ga0}.
We refer to \cite[\S 3]{KoelKCMP} for a more
elaborate discussion of this motivation.

\begin{thm}\label{thm:vNalgqgroupforSU11}
The pair $(M,\De)$ is a von Neumann
algebraic quantum group.
\end{thm}

Theorem \ref{thm:vNalgqgroupforSU11} is
\cite[Thm. 4.9]{KoelKCMP}, and the really hard
part is to prove the coassociativity
$\Id\ot\De\circ\De = \De\ot\Id\circ\De$.
For this part the choice of sign $s(\cdot,\cdot)$
in Definition \ref{def:functionap} of the function $a_p$ is
essential. It should be noted that the results are obtained
in different order in \cite{KoelKCMP} than presented here.

All of the above is included in \cite{KoelKCMP}, but we
additionally need the action of the dual modular
conjugation\index{dual modular
conjugation} $\hat{J}$\index{J@$\hat{J}$} in the GNS-space $\cK$.
Explicitly, we have
\begin{equation} \label{eq:expressionhatJonfmpt}
\hat{J}\, f_{m,p,t} = \sgn(p)^{\chi(p)} \, \sgn(t)^{\chi(t)}\, (-1)^m \,
f_{-m,p,t}, \qquad p,t\in I_q,\ m\in\Z.
\end{equation}
This can be proved from the results in \cite{KoelKCMP} as follows.
Since the right invariant weight equals the left invariant
weight, we have $\hat{J} \La(x) = \La (R(x)^\ast)$ for
$x\in \cN$, see \cite[Prop.~2.11]{KustVMS}.
Using \cite[Prop.~4.14]{KoelKCMP} for the explicit expression
of the unitary antipode $R$
we see that applying this expression with $x=\Phi(m,p,t)$
gives \eqref{eq:expressionhatJonfmpt}.


\section{The von Neumann algebra for the dual quantum group}
\label{sec:VonNeumanAlgDualQuantumGroup}

The general theory as described in Section \ref{sec:vNalgquantumgroups}
shows that there is a dual von Neumann algebraic quantum group
associated to the von Neumann algebraic quantum group $(M,\De)$
associated to the normalizer of $SU(1,1)$ in $SL(2,\C)$,
see Theorem \ref{thm:duallocallycompactquantumgroup}. Since we
have the von Neumann algebra $M$ explicitly given by
Definition \ref{def:vNalgMinBH} and Theorem \ref{thm:vNalgqgroupforSU11}
it is natural to ask for an explicit description in terms of generators
for the von Neumann algebra $\hat{M}$ of the dual von Neumann algebraic
quantum group. On the level of Hopf algebras, there is a
duality between $\cA_q^1$ and the quantized universal enveloping
algebra $\su$, see \cite{MasuMNNSU}
and \cite{CharP}. So it is natural to expect that the
quantized enveloping algebra $\su$ plays a role in an explicit description
of $\hat{M}$, but also that $\su$ will not suffice to describe
$\hat{M}$. This is made explicit in Theorem \ref{thm:generatorsforhatM}.

Let us first recall the quantized universal
enveloping algebra $\su$\index{quantized universal
enveloping algebra $\su$} in order to fix the notation.
The study of $\su$ goes back to Vaksman and
Korogodski\u\i\ \cite{VaksK}, and Masuda et al. \cite{MasuMNNSU},
see also Burban and Klimyk \cite{BurbK}.
Its representation theory is also
needed in this paper, and we recall the irreducible
admissible representations in Section \ref{sec:Casimiroperator},
where we decompose the GNS-space with respect to the
$\su$-action.
For general information on quantized universal enveloping
algebras one can consult e.g. \cite{CharP}, \cite{EtinS},
\cite{Kass}, \cite{Lusz}, \cite{SchmK}.

Recall that $\su$\index{U@$\su$} is the complex unital
$\ast$-algebra generated by $\KA$, $\KA^{-1}$, $\EA$ and
$\FA$ subject to
\begin{equation}\label{eq:defrelUqsu11}
\KA \KA^{-1}=1=\KA^{-1}\KA,
\quad \KA \EA=q\, \EA \KA,
\quad \KA \FA=q^{-1}\FA \KA,\quad \FA \EA-\EA \FA =
\frac{\KA^2-\KA^{-2}}{q-q^{-1}}
\end{equation}
and where the $\ast$-structure is defined by $\KA^\ast = \KA$,
$\EA^\ast=\FA$. Since we assume $0<q<1$, the $\ast$-structure is
easily seen to be compatible with \eqref{eq:defrelUqsu11}.
(We identify $(A,B,C,D)$ of \cite{KoelSRIMS} by
$(\KA,\EA,-\FA,\KA^{-1})$ and compared to
the notation of \cite{MasuMNNSU} we have  $e = \EA$, $f = -\FA$
and $k = \KA$.)
The algebra $\su$ has more structure, since
it can be made into a Hopf $\ast$-algebra. For completeness
we recall the action of the antipode $\SA$ and
the comultiplication $\DeA$ on the
generators;
\begin{equation}\label{eq:defSongeneratorssu}
\SA(\KA) = \KA^{-1}, \quad
\SA(\EA) = -q^{-1}\EA, \quad \SA(\FA) = -q \FA, \quad
\SA(\KA^{-1}) = \KA.
\end{equation}
and
\begin{equation}\label{eq:deDeltaongeneratorssu}
\begin{split}
&\DeA(\KA) = \KA \otimes \KA , \quad \DeA(\EA) = \KA\otimes\EA + \EA\otimes \KA^{-1} \\
&\DeA(\FA) = \KA\otimes\FA + \FA\otimes \KA^{-1}, \quad \DeA(\KA^{-1}) = \KA^{-1} \otimes \KA^{-1}.
\end{split}
\end{equation}

The Casimir element
\begin{equation}\label{eq:Casimir}
\OmA = \hf \Bigl( (q^{-1}-q)^2\FA \EA - q\KA^2-q^{-1}\KA^{-2}\Bigr)
= \hf \Bigl( (q^{-1}-q)^2\EA \FA -
q\KA^{-2}-q^{-1}\KA^{2}\Bigr)
\end{equation}
is a central self-adjoint element in $\su$.
In fact, we use a slightly renormalized version of
the operator used
in \cite{MasuMNNSU}. If $\mathbf{C}$ denotes the element introduced
in \cite[Part II, (1.9)]{MasuMNNSU}, one has
$\OmA = -\frac{1}{2}\,(q-q^{-1})^2\,\mathbf{C} - 1$.
The Casimir element $\OmA$ generates the center of $\su$.

In order to represent the algebra $\su$ on the Hilbert space
$\cK$ of the GNS-re\-pre\-sen\-tation
some care has to be taken, since the operators are in general
unbounded.
We  define the dense subspace $\cK_0$ of $\cK$ as the linear subspace
consisting of finite linear combinations of the orthonormal
basis elements $f_{mpt}$, see the definition in Section \ref{sec:quantumSU11description}.
Equivalently $\cK_0$ can also be viewed as the linear
span of  elements of the form $\ze^m \ot f$ with $m \in \Z$,
$f \in \cK(I_q \times I_q)$, where $\cK(I_q\times I_q)$ is the space
of compactly supported function on $I_q\times I_q$.
Note that  $\cK_0$ is dense
in $\cK$ and that $\cK_0$ inherits the inner product of $\cK$,
so we can look at the space of adjointable
operators $\cL^+(\cK_0)$ for $\cK_0$,
see \cite[Prop.~2.1.8]{Schm}.
Recall that
\[
\cL^+(\cK_0) = \{ T\colon \cK_0 \to \cK_0 \text{ linear}
\mid \exists S \colon \cK_0 \to \cK_0 \text{ linear
so that } \langle Tx,y\rangle = \langle x, Sy\rangle \
\forall\, x,y \in \cK_0\}
\]
The
$\ast$-operation in $\cL^+(\cK_0)$ will be
denoted by $\dag$.

\begin{defn} \label{def:actiongeneratorssu}
We define operators $E_0$, $K_0$ in $\cL^+(\cK_0)$
by
\begin{equation}\label{eq:defE0forsu}
\begin{split}
(q-q^{-1})\,E_0\,f_{mpt} = &
\ \sgn(t)\,\,q^{-\frac{m-1}{2}}\,|p/t|^{\frac{1}{2}}
\,\sqrt{1+\kappa(q^{-1} t)}\,\,\, f_{m-1,p,q^{-1}t} \\
 &\ -\, \sgn(p)\,\,q^{\frac{m-1}{2}}\,|t/p|^{\frac{1}{2}}
 \, \sqrt{1+\kappa(p)}\,\,\, f_{m-1,qp,t}
\end{split}
\end{equation}
and $K_0\,f_{mpt} = q^{-\frac{m}{2}}\,
|p/t|^{\frac{1}{2}}\,f_{mpt}$
for all $m \in \Z$, $p,t \in I_q$.
\end{defn}

Here $\sgn$ denotes the sign, and $\ka(x) = \sgn(x) x^2$, see
Definition \ref{defn:chikappanuands}.

Definition \ref{def:actiongeneratorssu} is motivated by formal
calculations based on \cite[Part II (1.11), (1.12)]{MasuMNNSU}.

One easily checks that $K_0^\dag = K_0$
and that $K_0$ is invertible in $\cL^+(\cK_0)$.
Also $E_0 \in \cL^+(\cK_0)$ and
\begin{equation}\label{eq:dualE0forsu}
\begin{split}
(q - q^{-1})\,E_0^\dag\,f_{mpt} =  &
\ \sgn(t)\,\,q^{-\frac{m+1}{2}}\,|p/t|^{\frac{1}{2}}\,
\sqrt{1+\kappa(t)}\,\,\, f_{m+1,p,q t} \\  &\ - \,
\sgn(p)\,\,q^{\frac{m+1}{2}}\,|t/p|^{\frac{1}{2}}\,
\sqrt{1+\kappa(q^{-1} p)}\,\,\, f_{m+1,q^{-1} p,t}
\end{split}
\end{equation}
for all $m \in \Z$, $p,t \in I_q$.

At this point we observe that modular conjugation $J$ preserves
$\cK_0$, since $J\,f_{mpt}=f_{-m,t,p}$ see Remark \ref{rmk:propertiessiJnabla}, and it
follows straightforwardly
\begin{equation}\label{eq:JE0JisminE0dag}
J\, E_0^\dag \, J = - E_0\ \ \text{and}\ \ J\, K_0\, J = K_0^{-1}
\ \ \text{in}\  \cL^+(\cK_0).
\end{equation}
Using \eqref{eq:actionunitaryantipode} we see that
\eqref{eq:JE0JisminE0dag} is in
correspondence with \eqref{eq:defSongeneratorssu}.

The next proposition shows that $E_0$ and $K_0$ do
satisfy the defining relations \eqref{eq:defrelUqsu11}
for the $\ast$-algebra $\su$.

\begin{prop}\label{prop:relationssuaresatisfied}
We have
$$
K_0\,E_0 = q\, E_0\,K_0 \hspace{10ex} \text{and}
\hspace{10ex} E_0^\dag \, E_0 - E_0\, E_0^\dag =
\frac{K_0^2- K_0^{-2}}{q-q^{-1}} \, .
$$
and the elements from
$\{ \,K_0^m E_0^k (E_0^\dag)^l \mid m \in \Z, k,l \in \NN \,\}$
are linearly independent.
\end{prop}

Proposition \ref{prop:relationssuaresatisfied} implies that
there exists a unique unital $\ast$-re\-pre\-sen\-ta\-tion
$\rho\colon \su\to \cL^+(\cK_0)$ so that $\EA \mapsto E_0$ and
$\KA \mapsto K_0$, hence $\cK_0$ is turned into a $\su$-module.
Define $\cU$ to be the unital $\ast$-subalgebra of
$\cL^+(\cK_0)$ generated by $K_0$, $K_0^{-1}$ and $E_0$.
This is a $\ast$-representation of $\su$ by unbounded
operators in the sense of \cite[Ch. 8]{Schm}, so that
in particular each element of $\cU$ is closable.
The Poincar\'e-Birkhoff-Witt theorem, see e.g. \cite{CharP},
implies that the $\ast$-representation $\su \to \cL^+(\cK_0)$ is
faithful and $\cU$ is a concrete realization of $\su$.

An essential role in the representation theory of $\su$ is
played by the Casimir
operator \eqref{eq:Casimir}. An elaborate discussion
about its role in decomposing $\cK_0$ into irreducible
$\su$-modules is given in Section \ref{sec:Casimiroperator}.
We define the Casimir element $\Om_0 \in \cU\subset \cL^+(\cK_0)$
as $\Om_0=\rho(\OmA)$, i.e.
$$
\Om_0 = \frac{1}{2}\,
\bigl(\,(q-q^{-1})^2\,E_0^\dag \, E_0 - q\,K_0^2 -
q^{-1}\,K_0^{-2}\,\bigr) = \frac{1}{2}\,
\bigl(\, (q-q^{-1})^2\,E_0 E_0^\dag -
q^{-1}\,K_0^2-q\,K_0^{-2}\,\bigr) \ .
$$
By Definition \ref{def:actiongeneratorssu} and \eqref{eq:dualE0forsu} we have
the explicit expression
\begin{equation}\label{eq:actioncasimir}
\begin{split}
&\ 2 \, \Om_0 \,f_{mpt}  =   - \sgn(pt)\,
\sqrt{(1+\kappa(p))(1+\kappa(t))}\,\, f_{m,qp,qt}
\\ &  + (q^{m-1} p |t| + q^{-m-1} t |p|)\,f_{mpt}
 - \sgn(pt)\, \sqrt{(1+\kappa(q^{-1}p))(1+\kappa(q^{-1}t))}\,
 \, f_{m,q^{-1}p,q^{-1}t}
\end{split}
\end{equation}
for all $m \in \Z$ and $p,t \in I_q$.

Recall that we are using a  renormalized version
(and terminology) of the operator used in \cite{MasuMNNSU}.
The renormalization is chosen in such a way that the continuous
spectrum of the relevant self-adjoint extension of $\Om_0$
is given by $[-1,1]$ and the point spectrum of this extension
has a maximal degree of symmetry with respect to
the origin.

Not $K_0$,
$E_0$ and $\Om_0$ are the operators relevant to the
dual locally compact
quantum group $(\hat{M},\hat{\De})$ introduced
in Section \ref{sec:vNalgquantumgroups}, but rather the
\emph{right} closed extensions of these operators. Now
$K_0$ is essentially self-adjoint, so it is clear what extension
of $K_0$ to use. At this moment, it is not clear what kind of
extension of $E_0$ we need, but Proposition \ref{prop:KEaffiliatedtohatM}
shows that the closure of $E_0$ is
the natural extension in this setting.
Next the Casimir operator is discussed.

\begin{defn} \label{def:closedgenerators}
We define the densely defined, closed, linear operators $E$\index{E@$E$} and
$K$\index{K@$K$} in $\cK$ as the closures of $E_0$ and $K_0$
respectively.
\end{defn}

One expects at least that $K$ and $E$ are affiliated to the
dual von Neumann algebra $\hat{M}$. This is indeed the
case.

\begin{prop} \label{prop:KEaffiliatedtohatM}
$K$ is an injective positive self-adjoint operator in $\cK$.
The operators $K$ and $E$ are affiliated to
the von Neumann algebra $\hat{M}$.
\end{prop}

Note that the spectrum $\si(K)$ consists of $q^{\hf \Z}\cup \{0\}$.
Moreover, $E^\ast$ is the closure of $E_0^\dag$, and there
exists a characterization of $E$ given in Proposition
\ref{prop:charEandrelationEstarwithE0dag}.

Next we want to define the Casimir operator on $\cK$
as the \emph{right} extension of $\Om_0$.
Since $\Om_0^\dag = \Om_0$, it is natural to look for a
self-adjoint extension of $\Om_0$ to be this
\emph{right} extension. But $\Om_0$ is not essentially
self-adjoint, which is discussed in Section \ref{sec:Casimiroperator},
thus, unlike the cases $E$ and $K$,
we can not merely use the closure of $\Om_0$.

\begin{defn} \label{def:closedCasimir}
We define the Casimir operator $\Om$\index{Casimir operator}\index{O@$\Om$} as the closure of the
operator
\[
\frac{1}{2}\,\, \bigl(\,(q-q^{-1})^2\,E^\ast  E -
q\,K^2 - q^{-1}\,K^{-2}\,\bigr) \, .
\]
\end{defn}

At this point it is not clear that Definition
\ref{def:closedCasimir} makes sense.

\begin{thm} \label{thm:Casimirwelldefinedandchar}
The Casimir operator $\Om$
is a well-defined self-adjoint operator. The Casimir operator commutes strongly with
the unbounded operators $E$ and  $K$.
Moreover, the Casimir operator $\Om$ is the unique self-adjoint
extension of $\Om_0$ that is affiliated to
the von Neumann algebra $\hat{M}$.
\end{thm}

The proofs of these statements are given Section \ref{sec:Casimiroperator}.
At the same time it will emerge that $\Om$ is
not the closure of $\Om_0$.

The Casimir element $\Om_0$ belongs to the center of $\cU$,
and hence commutes with $E_0$ and $K_0$ in $\cL^+(\cK_0)$.
On the Hilbert space level, this result has an analogue to the
extent that the Casimir operator $\Om$ strongly commutes with
$E$ and $K$, see Theorem \ref{thm:Casimirwelldefinedandchar}.
However, since $(M,\De)$ is a quantization of
the normalizer of $SU(1,1)$ in $SL(2, \C)$,
and not of $SU(1,1)$, it is to be
expected that the Casimir operator does not commute with
all elements of $\hat{M}$.
Indeed, the Casimir
operator satisfies a graded commutation relation with
the elements of $\hat{M}$, i.e. there exists a
decomposition $\hat{M}=\hat{M}_+\oplus\hat{M}_-$ such that
the Casimir operator commutes with the elements of
$\hat{M}_+$ and anti-commutes with elements of $\hat{M}_-$,
see Proposition \ref{prop:decompMintoM+andM-andgradedcommutation}.

In order to formulate the graded commutation relation
involving the Casimir operator we provide $\cK$ and $\hat{M}$
with a natural $\Z_2$-grading.

\begin{defn}\label{def:defKpmandMpm}
We define the closed subspaces $\cK_+,\cK_- \subseteq \cK$\index{K@$\cK_+$}\index{K@$\cK_-$} as
$$
\cK_{\pm} = \overline{\Span}\{\,f_{m,p,t} \mid m \in \Z, p,t \in I_q
\text{ so that } \sgn(p t) = \pm \,\},
$$
So $\cK=\cK_+\oplus\cK_-$.
We define the $\si$-weakly closed subspaces
$\hat{M}_+, \hat{M}_- \subseteq \hat{M}$ as
$$\index{M@$\hat{M}_+$}\index{M@$\hat{M}_-$}
\hat{M}_+ = \{\, x \in \hat{M} \mid x \,\cK_\pm \subseteq \cK_\pm\,\}
\hspace{7ex} \text{ and } \hspace{7ex}
\hat{M}_- = \{\, x \in \hat{M} \mid x\, \cK_\pm \subseteq \cK_\mp\,\}\ .
$$
\end{defn}

Then $\hat{M}_+$ is a von Neumann algebra, and
$\hat{M}_-$ is a self-adjoint subspace so that
$\hat{M}_\pm\,\hat{M}_\mp \subseteq \hat{M}_-$ and
$\hat{M}_- \,\hat{M}_- \subseteq \hat{M}_+$.
In order to get a real $\Z_2$-grading on $\hat{M}$,
we need the following result.

\begin{prop}\label{prop:decompMintoM+andM-andgradedcommutation}
$\hat{M} = \hat{M}_+ \oplus \hat{M}_-$.
Let $x \in \hat{M}_+$ and $y \in \hat{M}_-$, then
$x\,\Om \subseteq \Om\,x$ and $y\,\Om \subseteq -\,\Om\,y$.
\end{prop}

Proposition \ref{prop:decompMintoM+andM-andgradedcommutation}
implies that
$E$ and $K$ do not suffice to generate $\hat{M}$
because of Theorem \ref{thm:Casimirwelldefinedandchar}.
In order to determine $\hat{M}$,
Proposition \ref{prop:decompMintoM+andM-andgradedcommutation} also
provides the key ingredient once we have
determined the spectral decomposition of $\Om$
explicitly. Indeed, Proposition \ref{prop:decompMintoM+andM-andgradedcommutation}
implies that elements of $\hat{M}$ can be described
by mapping (generalized) eigenvectors for the
eigenvalue $\la$ of the Casimir operator to
(generalized) eigenvectors for the
eigenvalue $\pm\la$ of the Casimir operator.
For this we have to study the Casimir operator restricted to
suitable invariant subspaces on which the spectrum of
$\Om$ has simple spectrum, which is done in Section \ref{sec:Casimiroperator}.

We define bounded operators on the Hilbert space
$\cK$ of the GNS-representation using the multiplicative
unitary $W\in B(\cK\ot\cK)$. Using the normal
functionals $\om_{f,g}\in B(\cK)_\ast$ defined
by $\om_{f,g}(x) = \langle x\, f,g\rangle$, $f,g\in\cK$,  we define
\begin{equation}\label{eq:defQppn} \index{Q@$Q(p_1,p_2,n)$}
Q(p_1,p_2,n) =(\om_{f,g}\ot \Id )(W^\ast) \colon \cK \to
\cK, \qquad f = f_{0,p_1,1},\ g = f_{n,p_2,1}.
\end{equation}

\begin{prop}\label{prop:QppngeneratedualM}
The operators $Q(p_1,p_2,n) \in B(\cK)$, $p_1,p_2 \in I_q$, $n \in \Z$,
are in $\hat M$ and the linear span is strong-$\ast$ dense in
$\hat M$. Moreover, $Q(p_1,p_2,n)\in \hat{M}_{\sgn(p_1p_2)}$.
\end{prop}

Proposition \ref{prop:QppngeneratedualM} is the key to the proof
of Proposition \ref{prop:decompMintoM+andM-andgradedcommutation}, and
describes sufficiently many elements of $\hat{M}$.

Since the operators $Q(p_1,p_2,n)$ span $\hat{M}$ linearly, we
calculate the structure constants.\index{structure constants}

\begin{prop}\label{prop:structureconstQs}
For $p_1,p_2,r_1,r_2 \in I_q$, $n,m \in \Z$,  we have
$Q(p_1,p_2,n)\, Q(r_1,r_2,m) = 0$ in case $\vert\frac{p_2}{p_1}\vert \not= q^m$
 or  $\vert \frac{r_1}{r_2}\vert \not= q^n$. In case
$\vert \frac{p_2}{p_1}\vert = q^m$
and  $\vert \frac{r_1}{r_2}\vert = q^n$ we have
\[
Q(p_1,p_2,n)\, Q(r_1,r_2,m) = \sum_{x_1,x_2 \in I_q} a_{x_1}(r_1,p_1)\, a_{x_2}(r_2, p_2) \,
Q(x_1,x_2,n+m)
\]
where the coefficients $a_{x_i}(r_i,p_i)$, $i=1,2$, are defined in
Definition \ref{def:functionap}.
\end{prop}

Since $\hat{M}'=\hat{J}\hat{M}\hat{J}$, see \eqref{eq:JMJisMcommutant}
for the dual von Neumann algebra, Proposition \ref{prop:QppngeneratedualM}
leads to the following.

\begin{cor}\label{cor:propQppngeneratedualM}
The operators $\hat{J}Q(p_1,p_2,n)\hat{J} \in B(\cK)$, $p_1,p_2 \in I_q$, $n \in \Z$,
are in $\hat{M}'$ and the linear span is strong-$\ast$ dense in
$\hat{M}'$.
\end{cor}

The main problem in proving Proposition \ref{prop:QppngeneratedualM} is that
the operators $Q(p_1,p_2,n)$ do not preserve the dense subspace
$\cK_0$. We have the following polar-type decomposition of these
operators.

\begin{lemma}\label{lem:polardecompositionQppn}
For fixed $p_1,p_2\in I_q$, $n\in\Z$, there exists an orthogonal
projection $P=P(p_1,p_2,n)\in B(\cK)$, a continuous
function $H(\cdot)=H(\cdot;p_1,p_2,n)$\index{H@$H(x;p_1,p_2,n)$} and a partial isometry
$U=U^{\sgn(p_1),\sgn(p_2)}_n$ so that
\[
Q(p_1,p_2,n) =  U\, H(\Om) \, P.
\]
\end{lemma}

Since the elements $H(\Om)$ and $P$, as element of the spectral decomposition of $K$, are in
the von Neumann algebra generated by $E$ and $K$, we only need to incorporate
the partial isometries.
Now we can state the main theorem of this section, which gives an
explicit description of the von Neumann algebra for the
dual locally compact quantum group.

\begin{thm}\label{thm:generatorsforhatM}
The von Neumann algebra $\hat{M}$ is generated by
$K$, $E$, $U^{+-}_0$, $U^{-+}_0$.
\end{thm}

It is interesting to connect the comultiplication of
the dual quantum group as in Theorem \ref{thm:duallocallycompactquantumgroup}
with the comultiplication \eqref{eq:deDeltaongeneratorssu}
of the quantized universal enveloping algebra.

\begin{prop}\label{prop:dualcomultiplicationandcomultiplicationsu}
We have $\hat{\De}(K) = K \ot K$, and
$$
K_0\odot E_0\, +\, E_0\odot K_0^{-1} \, \subset \, \hat{\De}(E) \qquad \text{and} \qquad
K_0\odot E_0^\dag\, +\, E_0^\dag\odot \,K_0^{-1}  \subset\, \hat{\De}(E^*).
$$
\end{prop}

In Proposition \ref{prop:dualcomultiplicationandcomultiplicationsu} the
left hand side denotes the algebraic tensor product of the
unbounded operators which are defined on the domain $\cK_0\odot \cK_0
\subset \cK\ot\cK$.
So we see that the comultiplication of the dual quantum group
corresponds to the  comultiplication  of the
quantized universal enveloping algebra, see \eqref{eq:Hopfalgebracomultiplication}.
Note that for an element $x$ affiliated to $\hat{M}$ we can calculate
$\hat{\De}(x)$ as an affiliated element of $\hat{M}\ot\hat{M}$.

We can also calculate
the comultiplication on the elements $Q(p_1,p_2,n)$
spanning $\hat{M}$, see Proposition \ref{prop:QppngeneratedualM},
using the pentagonal equation.

\begin{prop}\label{prop:hatDeonQppn}
For $p_1,p_2\in I_q$, $n\in\Z$,  we
have
\[
\hat{\De}\bigl(Q(p_1,p_2,n)\bigr)
=\sum_{m\in\Z,\, p\in I_q} Q(p,p_2,n-m)\ot Q(p_1,p,m),
\]
where the sum converges in the $\si$-weak-topology of $\hat{M}\ot\hat{M}$.
\end{prop}

The action of the unitary antipode $\hat{R}$ and of the $\ast$-operator
on the generators $Q(p_1,p_2,n)$ of $\hat{M}$ is given in
Corollary \ref{cor:lemQppnareallinhatM}.


\section{The decomposition of the left regular corepresentation}
\label{sec:decompleftregularcorep}

As remarked in Section \ref{sec:vNalgquantumgroups} the multiplicative unitary acting in the
GNS-representation of the left invariant weight is the
analogue of the left regular representation. For
the Lie group $SU(1,1) \cong SL(2,\R)$, the decomposition
into irreducible representations involves the
principal unitary series and the discrete series,
see e.g. \cite{vDijk}, \cite{HoweT}, \cite{KoorEM}, \cite{Lang},
\cite{Sugi}. 
The decomposition is obtained by considering the action of
the Casimir operator, since its eigenspaces give invariant
subspaces as the Casimir operator is a central element.
Our next goal is to decompose  the left regular
corepresentation given by the multiplicative
unitary $W$ acting in the GNS-representation $\cK$
into irreducible corepresentations. We want to proceed
in a similar fashion, but as follows from
Proposition \ref{prop:decompMintoM+andM-andgradedcommutation}
we need to combine two eigenspaces of the Casimir operator.
We
first consider the discrete part, and next the
continuous part.

In Section \ref{sec:Casimiroperator} we decompose the GNS-space
$\cK$ into irreducible representations for $\su$ by decomposing the
action of the Casimir operator, and
since its generators are related to affiliated
operators to $\hat{M}$ we expect that this is a
building block in the decomposition. In this section we
describe the decomposition explicitly, and for each corepresentation
in the decomposition of the left regular corepresentation we indicate
its decomposition as $\su$-representation using its
representations as described in Section \ref{ssec:decompUq(su11)}.

In order to find the decomposition of the left regular corepresentation
we have to look for invariant subspaces of $(\om\ot\Id)(W)$, $\om\in M_\ast$,
which are the generators of $\hat{M}$. By Proposition
\ref{prop:decompMintoM+andM-andgradedcommutation} we can restrict to
eigenspaces for the Casimir operator for the eigenvalues $\la$ and
$-\la$. By considering combinations of such eigenspaces in
suitable invariant subspaces for the Casimir operator we can
determine invariant subspaces, hence irreducible corepresentations
occurring in the decomposition of the left regular corepresentation.
In this approach we have to distinguish between eigenvalues $\la$
of the Casimir operator $\Om$ satisfying $|\la|>1$, leading to
the analogue of discrete series representations of $SU(1,1)$, and
those satisfying $|\la|\leq 1$, leading to the analogue of
principal unitary series representations of $SU(1,1)$. The case $\la=0$
has to be considered separately.

In Section \ref{ssec:discreterepsinregularrep} we discuss the analogue
of the discrete series representations, and in
Section \ref{ssec:principalrepsinregularrep} we discuss the analogue
of the principal unitary series representations.
For the precise description of the results we need to use some
notation that is used in the proofs.

\subsection{Unitary corepresentations: discrete series}\label{ssec:discreterepsinregularrep} \index{corepresentation!discrete series}

In order to be able to describe the results we need to consider
the discrete spectrum of the Casimir operator. The complete
spectrum of the Casimir operator $\Om$ is described in
Section \ref{sec:Casimiroperator}, where for suitable
$\Om$-invariant subspaces $\cK(p,m,\ep,\et)\subset \cK$
the spectral decomposition of $\Om\vert_{\cK(p,m,\ep,\et)}$
is discussed in detail. The spectrum is simple and consists
of a continuous part $[-1,1]$ and a discrete part
depending on $\cK(p,m,\ep,\et)$ for $p\in q^\Z$, $m\in \Z$, $\ep,\et\in\{ \pm1\}$.
We refer to \eqref{eq:defKpmepeta} for the
definition of these subspaces.
Throughout this subsection we fix $p \in q^\Z$,
$\la \in - q^{-\N} \cup q^{-\N}$ and set $x = \mu(\la) = \hf (\la+\la^{-1})$. Thus,
$x$ is an isolated point of the spectrum of the Casimir operator $\Om$ if
$x\in \si_d(\Om)$, see Section \ref{ssec:spectraldecompCasimir}.
We denote $e^{\ep,\et}_m(p,x) \in D(\Om)\cap \cK(p,m,\ep,\et)$
to be the eigenvector of the Casimir operator $\Om$ for the
eigenvalue $\ep\et\,x$ in the subspace
$\cK(p,m,\ep,\et)$ of the GNS-space.  We
note that $e^{\ep,\et}_m(p,x) \not= 0$ if and only if
$\Om$ has an eigenvector with eigenvalue $\ep\et\,x$
inside $\cK(p,m,\ep,\et)$. By the results proved in Section \ref{ssec:spectraldecompCasimir} the eigenspace of
$\Om$ restricted to $\cK(p,m,\ep,\et)$ is at most one-dimensional,
so that $e^{\ep,\et}_m(p,x)$ is defined up to phase-factor
after putting $\|e^{\ep,\et}_m(p,x)\|=1$. The precise
choice is given in Section \ref{ssec:discreteseries}.

Recall we have to find closed invariant subspaces for the
action of $\hat{M}$, and we can define closed invariant subspaces
in terms of the eigenvectors of the Casimir operator $\Om$.
This is straightforward once we have described the actions of the
generators of $\hat{M}$ on the eigenvectors of $\Om$ in
Lemma \ref{lem:actionsofgeneratorsoneigvetsOm}.

\begin{lemma} \label{lem:closedsubspaceinregularrep}
We define the closed subspace $\cL_{p,x}$ of $\cK$ as
$$
\cL_{p,x} = \overline{\Span}\{\, e^{\ep,\et}_m(p,x) \mid  m \in \Z, \ep,\et \in \{-,+\}\,\}.
$$
The space $\cL_{p,x}$ is an invariant subspace of the
corepresentation $W$ of $(M,\De)$. If $\cL_{p,x} \not= \{0\}$ we say that
that $(p,x)$ determines a discrete series
corepresentation of $(M,\De)$. The element
$W_{p,x} = W\big\vert_{\cK \ot \cL_{p,x}}$ is
a unitary corepresentation of $(M,\De)$ on $\cL_{p,x}$.
\end{lemma}

Using the explicit actions of the generators of
$\hat{M}$ as described in Theorem \ref{thm:generatorsforhatM}
on the eigenvectors of the Casimir operator we can
classify the values of $(p,x)$ such that
$\cL_{p,x}$ is a discrete series corepresentation of
$(M,\De)$. The result is the following.

\begin{prop} \label{prop:discretesubrepresinregrepr}
Consider $p \in q^\Z$ and $x = \mu(\lambda)$ where
$\lambda \in - q^{2\Z+1} p \cup q^{2\Z+1} p$ and $|\lambda| > 1$.
Let $j,l \in \Z$ be such that $|\lambda| = q^{1-2j} p^{-1} = q^{1+2l} p$,
so $l < j$. Then $(p,x)$ determines a
discrete series corepresentation of $(M,\De)$
in the following 3 cases, and these are the only cases:
\begin{trivlist}
\item[\ \,(i)] If $x > 0$, in which case
$$
\{\,e^{++}_m(p,x) \mid m \in \Z\,\}\,
\cup\,\{\,e^{-+}_m(p,x) \mid m \in \Z,
m \leq l\,\}\,\cup\,\{\,e^{+-}_m(p,x) \mid
m \in \Z, m \geq j\,\}
$$
is an orthonormal basis for $\cL_{p,x}$.
\item[\ \,(ii)] If $x < 0$, $l \geq 0$ and $j > 0$, in which case
$$
\{\,e^{-+}_m(p,x) \mid m \in \Z\,\}\,\cup\,
\{\,e^{++}_m(p,x) \mid m \in \Z, m \leq l\,\}
\,\cup\,\{\,e^{--}_m(p,x) \mid
m \in \Z, m \geq j\,\}
$$
is an orthonormal basis for $\cL_{p,x}$.
\item[\ \,(iii)] If $x < 0$, $l < 0$ and
$j \leq 0$, in which case
$$
\{\,e^{+-}_m(p,x) \mid m \in \Z\,\}\,\cup\,\{\,e^{--}_m(p,x)
\mid m \in \Z, m \leq l\,\}\,\cup\,\{\,e^{++}_m(p,x) \mid
m \in \Z, m \geq j\,\}
$$
is an orthonormal basis for $\cL_{p,x}$.
\end{trivlist}
\end{prop}

Proposition \ref{prop:discretesubrepresinregrepr} gives a
complete list of discrete corepresentations occurring in the
left regular corepresentation.
In each of the cases listed in Proposition
\ref{prop:discretesubrepresinregrepr} we can consider the
representation of $\hat{M}$ as a representation of
$\su$ (by unbounded operators in the sense of \cite{Schm}), and then, by comparing
the action of $E$ and $K$ as given in Lemma \ref{lem:actionsofgeneratorsoneigvetsOm},
with the listing in Section \ref{ssec:decompUq(su11)}, we see
that $\cL_{p,x}$ in case (i), (ii) and (iii) of Proposition
\ref{prop:discretesubrepresinregrepr} corresponds to
\begin{equation}\label{eq:sudecompdiscretesubrepresinregrepr}
\pi^S_{\hf(\chi(p)-1)+j, \epsilon(p)} \oplus D^-_{-\hf\chi(p)-l}
\oplus D^+_{\hf\chi(p)+j}
\end{equation}
as $\su$-module, where the decomposition corresponds to
the order of the orthonormal basis. The notation for the $\su$-modules
is as in Section \ref{ssec:decompUq(su11)}.
Here $\chi(p)\in\Z$ is defined  in
Definition \ref{defn:chikappanuands} and $\epsilon(p) = \hf\chi(p)\mod 1$,\index{E@$\epsilon(p)$} so
$\epsilon(p)=0$ for $p\in q^{2\Z}$ and $\epsilon(p)=\hf$ for $p\in q^{2\Z+1}$, see \eqref{eq:epsilon(p)}.
So we see that a discrete series corepresentation in the
left regular corepresentation decomposes in the same way
as sum of three $\su$-representations involving a strange series
representation in combination with a positive and negative discrete series
representation.

\begin{prop}\label{prop:Wpxgivesirreduciblediscreteseries}
Assume that $(p,x)$ determines a discrete series corepresentation
of $(M,\De)$. Then $W_{p,x}$ is an irreducible corepresentation
of $(M,\De)$.
\end{prop}

\subsection{Unitary corepresentations: principal series}\label{ssec:principalrepsinregularrep} \index{corepresentation!principal series}

Next we discuss the irreducible corepresentations of $(M,\De)$ in the
left regular corepresentation $W$ corresponding to the continuous
spectrum of the Casimir operator $\Om$. We cannot obtain these
representations by restriction to closed subspaces, so we
have to use another approach.

Motivated by Lemma \ref{lem:actionsofgeneratorsoneigvetsOm} and the
admissible irreducible representations of $\su$ as discussed in
Section \ref{ssec:decompositionofGNSspace} we define for
 $x=\cos\te\in [-1,1]$ and $p \in q^\Z$ a Hilbert space $\cL_{p,x}$ by
\[
\cL_{p,x} = \bigoplus_{\ep,\et \in \{-,+\}} \ell^2_{\ep,\et}(p,x),
\]
where each space $\ell^2_{\ep,\et}(p,x)$ denotes a copy of $\ell^2(\Z)$ with standard orthonormal basis $\{ e_m^{\ep,\et}(p,x) \mid m\in \Z\}$. We define operators $K, E, U_0^{+-}, U_0^{-+}$ on  $\cL_{p,x}$ by
\begin{equation} \label{eq:introprincipalseries}
\begin{split}
K \, e_m^{\ep,\et}(p,x) &= p^\hf q^m \, e_m^{\ep,\et}(p,x),\\
(q^{-1}-q) E\, e_m^{\ep,\et}(p,x) & = q^{-m-\hf}p^{-\hf} |1+\ep\et p q^{2m+1} e^{i\te}| \, e_{m+1}^{\ep,\et}(p,x),\\
U_{0}^{+-}\, e_m^{\ep,\et}(p,x) & = \et\, (-1)^{\upsilon(p)} \, e_m^{\ep,-\et}(p,x),\\
U_{0}^{-+}\, e_m^{\ep,\et}(p,x) & = \ep \et^{\chi(p)}(-1)^m \, e_m^{-\ep,\et}(p,x).
\end{split}
\end{equation}
Here $\chi(p) =\log_q (p)$ is defined in Definition \ref{defn:chikappanuands} and
$\upsilon(p)$ is defined in \eqref{eq:defupsilon}. Explicitly, for $p=q^{2k}$ or $p=q^{2k-1}$ with $k \in \Z$
we have $\upsilon(p)=k$.
The operators $E$ and $K$ are unbounded closable operators with dense core the finite linear combinations of the orthonormal basis vectors $e_m^{\ep,\et}(p,x)$, $m \in \Z$, $\ep,\et \in \{-,+\}$.
The operators $U_0^{+-}$ and $U_0^{-+}$ are bounded; they are isometries.

\begin{prop} \label{prop:introprincipalseriescorepresentation}
The operators $E, K, U_0^{+-}, U_0^{-+}$ defined by \eqref{eq:introprincipalseries} generate a von Neumann algebra $\hat M_{p,x}$ that is isomorphic to $\hat M$. Consequently, \eqref{eq:introprincipalseries} determines a unitary corepresentation $W_{p,x}$ of $(M,\De)$.
The corepresentation $W_{p,x}$ is reducible, and its decomposition into irreducible corepresentations
is given by
\[
\begin{split}
&W_{p,x} \, = \, W^1_{p,x}\, \oplus\, W^2_{p,x}, \qquad \text{in case } x\not=0, \, \text{ or }
p\in q^{2\Z+1}, \\
&W_{p,0} \, = \, W^{1,1}_{p,0}\, \oplus\, W^{1,2}_{p,0}\, \oplus\,W^{2,1}_{p,0}\, \oplus\,W^{2,2}_{p,0}, \qquad \text{in case } p\in q^{2\Z}. \\
\end{split}
\]
\end{prop}

\begin{remark}\label{rmk:propintroprincipalseriescorepresentation}
Denoting the corresponding invariant subspaces by $\cL^j_{p,x}$ and $\cL^{j,k}_{p,0}$ of
Proposition \ref{prop:introprincipalseriescorepresentation}, which
are described explicitly in Section \ref{ssec:principalseries}, we can consider these
irreducible constituents of Proposition \ref{prop:introprincipalseriescorepresentation}
as representations of $\su$.
If we consider the irreducible representations $W_{p,x}^j$, $j=1,2$, of $\hat M$ as representations of $\su$, they  decompose into irreducible principal series $\su$-representations as $\pi_{b(-x),\epsilon(p)}\oplus \pi_{b(x),\epsilon(p)}$, where $b(x)$ is determined by $\mu(q^{2ib(x)})=x$ and $\epsilon(p)=\hf \chi(p) \mod 1$,
as in the  decomposition  of the discrete series subcorepresentation of $W$ into $\su$-modules in
Section \ref{ssec:discreterepsinregularrep}
and the representations of $\su$ are described in Section \ref{ssec:decompUq(su11)}.
Similarly, it follows that for $x=0$ and $\epsilon(p)=0$, the irreducible representations $W_{p,0}^{j,k}$, $j,k=1,2$, of $\hat M$ can be considered as irreducible principal series representations $\pi_{b(0),0}$ of $\su$, where $b(0)=-\frac{\pi}{4\ln q}$.
This follows directly from the explicit description of the spaces $\cL^j_{p,x}$ and $\cL^{j,k}_{p,0}$
in Section \ref{ssec:principalseries} and \eqref{eq:introprincipalseries} compared to the listing
of irreducible representations of $\su$ in
Section \ref{ssec:decompUq(su11)}.
\end{remark}

In Section \ref{sec:Casimiroperator} we discuss for suitable
$\Om$-invariant subspaces $\cK(p,m,\ep,\et)\subset \cK$
the spectral decomposition of $\Om\vert_{\cK(p,m,\ep,\et)}$, and
we denote by $\cK_c(p,m,\ep,\et)\subset \cK(p,m,\ep,\et)\subset \cK$
the subspace corresponding to the continuous spectrum $[-1,1]$ of
$\Om\vert_{\cK(p,m,\ep,\et)}$.

\begin{prop} \label{prop:introdecompWprincipalseries}
For $p \in q^\Z$ let $\cK_c(p) \subset \cK$ be the subspace defined by
\[
\cK_c(p)=\bigoplus_{\substack{\ep,\et \in \{-,+\}\\ m \in \Z}}\cK_c(p,m,\ep,\et),
\]
then
\[
W\big\vert_{\cK \otimes \cK_c(p)} \cong \int_{-1}^1 W_{p,x}\, dx.
\]
\end{prop}
For direct integrals of (co)representations we refer to \cite[Ch.8]{Schm}.

\subsection{Decomposition of the left regular corepresentation}\label{ssec:} \index{corepresentation!decomposition of the left regular}

Since $\cK=\cK_c\oplus\cK_d$ with $\cK_d$, respectively $\cK_c$,
the subspace corresponding to the discrete, respectively continuous, spectrum of
the Casimir operator, we find by
combining Propositions \ref{prop:Wpxgivesirreduciblediscreteseries}
and \ref{prop:introdecompWprincipalseries} the following decomposition of
the left regular corepresentation $W$ of $(M,\De)$.

\begin{thm}\label{thm:decompositionofleftregularcorep}
\[
W \cong \bigoplus_{p \in q^\Z} \Big( \int_{-1}^1 W_{p,x} dx \oplus \bigoplus_{x \in \si_d(\Om_p)} W_{p,x} \Big),
\]
where $\displaystyle \Om_p = \bigoplus_{\substack{\ep,\et \in \{-,+\} \\ m \in \Z}} \Om\vert_{\cK(p,m,\ep,\et)}$
and $\si_d$ denotes the discrete spectrum.
\end{thm}

It is well-known that in the decomposition of the left regular representation
of $SU(1,1)\cong SL(2,\R)$ the discrete series representation and the principal
unitary series occur, and in this sense Theorem \ref{thm:decompositionofleftregularcorep}
is the appropriate analogue of this result. In case of the group $SU(1,1)\cong SL(2,\R)$
we also have complementary series representations, which do not occur in the
decomposition of the left regular representation, but which can be obtained by
continuation from the principal unitary series representation. For the quantum group
analogue of the normalizer of $SU(1,1)$ in $SL(2,\C)$ we have a similar result. So
we can obtain unitary complementary series corepresentations of $(M,\De)$, and the
approach is sketched in Section \ref{ssec:complementaryseries}.

%

\section{Results for special functions of basic hypergeometric type}\label{sec:BHSresults}

This section is separately readable from the remainder of the paper. This
section is meant to give a couple of examples of rather complicated identities for special
functions of basic hypergeometric type ${}_1\vp_1$ and type ${}_2\vp_1$, see \cite{GaspR}.
We assume that the reader of this section is familiar with the notation for basic
hypergeometric series \cite{GaspR}, but the definition is recalled in
Appendix \ref{app:specialfunctions}. In the first subsection we introduce the notation for
special functions, and we recall some elementary properties. The first subsection
introduces notation and special functions that are used throughout the paper, whereas
the following subsections give explicit highly non-trivial results for these
special functions. These identities follow from the quantum group theoretic interpretation.

\subsection{Definition of some special functions}\label{ssec:defsomespecialf}
The set of natural numbers (without 0) will be denoted by $\N$ and
$\NN = \N \cup \{0\}$. We write, as in Section \ref{sec:quantumSU11description}, $I_q= -q^\N \cup q^\Z$.\index{I@$I_q$}
We use the following functions frequently.

\begin{defn}\label{defn:chikappanuands}
\textrm{(i)} $\chi \colon -q^\Z \cup q^\Z \rightarrow \Z$\index{C@$\chi(x)$} such that
$\chi(x) = \log_q(|x|)$ for all $x \in -q^\Z \cup q^\Z$; \\
\textrm{(ii)} $\kappa \colon \R \rightarrow \R$\index{K@$\kappa(x)$} such that
$\kappa(x) = \sgn(x)\,x^2$ for all $x \in \R$; \\
\textrm{(iii)} $\nu \colon -q^\Z \cup q^\Z  \rightarrow \R^+$\index{N@$\nu(x)$} such that
$\nu(t) = q^{\frac{1}{2}(\chi(t)-1)(\chi(t) -
2)}$ for all $t \in -q^\Z \cup q^\Z$; \\
\textrm{(iv)}  $s \colon \R_0 \times \R_0 \rightarrow \{-1,1\}$\index{S@$s(x,y)$}
is defined such that
$$
s(x,y) = \begin{cases} -1 & \text{if } x > 0 \text{ and } y < 0
\\ 1 & \text{ if } x < 0 \text{ or } y > 0  \end{cases}
$$
for all $x,y \in \R_0=\R\backslash\{0\}$.\\
\textrm{(v)} $\mu : \C\setminus\{0\} \rightarrow \C\setminus\{0\}$\index{M@$\mu$}
such that $\mu(y)=\hf(y+y^{-1})$ for all $y \in \C\setminus\{0\}$.
\end{defn}

For $a,b,z \in \C$, we define
\begin{equation}\label{eq:defPsi}\index{P@$\Psi$}
\Psis{a}{b}{q,\,z}
= \sum_{n=0}^\infty \,\frac{(a;q)_n\,(b\,q^n;q)_\infty}{(q\,;q)_n}\,(-1)^n\,q^{\hf n(n-1)}\, z^n \
= (b;q)_\infty \ \rphis{1}{1}{a}{b}{q,z}.
\end{equation}
This is an entire function in $a,b$ and $z$. Here we have used the standard notation
for basic hypergeometric series \cite{GaspR}, or see Appendix \ref{ssecB:BHS}.

We  use the normalization constant
$c_q = (\sqrt{2}\,q\,(q^2,-q^2;q^2)_\infty)^{-1}$.
Then the following definition is \cite[Def. 3.1]{KoelKCMP},
and the notations as in Definition \ref{defn:chikappanuands}
are used.

\begin{defn} \label{def:functionap} \index{A@$a_p(x,y)$}
If $p \in I_q$, we define the function
$a_p \colon I_q \times I_q \rightarrow \R$
such that $a_p$ is supported on the set
$\{\,(x,y) \in I_q \times I_q \mid \sgn(xy) = \sgn(p)\,\}$
and  is given by
\begin{equation*}
\begin{split}
a_p(x,y)\, =\, & \, c_q\,s(x,y)\,(-1)^{\chi(p)}\,(-\sgn(y))^{\chi(x)}\,|y|\,\,\nu(py/x)\,\,
\,\, \sqrt{\frac{(-\kappa(p),-\kappa(y);q^2)_\infty}{(-\kappa(x);q^2)_\infty}}\,\,
\\ &\times \Psis{-q^2/\kappa(y)}{q^2
\kappa(x/y)}{q^2,\,q^2 \kappa(x/p)}
\end{split}
\end{equation*}
for all $(x,y) \in I_q \times I_q$ satisfying $\sgn(xy) = \sgn(p)$.
\end{defn}

The functions $a_p(x,y)$ for $p,x,y\in I_q$ have been introduced in
\cite[\S 3]{KoelKCMP}, motivated by their occurrence as Clebsch-Gordan coefficients.
Depending on the choices of the sign, these functions can be identified with
well-known special functions of basic-hypergeometric type. In particular, for
$\sgn(x) = \sgn(y)$ the functions $a_p(x,y)$ can be identified with the
$q$-Laguerre polynomials in case $\sgn(x) = \sgn(y)= -$ and with the associated
big $q$-Bessel functions in case $\sgn(x) = \sgn(y)= +$, see \cite{CiccKK}.
The $q$-Laguerre polynomials correspond to an indeterminate moment problem,
and the big $q$-Bessel functions form a complementary orthogonal basis
to the orthogonal polynomials for an explicit solution to the moment
problem corresponding to Ramanujan's ${}_1\psi_1$-summation formula,
see \cite{CiccKK} for details. For
$\sgn(x) = - \sgn(y)$, the functions $a_p(x,y)$ can be matched with
Al-Salam--Carlitz polynomials and $q$-Charlier polynomials, see
\cite{KoekS} for their definition.

For completeness we recall the orthogonality properties of these functions,
see \cite[Prop. 3.2, 3.3]{KoelKCMP}.
For $\theta \in -q^\Z\cup q^\Z$ we define
$\ell_\theta = \{\,(x,y) \in I_q \times I_q \mid y = \theta x \,\}$.\index{L@$\ell_\theta$}

\begin{prop}\label{prop:orthorelapxy}
Consider $\theta \in -q^\Z\cup q^\Z$. Then the family  $\{\,a_p\vert_{\ell_\theta}\, \mid p \in I_q \text{ such that }
\sgn(p) = \sgn(\theta)\,\}$ is an orthonormal basis for $l^2(\ell_\theta)$. In particular,
\[
\sum_{x\in I_q \text{ so that } \theta x\in I_q} a_p(x, \theta x)\, a_r(x, \theta x) = \de_{p,r},
\qquad p,r \in I_q.
\]
\end{prop}

\begin{prop} \label{prop:dualorthorelapxy}
Consider $\theta  \in -q^\Z\cup q^\Z$ and define $J = q^\Z \subset I_q$ if $\theta > 0$ and $J = -q^\N \subset I_q$ if $\theta < 0$. For every $(x,y) \in
\ell_\theta$ we define the function $e_{(x,y)} : J \rightarrow \R$ such that $e_{(x,y)}(p) = a_p(x,y)$ for all $p \in
J$. Then the family $\{\,e_{(x,y)} \mid (x,y) \in \ell_\theta\,\}$ forms an orthonormal basis for $l^2(J)$.
In particular,
\[
\sum_{p\in J} a_p(x, \theta x)\, a_p(y, \theta y) = \de_{x,y}, \qquad x, y \in I_q.
\]
\end{prop}

For convenience we state the following
symmetry relations
for the functions $a_p(x,y)$, see \cite[Prop.~3.5]{KoelKCMP}:
\begin{equation}\label{eq:symmetryforapxy}
\begin{split}
a_p(x,y) & =  (-1)^{\chi(yp)} \sgn(x)^{\chi(x)}
\left|\frac{y}{p}\right| a_y(x,p); \\
a_p(x,y) & =  \sgn(p)^{\chi(p)} \sgn(x)^{\chi(x)}
\sgn(y)^{\chi(y)} a_p(y,x); \\
a_p(x,y) & =  (-1)^{\chi(xp)} \sgn(y)^{\chi(y)}
\left|\frac{x}{p}\right| a_x(p,y).
\end{split}
\end{equation}


\subsection{Summation and transformation formulas for $a_p(x,y)$}

The functions $a_p(x,y)$, which as noted above are closely related to some
well-known orthogonal polynomials of basic hypergeometric type, are
used in the definition of the so-called multiplicative unitary $W$,
see \eqref{eq:Wexplicitinapxy}. In the general theory of locally compact
groups, the multiplicative unitary $W$ plays an important role. In
particular, it satisfies the pentagonal equation, a relation that
is essential in proving Propositions \ref{prop:structureconstQs} and
\ref{prop:hatDeonQppn}. The result in these propositions lead to
operator identities in suitable Hilbert spaces, and taking
matrix coefficients then essentially lead to Theorems
\ref{thm:symmetryforsinglesumfourapxys} and \ref{thm:doublesumfiveapxysequalsthreeapxys}
in this section. The details of the proofs are given in
Section \ref{ssec:summationformula}.

\subsubsection{Representing the structure of $\hat{M}$}
By taking the non-trivial structure constants of Proposition \ref{prop:structureconstQs}
and considering matrix coefficients at both sides we obtain the following
theorem.

\begin{thm}\label{thm:symmetryforsinglesumfourapxys}
For $p_1,p_2,r_1,r_2 \in I_q$, $l,n,m \in \Z$, $\ep,\et\in \{\pm\}$ and
with $z\in I_q$ so that $\sgn(z)=\ep$ and $\ep\et pq^lz\in I_q$ and
with $w\in I_q$ so that $\sgn(w)=\ep\sgn(r_1p_1)$ and $\ep\et \sgn(r_1p_1r_2p_2)pq^{l+m+n}w\in I_q$
we have
\[
\begin{split}
&\sum_{\scriptstyle{\stackrel{\scriptstyle{x\in I_q\text{ so that } \sgn(x)=\sgn(r_1p_1)}}{\text{and }|x|\sgn(r_2p_2)pq^{2l+m+n}\in I_q}}}
a_z(x,w)\, a_x(r_1,p_1)\, a_{|x|\sgn(r_2p_2)pq^{2l+m+n}}(r_2, p_2) \,
\\ & \qquad\quad  \times\, a_{\ep\et pq^lz}(|x|\sgn(r_2p_2)pq^{2l+m+n}, \sgn(r_1p_1r_2p_2)\ep\et pq^{l+m+n}w) =
\de_{|\frac{r_1}{r_2}|p, q^{-2l-m}}\, \de_{|\frac{p_1}{p_2}|p, q^{-2l-2m-n}}  \\
&\times \sum_{\scriptstyle{\stackrel{\scriptstyle{u\in I_q\text{ so that } \sgn(u)=\sgn(r_1)\ep}}{\text{and }\ep\et \sgn(r_1r_2)pq^{l+m}u\in I_q}}}
a_z(r_1,u)\, a_u(p_1, w)\, a_{\ep\et pq^lz}(r_2, \ep\et \sgn(r_1r_2)pq^{l+m}u) \\
&\qquad\quad \times\, a_{\ep\et \sgn(r_1r_2)pq^{l+m}u}(p_2, \sgn(r_1p_1r_2p_2)\ep\et pq^{l+m+n}w),
\end{split}
\]
where the series on both sides converge absolutely.
\end{thm}

\begin{remark}\label{rmk:thmsymmetryforsinglesumfourapxys}
(i) The formula of Theorem \ref{thm:symmetryforsinglesumfourapxys} contains many
special cases involving $q$-Laguerre polynomials, big $q$-Bessel functions,
Al-Salam--Carlitz polynomials and $q$-Charlier polynomials as special cases by suitable specializing
the signs in the formula. Note moreover that in all cases the sums are essentially sums
over $q^\Z$ or $q^\N$. For each particular choice of the signs the square roots
occurring in Definition \ref{def:functionap}
in Theorem \ref{thm:symmetryforsinglesumfourapxys} will cancel or can be taken together.
It would be of interest to find a direct analytic proof. \par\noindent
(ii) As stated before, the functions $a_p(x,y)$ can be interpreted as Clebsch-Gordan
coefficients related to representations of the quantized function algebra, which
has no classical counterpart. For the case of the quantum $SU(2)$ group the
corresponding Clebsch-Gordan coefficients are Wall polynomials, which are
special cases of little $q$-Jacobi polynomials and also can be interpreted as
$q$-analogues of Laguerre polynomials, see \cite{KoorSIAM91}.
The classical Clebsch-Gordan coefficients
also satisfy summation formulas involving the product of
four Clebsch-Gordan coefficients, see e.g. \cite[Ch.~8.7]{VarsMK}, but the structure
of the summations is quite different.
Relations as in Theorems \ref{thm:symmetryforsinglesumfourapxys} and
\ref{thm:doublesumfiveapxysequalsthreeapxys}, if proved directly, might give
a hint of proving directly that the corresponding $q$-analogues of the
Racah coefficients are zero at the appropriate places, leading to a direct
proof of the coassociativity for $M$, see the discussion \cite[p.~289]{KoelKCMP}.
\end{remark}

Theorem \ref{thm:symmetryforsinglesumfourapxys} can be used to obtain
positivity results for sums where the summands have four of the functions
$a_p(x,y)$. The result is contained in Corollary \ref{cor:thmsymmetryforsinglesumfourapxys}.
We give the case corresponding to the $q$-Laguerre polynomials explicitly, and we refer
to Askey \cite[Lecture 5]{Aske} for more information on the related positivity results for
the Laguerre polynomials. The $q$-Laguerre polynomials are defined by,
\begin{equation}\label{eq:defqLaguerrepols}
L_n^{(\al)}(x;q) = \frac{(q^{\al+1};q)_n}{(q;q)_n}\, \rphis{1}{1}{q^{-n}}{q^{\al+1}}{q,-q^{1+\al}x},
\end{equation}
in this application we only consider the case $\al=0$.

\begin{cor}\label{cor:thmsymmetryforsinglesumfourapxys}
For $r_1,r_2 \in I_q$, $l,m \in \Z$ and
with $z\in I_q$ so that $\sgn(z)=\ep$ and $\ep\et |\frac{r_2}{r_1}|q^{-m-l}z\in I_q$ and
we have
\[
\begin{split}
&\qquad\qquad (-\et)^{l+m} (\et\sgn(r_1))^{\chi(r_1)}\,
(\et\sgn(r_2))^{\chi(r_2)} \, (\ep\et)^{\chi(z)}  \\
&\sum_{x\in q^\Z} x^2\, a_x(r_1,r_1)\, a_x(z,z)
\, a_{xq^{-m}|\frac{r_2}{r_1}|}(r_2,r_2)\, a_{xq^{-m}|\frac{r_2}{r_1}|}(\ep\et |\frac{r_2}{r_1}|q^{-m-l}z,\ep\et |\frac{r_2}{r_1}|q^{-m-l}z) >0
\end{split}
\]
and for $a\in \Z$ and $n_1,n_2,n_3,n_4\in \NN$ we have
\[
\sum_{k\in \Z} \frac{q^{k}}{(-q^{k}, -q^{k+a};q)_\infty}
L^{(0)}_{n_1}(q^{k};q)\, L^{(0)}_{n_2}(q^{k};q)\, L^{(0)}_{n_3}(q^{k+a};q)\,
L^{(0)}_{n_4}(q^{k+a};q)\,  > 0.
\]
\end{cor}

Note that the sum is closely related to one of the orthogonality measures for
the $q$-Laguerre polynomials, which correspond to an indeterminate moment
problem. A similar positivity result can be obtained for the $q$-Bessel
functions involved.

\subsubsection{Representing the comultiplication in $\hat{M}$}
The explicit expression for $\hat{\De}$ in the dual quantum group $\hat{M}$
as given in Proposition \ref{prop:hatDeonQppn}, or better the expression  \eqref{eq:generalobsforcomultiplication-2} in the proof of Proposition
\ref{prop:hatDeonQppn}, leads to a formula for its matrix elements.
The result is the following theorem.

\begin{thm}\label{thm:doublesumfiveapxysequalsthreeapxys}
For fixed $r\in q^\Z$, $m_1,m_2,M,n\in \Z$, $p_1,p_2\in I_q$,
$\ep_1,\ep_2, \et_1,\et_2,\si\in \{\pm\}$ and for $z_1,z_2,w_1, w_2 \in I_q$
satisfying
\[
\begin{split}
&\sgn(z_i)=\ep_i, \ (i=1,2), \quad \ep_1\et_1q^{m_1}r z_1\in I_q, \quad
\ep_2\et_2q^{-2m_1-m_2-n}\frac{z_2|p_2|}{r|p|}\in I_q, \\
&\sgn(w_1) =\sgn(p_1)\ep_1, \quad \sgn(w_2)=\si \ep_2, \qquad
\si\sgn(p_1)\ep_1\et_1 q^{m_1+M}r_1w_1\in I_q, \\
&\si\sgn(p_2)\ep_2\et_2 q^{-2m_1-m_2-M}\frac{w_2|p_2|}{r|p|}\in I_q
\end{split}
\]
and such that $a_{z_1}(p_1,w_1)\not= 0$
we have
\[
\begin{split}
&\frac{1}{w_2^2} a_{ep_1\et_1q^{m_1}rz_1} (\si |p_1|rq^{2m_1+M}, \ep_1\et_1\si \sgn(p_1) w_1r q^{m_1+M})
\, a_{z_2}(\si |p_1|rq^{2m_1+M}, w_2) \\
&\qquad\qquad \times \, a_{\ep_2\et_2|\frac{p_2}{p_1}|\frac{z^2}{r} q^{-2m_1-m_2-n}} (p_2, \ep_2\et_2\si \frac{p_2w_2}{|p_1|r}
q^{-2m_1-m_2-M})
= \\
&\sum_{\stackrel{\scriptstyle{y,x\in I_q\text{ so that }\sgn(y)=\ep_2\et_1}\text{ and}}
{\scriptstyle{\sgn(p_1p_2)q^nxw_1/z_1\in I_q, \ \ep_1\ep_2\et_1\et_2q^{-m_1-m_2}yx/rz_1\in I_q}}}
\hskip-.3truecm\frac{1}{y^2} \, a_{z_2}(\ep_1\et_1q^{m_1}rz_1,y) \, a_{w_2}(\sgn(p_1)\si \ep_1\et_1q^{m_1+M}rw_1,y)
\\ &\times  \,
a_{\ep_2\et_2\frac{z_2|p_2|}{r|p_1|}q^{-2m_1-m_2-n}}(x,\ep_1\ep_2\et_1\et_2q^{-m_1-m_2}\frac{yx}{rz_1})
\, a_x(p_2,\sgn(p_1p_2)q^n\frac{xw}{z_1}) \\
&\times \, a_{\si\sgn(p_2)\ep_2\et_2\frac{w_2|p_2|}{r|p_1|}q^{-2m_1-m_2-M}} (\sgn(p_1p_2)q^n\frac{xw_1}{z_1},
\ep_1\ep_2\et_1\et_2q^{-m_1-m_2}\frac{yx}{rz_1})
\end{split}
\]
where the left-hand-side is considered to be zero in case $\si |p_1|rq^{2m_1+M}\notin I_q$.
The series converges absolutely.
\end{thm}

\begin{remark}\label{rmk:doublesumfiveapxysequalsthreeapxys}
(i) First note that the largest part of Remark \ref{rmk:thmsymmetryforsinglesumfourapxys}(i)
is also applicable to Theorem \ref{thm:doublesumfiveapxysequalsthreeapxys}, except for the
fact that the summation is more involved. Viewing the summation as a sum over an
area in $I_q\times I_q \subset \R^2$ (with $x$ on the horizontal axis and $y$ on the vertical axis),
we see that the summation area is a subset of $I_q\times I_q$ bounded by a vertical line and
a hyperbola. Depending on the sign choices there are eight possibilities
for the location of the vertical line and the hyperbola.  \par\noindent
(ii) Theorem \ref{thm:doublesumfiveapxysequalsthreeapxys} follows from
the operator identity in Proposition \ref{prop:hatDeonQppn}, but the single term in
the left hand side of Theorem \ref{thm:doublesumfiveapxysequalsthreeapxys} corresponds
to summation on the left hand side of Proposition \ref{prop:hatDeonQppn}, whereas
the double sum on the right hand side of Theorem \ref{thm:doublesumfiveapxysequalsthreeapxys}
corresponds to the single term on the right hand side of Proposition \ref{prop:hatDeonQppn}.
\par\noindent
(iii) Since the results in Theorems \ref{thm:symmetryforsinglesumfourapxys} and
\ref{thm:doublesumfiveapxysequalsthreeapxys} both reflect the pentagonal equation for
the multiplicative unitary, one might expect the resulting identities to be equivalent
by using the orthogonality relations of Propositions \ref{prop:orthorelapxy} and
\ref{prop:dualorthorelapxy}. However, this is not the case as follows by considering
the dependence of both results on the free parameters.
\end{remark}


\subsection{Formulas involving ${}_2\vp_1$-series}\label{ssec:formulas2phi1}

In Section \ref{sec:generatorshatM} we show that with respect to the spectral
decomposition of the Casimir operator $\Om$ the operators $Q(p_1,p_2,n)$
generating $\hat{M}$, see Proposition \ref{prop:QppngeneratedualM},
act by multiplication by a ${}_2\vp_1$-series up to a sign-change in the
argument. Since we also have another explicit expression for the
action of $Q(p_1,p_2,n)$ by Lemma \ref{lem:explicitactionQppn}, we have two different
explicit expressions for the action of $Q(p_1,p_2,n)$.
This leads
to the following theorem, where the functions $\Psi$ are
essentially ${}_1\vp_1$-functions as defined in \eqref{eq:defPsi}. Actually,
we have written out two of several options depending on several sign choices.

\begin{thm} \label{thm:introsummationformula1}
Let $m,n \in \Z$, $p_1,p_2 \in q^\Z$ and $\la \in \T$.
\begin{enumerate}[(i)]
\item For $k \in \Z$,
\[
\begin{split}
&\sum_{l=-\infty}^\infty (-1)^{l+k+n} \Big(p_2^2q^{2n-2k-3}\Big)^l q^{l^2} (-q^{2l-2m}p_2^2/p_1^2;q^2)_\infty \\
& \qquad \times (q^{2-2m-2n};q^2)_\infty \rphis{2}{1}{q^{1-n}p_1\la/p_2, q^{1-n}p_1/p_2\la}{q^{2-2m-2n}}{q^2, -q^{2-2l}} \\
& \qquad\times \Psis{-q^{2-2l}}{q^{2+2k-2l}}{q^2,q^{2+2k}/p_1^2} \Psis{-q^{2-2l+2m}p_1^2/p_2^2}{q^{2+2k-2n-2l}}{q^2, q^{2+2k-2m-2n}/p_1^2}\\
&= p_2^{2k} q^{2n-3k} q^{-k^2} (q^2, -q^2/p_2^2;q^2)_\infty \frac{ (p_1q^{1-n}\la/p_2,p_1q^{1-n}\la/p_2;q^2)_n }{ (p_2q^{1-n}\la/p_1,p_2q^{1-n}\la/p_1;q^2)_n } \\
& \qquad \times (q^{2-2n};q^2)_\infty \rphis{2}{1}{p_2q^{1-n}\la/p_1, p_2q^{1-n}/p_1\la}{q^{2-2n}}{q^2, -q^2/p_2^2} \\
& \qquad\times (q^{2-2m};q^2)_\infty \rphis{2}{1}{p_1q^{1+n}\la/p_2, p_1q^{1+n}/p_2\la}{q^{2-2m}}{q^2, -q^{2-2k}}
\end{split}
\]
where the sum converges absolutely.
\item Assume $q^{-m}p_2/p_1 \leq 1$ and $q^{-m-n}p_2/p_1 \leq 1$, then for $k \in \NN$,
\[
\begin{split}
& \sum_{l=0}^\infty \frac{p_2^{2l}q^{2(k-l)} q^{(l-k)(l-k-1)}}{(q^2;q^2)_l} \rphis{3}{2}{q^{-2l}, q^{1+n}p_2\la/p_1, q^{1+n}p_2/p_1 \la}{q^{2-2m}p_2^2/p_1^2, 0}{q^2,q^2} \\
& \qquad \times \Psis{q^{-2l}}{q^{2+2k-2l}}{q^2, -q^{4+2k}/p_1^2} \Psis{q^{2m-2l}p_1^2/p_2^2}{q^{2+2k-2n-2l}}{q^2, -q^{4+2k-2m-2n}/p_1^2} \\
&= q^{2n(k-m+1)}  q^{-n(n-1)} p_1^{2k-2n}(q^{2m}p_1^2/p_2^2;q^2)_n (q^{2+2k},-q^2/p_2^2;q^2)_\infty \\
& \qquad \times (q^{2+2n};q^2)_\infty \rphis{2}{1}{q^{1-n}p_2\la/p_1, q^{1-n}p_2/p_1 \la}{q^{2+2n}}{q^2, -q^2/p_2^2} \\
& \qquad \times \rphis{3}{2}{q^{-2k}, q^{1-n}p_2 \la/p_1, q^{1-n}p_2/p_1\la}{q^{2-2m-2n}p_2^2/p_1^2,0}{q^2,q^2}
\end{split}
\]
where the sum converges absolutely.
\end{enumerate}
\end{thm}

\begin{remark}\label{rmk:thmintrosummationformula1}
(i) The $_2\varphi_1$-function inside the sum in Theorem \ref{thm:introsummationformula1}(i) is essentially the little $q$-Jacobi function  $f_l(\mu(\la);q^{2-2m-2n},q^{1-n}p_1/p_2;-q^2|q^2)$, see \eqref{eq:2ndordqdiffeqBHS}, and the summations formula remains valid if $\mu(\la)$ is a discrete mass point of the corresponding orthogonality measure $\nu$,
see Appendix \ref{ssecB:litteqJacobi}.
In Theorem \ref{thm:introsummationformula1}(ii) the ${}_3\vp_2$-series is essentially an Al-Salam--Chihara polynomial, and the same
remark applies using the orthogonality measure described in Appendix \ref{ssecB:AlSCpol}.  Note that
the ${}_3\vp_2$-series can be transformed to a ${}_2\vp_1$-series by \eqref{eq:AlSCpol}.
\par\noindent
(ii) If we multiply the formula (i) by  $f_{l'}(\mu(\la);q^{2-2m-2n},q^{1-n}p_1/p_2;-q^2|q^2)$ and we use the orthogonality relations, see Appendix \ref{ssecB:litteqJacobi}, it follows that the above identity is equivalent to an integral identity of the form $\int \, {}_2\varphi_1\, {}_2\varphi_1 \, {}_2\varphi_1\, d\nu = \Psi \, \Psi$. The integral can be written as an integral over $[-1,1]$ plus an infinite sum. The same remark applies
for (ii) but this time using the orthogonality relations, see Appendix \ref{ssecB:AlSCpol}, for the
Al-Salam--Chihara polynomials. \par\noindent
(iii) Note that we can view the $\Psi$-functions as $q$-analogues of the Bessel function, cf.~the
discussion in Section \ref{ssec:defsomespecialf},
and since we can do the same for the ${}_2\vp_1$-series involved in (i) we
may also consider Theorem \ref{thm:introsummationformula1}(i) as an identity for $q$-Bessel functions.
\end{remark}

The following result follows from the structure constants formula of Proposition
\ref{prop:structureconstQs}. Note that Theorem \ref{thm:symmetryforsinglesumfourapxys} also
follows from Proposition \ref{prop:structureconstQs}, but now we use again the fact that we
can realize $Q(p_1,p_2,n)$ as multiplication operators by a ${}_2\vp_1$-series up to a sign-change in the
argument.

\begin{thm} \label{thm:introsumfromstructureformula}
Let $\la \in \T$, $p_1,p_2,r_1,r_2 \in I_q$, $n,m \in \Z$, and assume that $|\frac{p_2}{p_1}|=q^m$ and $|\frac{r_1}{r_2}|=q^n$. Then
\[
\begin{split}
&\sgn(r_1)^{\hf(1-\sgn(p_1))} \sgn(r_2)^{\hf(1-\sgn(p_2))+n} r_2^m p_2^n |r_1r_2| \nu(r_1)\nu(r_2)\nu(p_1)\nu(p_2) \\
\times &  (q^2, -\sgn(r_1)r_1^2, -\sgn(r_2)r_2^2, -\sgn(r_2)q^2/r_2^2, -\sgn(p_2)q^2/p_2^2;q^2)_\infty \\
\times & \frac{ (-\sgn(r_1p_1) q^{-m-n-1}/\la, -\sgn(r_1p_1) q^{3+m+n}\la, -\sgn(r_1r_2)\la q^{3-n}/p_1p_2;q^2)_\infty } { (-\sgn(r_1r_2p_1p_2) q^{m+n-1}/\la, -\sgn(r_1r_2p_1p_2)q^{1-m-n} \la, -\sgn(r_1r_2)p_1|p_2|q^{-n-1}/\la, ;q^2)_\infty }\\
\times & \frac{ (-\sgn(r_1r_2)p_1p_2 q^{n-1}/\la, -\la q^{3-m}/r_1r_2, -r_1r_2q^{m-1}/\la;q^2)_\infty}{ (-\sgn(r_1r_2)\la q^{3+n}/p_1|p_2|, -r_1|r_2|q^{-m-1}/\la, -\la q^{m+3}/r_1|r_2|;q^2)_\infty} \\
\times & (\sgn(p_1p_2)q^{2+2n};q^2)_\infty \rphis{2}{1}{ \sgn(r_1r_2)p_2 q^{1+n}/ p_1 \la, \sgn(r_1r_2)p_2 q^{1+n} \la/ p_1}{ \sgn(p_1p_2)q^{2+2n} }{q^2, -\sgn(p_2)\frac{q^2}{p_2^2} } \\
\times & (\sgn(r_1r_2)q^{2+2m};q^2)_\infty \rphis{2}{1}{ r_2 q^{1+m}/ r_1 \la,r_2 q^{1+m} \la/ r_1}{ \sgn(r_1r_2)q^{2+2m} }{q^2, -\sgn(r_2)\frac{q^2}{r_2^2} } \\
=&\sum_{(x_1,x_2) \in \mathcal A}
x_2^{m+n} |x_1|^2 \nu(x_1)^2 \nu(x_1p_1/r_1) \nu(x_2p_2/r_2) (\sgn(r_2p_2)q^{-2m-2n})^{\chi(x_1)}\\
\times & (-\sgn(r_1p_1)x_1^2, -\sgn(r_2p_2) x_1^2, -\sgn(r_2p_2)q^2/x_1^2, \sgn(r_1r_2p_1p_2)q^{2+2m+2n};q^2)_\infty \\
\times & \rphis{2}{1}{ \sgn(r_1r_2p_1p_2) q^{1+m+n}/ \la, \sgn(r_1r_2p_1p_2) q^{1+m+n} \la} {\sgn(r_1r_2p_1p_2)q^{2+2m+2n}}{q^2, -\sgn(r_2p_2)\frac{q^2}{x_1^2}} \\
\times & \Psis{-\sgn(p_1)q^2/p_1^2}{ \sgn(r_1p_1)q^2r_1^2/p_1^2 }{q^2, \sgn(p_1)\frac{q^2r_1^2}{x_1^2}}  \Psis{-\sgn(p_2)q^2/p_2^2}{ \sgn(r_2p_2)q^2r_2^2/p_2^2 }{q^2, \sgn(p_2)\frac{q^2r_2^2}{x_1^2}}
\end{split}
\]
where the sum converges absolutely.
Here $\mathcal A \subset I_q \times I_q$ is given by
\[
\mathcal A = \Big\{ (x_1,x_2) \in I_q \times I_q \mid \sgn(x_1) = \sgn(p_1r_1),\ \sgn(x_2) = \sgn(p_2r_2),\  |x_1|=|x_2|\Big\}.
\]
\end{thm}
From Theorem \ref{thm:introsumfromstructureformula} we obtain another positivity result.
\begin{cor} \label{cor:positivity}
Let $p_1,p_2 \in I_q$ and $\la \in \T$, then
\[
\begin{split}
0< \sum_{x \in q^\Z} &  \nu(x)^2  (-x^2;q^2)_\infty  \rphis{2}{1}{ q/ \la, q\la} {q^{2}}{q^2, -\frac{q^2}{x^2}} \\
& \times  \Psis{-\sgn(p_1)q^2/p_1^2}{ q^2 }{q^2, \sgn(p_1)\frac{q^2p_1^2}{x^2}}  \Psis{-\sgn(p_2)q^2/p_2^2}{ q^2 }{q^2, \sgn(p_2)\frac{q^2p_2^2}{x^2}}.
\end{split}
\]
\end{cor}

\subsection{Biorthogonality relations for $_2\varphi_1$-functions}
We have explicit expressions for the matrix elements of the principal series corepresentations $W_{p,x}$, $p\in q^\Z$, $x=\mu(\la) \in [-1,1]$, in terms of $_2\varphi_1$-functions. Unitarity of $W_{p,x}$ leads to orthogonality relations for the matrix elements. By analytic continuation these orthogonality relations remain valid for other values of $\la$.

Let $m \in \Z$ and $\la \in \C\setminus\{0\}$, and define $s(\,\cdot\,,\,\cdot\,;\la,m):I_q \times I_q \to \C$ by
\[
\begin{split}
s(p_1,p_2;\la,m) =&\, c_q^2 |p_1p_2| p_2^{\chi(p_1p_2)+m}
\nu(p_1p_2q^{m+1}) \nu(p_1) \nu(p_2) \sqrt{
(-\ka(p_2),-\ka(p_2);q^2)_\infty } \\
& \times \frac{(q^2,-q^2/\ka(p_2),-\la q^{3-m}/p_1^2 \ka(p_2),
-p_1^2\ka(p_2) q^{m-1}/\la, q^{1-m}/\ka(p_2)\la;q^2)_\infty
}{(\ka(p_1)q^{1+m}/\la,-q^{-m-1}/\la, -q^{3+m}\la;q^2)_\infty} \\
& \times (\ka(p_1p_2)q^{2+2m};q^2)_\infty \rphis{2}{1}{\ka(p_2)
q^{1+m}\la,  \ka(p_2)
q^{1+m}/\la}{\ka(p_1p_2)q^{2+2m}}{q^2,-\frac{q^2}{\ka(p_2)}},
\end{split}
\]
for $p_1,p_2 \in I_q$. From this expression it is not clear that the function is defined
for all values of $p_2\in I_q$, but an application of Jackson's transformation formula
\cite[(III.4)]{GaspR} shows how to extend to all values of  $p_2\in I_q$.

\begin{thm} \label{thm:biorthogonality}
The following biorthogonality relations hold:
\[
\begin{split}
\sum_{p_1 \in I_q} s(p_1,p_2;\la,m) s(p_1,p_2';\la^{-1},m) &= \de_{p_2,p_2'},\\
\sum_{p_2 \in I_q} s(p_1,p_2;\la,m) s(p_1',p_2;\la^{-1},m) &= \de_{p_1,p_1'}.
\end{split}
\]
\end{thm}

\begin{remark}\label{rmk:thmbiorthogonality}
The two biorthogonality relations  Theorem \ref{thm:biorthogonality} are actually equivalent.
Also, for $\la \in \T$ the biorthogonality relations are orthogonality relations.
\end{remark}

\section*{Proofs}

\section{Extensions of the generators of $\su$}\label{sec:extensionsgeneratorssu}

\subsection{Decomposition of the GNS-space} \label{ssec:decompositionofGNSspace}
The operators $K_0$ and $E_0$, and therefore also $\Omega_0$, are defined on the dense subspace $\mathcal K_0$ of the Hilbert space $\mathcal K$ of the GNS-construction for the left-invariant weight $\vp$. In order to obtain the \emph{right} closures of the operators $K_0, E_0$, $\Omega_0$ we first give a convenient decomposition of $\mathcal K_0$.

Let $p \in q^\Z$, $m \in \Z$ and $\ep,\et \in \{-,+\}$, and define
\begin{equation}\label{eq:defKpmepeta}
\begin{split}
J(p,m,\ep,\et) &
= \{\,z\in I_q\mid \ep\et\, q^mpz\in I_q \text{ and } \sgn(z)=\ep\}, \\
\cK_0(p,m,\ep,\et) &
= \mathrm{span} \{\, f_{-m,\ep\et\,q^m p z, z} \mid z \in J(p,m,\ep,\et) \, \}.
\end{split}
\end{equation}\index{J@$J(p,m,\ep,\et)$}\index{K@$\cK_0(p,m,\ep,\et)$}
We denote by $\cK(p,m,\ep,\et)$\index{K@$\cK(p,m,\ep,\et)$} the closure of $\cK_0(p,m,\ep,\et)$ inside $\cK$. Then $\cK(p,m,\ep,\et)\cong \ell^2(J(p,m,\ep,\et))$, and we consider $v \in \cK(p,m,\ep,\et)$ as a function $v \colon J(p,m,\ep,\et) \to \C$ by setting
\begin{equation}\label{eq:conventioninnerproductfunction}
v(z) = \langle v , f_{-m,\ep\et\,q^m pz,z} \rangle,
\qquad z \in J(p,m,\ep,\et).
\end{equation}
By convention, for $z \in \pm q^\Z \setminus
J(p,m,\ep,\et)$ we set $v(z) = 0$. Note that $J(p,m,\ep,\et) = I_q^+=q^\Z$
if $\ep=\et=+$. If $\ep = -$ or $\et = -$, then $J(p,m,\ep,\et)$ is a bounded $q$-halfline with $0$ as only accumulation point. In this case $J(p,m,\ep,\et)$ is of the form $\ep C(p,m) q^\N$ for some
constant $C(p,m)\in q^\N$ depending on $p$ and $m$. Explicitly,
\[
C(p,m)\, =\, \begin{cases} 1, &\text{if } (\ep=-, \et=+) \text{ or }
(\ep=-, \et=- \text{ and }  q^mp\leq 1),  \\
q^{-m}p^{-1} &\text{if } (\ep=+, \et=-) \text{ or } (\ep=-, \et=- \text{ and }
q^mp\geq 1). \end{cases}
\]
In particular, the sign of the bounded $q$-halfline is determined by $\ep$.
Note that for the modular conjugation $J$ we have
\begin{equation}\label{eq:JonKpmepet}
J\colon \cK(p,m,\ep,\et) \to \cK(\frac{1}{p}, -m ,\et,\ep),\qquad
J\, f_{-m,\ep\et q^mpz,z} = f_{m,\et\ep q^{-m}p^{-1} (\ep\et q^mpz), \ep\et q^m p z}.
\end{equation}

We have an algebraic direct sum decomposition
\[
\cK_0 = \bigoplus_{\substack{\ep,\et \in \{-,+\}\\p \in q^\Z,m \in \Z}} \mathcal K_0(p,m,\ep,\et).
\]
By Definition \ref{def:actiongeneratorssu} and \eqref{eq:dualE0forsu} the actions of $K_0$, $E_0$ and $E_0^\dagger$ on the basis elements of $\mathcal K_0(p,m,\ep,\et)$ are given explicitly by
\begin{align}
K_0 f_{-m,\ep,\et q^m p z, z}  = &q^m \sqrt p f_{-m,\ep,\et q^m p z, z}, \nonumber \\
(q-q^{-1})\, E_0 f_{-m,\ep,\et q^m p z, z} = &\,
\ep q^m (pq)^\hf \sqrt{1+\ep z^2 q^{-2}}\, f_{-m-1, \ep\et q^{m+1}p (z/q), z/q} \label{eq:defE0onbasisKpmepet} \\
&\,  - \et q^{-m} (pq)^{-\hf} \sqrt{1+\et q^{2m} p^2 z^2}\,
f_{-m-1, \ep\et q^{m+1}p z, z},\nonumber \\
(q-q^{-1})\, E_0^\dag f_{-m,\ep,\et q^m p z, z} = &\,
\ep q^m (p/q)^\hf \sqrt{1+\ep z^2}\, f_{-m+1, \ep\et q^{m-1}p (zq), zq} \label{eq:defE0dagonbasisKpmepet} \\
&\,  - \et q^{-m} (p/q)^{-\hf} \sqrt{1+\et q^{2m-2} p^2 z^2}\,
f_{-m+1, \ep\et q^{m-1}p z, z}, \nonumber
\end{align}
so that
\begin{equation}\label{eq:E0E0dagKoncK0}
\begin{split}
&K_0= q^m \sqrt p\, \Id  \colon \cK_0(p,m,\ep,\et) \to \cK_0(p,m,\ep,\et),\\
&E_0 \colon \cK_0(p,m,\ep,\et) \to \cK_0(p,m\!+\!1,\ep,\et),\\
&E_0^\dag \colon \cK_0(p,m,\ep,\et) \to \cK_0(p,m\!-\!1,\ep,\et),\\
&\Om_0 \colon \cK_0(p,m,\ep,\et) \to \cK_0(p,m,\ep,\et).
\end{split}
\end{equation}
For the action of $\Om_0$ on the basis elements, see \eqref{eq:actioncasimir}.

\begin{proof}[Proof of Proposition \ref{prop:relationssuaresatisfied}.]
We need to show that the relations
\[
K_0E_0= q E_0K_0, \qquad E_0^\dagger E_0 - E_0 E_0^\dagger = \frac{ K_0^2 - K_0^{-2} }{q-q^{-1}}
\]
are valid when acting on the basis elements $f_{mpt}$ of $\mathcal K_0$. The first relation follows immediately from \eqref{eq:E0E0dagKoncK0}. Using \eqref{eq:defE0forsu} and \eqref{eq:dualE0forsu}, we obtain
\[
\begin{split}
(q-q^{-1})^2\,& \bigl(\, E_0^\dag E_0 - E_0^\dag E_0\,\bigr)\,f_{mpt}
\\  =&\, \Bigl[\ q^m\,|\frac{t}{p}|\,
\bigl(q^{-1}(1+\kappa(p)) - q (1+ \kappa(q^{-1} p))\bigr)\\
&\quad + q^{-m}\,|\frac{p}{t}|\bigl(\,q(1+\kappa(q^{-1}t)) -
 q^{-1} (1+ \kappa(t))\,\bigr)\Bigr] f_{mpt}
\\  =&\, (q - q^{-1}) \,
\Bigl[ q^{-m}\,|\frac{p}{t}| - q^m \, |\frac{t}{p}|\,\Bigr]\,\,f_{mpt},
\end{split}
\]
and then Definition \ref{def:actiongeneratorssu} proves the second relation.

In order to prove the linear independence of the operators $K_0^n E_0^k (E_0^\dagger)^l$, $n \in \mathbb Z$, $k,l \in \mathbb N_0$, we assume that the sum
\[
\sum_{\substack{n \in \mathbb Z \\ k,l \in \mathbb N_0}} c_{nkl} K_0^n E_0^k (E_0^\dag)^l
\]
with only finitely many non-zero coefficients $c_{nkl}$, equals zero as operator on $\cK_0$. By \eqref{eq:E0E0dagKoncK0} we have
\[
K_0^n E_0^k (E_0^\dag)^l = \bigl(q^{m-l+k}p^\hf\bigr)^n
E_0^k (E_0^\dag)^l \colon \cK_0(p,m,\ep,\et)
\to \cK_0(p,m-l+k,\ep,\et).
\]
So for fixed $r \in \Z$, the sum $\sum_{k-l=r} c_{n,k,l}\bigl(q^{m+r}p^\hf\bigr)^n E_0^k (E_0^\dag)^l\colon \cK_0(p,m,\ep,\et)\to \cK_0(p,m+r,\ep,\et)$ equals zero. We fix such an $r$ and we take $\ep=+=\et$. From \eqref{eq:defE0onbasisKpmepet} and \eqref{eq:defE0dagonbasisKpmepet} we see that $E_0^k (E_0^\dag)^{k-r} f_{-m,q^mpz,z}= \sum_{s=-k}^{k-r} a_s^{mpz} f_{-m-r,q^{m+r}pzq^s,zq^s}$ for certain coefficients $a_s^{mpz}$. Let $k_0$ be the maximum of the $k$'s such that $c_{n,k,k-r}\not=0$, then it follows that
\[
\begin{split}
0=\Bigg\langle
&\, \sum_{\substack{n \in \Z\\k,l \in \mathbb N_0, k-l=r}} c_{n,k,l} \bigl(q^{m+r}p^\hf\bigr)^n E_0^k (E_0^\dag)^{l} \,
f_{-m,q^mpz,z}\,,\, f_{-m-r,q^{m+r}pzq^{-k_0},zq^{-k_0}}\Bigg\rangle \\
=&\, \sum_{n \in \Z} c_{n,k_0,k_0-r} \bigl(q^{m+r}p^\hf\bigr)^n \Big\langle E_0^{k_0} (E_0^\dag)^{k_0-r} \, f_{-m,q^mpz,z}, f_{-m-r,q^{m+r}pzq^{-k_0},zq^{-k_0}}\Big\rangle
\end{split}
\]
The coefficient $a_{-k_0}^{mpz}=\langle E_0^{k_0} (E_0^\dag)^{k_0-r}
f_{-m,q^mpz,z}, f_{-m-r,q^{m+r}pzq^{-k_0},zq^{-k_0}}\rangle$ can be
explicitly calculated from \eqref{eq:defE0onbasisKpmepet} and
\eqref{eq:defE0dagonbasisKpmepet} as the product of
$\langle (E_0^\dag)^{k_0-r} f_{-m,q^mpz,z}, f_{-m+k_0-r,q^{m-k_0+r}pzq,z}\rangle$ and
$\langle E_0^{k_0} f_{-m+k_0-r,q^{m-k_0+r}pzq,z}, f_{-m-r,q^{m+r}pzq^{-k_0},zq^{-k_0}}\rangle$.
These matrix coefficients are non-zero for all $p$, $m$, $z$ and can be calculated
explicitly in terms of $q$-shifted factorials. This leaves us with the identity $\sum_{n \in \Z} c_{n,k_0,k_0-r} \bigl(q^{m+r}p^\hf\bigr)^n=0$ for all $m$ and $p$,
from which we conclude that the coefficients $c_{n,k_0,k_0-r}$ are zero.
\end{proof}

After these considerations we can start considering the closures of $E_0$ and $K_0$. From the results in Appendix \ref{ssecA:summationoperators} it follows that the closure of $K_0$ is given by the direct sum of $q^m\sqrt{p}\,\,\text{Id}\big\vert_{\cK_0(p,m,\ep,\et)}$, see also \eqref{eq:defKanddomain}.

Let us now consider the closure of $E_0$. Since
\begin{equation}\label{eq:E0andE0ddagareadjointonK0}
\langle E_0\, v, w\rangle = \langle v, E_0^\dag \, w\rangle , \qquad
\forall\, v,w \in \cK_0,
\end{equation}
we see that $\cK_0 \in D(E_0^\ast)$, so that $E_0^\ast$ is
densely defined. This means that $E_0$ is closable, and
its closure is $E=(E_0^{\ast})^{\ast}$, and similarly for
$E_0^\dag$.
From \eqref{eq:E0andE0ddagareadjointonK0} one obtains
\begin{equation}
E_0 \subset (E_0^\dag)^\ast \ \Longrightarrow\  E\subset (E_0^\dag)^\ast, \qquad
\text{and} \ E_0^\dag \subset E_0^\ast.
\end{equation}
Moreover, putting $E^{\ep,\et}_{p,m} =  E\big\vert_{\cK(p,m,\ep,\et)}$ we see from the explicit action \eqref{eq:defE0onbasisKpmepet} of $E_0$ on the basis elements of $\mathcal K_0$
that the closure of $E_0\vert_{\cK_0(p,m,\ep,\et)}$ gives
$E^{\ep,\et}_{p,m}\colon \cK(p,m,\ep,\et) \to \cK(p,m+1,\ep,\et)$. It follows
that
\begin{equation} \label{eq:decompE}\index{E@$E^{\ep,\et}_{p,m}$}
E = \bigoplus_{\substack{\ep,\et\in \{-,+\}\\ p\in q^\Z,\, m\in\Z}} E^{\ep,\et}_{p,m},
\end{equation}
and so $v\in D(E)$ if and only if $P^{\ep,\et}_{p,m}v\in D(E^{\ep,\et}_{p,m})$
for all  $\ep,\et\in \{\pm\}$, $p\in q^\Z$, $m\in\Z$, where
$P^{\ep,\et}_{p,m}\in B(\cK)$ is the orthogonal projection onto
$\cK(p,m,\ep,\et)$, see Appendix \ref{ssecA:summationoperators}. The operator $E^\ast$, and the closures of $E_0^\dag$ and $(E_0^\dag)^*$ have similar decompositions.

Examining coefficients in \eqref{eq:defE0onbasisKpmepet}, we see that $E_0\big\vert_{\cK_0(p,m,\ep,\et)}$ extends to a bounded operator $E^{\ep,\et}_{p,m} \colon \cK(p,m,\ep,\et)
\to \cK(p,m+1,\ep,\et)$ unless $\ep=+=\et$.
Similarly, from \eqref{eq:defE0dagonbasisKpmepet} it follows that $E_0^\dag\big\vert_{\cK_0(m,p,\ep,\et)}$ extends to a bounded
operator $\cK(p,m,\ep,\et) \to \cK(p,m-1,\ep,\et)$ unless $\ep=+=\et$, and this bounded operator
is indeed equal to the adjoint $(E^{\ep,\et}_{p,m-1})^\ast \colon \cK(p,m,\ep,\et) \to \cK(p,m-1,\ep,\et)$. The case $\ep=+=\et$ is more delicate, and we study this case later on in Section \ref{ssec:AffiliationOfKandE}.

\subsection{The multiplicative unitary and related operators}
\label{ssec:multiplicativeunitary}
Next we study the operators $Q(p_1,p_2,n) \in B(\cK)$, defined by
\eqref{eq:defQppn}, restricted to the subspaces $\mathcal K(p,m,\ep,\et)$. The definition of the operators $Q(p_1,p_2,n)$ involves the multiplicative unitary $W\in B(\cK\ot \cK)$. A for our purposes useful description of $W$ in terms of the functions $a_p(\cdot,\cdot)$ can be found in \cite[Prop.~4.5, 4.10]{KoelKCMP};
\begin{equation}\label{eq:Wexplicitinapxy}
\begin{split}
W^\ast (f_{m_1,p_1,t_1}\ot f_{m_2,p_2,t_2})
= &\sum_{\stackrel{\scriptstyle{y,z\in I_q}}
{\scriptstyle{\sgn(p_2t_2) yzq^{m_2}/p_1\in I_q}}}
\left| \frac{t_2}{y}\right|
a_{t_2}(p_1,y) a_{p_2}(z, \sgn(p_2t_2) yzq^{m_2}/p_1)\, \\
\ &\qquad \times f_{m_1+m_2-\chi(p_1p_2/t_2z),z,t_1} \ot
f_{\chi(p_1p_2/t_2z),  \sgn(p_2t_2) yzq^{m_2}/p_1,y}.
\end{split}
\end{equation}
The functions $a_p(\cdot,\cdot)$ are defined in Definition \ref{def:functionap}.
For convenience we state the corresponding result for $W$ as well, which
follows directly from \eqref{eq:Wexplicitinapxy}:
\begin{equation}\label{eq:Wactionexplicitinapxy}
\begin{split}
W (f_{m_1,p_1,t_1}\ot f_{m_2,p_2,t_2})
= &\sum_{\stackrel{\scriptstyle{r, s\in I_q}}
{\scriptstyle{\sgn(rp_2t_2) sp_1q^{m_2}\in I_q}}}
\left| \frac{s}{t_2}\right|
a_{s}(\sgn(rp_2t_2)sp_1q^{m_2}, t_2)\,  a_{r}(p_1,p_2)\, \\
\ &\qquad \times f_{m_1-\chi(sp_2/t_2),\sgn(rp_2t_2)sp_1q^{m_2},t_1} \ot
f_{m_2+\chi(sp_2/t_2), r,s}.
\end{split}
\end{equation}

\begin{lemma}\label{lem:explicitactionQppn}
Let $p\in q^\Z$, $p_1,p_2 \in I_q$, $n,m \in \Z$, and $\ep,\et \in \{-,+\}$. If $q^{2m}p\not= q^{-n}|p_2/p_1|$, then
\[
Q(p_1,p_2,n) \bigl( \cK(p,m,\ep,\et)\bigr)=\{0\}.
\]
If $q^{2m}p = q^{-n}|p_2/p_1|$, then
\[
Q(p_1,p_2,n) \colon \cK(p,m,\ep,\et) \to \cK(p,m+n,\sgn(p_1)\ep,\sgn(p_2)\et),
\]
and $Q(p_1,p_2,n)$ is given explicitly by
\[
\begin{split}
&Q(p_1,p_2,n)f
=  (-1)^{m'} (\et')^{\chi(p_1p_2)+m} \frac{|p_1p_2|}{q^m p} \\
&\times\sum_{w\in J(p,m',\ep',\et')}
\Bigl(\frac{(\ep'\et')^{\chi(w)}}{|w|}
 \sum_{z \in J(p,m,\ep,\et)}\, \frac{f(z)}{|z|}\,\,
a_{p_1}(z,w)\,a_{p_2}(\ep\,\et\,q^m p\, z ,
\ep' \et' q^{m'} p\, w)\Bigr)
f_{-m',\ep'\et' q^{m'}pw,w},
\end{split}
\]
where $f \in \cK(p,m,\ep,\et)$, $\ep'=\sgn(p_1)\ep$, $\et'=\sgn(p_2)\et$, $m'=m+n$.
\end{lemma}
Recall here that $f(z) = \langle f, f_{-m,\ep\et q^mpz,z} \rangle$ for
$f \in \cK(p,m,\ep,\et)$ using the convention \eqref{eq:conventioninnerproductfunction}. We prove Lemma \ref{lem:explicitactionQppn} at the end of this subsection. First we look at a few consequences.

By the definition of $\cK_\pm$, see Definition \ref{def:defKpmandMpm}, we have
$\cK_{\pm} = \bigoplus_{p\in q^\Z, m\in\Z, \ep\et=\pm}\cK(p,m,\ep,\et)$,\index{K@$\cK_+$}\index{K@$\cK_-$}
and then Lemma \ref{lem:explicitactionQppn} implies that
\[
Q(p_1,p_2,n) \colon \cK_\ep \to \cK_{\sgn(p_1p_2)\ep}, \qquad \ep\in\{-,+\}.
\]
This proves the last statement of Proposition \ref{prop:QppngeneratedualM} assuming
we know that $Q(p_1,p_2,n)\in\hat{M}$.

Recall the action \eqref{eq:expressionhatJonfmpt} of the dual modular conjugation
$\hat{J}$, so that
\begin{equation}\label{eq:hatJinbasisKmpepet}
\hat{J}\, f_{-m,\ep\et q^mpz,z} =
\et^{m+\chi(pz)}\ep^{\chi(z)} (-1)^m \, f_{-(-m),\ep\et q^{-m}pq^{2m}z,z}
\end{equation}
and thus $\hat{J}\colon \cK(p,m,\ep,\et)\to \cK(q^{2m}p,-m,\ep,\et)$. Now Lemma \ref{lem:explicitactionQppn} implies the following.

\begin{cor}\label{cor:lemexplicitactionQppn2}
Let $p\in q^\Z$, $p_1,p_2 \in I_q$, $m,n \in \Z$ and $\ep,\et \in \{-,+\}$. If $p\not= q^{-n}|p_2/p_1|$, then
\[
\hat{J}Q(p_1,p_2,n)\hat{J} \bigl(\cK(p,m,\ep,\et)\bigr)=\{0\}.
\]
If $p = q^{-n}|p_2/p_1|$, then
\[
\hat{J}Q(p_1,p_2,n)\hat{J} \colon \cK(p,m,\ep,\et) \to
\cK(q^{2n}p,m-n,\sgn(p_1)\ep,\sgn(p_2)\et)
\]
and
\[
\begin{split}
&\hat{J}\, Q(p_1,p_2,n) \, \hat{J} f =
(-1)^{m} \et^{m+\chi(p)} \, \frac{|p_1p_2|}{q^m p}
\sum_{w\in J(q^{2n}p, m-n, \ep',\et' )} \frac{1}{|w|} \\
&\times \Bigl( \sum_{z\in J(p,m,\ep,\et)}
\frac{f(z)}{|z|} (\ep\et)^{\chi(z)} a_{p_1}(z,w)
a_{p_2} (\ep\et q^mpz, \ep'\et' q^{m+n}pw) \Bigr)
\, f_{n-m,\ep'\et' q^{m+n}pw,w},
\end{split}
\]
where $f \in \cK(p,m,\ep,\et)$, $\ep'=\sgn(p_1)\ep$ and $\et'=\sgn(p_2)\et$.
\end{cor}
Again we postpone the proof until the end of this subsection.

Let us state the matrix elements of $Q(p_1,p_2,n)$ and
$\hat{J}\, Q(p_1,p_2,n)\, \hat{J}$ explicitly;
\begin{equation}\label{eq:matrixelementsofQppnandhatJQppnhatJ}
\begin{split}
\langle Q(p_1,p_2,n) \, f_{uvw}, f_{lrs} \rangle
& \, =\,  \de_{u-l,n}\de_{l, \chi(p_1v/p_2w)}\de_{r,\sgn(vw)sp_2q^u/p_1}
\left|\frac{w}{s}\right|\, a_w(p_1,s)\, a_v(p_2,r),\\
\langle \hat{J} Q(p_1,p_2,n)\hat{J} \, f_{uvw}, f_{lrs} \rangle
& \, =\,  \de_{l-u,n}\de_{l, \chi(p_2w/p_1v)}\de_{r,\sgn(vw)sp_2q^{-u}/p_1}
\sgn(r)^{\chi(r)}\sgn(s)^{\chi(s)}
\\ &\quad\times \sgn(v)^{\chi(v)}\, \sgn(w)^{\chi(w)}
(-1)^{l+u}
\left|\frac{w}{s}\right|\, a_w(p_1,s)\, a_v(p_2,r),
\end{split}
\end{equation}
to which one may apply the symmetry relations
\eqref{eq:symmetryforapxy}.

The remainder of this subsection is devoted to the proofs of Lemma \ref{lem:explicitactionQppn} and
Corollary \ref{cor:lemexplicitactionQppn2}.

\begin{proof}[Proof of Lemma \ref{lem:explicitactionQppn}.]
We start by considering matrix elements of the more generally defined operator
\[
(\om_{f_{m_1,p_1,t_1},f_{m_2,p_2,t_2}} \ot \Id)(W^*) \in B(\cK),
\]
with $m_1,m_2 \in \Z$ and $p_1,p_2,t_1,t_2 \in I_q$. For $n_1,n_2\in\Z$ and $r_1,r_2,s_1,s_2\in I_q$ we have
\[
\begin{split}
\big\langle (&\om_{f_{m_1,p_1,t_1},f_{m_2,p_2,t_2}} \ot \Id)(W^*)\, f_{n_1,r_1,s_1}, f_{n_2,r_2,s_2} \big\rangle \\
=& \big\langle W^\ast f_{m_1,p_1,t_1} \ot f_{n_1,r_1,s_1},  f_{m_2,p_2,t_2}\ot f_{n_2,r_2,s_2} \big\rangle \\
= &\, \de_{t_1,t_2}\de_{n_1-n_2,m_2-m_1} \de_{n_2,\chi(p_1r_1/s_1p_2)}
\de_{r_2,\sgn(r_1s_1)s_2p_2q^{n_1}/p_1} \\
& \times \left| \frac{s_1}{s_2}\right| a_{s_1}(p_1,s_2) \, a_{r_1}(p_2,\sgn(r_1s_1)s_2p_2q^{n_1}/p_1),
\end{split}
\]
where we used expression \eqref{eq:Wexplicitinapxy} for $W^\ast$. The dependence on $t_1, t_2\in I_q$ and $m_1,m_2\in\Z$ of the right hand side occurs only in the first two Kronecker deltas, so by \eqref{eq:defQppn} we have
\begin{equation}\label{eq:generalcasereducestoQppn}
(\om_{f_{m_1,p_1,t_1},f_{m_2,p_2,t_2}} \ot \Id)(W^*) = \de_{t_1,t_2}
Q(p_1,p_2,m_2-m_1).
\end{equation}
We see that it suffices to restrict to the case $t_1=t_2=1$, $m_1=0$, $m_2=n$, and we switch to
the basis elements of $\cK(p,m,\ep,\et)$, see Section \ref{ssec:decompositionofGNSspace},
i.e., we replace $(n_1,r_1,s_1)$ by $(-m,\ep\et q^mpz,z)$ and $(n_2,r_2,s_2)$ by $(-m',\ep'\et' q^{m'}p'z',z')$, where $p,p'\in q^\Z$ and $\ep,\et,\ep',\et \in \{+,-\}$. Then we find
\[
\begin{split}
\big\langle Q&(p_1,p_2,n) f_{-m,\ep\et q^mpz,z}, f_{-m',\ep'\et' q^{m'}p'z',z'} \big\rangle
= \\
& \de_{n,m'-m} \de_{-m',m+\chi(p_1p/p_2)}
\de_{\ep'\et' q^{m'}p', \ep\et p_2q^{-m}/p_1}
\left| \frac{z}{z'}\right| a_{z}(p_1,z') \, a_{\ep\et q^mpz}(p_2,\ep\et z'p_2q^{-m}/p_1).
\end{split}
\]
The first two Kronecker deltas always give zero unless $m'= n+m =-m-\chi(p_1p/p_2)$, or equivalently $q^{2m}p=q^{-n}|p_2/p_1|$,
which is the first statement of Lemma \ref{lem:explicitactionQppn}. Assuming that this condition is valid we see the third Kronecker delta becomes $\de_{\ep'\et'p',\ep\et\,\sgn(p_1p_2)p}$. Since $p,p'\in q^\Z$, we find that we need $p=p'$ and $\ep'\et'=\sgn(p_1p_2)\ep\et$. Assuming these conditions and using the last symmetry of \eqref{eq:symmetryforapxy} we find that
\[
\begin{split}
&\langle Q(p_1,p_2,n) f_{-m,\ep\et q^mpz,z}, f_{-m',\ep'\et' q^{m'}p'z',z'} \rangle
= \\
& (-1)^{m+n} (\ep'\et')^{\chi(z')} (\et')^{\chi(p)+m+n}
\frac{p_1p_2}{q^mp|zz'|} a_{p_1}(z,z')\, a_{p_2}(\ep\et q^mpz,\ep'\et' q^{m'}p'z').
\end{split}
\]
Now the product of the functions $a_p$ is zero unless $\ep'=\sgn(p_1)\ep$ and $\et'=\sgn(p_2)\et$,
see Definition \ref{def:functionap}. So in case  $q^{2m}p=q^{-n}|p_2/p_1|$
we find $Q(p_1,p_2,n)\colon \cK(p,m,\ep,\et) \to \cK(p,m+n,\ep\sgn(p_1),\et\sgn(p_2))$
and with $m'=m+n$, $\ep'=\ep\sgn(p_1)$, $\et'=\et\sgn(p_2)$ we find
\begin{equation}\label{eq:defQonbasisvector}
\begin{split}
&Q(p_1,p_2,n)\,f_{-m,\ep\et\,q^m p\,z,z} =
(-1)^{m'}\,(\et')^{\chi(p)+m'}\,\frac{|p_1 p_2|}{q^m p}\,\frac{1}{|z|}\,\,
\\ &\qquad
\times \sum_{z' \in J(p,m',\ep',\et')}\,
\frac{(\ep'\et')^{\chi(z')}}{|z'|}\,a_{p_1}(z,z')
\,a_{p_2}(\ep\et\,q^m p\,z,\ep'\et'\,q^{m'}
p\,z')\,\,f_{-m',\ep'\et'\,q^{m'}p\,z',z'}.
\end{split}
\end{equation}
This gives the required expression leading to the last statement of
Lemma \ref{lem:explicitactionQppn} after taking into account
$q^{2m}p=q^{-n}|p_2/p_1|$.
\end{proof}

\begin{proof}[Proof of Corollary \ref{cor:lemexplicitactionQppn2}.]
The first statements are immediate from Lemma \ref{lem:explicitactionQppn} and
\eqref{eq:hatJinbasisKmpepet}, and assuming the condition
$p=q^{-n}|p_2/p_1|$ we get, with $\ep'=\sgn(p_1)\ep$ and $\et'=\sgn(p_2)\et$,
\[
\begin{split}
&\hat{J} Q(p_1,p_2,n)\hat{J}\,f_{-m,\ep\et\,q^m p\,z,z} =
(\ep\et)^{\chi(z)}\,\et^{m+\chi(p)} (-1)^m \frac{|p_1 p_2|}{q^m p}\,\frac{1}{|z|}\,\,
\\ &\qquad
\times \sum_{z' \in J(q^{2n}p,m-n,\ep',\et')}\,
\frac{1}{|z'|}\,a_{p_1}(z,z')
\,a_{p_2}(\ep\et\,q^m p\,z,\ep'\et'\,q^{m+n}
p\,z')\,\,f_{-(m-n),\ep'\et'\,q^{m+n}p\,z',z'}
\end{split}
\]
using \eqref{eq:hatJinbasisKmpepet}, \eqref{eq:defQonbasisvector}
and $J(q^{2m}p,n-m,\ep',\et')=J(q^{2n}p,m-n,\ep',\et')$.
This implies the last statement of
Corollary \ref{cor:lemexplicitactionQppn2}.
\end{proof}

\subsection{A basis for the  dual von Neumann algebra}\label{ssec:commutantofdualvNalg}
In this subsection we give a proof of Proposition \ref{prop:QppngeneratedualM} and
Corollary \ref{cor:propQppngeneratedualM}.
For this we use the description of $\hat M$ as in \eqref{eq:defdualhatM}.

\begin{lemma}\label{lem:QppnareallinhatM}
The operators $Q(p_1,p_2,n)$, $p_1,p_2\in I_q$, $n\in\Z$, are in $\hat{M}$, and the linear
span of the operators  $Q(p_1,p_2,n)$, $p_1,p_2\in I_q$, $n\in\Z$,
is strong-$\ast$ dense in $\hat{M}$. Moreover, for $x\in \hat{M}$ there
exists a net $\{ x_i\}_{i\in I}$ in this linear span such that $x_i\to x$ in the
strong $\ast$-topology with $\|x_i\| \leq \|x\|$.
\end{lemma}

Lemma \ref{lem:QppnareallinhatM} proves Proposition \ref{prop:QppngeneratedualM} except for the last statement, which was proved in Section \ref{ssec:multiplicativeunitary} after Lemma \ref{lem:explicitactionQppn}. By the general Tomita-Takesaki theory, see \cite[Vol.~II]{Take}, cf. \eqref{eq:JMJisMcommutant}, we have that the commutant satisfies $\hat{M}' = \hat{J}\, \hat{M}\, \hat{J}$, and so Corollary \ref{cor:propQppngeneratedualM} follows.

\begin{proof}
By \eqref{eq:defdualhatM} and Theorem \ref{thm:duallocallycompactquantumgroup}
we have to consider
\[
\begin{split}
&\, (\om_{f_{m_1,p_1,t_1},f_{m_2,p_2,t_2}} \ot \Id)(W)
= (\om_{f_{m_1,p_1,t_1},f_{m_2,p_2,t_2}} \ot \Id) \Bigl(
\bigl( \hat{J}\ot J\bigr)\, W^\ast \, \bigl( \hat{J}\ot J\bigr) \Bigr) \\
= &\, J\,
(\om_{\hat{J}\, f_{m_1,p_1,t_1},\hat{J}\, f_{m_2,p_2,t_2}} \ot \Id) \bigl( W^\ast\bigr)\, J \\
= &\, \sgn(p_1)^{\chi(p_1)} \sgn(p_2)^{\chi(p_2)}\sgn(t_1)^{\chi(t_1)}
\sgn(t_2)^{\chi(t_2)} (-1)^{m_1+m_2}
J\, (\om_{f_{-m_1,p_1,t_1}, f_{-m_2,p_2,t_2}} \ot \Id) \bigl( W^\ast\bigr)\, J
\end{split}
\]
using \eqref{eq:hatJotJWhatJotJisWster}, $\hat{J}^2=\Id$,
$\langle \hat{J}f, \hat{J}g\rangle= \langle g, f\rangle$, $J$ being antilinear,
and \eqref{eq:expressionhatJonfmpt}.
It follows from the proof of Lemma \ref{lem:explicitactionQppn}, in
particular from \eqref{eq:generalcasereducestoQppn}, that we can restrict to
the case $t_1=t_2=1$, $m_1=0$, $m_2=-n$. By \eqref{eq:defdualhatM} and $J^2=\Id$ we
see that, up to a sign, $Q(p_1,p_2,n)$ equals
$J\, (\om \ot \Id)(W) \, J$ for $\om \in B(\cK)_\ast$.
Recall from \eqref{eq:actionunitaryantipode} that the unitary antipode
$\hat{R}$ for the dual quantum group is given by $\hat{R}(x)=Jx^\ast J$, so that for $x\in \hat{M}$ we have $J\, x\, J = \hat{R}(x^\ast)\in \hat{M}$. Now we see that $Q(p_1,p_2,n)\in \hat{M}$.

In order to prove the density statement, we recall that there exists
a dense $\ast$-subalgebra $M_\ast^\sharp$ of the predual $M_\ast$ such
such that $\{ (\om\ot\Id)(W) \mid \om \in M_\ast^\sharp\}$ is $\si$-strong-$\ast$
dense $\ast$-subalgebra of $\hat{M}$, see \cite[p.~79]{KustVMS}. The subspace
$M_\ast^\sharp$ consists of those normal functionals $\om$ such that $\bar\om \circ S$
is again a normal functional, where $\bar\om(x)=\overline{\om(x^\ast)}$, and
the $\ast$-operator for $\om\in M_\ast^\sharp$ defined as $\om^\ast = \bar\om \circ S$.
Now we apply the Kaplansky density theorem, see e.g.~\cite[Vol I, Ch.~II, Thm. 4.8]{Take},
to obtain a net $\{ \om_i\}_{i\in I}$ in $M_\ast^\sharp$ with the properties
$\| (\om_i \ot\Id)(W)\| < \|x^\ast\|=\|x\|$ for all $i\in I$ and such that
$(\om_i \ot\Id)(W) \to x^\ast$ in the strong-$\ast$ topology, so that also
$(\bar\om_i \ot\Id)(W^\ast) \to x$ in the strong-$\ast$ topology.

Let $L$ be the linear span of the normal functionals
$\om_{f_{m_1,p_1,t_1},f_{m_1,p_1,t_1}}$ for $p_1,p_2,t_1,t_2\in I_q$ and
$m_1,m_2\in\Z$, then $L$ is norm dense in $M_\ast$
and $\overline{\om_{f,g}}= \om_{g,f}$ so $L$ is closed under $\om\mapsto \bar\om$.
Now define the index set $I_0=I\times\N$, and make this a directed (or
upward filtering) set by the product order, i.e.~$(i_1,k_1)\leq (i_2,k_2)$ whenever $i_1\leq i_2$ in $I$ and $k_1\leq k_2$. For $j=(i,k)\in I_0$ we can pick $\et_j\in L$ such that $\| (\et_j\ot\Id)(W^\ast) - (\bar \om_i\ot\Id)(W^\ast)\|\leq 1/k$ and $\| (\et_j\ot\Id)(W^\ast)\|<\|x\|$. For such $j\in I_0$ set $x_j = (\et_j\ot\Id)(W^\ast)$ in the linear span of the operators  $Q(p_1,p_2,n)$, $p_1,p_2\in I_q$, $n\in\Z$, and the net $\{ x_j\}_{j\in I_0}$ satisfies all required properties.
\end{proof}

\begin{cor} \label{cor:lemQppnareallinhatM} With $\hat{R}$ the unitary
antipode for the dual locally compact quantum group we have
\[
\begin{split}
\hat{R}\bigl(Q(p_1,p_2,n)\bigr) &\, = (-1)^n \, \sgn(p_1)^{\chi(p_1)} \sgn(p_2)^{\chi(p_2)}
\, Q(p_2,p_1,n), \\
Q(p_1,p_2,n)^\ast &\, = (-1)^n\, \sgn(p_1)^{\chi(p_1)} \sgn(p_2)^{\chi(p_2)}\,
J\, Q(p_2,p_1,n)\, J.
\end{split}
\]
\end{cor}

\begin{proof}
Note that the statements are equivalent because $Q(p_1,p_2,n)\in \hat{M}$ by
Lemma \ref{lem:QppnareallinhatM} and  $\hat{R} x = J x^\ast J$ for $x\in \hat{M}$, see \eqref{eq:actionunitaryantipode}.

For $f,g \in \cK$ we set $T=(\om_{f,g}\ot\Id)(W^\ast)$. Then $T^\ast = (\om_{g,f}\ot\Id)(W)$, so that
$(\hat{J}\ot J)(W^\ast)(\hat{J}\ot J)= W$ gives
\[
T^\ast = J\, (\om_{\hat{J} g,\hat{J}f}\ot\Id)(W^\ast) \, J
\]
as in the first part of the proof of Lemma \ref{lem:QppnareallinhatM}.
Specializing $f= f_{0,p_1,1}$, $g=f_{n,p_2,1}$ gives $T=Q(p_1,p_2,n)$, and using the
action \eqref{eq:expressionhatJonfmpt} of $\hat J$ on $f_{mpt}$ we obtain
\[
\begin{split}
Q(p_1,p_2,n)^\ast &\, = (-1)^n \, \sgn(p_1)^{\chi(p_1)}\, \sgn(p_2)^{\chi(p_2)}
J\, (\om_{f_{-n,p_2,1}, f_{0,p_1,1}}\ot\Id)(W^\ast) \, J \\ &\, =
(-1)^n \, \sgn(p_1)^{\chi(p_1)}\, \sgn(p_2)^{\chi(p_2)}
J\, Q(p_2,p_1,n)\, J,
\end{split}
\]
where the last equality follows from \eqref{eq:generalcasereducestoQppn}.
\end{proof}

We finish the subsection by establishing the structure constants for
the operators $Q(p_1,p_2,n)$ as a linear basis for $\hat{M}$. First observe
that as elements of $B(\cK)$
\[
\bigl( (\om_{f,g}\ot \Id)(W^\ast)\bigr) \, \bigl( (\om_{\xi,\et}\ot \Id)(W^\ast)\bigr)
\, =\, \bigl( \om_{\xi,\et}\ot\om_{f,g}\ot\Id\bigr) (W^\ast_{23}W^\ast_{13})
\]
for arbitrary vectors $f,g,\xi,\et\in\cK$. Using the pentagonal equation
$W_{12}W_{13}W_{23}=W_{23}W_{12}$ this can be rewritten in the compact form
\begin{equation}\label{eq:productofQs1}
\bigl( (\om_{f,g}\ot \Id)(W^\ast)\bigr) \, \bigl( (\om_{\xi,\et}\ot \Id)(W^\ast)\bigr)
\, =\,
\bigl( \om_{W(\xi\ot f), W(\et\ot g)}\ot \Id\bigr) (\Id\ot W^\ast).
\end{equation}

\begin{proof}[Proof of Proposition \ref{prop:structureconstQs}]
We start with the choice $f=f_{0,p_1,1}$, $g=f_{n,p_2,1}$,
$\xi=f_{0,r_1,1}$, $\et=f_{m,r_2,1}$, so that the left hand side of
\eqref{eq:productofQs1} equals $Q(p_1,p_2,n)\, Q(r_1,r_2,m)$.
In order to evaluate the right hand side of \eqref{eq:productofQs1} we use
\eqref{eq:Wactionexplicitinapxy}, which leads to
\begin{equation}\label{eq:pfpropstructureconstQs1}
\begin{split}
&\sum_{\stackrel{\scriptstyle{x_1,y_1\in I_q}}{\scriptstyle{\text{so that }\sgn(x_1p_1)y_1r_1\in I_q}}}
\sum_{\stackrel{\scriptstyle{x_2,y_2\in I_q}}{\scriptstyle{\text{so that }\sgn(x_2p_2)y_2r_2\in I_q}}}
|y_1y_2| \, a_{x_1}(r_1,p_1)\,  a_{x_2}(r_2,p_2)\\
&\qquad \times
\,  a_{y_1}(\sgn(x_1p_1)y_1r_1, 1) \, a_{y_2}(\sgn(x_2p_2)y_2r_2q^n, 1) \\
&\qquad \times \langle f_{-\chi(y_1p_1),\sgn(x_1p_1)y_1r_1,1}, f_{m-\chi(y_2p_2),\sgn(x_2p_2)y_2r_2q^n,1}\rangle \\
&\qquad \times \bigl( \om_{f_{\chi(y_1p_1),x_1,y_1},f_{n+\chi(y_2p_2),x_2,y_2}}\ot\Id\bigr) (W^\ast).
\end{split}
\end{equation}
The inner product in the summand of \eqref{eq:pfpropstructureconstQs1} leads to
$\de_{-\chi(y_1p_1),m-\chi(y_2p_2)}\de_{\sgn(x_1p_1)y_1r_1,\sgn(x_2p_2)y_2r_2q^n}$, whereas,
by \eqref{eq:generalcasereducestoQppn}, the last term in the summand is
$\de_{y_1,y_2} \, Q(x_1,x_2, n+\chi(p_2)-\chi(p_1))$. Combining this we see that the
last two terms in the summand of \eqref{eq:pfpropstructureconstQs1} equal
\[
\de_{\chi(p_2),m+\chi(p_1)}\, \de_{\sgn(x_1p_1)r_1,\sgn(x_2p_2)r_2q^n}\,
\de_{y_1,y_2} \, Q(x_1,x_2,n+m),
\]
which is zero in case $|\frac{p_2}{p_1}|\not= q^m$ independent of $x_1$, $y_1$, $x_2$, $y_2$.

Assuming $|\frac{p_2}{p_1}|= q^m$
and inserting this into
\eqref{eq:pfpropstructureconstQs1} leads to
\begin{equation}\label{eq:pfpropstructureconstQs2}
\begin{split}
&\sum_{x_1,x_2\in I_q}
 \, a_{x_1}(r_1,p_1)\,  a_{x_2}(r_2,p_2) \, Q(x_1,x_2,n+m) \\
& \times
\Bigl( \sum_{\stackrel{\scriptstyle{y_1\in I_q}\text{ so that }}
{\scriptstyle{\sgn(x_1p_1)y_1r_1=\sgn(x_2p_2)y_1r_2q^n\in I_q}}}
y_1^2\,  a_{y_1}(\sgn(x_1p_1)y_1r_1, 1) \, a_{y_1}(\sgn(x_2p_2)y_2r_2q^n, 1)
\Bigr),
\end{split}
\end{equation}
where empty sums are zero.
For the expression in \eqref{eq:pfpropstructureconstQs2} to be non-zero result we require $\sgn(x_1)=\sgn(r_1p_1)$ and
$\sgn(x_2)=\sgn(r_2p_2)$, see Definition \ref{def:functionap}. Then we see that
$\sgn(x_1p_1)y_1r_1=y_1|r_1|$ and $\sgn(x_2p_2)y_1r_2q^n=y_1|r_2|q^n$, and so the
inner sum is zero unless $|\frac{r_1}{r_2}|= q^n$. In this case the inner sum
equals
\[
 \sum_{\stackrel{\scriptstyle{y_1\in I_q}\text{ so that }}
{\scriptstyle{y_1|r_1|\in I_q}}}
y_1^2\,  \bigl(a_{y_1}(y_1|r_1|, 1)\bigr)^2 =
 \sum_{\stackrel{\scriptstyle{y_1\in I_q}\text{ so that }}
{\scriptstyle{y_1|r_1|\in I_q}}}
 \bigl(a_{1}(y_1, y_1|r_1|)\bigr)^2 = 1,
\]
where the first equality follows from the symmetry relations \eqref{eq:symmetryforapxy},
and the second equality is a special case of Proposition \ref{prop:orthorelapxy} (with
$p=1$ and $\theta=|r_1|$).

Collecting the results finishes the proof of Proposition \ref{prop:structureconstQs}.
\end{proof}

\subsection{Affiliation of $K$ and $E$ to $\hat{M}$} \label{ssec:AffiliationOfKandE}

The purpose of this subsection is to prove Proposition \ref{prop:KEaffiliatedtohatM}. First we focus on the operator $K$.

By Definition \ref{def:closedgenerators} $K$ is the closure
of $(K_0, \cK_0)$, with $K_0$ given by Definition \ref{def:actiongeneratorssu}.
Since $K_0$ acts diagonally on basis elements $f_{mpt}$, $m \in \Z$, $p,t \in I_q$, we find from
Definition \ref{def:actiongeneratorssu}
\begin{equation}\label{eq:defKanddomain}
\begin{split}
&\, D(K) = \Bigl\{ \sum_{m\in \Z, p,t\in I_q} c_{mpt}\, f_{mpt}  \mid
\sum_{m\in \Z, p,t\in I_q} |c_{mpt}|^2 q^{-m} \left|\frac{p}{t}\right| <\infty\Bigr\}, \\
&\, K \Bigl( \sum_{m\in \Z, p,t\in I_q} c_{mpt}\, f_{mpt} \Bigr) =
\sum_{m\in \Z, p,t\in I_q} q^{-\hf m} \left| \frac{p}{t}\right|^{\hf}\, c_{mpt}\, f_{mpt}.
\end{split}
\end{equation}
It is now straightforward from \eqref{eq:defKanddomain} to check
that $K$ is an injective positive self-adjoint operator, establishing
the first statement of Proposition \ref{prop:KEaffiliatedtohatM}. We now prove the
second statement for the operator $K$.

\begin{prop}\label{prop:KaffiliatohatM}
$K$ is affiliated to $\hat{M}$.
\end{prop}

\begin{proof}
Note that $K$ restricted to $\cK_0(p,m,\ep,\et)$ acts as $q^m\sqrt{p}\, \Id$ by Definition \ref{def:actiongeneratorssu} and \eqref{eq:defKpmepeta}. It follows that $\cK(p,m,\ep,\et)\subset D(K)$. So $\hat{J}\,Q(p_1,p_2,n)\hat{J}\,f_{mpt}\in D(K)$ by Corollary \ref{cor:lemexplicitactionQppn2} and
\[
K\bigl(\hat{J}\,Q(p_1,p_2,n)\hat{J}\,f_{mpt}\bigr) =
\hat{J}\,Q(p_1,p_2,n)\hat{J}\,K\,f_{mpt}
\]
since the action of $K$ on $\cK(p,m,\ep,\et)$ is the same as on $\cK(q^{2n}p,m-n,\sgn(p_1)\ep,\sgn(p_2)\et)$ in case $p=q^{-n}|p_2/p_1|$. In case this is not true, both sides equal zero.

Since $\cK_0$ is a core for $K$, we can take for
$f\in D(K)$ a sequence $\cK_0\ni f_i\to f$ and $Kf_i\to g = Kf$.
Then $\hat{J}\,Q(p_1,p_2,n)\hat{J} \, f_i\to \hat{J}\,Q(p_1,p_2,n)\hat{J}\, f$ by
continuity, and $K \hat{J}\,Q(p_1,p_2,n)\hat{J} \, f_i =
\hat{J}\,Q(p_1,p_2,n)\hat{J}\, K f_i \to \hat{J}\,Q(p_1,p_2,n)\hat{J}\, g$.
Since $K$ is closed, we conclude
$\hat{J}\,Q(p_1,p_2,n)\hat{J}\, f \in D(K)$ and
$K\, \hat{J}\,Q(p_1,p_2,n)\hat{J}\, f = \hat{J}\,Q(p_1,p_2,n)\hat{J}\, Kf$.
This means
\[
\hat{J}\,Q(p_1,p_2,n)\hat{J}\, K
\subset K\, \hat{J}\,Q(p_1,p_2,n)\, \hat{J},
\]
so $K$ commutes with the
generators of $\hat{M}'$, see Appendix \ref{ssecA:commutationoperators}.

To see that $K$ commutes with an arbitrary element $T\in \hat{M}'$,
pick $T_i$ from the linear span of $\hat{J}\, Q(p_1,p_2,n)\, \hat{J}$
such that $T_i\to T$ strongly, see Corollary \ref{cor:propQppngeneratedualM}.
Take any $f\in D(K)$, so
that $T_if\to Tf$ and since $T_if\in D(K)$ (by $T_iK\subset KT_i$) we
have $KT_if= T_iKf\to TKf$ by the strong convergence. Again by the
closedness of $K$ we conclude that $Tf\in D(K)$ and $KTf=TKf$, or
$TK\subset KT$. Since $T\in \hat{M}'$ is arbitrary, $K$ is affiliated to $\hat{M}$, see Appendix \ref{ssecA:affiliationandgenerators}.
\end{proof}

In order to show that $E$ is affiliated to $\hat{M}$ we need to work
more carefully.  We start with a useful property of the operator $E_0$.

\begin{lemma} \label{lem:firststepEaffiliatedhatM}
Let $p_1,p_2 \in I_q$ and $n \in \Z$. Then
\[
\langle \, \hat{J}\,Q(p_1,p_2,n) \hat{J} v , E_0^\dag \, w \rangle =
\langle \, \hat{J}\,Q(p_1,p_2,n) \hat{J} E_0 \, v ,  w \rangle, \qquad
\forall\, v,w \in \cK_0.
\]
\end{lemma}

We relegate the proof of Lemma \ref{lem:firststepEaffiliatedhatM} to
Appendix \ref{ssecC:pflemfirststepEaffiliatedhatM}, since it is a tedious check.

By Lemma \ref{lem:firststepEaffiliatedhatM} we have for the closure $E$ of $E_0$
the equality
\[
\langle \, \hat{J}\,Q(p_1,p_2,n) \hat{J} v , E_0^\dag \, w \rangle =
\langle \, \hat{J}\,Q(p_1,p_2,n) \hat{J} E \, v ,  w \rangle
\]
for $v,w\in \cK_0$. Now fix $v=f_{-m,\ep\et q^mpz,z}\in \cK_0(p,m,\ep,\et)$
and put $u= \hat{J}\,Q(p_1,p_2,n) \hat{J} v$. It follows that
$u\in D\bigl( (E_0^\dag)^\ast\bigr)$ and
\begin{equation} \label{eq:E0dag*u}
(E_0^\dag)^\ast u = \hat{J}\,Q(p_1,p_2,n) \hat{J} E \, f_{-m,\ep\et q^mpz,z}.
\end{equation}
This equality can be extended in the following way.

\begin{lemma}\label{lem:secondstepEaffiliatedhatM}
Let $u = \hat{J}\,Q(p_1,p_2,n) \hat{J}\, f_{-m,\ep\et q^mpz,z}$, then
$\langle (E_0^\dag)^\ast u, w\rangle = \langle u, E^\ast w\rangle$ for all
$w\in D(E^\ast)$.
\end{lemma}

Before proving Lemma \ref{lem:secondstepEaffiliatedhatM} we show how
it implies that $E$ is affiliated to $\hat{M}$, which finishes the proof of Proposition \ref{prop:KEaffiliatedtohatM}.

\begin{prop}\label{prop:EaffiliatedtohatM}
$E$ is affiliated to $\hat{M}$.
\end{prop}

\begin{proof}
Since $E$ is the closure of $E_0$, it follows from Lemma \ref{lem:secondstepEaffiliatedhatM}
that $u \in D(E^{\ast\ast})= D(E)$, and
\[
E\,\hat{J}\,Q(p_1,p_2,n) \hat{J}\, f_{-m,\ep\et q^mpz,z}=E\, u = (E_0^\dag)^\ast u = \hat{J}\,Q(p_1,p_2,n) \hat{J} E \, f_{-m,\ep\et q^mpz,z}
\]
by \eqref{eq:E0dag*u}. This shows that $E \hat{J}\,Q(p_1,p_2,n) \hat{J} = \hat{J}\,Q(p_1,p_2,n) \hat{J} E$ on $\cK_0$.
Now the proof is finished as in the last stage of
Proposition \ref{prop:KaffiliatohatM} using the closedness of $E$, $\cK_0$
being a core for $E$, and the strong-$\ast$ denseness of the operators $\hat{J}\,Q(p_1,p_2,n) \hat{J}$ in $\hat{M}'$ by Corollary \ref{cor:propQppngeneratedualM}.
\end{proof}

Before we turn to the proof of Lemma \ref{lem:secondstepEaffiliatedhatM}, recall the decomposition \eqref{eq:decompE} of $E$ into operators $E^{\ep,\et}_{p,m} \colon \cK(p,m,\ep,\et) \to \cK(p,m+1,\ep,\et)$. The operators $E^{\ep,\et}_{p,m}$ are bounded, unless $\ep=+=\et$. We study the case $\ep=+=\et$ by considering truncated inner products.
Define for $x \in q^\Z$ a truncated inner product by
\begin{equation} \label{eq:TruncatedInnerProduct}
\langle v, w\rangle_x \, = \, \sum_{\substack{z\in J(p,m,+,+)\\z\leq x}}
v(z) \, \overline{w(z)},\qquad v,w\in \cK(p,m,+,+).
\end{equation}
For $x \rightarrow \infty$ this gives back the inner product on $\cK(p,n,+,+)$. Let us remark that all coefficients in \eqref{eq:defE0onbasisKpmepet} and \eqref{eq:defE0dagonbasisKpmepet} remain bounded for $z\to 0$, $z\in q^\Z$, so we do not need to consider a truncated inner product of the form \eqref{eq:TruncatedInnerProduct} with the terms $z \leq y$ cut off, for some $y \in q^\Z$, $y < x$.

\begin{lemma}\label{lem:truncatedinnerproductforepet+}
Let $w\in D(E^\ast)\cap \cK(p,m,+,+)$, $u\in D((E_0^\dag)^\ast)
\cap\cK(p,m-1,+,+)$, then, with $x\in q^\Z$,
\[
\langle (E_0^\dag)^\ast u, w \rangle_x - \langle u, E^\ast w\rangle_x
= \frac{q^{m-1} (pq)^\hf \sqrt{1+x^{-2} q^2}}{q-q^{-1}} \, \frac{x}{q} u(x/q)
\, w(x)
\]
using the convention \eqref{eq:conventioninnerproductfunction}.
\end{lemma}

\begin{proof}
By \eqref{eq:defE0onbasisKpmepet} and \eqref{eq:defE0dagonbasisKpmepet} for the case $\ep=+=\et$, using the boundedness of the coefficients as $z\to 0$, $z\in q^\Z$, we obtain
\[
\begin{split}
&\, (q-q^{-1})\, \Bigl(\langle (E_0^\dag)^\ast u, w \rangle_x - \langle u, E^\ast w\rangle_x\Bigr)  \\
=&\,  \sum_{z\in q^\Z, z\leq x} \Bigl( q^{m-1} (pq)^\hf \sqrt{1+z^2 q^{-2}} \, u(\frac{z}{q})
\overline{w(z)} - q^{1-m} (pq)^{-\hf} \sqrt{1+q^{2m-2}p^2 z^2}\, u(z) \overline{w(z)}\Bigr) \\
&\, -  \sum_{z\in q^\Z, z\leq x} \Bigl( q^{m} (p/q)^\hf \sqrt{1+z^2} \, u(z)
\overline{w(qz)} - q^{-m} (p/q)^{-\hf} \sqrt{1+q^{2m-2}p^2 z^2}\, u(z) \overline{w(z)}\Bigr) \\
=&\,  q^{m-1} (pq)^\hf \sqrt{1+x^2 q^{-2}} \, u(\frac{x}{q})
\overline{w(x)}
\end{split}
\]
giving the required expression.
\end{proof}

The following result will be useful when we want to take the limit $x \rightarrow \infty$, $x \in q^\Z$, in the previous lemma. Recall the convention \eqref{eq:conventioninnerproductfunction}.

\begin{lemma}\label{lem:behaviourhatJQppnhatJincase++}
\begin{enumerate}[(i)]
\item Let $v= Q(p_1,p_2,n) f_{-m,\ep\et pq^mz,z}$ and assume $q^{2m}p = q^{-n}|p_2/p_1|$, so
that $v\in\cK(p,m+n,\ep\sgn(p_1),\et\sgn(p_2))$ is non-zero. If $\sgn(p_1)\ep=+=\sgn(p_2)\et$, then there exists a continuous function $h\colon \R_{\geq 0}\to \R$ such that $x\, v(x) = h(x^{-2})$ for $x\in I_q^+$.
In case $m+n=0$, $h\colon \R_{\geq 0}\to \R$ is differentiable, in particular at $0$.
\item Let $u = \hat{J}\, Q(p_1,p_2,n) \, \hat{J} \, f_{-m,\ep\et pq^mz,z}$ and assume $p = q^{-n}|p_2/p_1|$, so that $u\in\cK(q^{2n}p,m-n,\ep,\sgn(p_1),\et\sgn(p_2))$ is non-zero. If $\sgn(p_1)\ep=+=\sgn(p_2)\et$, then there exists a continuous function $h\colon \R_{\geq 0}\to \R$ such that $x\, u(x) = h(x^{-2})$ for $x\in I_q^+$.
In case $m-n=0$, $h\colon \R_{\geq 0}\to \R$ is differentiable, in particular at $0$.
\end{enumerate}
\end{lemma}

\begin{proof}
We prove the second statement; the first statement is proved in the same way.
It follows from Corollary \ref{cor:lemexplicitactionQppn2} or \eqref{eq:matrixelementsofQppnandhatJQppnhatJ} that $u\in\cK(q^{2n}p,m-n,\ep,+,+)$ and for $x\in I_q^+$
\[
x\, u(x) =
(-1)^m \et^{m+\chi(p)} (\ep\et)^{\chi(z)}
\frac{|p_1p_2|}{q^m p |z|}\, a_{p_1}(z,x) \,
a_{p_2}(\theta z, \theta' x)
\]
where $\theta = \ep\et q^m p$, $\theta' = q^{m+n} p$.
Lemma \ref{lemB:apxyforyinqZ} gives
$a_p(z,x) = x^{\chi(p/z)} f_1(x^{-2})$ as well as $a_p(z,x) = x^{\chi(z/p)} f_2(x^{-2})$ for certain
differentiable functions $f_1, f_2\colon \R_{\geq 0}\to\R$ using the last equation
of the symmetry relations \eqref{eq:symmetryforapxy} and then Lemma \ref{lemB:apxyforyinqZ}. Now we find, with
$C$ a generic non-zero constant not depending on $x$,
\[
\begin{split}
x\, u(x) &\, =\,  C\, x^{\chi(z/p_1)} f_1(x^{-2})\,
(\theta' x)^{\chi(p_2/\theta z)} f_2((\theta')^{-2}x^{-2})
= C\, x^{\chi(p_2/p_1\theta)} f_1(x^{-2})\,
f_2((\theta')^{-2}x^{-2}) \\
&\, =\,
C\, x^{n-m} f_1(x^{-2})\,
f_2((\theta')^{-2}x^{-2})
\end{split}
\]
using $|p_2/\theta p_1|=q^{n-m}$ as follows from the assumption
$q^np =|p_2/p_1|$. This proves the statement in case $m-n\geq 0$,
since we can take $h(t) = C\, t^{\hf(m-n)} f_1(t)\,
f_2((\theta')^{-2}t)$. In case $n=m$ the statement on the
differentiability of $h$ follows immediately.

Similarly, we find, for other functions $f_1$, $f_2$,
\[
\begin{split}
x\, u(x) &\, =\,  C\, w^{\chi(p_1/z)} f_1(x^{-2})\,
(\theta' x)^{\chi(\theta z/p_2)} f_2((\theta')^{-2}x^{-2})
= C\, x^{\chi(p_1\theta/p_2)} f_1(x^{-2})\,
f_2((\theta')^{-2}x^{-2}) \\
&\, =\,
C\, x^{m-n} f_1(x^{-2})\,
f_2((\theta')^{-2}x^{-2})
\end{split}
\]
using $|p_2/\theta p_1|=q^{n-m}$ again.
This proves the statement in case $m-n\leq 0$,
since we can take $h(t) = C\, t^{\hf(n-m)} f_1(t)\,
f_2((\theta')^{-2}t)$.
\end{proof}

We are now ready to prove Lemma \ref{lem:secondstepEaffiliatedhatM}.

\begin{proof}[Proof of Lemma \ref{lem:secondstepEaffiliatedhatM}.]
We set $p'=q^{2n}p$, $m'=m-n$, $\ep'=\sgn(p_1)\ep$, $\et'=\sgn(p_2)\et$, then $u=\hat{J}\,Q(p_1,p_2,n) \hat{J}\, f_{-m,\ep\et q^mpz,z}\in \cK(p',m',\ep',\et')$ by Lemma \ref{cor:lemexplicitactionQppn2}. Using the decomposition of $(E_0^\dag)^\ast$, cf.~\eqref{eq:decompE},
\[
(E_0^\dag)^\ast = \bigoplus_{\substack{\al,\be\in\{-,+\}\\r\in q^\Z,\,l\in\Z}}
\bigl( E_0^\dag\big\vert_{\cK_0(r,l,\al,\be)}\bigr)^\ast,
\]
we find $u \in D\bigl( (E_0^\dag\big\vert_{\cK_0(p',m'+1,\ep',\et')})^\ast\bigr)$.
Using the similar decomposition for $E^\ast$ we find that
$w'=P^{\ep',\et'}_{p',m'+1}w\in D\bigl( (E^{\ep',\et'}_{p',m'})^\ast\bigr)$, where
$P^{\ep,\et}_{p,m}\in B(\cK)$ is the orthogonal projection onto
$\cK(p,m,\ep,\et)$ as in Section \ref{ssec:decompositionofGNSspace}. This gives
\[
\big\langle (E_0^\dag)^\ast u, w\big\rangle - \big\langle u, E^\ast w\big\rangle =
\big\langle (E_0^\dag\big\vert_{\cK_0(p',m'+1,\ep',\et')})^\ast u, w'\big\rangle
- \big\langle u, (E^{\ep',\et'}_{p',m'})^\ast w'\big\rangle.
\]

In case $\ep'=-$ or $\et'=-$, $E^{\ep',\et'}_{p',m'}$ is bounded. Therefore
$\bigl( E^{\ep',\et'}_{p',m'}\bigr)^\ast$ is the unique continuous extension
of $E_0^\dag\big\vert_{\cK_0(p',m'+1,\ep',\et')}$, so
$(E_0^\dag\big\vert_{\cK_0(p',m'+1,\ep',\et')})^\ast = E^{\ep',\et'}_{p',m'}$
and hence the right hand side is zero, as required.

It remains to consider the case $\ep'=+=\et'$. In this case we consider the truncated inner product. Using Lemma \ref{lem:truncatedinnerproductforepet+} we find for $x\in q^\Z=I_q^+$,
\[
\langle (E_0^\dag\big\vert_{\cK_0(p',m'+1,+,+)})^\ast u, w'\rangle_x
- \langle u, (E^{+,+}_{p',m'})^\ast w'\rangle_x
= \frac{q^{m'-1} (p'q)^\hf \sqrt{1+x^{-2}q^2}}{q-q^{-1}}
\, \frac{x}{q} \, u(x/q)\, w'(x),
\]
and we need to show that the right hand side tends to zero as $x\to \infty$
through $I_q^+$. Since $w'\in \cK(p',m'+1,+,+) \cong \ell^2(I_q^+)$ it follows
that $w'(x)\to 0$ as $x \rightarrow \infty$, so the required result follows from
Lemma \ref{lem:behaviourhatJQppnhatJincase++} which implies
$\frac{x}{q} \, u(x/q)$ is bounded as $x\to\infty$ in $I_q^+$.
\end{proof}

\subsection{The comultiplication on $\hat{M}$}\label{ssec:com,ultiplicationhatM}

In order to calculate the action of the comultiplication of the dual
quantum group  on the elements $Q(p_1,p_2,n)$, we note that this can
be done in greater generality. First observe
\begin{equation}\label{eq:generalobsforcomultiplication}
\begin{split}
W\, \bigl( (\om_{f,g}\ot \Id)(W^\ast)\ot\Id\bigr) W^\ast\, &=\,
(\om_{f,g}\ot \Id\ot \Id)(W_{23}W^\ast_{12}W^\ast_{23})\\  \, &= \,
(\om_{f,g}\ot \Id\ot \Id)(W_{13}^\ast W_{12}^\ast).
\end{split}
\end{equation}
The first equality is straightforward, and the second follows
from the pentagonal equation for the multiplicative unitary,
see Section \ref{sec:vNalgquantumgroups}.  Using an orthonormal basis $\{ e_k\}$ for the
Hilbert space $\cK$, so that we have
$\langle x,y\rangle = \sum_k \langle x, e_k\rangle\langle e_k,y\rangle$
we get
\begin{equation}\label{eq:generalobsforcomultiplication-2}
\begin{split}
\Sigma\, \hat{\De}\bigl( (\om_{f,g}\ot\Id)(W^\ast)\bigr) \Sigma\, &=\,
(\om_{f,g}\ot \Id\ot \Id)(W_{13}^\ast W_{12}^\ast) \\ \, &=\,
 \sum_k (\om_{f,e_k}\ot \Id)(W^\ast) \, \ot\, (\om_{e_k,g}\ot \Id)(W^\ast).
\end{split}
\end{equation}
using the definition of $\hat{\De}$ and notation as in
Theorem  \ref{thm:duallocallycompactquantumgroup}.

\begin{proof}[Proof of Proposition \ref{prop:hatDeonQppn}]
We use the general formula \eqref{eq:generalobsforcomultiplication-2} with
$f=f_{0,p_1,1}$, $g=f_{n,p_2,1}$, the orthonormal basis
$f_{m,p,t}$ ($m,\in\Z$, $p,t\in I_q$) and next
\eqref{eq:generalcasereducestoQppn} to rewrite the right hand side
in terms of the operators $Q(p_1,p_2,n)$. The series converges
in the von Neumann algebra $\hat{M}\ot \hat{M}$, so that we find convergence
in the $\si$-weak topology.
\end{proof}

Next we prove the link between the comultiplication $\hat{\De}$ of the dual
quantum group $\hat{M}$ and the comultiplication \eqref{eq:Hopfalgebracomultiplication}
$\bDe$ of the Hopf $\ast$-algebra $\su$ as given in
Proposition \ref{prop:dualcomultiplicationandcomultiplicationsu}.

\begin{proof}[Proof of Proposition \ref{prop:dualcomultiplicationandcomultiplicationsu}.]
The comultiplication for the dual locally compact quantum group is given by
$\hat{\De}(x) = \Sigma \, W  (x\ot 1) W^\ast\, \Sigma$, $x\in\hat{M}$, see
Theorem \ref{thm:duallocallycompactquantumgroup}. We use the same formula
for the elements $K$ and $E$ affiliated to $\hat{M}$, see Proposition
\ref{prop:KEaffiliatedtohatM}. In order to prove that $\hat{\De}(K)\, =\, K\ot K$ we
need to show
$\Sigma \, W  (K\ot 1) W^\ast\, \Sigma\, =\, K\ot K$, or
$W  (K\ot 1) W^\ast\, =\, \Sigma\,K\ot K\, \Sigma\,=\, K\ot K$. So it
suffices to check that $(K\ot 1) W^\ast\, =\, W^\ast (K\ot K)$.

Now by \eqref{eq:Wexplicitinapxy} and \eqref{eq:defKanddomain}
we check this formula first by evaluating
it on the orthonormal basis of $\cK\ot\cK$. So  for
arbitrary $m_1, m'_1, m_2, m'_2\in\Z$, $p_1, p'_1, p_2, p'_2\in I_q$, $m_1, m'_1, m_2, m'_2\in I_q$
\[
\begin{split}
&\qquad\, \langle W^* (f_{m_1,p_1,t_1} \ot f_{m_1',p_1',t_1'}) , (K \ot 1)(f_{m_2,p_2,t_2} \ot f_{m_2',p_2',t_2'}) \rangle\,  =\, \, \sqrt{|\frac{p_2}{t_2}|}\,q^{-\hf m_2} \\
&\times \,\,\de_{t_1,t_2}\,\de_{|p_1 p_1'/p_2 t_1'|,q^{m_2'}}
 \, \de_{m_1+m_1',m_2+m_2'}\,\de_{\sgn(p_1' t_1') q^{m_1'}
p_2/p_1,p_2'/t_2'}\,\,|\frac{t_1'}{t_2'}|\,a_{t_1'}(p_1,t_2')\,a_{p_1'}(p_2,p_2')
\end{split}
\]
and similarly
\[
\begin{split}
&\qquad \langle W^* (K \ot K)(f_{m_1,p_1,t_1} \ot f_{m_1',p_1',t_1'}) , f_{m_2,p_2,t_2} \ot f_{m_2',p_2',t_2'} \rangle \, =\,   \sqrt{|\frac{p_1 p_1'}{t_1 t_1'}|}\,q^{-\hf(m_1+m_1')}
\, \\ &\times \,\,\de_{t_1,t_2}\,\de_{|p_1 p_1'/p_2 t_1'|,q^{m_2'}}
 \, \de_{m_1+m_1',m_2+m_2'}\,\de_{\sgn(p_1' t_1') q^{m_1'}
p_2/p_1,p_2'/t_2'}\,\,|\frac{t_1'}{t_2'}|\,a_{t_1'}(p_1,t_2')\,a_{p_1'}(p_2,p_2').
\end{split}
\]
These expressions are equal by inspection using the Kronecker deltas.
The linear span of elements $f_{m,p,t}$ forms a core for the operator $K$.
So it follows that
\[
\langle W^* (f_{m_1,p_1,t_1} \ot f_{m_1',p_1',t_1'}) , (K \ot 1) w
\rangle = \langle W^* (K \ot  K) (f_{m_1,p_1,t_1} \ot f_{m_1',p_1',t_1'}) , w \rangle
 \]
for all $w \in D(K \ot 1)$. Hence, $W^*(K_0 \odot K_0)\, \subset\,
(K \ot 1)W^*$, thus $K \ot K \, \subset\, W(K \ot 1)W^*$ using that $K$ is self-adjoint and
that the closure of $K_0\odot K_0$ equals $K\ot K$.
Since both operators are self-adjoint, this inclusion is an
equality. This proves the statement for $\hat{\De}(K)$.

Let us now prove the more complicated second statement. Choose $p_1,p_1',t_1,t_1' \in I_q$ and $m_1,m_1' \in \Z$.
Take also $p_2,p_2',t_2,t_2' \in I_q$ and $m_2,m_2' \in \Z$. Since $E_0 \subseteq E$ and $E_0^\dag \subseteq E^*$,
\eqref{eq:dualE0forsu} and \eqref{eq:Wexplicitinapxy} imply that
\begin{equation}\label{eq:comultiE2}
\begin{split}
&  (q-q^{1})\,\,\langle W^* (E_0 \odot K_0 + K_0^{-1} \odot E_0) (f_{m_1,p_1,t_1} \ot f_{m_1',p_1',t_1'}) ,
f_{m_2,p_2,t_2} \ot f_{m_2',p_2',t_2'} \rangle
\\ & = \de_{t_1,q t_2}\,\de_{|p_1 p_1'/p_2
t_1'|,q^{m_2'}}\,\,\de_{m_1+m_1'-1,m_2+m_2'}\,\,\de_{\sgn(p_1' t_1') q^{m_1'} p_2/p_1, p_2'/t_2'}
\\ & \qquad\times \ \sgn(t_1)\,q^{-\hf(m_1+m_1'-1)}\,
\sqrt{|\frac{p_1 p_1'}{t_1t_1'}|}\,\,\sqrt{1+\kappa(q^{-1} t_1)}\,\,|\frac{t_1'}{t_2'}|\,\,a_{t_1'}(p_1,t_2')\,a_{p_1'}(p_2,p_2')
\\ &  \ \ \ - \,\de_{t_1,t_2}\,\de_{|p_1 p_1'/p_2
t_1'|,q^{m_2'-1}}\,\,\de_{m_1+m_1'-1,m_2+m_2'}\,\,\de_{\sgn(p_1' t_1') q^{m_1'-1} p_2/p_1, p_2'/t_2'}
\\ & \qquad \times \ \sgn(p_1)\,q^{\hf(m_1-m_1'-1)}\,\sqrt{|\frac{t_1 p_1'}{p_1 t_1'}|}\,\,
\sqrt{1+\kappa(p_1)}\,\,|\frac{t_1'}{t_2'}|\,\,a_{t_1'}(q p_1,t_2')\,a_{p_1'}(p_2,p_2')
\\ &  \ \ \ + \,\de_{t_1,t_2}\,\de_{|p_1 p_1'/p_2
t_1'|,q^{m_2'-1}}\,\,\de_{m_1+m_1'-1,m_2+m_2'}\,\,\de_{\sgn(p_1' t_1') q^{m_1'-1} p_2/p_1, p_2'/t_2'}
\\ & \qquad \times \ \sgn(t_1')\,q^{\hf(m_1-m_1'-1)}\,\sqrt{|\frac{t_1 p_1'}{p_1 t_1'}|}\,\,
\sqrt{1+\kappa(q^{-1}t_1')}\,\,|\frac{t_1'}{t_2'}|\,\,a_{q^{-1}t_1'}(p_1,t_2')\,a_{p_1'}(p_2,p_2')
\\ &  \ \ \ - \,\de_{t_1,t_2}\,\de_{|p_1 p_1'/p_2
t_1'|,q^{m_2'-1}}\,\,\de_{m_1+m_1'-1,m_2+m_2'}\,\,\de_{\sgn(p_1' t_1') q^{m_1'-1} p_2/p_1, p_2'/t_2'}
\\ & \qquad \times \ \sgn(p_1')\,q^{\hf(m_1+m_1'-1)}\,\sqrt{|\frac{t_1 t_1'}{p_1 p_1'}|}\,\,
\sqrt{1+\kappa(p_1')}\,\,|\frac{t_1'}{t_2'}|\,\,a_{t_1'}(p_1,t_2')\,a_{q p_1'}(p_2,p_2')
\end{split}
\end{equation}
and
\begin{equation}\label{eq:comultiE1}
 \begin{split}
&  (q-q^{-1})\,\,\langle W^* (f_{m_1,p_1,t_1} \ot f_{m_1',p_1',t_1'}) , (E^* \ot 1) (f_{m_2,p_2,t_2} \ot
f_{m_2',p_2',t_2'}) \rangle
\\  &  = \de_{t_1,q t_2}\,\de_{|p_1 p_1'/p_2
t_1'|,q^{m_2'}}\,\de_{m_1+m_1'-1,m_2+m_2'}\,\de_{\sgn(p_1' t_1') q^{m_1'} p_2/p_1, p_2'/t_2'}
\\  & \qquad\qquad  \times\ \sgn(t_1)\, q^{-\hf(m_2+1)}\,\sqrt{|\frac{p_2}{t_2}|}\,\sqrt{1+\kappa(t_2)}\,\, |\frac{t_1'}{t_2'}|\,
a_{t_1'}(p_1,t_2')\,\,a_{p_1'}(p_2,p_2')
\\ & \ \ \ -\,\de_{t_1, t_2}\,\de_{|p_1 p_1'/p_2
t_1'|,q^{m_2'-1}}\,\,\de_{m_1+m_1'-1,m_2+m_2'}\,\,\de_{\sgn(p_1' t_1') q^{m_1'-1} p_2/p_1, p_2'/t_2'}
\\ & \qquad\qquad \times\ \sgn(p_2)\, q^{\hf(m_2+1)}\,\sqrt{|\frac{t_2}{p_2}|}\,\sqrt{1+\kappa(q^{-1} p_2)}\,\, |\frac{t_1'}{t_2'}|\, a_{t_1'}(p_1,t_2')\,\,a_{p_1'}(q^{-1}
p_2,p_2')\ .
\end{split}
\end{equation}
One sees immediately that the right hand side of \eqref{eq:comultiE1} and agrees with the first
two terms on the right hand side of \eqref{eq:comultiE2} agree.
Thus in order to prove that the left hand sides of \eqref{eq:comultiE1} and \eqref{eq:comultiE2} agree
it suffices to show that, under the conditions $t_1 = t_2$, $m_1+m_1' = m_2+m_2'+1$, $|p_1 p_1'/p_2 t_1'| =
q^{m_2'-1}$ and $\sgn(p_1' t_1')\,q^{m_1'-1}\, p_2/p_1 = p_2'/t_2'$,
\begin{equation}\label{eq:comultiE3}
 \begin{split}
 0 & =   \sgn(p_2)\, q^{\frac{m_2+1}{2}}\,|t_2/p_2|^{\frac{1}{2}}\,\sqrt{1+\kappa(q^{-1} p_2)}\,\,
a_{t_1'}(p_1,t_2')\,\,a_{p_1'}(q^{-1} p_2,p_2')
\\ &    -\, \sgn(p_1)\,q^{\frac{m_1-m_1'-1}{2}}\,|t_1 p_1'/p_1 t_1'|^{\frac{1}{2}}\,\,
\sqrt{1+\kappa(p_1)}\,\,a_{t_1'}(q p_1,t_2')\,a_{p_1'}(p_2,p_2')
\\ &   +\, \sgn(t_1')\,q^{\frac{m_1-m_1'-1}{2}}\,|t_1 p_1'/p_1 t_1'|^{\frac{1}{2}}\,\,
\sqrt{1+\kappa(q^{-1}t_1')}\,\,a_{q^{-1}t_1'}(p_1,t_2')\,a_{p_1'}(p_2,p_2')
\\ &   - \sgn(p_1')\,q^{\frac{m_1+m_1'-1}{2}}\,|t_1 t_1'/p_1 p_1'|^{\frac{1}{2}}\,\,
\sqrt{1+\kappa(p_1')}\,\,a_{t_1'}(p_1,t_2')\,a_{q p_1'}(p_2,p_2')
\end{split}
\end{equation}
For this purpose we can use the $q$-contiguous relations
$$\sgn(p)\,\sqrt{1+\kappa(p)}\,\,a_{qp}(x,y) = \sgn(x)\,\sqrt{1+\kappa(q^{-1} x)}\,\,a_p(q^{-1} x,y) -
\frac{xp}{qy}\,\,a_p(x,y)$$ and
$$\sgn(p)\,\sqrt{1+\kappa(q^{-1}p)}\,\,a_{q^{-1}p}(x,y) = \sgn(x)\,\sqrt{1+\kappa(x)}\,\,a_p(qx,y) -
\frac{xp}{y}\,\,a_p(x,y)$$ for all $x,y,p \in I_q$ which follow from
Lemma \ref{lemB:qcontiguousrelapxy} and the symmetry relations \eqref{eq:symmetryforapxy}.
If one uses the first equality to replace $a_{q p_1'}(p_2,p_2')$ and the second one to replace
$a_{q^{-1}t_1'}(p_1,t_2')$ one checks that the above equality holds. Thus, we see that
 \eqref{eq:comultiE3} holds.

The linear span of elements $f_{m,p,t}$ forms a core for $E^*$.
So it follows that
\[
\langle W^* (f_{m_1,p_1,t_1} \ot f_{m_1',p_1',t_1'}) , (E^* \ot 1) v
\rangle = \langle W^* (E_0 \odot K_0 + K_0^{-1} \odot E_0)
(f_{m_1,p_1,t_1} \ot f_{m_1',p_1',t_1'}) , v \rangle
\]
for all $v \in D(E^* \ot 1)$. Hence, $W^* (E_0 \odot K_0 + K_0^{-1} \odot E_0) \subset (E \ot
1)W^*$, or
\[
E_0 \odot K_0 + K_0^{-1} \odot E_0 \subset W(E \ot
1)W^*\ \Longrightarrow\ K_0 \odot E_0 + E_0 \odot K_0^{-1} \subset \Sigma\, W(E \ot
1)W^*\,\Sigma\, =\, \hat{\De}(E).
\]

The last statement of Proposition \ref{prop:dualcomultiplicationandcomultiplicationsu}
is proved in the same way.
\end{proof}

\section{The Casimir operator}\label{sec:Casimiroperator}

\subsection{Definition of the Casimir operator}\label{ssec:Casimiroperator}
In this section we prove Theorem \ref{thm:Casimirwelldefinedandchar}. In order to show that the Casimir operator $\Omega$ as defined in Definition \ref{def:closedCasimir} is well-defined, we need to study the commutation relation between $K$ and $E$.

\begin{lemma} \label{lem:KandEstarEstronglycommute}
If $s \in \R$, then $K^{is} E = q^{is}\, E\,K^{is}$. Consequently,
$K$ and $E^\ast E$ strongly commute.
\end{lemma}

\begin{proof} By Definitions
\ref{def:actiongeneratorssu}, \ref{def:closedgenerators} and
Proposition \ref{prop:KEaffiliatedtohatM}, we find
$K^{is}\, f_{mpt} = q^{- \hf ism} |p/t|^{\hf is}\, f_{mpt}$
for $m\in \Z$ and $p,t \in I_q$. So the vector
$K^{is} f_{mpt} \in D(E)$ and
$K^{is} E \,f_{mpt}= q^{is}\, E\,K^{is}\,f_{mpt}$
by Definition \ref{def:actiongeneratorssu}. Since $E$ is the closure of $E_0$, with domain $D(E_0)$ the finite linear span of the $f_{mpt}$, and $K^{is}$ is bounded, this implies $K^{is} E \subseteq q^{is}\,E\,K^{is}$. Using Proposition \ref{prop:KEaffiliatedtohatM} we multiply this result
with the bounded operator $K^{-is}$ from the left and from the right to find $E\, K^{-is} \subseteq q^{is}\,K^{-is}\, E$, and since $s\in\R$ is arbitrary we have $K^{is} E = q^{is}\, E\,K^{is}$.

Taking adjoints we get $K^{is} E^\ast = q^{-is}\,E^\ast\,K^{is}$, and consequently $K^{is} E^\ast E = E^\ast E\,K^{is}$, so $K^{is}$ commutes with all spectral projections of the self-adjoint operator $E^\ast E$. In particular, $K$ and $E^\ast E$ are resolvent commuting, hence they strongly commute, see Appendix \ref{ssecA:commutationoperators}.
\end{proof}

Lemma \ref{lem:KandEstarEstronglycommute} leads to a proof of a part of Theorem \ref{thm:Casimirwelldefinedandchar}.

\begin{prop}\label{prop:OmwelldefinedaffiliatedKEstarEcommstrongly}
The Casimir operator $\Om$ as defined in Definition \ref{def:closedCasimir} is a well-defined
self-adjoint operator. Moreover, $\Om$ is affiliated to $\hat M$ and commutes strongly with $K$ and $E^*E$.
\end{prop}

\begin{proof}
Since $K$ and $E^\ast E$ are strongly commuting self-adjoint operators, see Appendix \ref{ssecA:commutationoperators}, we see that the closure $\Om$ of
\[
\frac{1}{2}\,\, \bigl(\,(q-q^{-1})^2\,E^\ast  E -
q\,K^2 - q^{-1}\,K^{-2}\,\bigr) \, .
\]
is a well-defined self-adjoint operator. Moreover, by Appendix \ref{ssecA:affiliationandgenerators} and Proposition  \ref{prop:KEaffiliatedtohatM}, the operators $K^2$, $K^{-2}$ and $E^\ast E$ are affiliated to $\hat{M}$. It follows that $\Om$ is affiliated to $\hat{M}$ and that $\Om$
commutes strongly with $K$ and $E^*E$.
\end{proof}

In order to show that $\Om$ also strongly commutes with $E$ we first need some preliminary results. Along the way we also prove the last statement of Theorem \ref{thm:Casimirwelldefinedandchar}.

Recall the decomposition of the Hilbert space $\cK$ into components $\cK(p,m,\ep,\et)$ with $p \in q^\Z$, $m \in \Z$ and $\ep,\et \in \{-,+\}$, and the corresponding decomposition \eqref{eq:decompE} of the operator $E$ into operators $E_{p,m}^{\ep,\et}$.

\begin{lemma} \label{lem:commutationEE*pmepet}
Let $p \in q^\Z$, $m \in \Z$ and $\ep,\et \in \{-,+\}$. Then
\[
(E_{p,m}^{\ep,\et})^*\,E_{p,m}^{\ep,\et} =
E_{p,m-1}^{\ep,\et}\, (E_{p,m-1}^{\ep,\et})^* + \, \frac{q^{2m}\,p - q^{-2m}
p^{-1} }{q-q^{-1}}\Id.
\]
\end{lemma}

\begin{proof}
Proposition \ref{prop:relationssuaresatisfied} implies that
\begin{equation} \label{eq:commutationEE*decomp}
E_0^\dag\vert_{\cK_0(p,m+1,\ep,\et)}\,
E_0\vert_{\cK_0(p,m,\ep,\et)} =
E_0\vert_{\cK_0(p,m-1,\ep,\et)}
\,E_0^\dag\vert_{\cK_0(p,m,\ep,\et)} + \, \frac{q^{2m}\,p -
q^{-2m} \,p^{-1}}{q-q^{-1}}\Id.
\end{equation}
If $\ep = -$ or $\et = -$, the lemma follows from this equality by the continuity of the operators involved.

It remains to deal with the case $\ep = \et = +$. From \eqref{eq:commutationEE*decomp} we see that the operators
\[
S_1 = (E_{p,m}^{++})^*\,E_{p,m}^{++} \qquad \text{and} \qquad
S_2 = E_{p,m-1}^{++}\, (E_{p,m-1}^{++})^* + \,
\frac{q^{2m}\,p - q^{-2m} p^{-1} }{q-q^{-1}}\Id
\]
are both self-adjoint extensions of $S := E_0^\dag\vert_{\cK_0(p,m+1,+,+)} \,E_0\vert_{\cK_0(p,m,+,+)}$. We will prove that they are the same by linking $S$ to a Jacobi operator, which is studied in Appendix \ref{app:jacobi}.

Set $\theta = q^m p$. By \eqref{eq:defE0forsu} and \eqref{eq:dualE0forsu}, we get for
$v \in \cK_0(p,m,+,+)$ and $x \in I_q^+$,
\[
\begin{split}
(q-q^{-1})^2&\,(S v)(x) =
[\,q^{m+1}\,\theta\,(1+q^{-2}x^2) + q^{-(m+1)}\,\theta^{-1}\, (1+\theta^2 x^2)\,] \, v(x) \\
& -  \sqrt{(1+q^{-2}x^2)\, (1+\theta^2 q^{-2} x^2)}\,\, v(q^{-1} x)  \,- \, \sqrt{(1+x^2)\,(1+\theta^2 x^2)}\,\, v(q x)
\end{split}
\]
Let $\{e_k\}_{k \in \Z}$ be the standard orthonormal basis of $\ell^2(\Z)$, and let $\cK(\Z)$ be the dense subspace consisting of finite linear combinations of the $e_k$'s. For $k \in \mathbb Z$ we denote $f_k = \langle f,e_k\rangle_{\ell^2(\Z)}$ for any $f \in \ell^2(\Z)$. We define the unitary transformation $U : \ell^2(\Z) \rightarrow \cK(p,m,+,+)$ so that $(U f)(q^k) = f_k$ for all $f \in \cK(\Z)$ and $k \in \Z$. So $U^* S U \in \mathrm{End}(\cK(\Z))$ is given by
\[
\begin{split}
(q-q^{-1})^2\,(\,U^* S U f)_k  = [\,q^{m+1}\,\theta\,(1+q^{2(k-1)}) + q^{-(m+1)}\,\theta^{-1}\,(1+\theta^2
q^{2k})\,] \, f_k \\
 - \, \sqrt{(1+q^{2(k-1)})\,(1+\theta^2 q^{2(k-1)})}\,\, f_{k-1}
\,  - \, \sqrt{(1+q^{2k})\,(1+\theta^2 q^{2k})}\,\, f_{k+1}
\end{split}
\]
for all $f \in \cK(\Z)$, $k \in \Z$. After a close inspection, one sees that
\[
U^* S U = (q-q^{-1})^{-2}\,\left(\, (q^{m+1} \theta +
q^{-m-1}\theta^{-1}) \,\,\text{Id} \, - \, 2\, L \,\right)
\]
where $L=L(q^{2+2|m|}, \theta^{-1},-q^2\mid q^2)$ is the Jacobi operator of Appendix \ref{ssecB:litteqJacobi},
see \eqref{eq:JacobioperatorreclqJacf}.

If $m \not= 0$, then $c=q^{2+2|m|} \leq q^4$ which by
Theorem \ref{thm:spectraldecompL} implies that $L$ and thus $S$
is essentially self-adjoint. Therefore $S_1 = S_2$ in this case.

Now assume that $m = 0$, so $c = q^2$. In this case $L$ is not essentially self-adjoint,
but we can use Theorem \ref{thm:Jacobioperatorforcasec=q} to prove that $S_1$ and $S_2$ are equal. From Proposition \ref{prop:EaffiliatedtohatM} or Proposition \ref{prop:KEaffiliatedtohatM}
we know that $E$ is affiliated to $\hat{M}$, implying that $E^* E$ and $E E^*$ are also affiliated to $\hat{M}$. This guarantees that
\[
\hat{J}\,Q(1,p,0)\hat{J} \, X \subseteq X\,\hat{J}\,Q(1,p,0)\hat{J}
\]
for $X = E^* E$ and $X = E E^*$, see Appendix \ref{ssecA:affiliationandgenerators}. Since $f_{0,p,1}$
belongs to $D(E^*E)$ and $D(E E^*)$, it
follows that the vector $w:=\hat{J}\,Q(1,p,0) \hat{J} f_{0,p,1}$
belongs to $D(E^* E)$ and $D(E E^*)$. As a
consequence, $w$ belongs to
$D((E_{p,0}^{++})^* E_{p,0}^{++}) = D(S_1)$ and
to $D(E_{p,-1}^{++}\,(E_{p,-1}^{++})^*) = D(S_2)$.

As in the proof of Lemma \ref{lem:behaviourhatJQppnhatJincase++}(ii) we see that
Corollary \ref{cor:lemexplicitactionQppn2} implies
\[
w \, = \, \hat{J}\, Q(p_1,p_2,n)\, \hat{J} \,
= \,  \sum_{x\in J(p,0,+,+)} \frac{1}{x}\, a_1(1,x)\, a_p(p,x)\, f_{0,px,x}
\]
so that $x\, w(x)\, = \, a_1(1,x)\, a_p(p,x) \, = \, h(x^{-2})$. By Lemma
\ref{lemB:apxyforyinqZ} the function $h\colon \R_{\geq 0}\to \R$ is differentiable
and $h(0)\not=0$.

So  $U^* w$ belongs to $D(U^* S_1 U)$ and
$D(U^* S_2 U)$  and
$(U^* w)_{-k} = (q^2)^{\frac{k}{2}}\, h(q^{2k})$
for all $k \in \Z$.
Since $U^* S_1 U$ and $U^* S_2 U$ are
both self-adjoint extensions of the operator
$$
(q-q^{-1})^{-2}\,\left(\,
(q^{m+1} \theta + q^{-m-1}\theta^{-1}) \,\,\text{Id} \, - \, 2\,
L \,\right) \, .
$$
Theorem \ref{thm:Jacobioperatorforcasec=q} now guarantees that
$U^* S_1 U = U^* S_2 U$ and we are done.
\end{proof}
Using Lemma \ref{lem:commutationEE*pmepet} we can describe the relation between $E^*$ and $E_0^\dagger$, and we can give a characterization of the operator $E$.

\begin{prop} \label{prop:charEandrelationEstarwithE0dag}
\begin{enumerate}[(i)]
\item $E^\ast$ is the closure of $E_0^\dag$.
\item $E$ is the unique closed,
linear operator in $\cK$ so that $E_0 \subseteq E$ and
$E_0^\dag \subseteq E^\ast$.
\end{enumerate}
\end{prop}

\begin{proof}
To prove the first statement, choose $p \in q^\Z$, $m \in \Z$ and
$\ep,\et \in \{-,+\}$. By Lemma \ref{lem:commutationEE*pmepet} there exists a constant $c \in \R$ such that
\begin{equation} \label{eq:E*E=EE^*+cId}
(E_{p,m}^{\ep,\et})^*\,E_{p,m}^{\ep,\et} =
E_{p,m-1}^{\ep,\et}\, (E_{p,m-1}^{\ep,\et})^* + c\,\Id.
\end{equation}
Thus, $D(\,(E_{p,m}^{\ep,\et})^*\,E_{p,m}^{\ep,\et})= D(E_{p,m-1}^{\ep,\et}\, (E_{p,m-1}^{\ep,\et})^*)$ and since these sets form a core for $E_{p,m}^{\ep,\et}$ and
$(E_{p,m-1}^{\ep,\et})^*$ respectively, \eqref{eq:E*E=EE^*+cId} implies that $D(E_{p,m}^{\ep,\et}) = D((E_{p,m-1}^{\ep,\et})^*)$ and $\|E_{p,m}^{\ep,\et}\,v\|^2 =\|(E_{p,m-1}^{\ep,\et})^*\,v\|^2 + c\,\|v\|^2$ for all $v \in D(E_{p,m}^{\ep,\et})$. Because $\cK_0(p,m,\ep,\et)$ is a
core for $E_{p,m}^{\ep,\et}$, this in turn guarantees that $\cK_0(p,m,\ep,\et)$ is a core for
$(E_{p,m-1}^{\ep,\et})^*$. In other words, $(E_{p,m-1}^{\ep,\et})^*$ is the closure of
$E_0^\dag\vert_{\cK_0(p,m,\ep,\et)}$. Thus,
\begin{eqnarray*}
E^* & = & \bigoplus_{\substack{p \in q^\Z, m \in \Z \\ \ep,\et \in\{-,+\}}}
\,(E_{p,m-1}^{\ep,\et})^*
= \bigoplus_{\substack{p \in q^\Z, m \in \Z \\ \ep,\et \in\{-,+\}}}\,
\overline{E_0^\dag\vert_{\cK_0(p,m,\ep,\et)}} \\
&  = & \Big(\,\sum_{\substack{p \in q^\Z, m \in \Z \\ \ep,\et \in\{-,+\}}}\,E_0^\dag\vert_{\cK_0(p,m,\ep,\et)}\,
\Big)^{\overline{\ \ }} =\overline{E_0^\dag}\ .
\end{eqnarray*}

For the second statement, we take a closed linear operator $F$ in
$\cK$ such that $E_0 \subseteq F$ and $E_0^\dag \subseteq F^*$.  Since, by
definition, $E$ is the closure of $E_0$ and $F$ is a
closed extension of $E_0$, we must have that $E \subseteq F$. By part (i)
we know that $E^*$ is the closure of
$E_0^\dag$. Since $F^*$ is a closed extension of
$E_0^\dag$, this implies that $E^* \subseteq F^*$ and
by taking the adjoint of this inclusion, we see that $F \subseteq
E$. Thus, $F = E$.
\end{proof}

We define, for $p\in q^\Z$, $m \in \Z$, $\ep,\et \in \{-,+\}$, self-adjoint operators in $\cK(p,m,\ep,\et)$ by
\begin{equation} \label{eq:defOmegapmepet}\index{O@$\Om_{p,m}^{\ep,\et}$}
\Om_{p,m}^{\ep,\et} = \frac12\Big( (q-q^{-1})^2 (E_{p,m}^{\ep,\et})^* E_{p,m}^{\ep,\et} - (q^{2m+1}p + q^{-2m-1}p^{-1})\Id \Big).
\end{equation}
Now we have the following decomposition of the Casimir operator;
\begin{equation} \label{eq:decompCasimir}
\Om = \bigoplus_{\substack{\ep,\et \in \{-,+\}\\p\in q^\Z, m \in \Z}} \Om_{p,m}^{\ep,\et},
\end{equation}
see Appendix \ref{ssecA:summationoperators}.

\begin{lemma} \label{lem:extensionsofOmega0}
Let $p\in q^\Z, m \in \Z, \ep,\et \in \{-,+\}$.
\begin{enumerate}[(i)]
\item If ($\ep = -$ or $\et=-$) or ($\ep=\et=+$ and $m \neq 0$), then $\Om^{\ep,\et}_{p,m}$ is the closure of the essentially self-adjoint operator $\Om_0\vert_{\cK_0(p,m,\ep,\et)}$.
\item $\Om^{+,+}_{p,0}$ is a self-adjoint extension of
$\Om_0\vert_{\cK_0(p,0,+,+)}$.
\end{enumerate}
\end{lemma}

Let us remark that $\Om^{+,+}_{p,0}$ is not the closure of $\Om_0\vert_{\cK_0(p,0,+,+)}$.

\begin{proof}
Take $p \in q^\Z$, $m \in \Z$ and $\ep,\et \in \{-,+\}$. If $\ep = -$ or $\et = -$, then \eqref{eq:actioncasimir}
and \eqref{eq:defKpmepeta} imply that $\Om_0\vert_{\cK_0(p,m,\ep,\et)}$ is bounded, hence essentially self-adjoint, and
$\Om_{p,m}^{\ep,\et}$ must be the closure of $\Om_0\vert_{\cK_0(p,m,\ep,\et)}$.

Now assume that $\ep = \et = +$ and set $\theta = q^m p$. As in the second half of the proof of Lemma \ref{lem:commutationEE*pmepet}, one sees that $\Om_0\vert_{\cK_0(p,m,\ep,\et)}$ is unitarily equivalent to $-L$, where $L=L(c,d,z\mid q)$ is the Jacobi operator
of Appendix \ref{app:jacobi} in base $q^2$ and with
parameters $c = q^{2+2|m|}$, $d = \theta^{-1}\,q^{|m|+1}$ and $r = -q^2$. If $m \not=0$, this implies, see \cite[Prop.~4.5.3]{KoelLaredo} and Appendix \ref{app:jacobi}, that $\Om_0\vert_{\cK_0(p,m,\ep,\et)}$ is essentially self-adjoint and $\Om_{p,m}^{\ep,\et}$ must be the closure of $\Om_0\vert_{\cK_0(p,m,\ep,\et)}$.

If $m=0$ the reasoning of the last part of the proof of Proposition \ref{lem:commutationEE*pmepet} shows,
since $\Om$ is affiliated to $\hat{M}$, that $\Om^{++}_{p,0}$ must be unitarily equivalent to the self-adjoint
extension of $-L_q$ described in Theorem \ref{thm:Jacobioperatorforcasec=q}.
\end{proof}
The proof of Lemma \ref{lem:extensionsofOmega0} and the last statement of Theorem \ref{thm:Jacobioperatorforcasec=q} lead to the following result, which will be useful later on and
for this reason it is stated separately.  Again we use the convention
\eqref{eq:conventioninnerproductfunction}.

\begin{lemma} \label{lem:vinD(Om)}
Consider $p \in q^\Z$, $m \in \Z$, $\ep,\et \in \{-,+\}$ and
$v \in D(\Om_0^*) \cap \cK(p,m,\ep,\et)$. Assume moreover
that if $m=0$ and $\ep= \et = +$, there exists a function
$h : \R_{\geq 0}  \to \R$ that is differentiable at $0$ and
satisfies $v(x) = x^{-1}\,h(x^{-2})$ for all $x \in I_q^+$.
Then $v$ belongs to $D(\Om)$ and $\Om\, v = \Om_0^*\, v$.
\end{lemma}

We are now ready to prove the last statement of Theorem \ref{thm:Casimirwelldefinedandchar}.

\begin{prop} \label{prop:OmuniqueextensionofOm0}
The Casimir operator $\Om$ is the unique self-adjoint extension of $\Om_0$ that is affiliated to $\hat M$.
\end{prop}

\begin{proof}
Choose a self-adjoint operator $C$ in $\cK$ so that $C$ is affiliated to $\hat{M}$ and $\Om_0 \subseteq C$.
We have to show that $C=\Om$.

We divide $\Om_0$ into two parts, For this purpose define
\[
\begin{split}
L = \Big\{(p,m,\ep,\et) \mid&\ p \in q^\Z,\ m \in \Z,\ \ep,\et \in \{-,+\} \\
&\text{ s.t. } (\ep = - \text{ or } \et = -)
\text{ or } (\ep = \et = + \text{ and } m \not= 0 ) \Big\}.
\end{split}
\]
Now set
\[
\Om_0^{(1)} = \sum_{(p,m,\ep,\et) \in L}\,
\Om_0\vert_{\cK_0(p,m,\ep,\et)}
\hspace{10ex} \text{and} \hspace{10ex}
\Om_0^{(2)} = \sum_{p \in q^\Z}\,
\Om_0\vert_{\cK_0(p,0,+,+)}\
\]
and define respective self-adjoint extensions
\[
\Om^{(1)} = \bigoplus_{(p,m,\ep,\et) \in L}\,
\Om^{\ep,\et}_{p,m} \hspace{10ex} \text{and} \hspace{10ex}
\Om^{(2)} = \bigoplus_{p \in q^\Z}\, \Om^{++}_{p,0}\ .
\]
By Lemma \ref{lem:extensionsofOmega0} we know that  $\Om_0^{(1)}$ is
essentially self-adjoint with closure $\Om^{(1)}$. Since
\[
\begin{split}
\text{Ker}(\Om_0^* \pm i\,\Id) &=
\text{Ker}( (\Om^{(1)} \oplus (\Om_0^{(2)})^*) \pm i\,\Id )\\
& = \text{Ker}\bigl( (\Om^{(1)} \pm i\,\Id) \oplus
((\Om_0^{(2)})^* \pm i\,\Id)\bigr)\\
& = \{0\} \oplus \text{Ker}\bigl((\Om^{(2)}_0)^* \pm i\,\Id\bigr) \ ,
\end{split}
\]
the theory of self-adjoint extensions via the deficiency spaces, see
\cite[\S XII.4]{DunfS},  implies the existence of a self-adjoint extension
$D$ of $\Om_0^{(2)}$ so that $C = \Om^{(1)} \oplus D$.

We have seen in Proposition \ref{prop:KaffiliatohatM} that $K$ is affiliated to $\hat{M}$ implying that $\hat{J} K \hat{J}$ is affiliated to $\hat{M}'$. By Definitions \ref{def:actiongeneratorssu}, \ref{def:closedgenerators} and \eqref{eq:expressionhatJonfmpt}, we know that $\cK_0$ is a core for $\hat{J} K \hat{J}$ and $\hat{J} K \hat{J} f_{-m,\ep\et\,q^m p t,t} = \sqrt{p}\,f_{-m,\ep\et\,q^m p t,t}$
for all $p \in q^\Z$, $m \in \Z$ and $\ep,\et \in \{-,+\}$. Thus, for each $p \in q^\Z$, the orthogonal projection $P_p$ of $\cK$ onto $\bigoplus_{m \in \Z,\ep,\et\in\{-,+\}}\,\,\cK(p,m,\ep,\et)$ belongs to $\hat{M}'$, since it is the spectral projection of $\hat{J} K \hat{J}$ with respect to the eigenvalue $\sqrt{p}$. Because $C$ is affiliated to $\hat{M}$, the operator $C$ commutes with each projection $P_p$. As a consequence, there exists for every $p \in q^\Z$ a self-adjoint extension $D_p$ of $\Om_0\vert_{\cK_0(p,0,+,+)}$
so that $D = \bigoplus_{p \in q^\Z}\,\,D_p$. As in the proof of Lemma \ref{lem:extensionsofOmega0}, the fact that $C$ is affiliated to $\hat{M}$ implies for every $p \in q^\Z$ that $D_p$ is unitarily equivalent to the self-adjoint extension described in Theorem \ref{thm:Jacobioperatorforcasec=q} and hence, $D_p = \Om^{++}_{p,0}$. Thus, we conclude that $\Om = C$.
\end{proof}

To finish the proof of Theorem \ref{thm:Casimirwelldefinedandchar} we need to prove the following result.

\begin{prop}\label{prop:EandOmstronglycommute}
The operators $E$ and $\Om$ strongly commute.
\end{prop}

Before embarking on the proof of Proposition \ref{prop:EandOmstronglycommute}, we first
collect all the elements for the proof of Theorem \ref{thm:Casimirwelldefinedandchar}.

\begin{proof}[Proof of Theorem \ref{thm:Casimirwelldefinedandchar}.]
By Proposition \ref{prop:OmwelldefinedaffiliatedKEstarEcommstrongly} the Casimir
operator is a well-defined self-adjoint operator affiliated to $\hat{M}$, and
by Proposition \ref{prop:OmuniqueextensionofOm0} the Casimir operator is the unique
self-adjoint extension of $\Om_0$ affiliated to $\hat{M}$. By
Proposition \ref{prop:OmwelldefinedaffiliatedKEstarEcommstrongly} the Casimir operator
commutes strongly with $K$, and by Proposition \ref{prop:EandOmstronglycommute} it
also commutes strongly with $E$.
\end{proof}

\begin{proof}[Proof of Proposition \ref{prop:EandOmstronglycommute}.]
By Proposition \ref{prop:OmwelldefinedaffiliatedKEstarEcommstrongly} the Casimir operator
$\Om$ is self-adjoint, and
we have to prove that
\[
E_\Om(B)\, E\subset E\, E_\Om(B)
\]
for all Borel sets $B\subset \R$, where
$E_\Om$ is the spectral decomposition of $\Om$, see Appendix \ref{ssecA:commutationoperators}.
Using the decompositions \eqref{eq:decompE}, \eqref{eq:defOmegapmepet}, \eqref{eq:decompCasimir}
and Lemma \ref{lem:extensionsofOmega0} it suffices to show
\[
E_{\Om^{\ep,\et}_{p,m+1}(B)}\, E^{\ep,\et}_{p,m}\, \subset\,  E^{\ep,\et}_{p,m}\, E_{\Om^{\ep,\et}_{p,m}}(B).
\]
for $p\in q^\Z$, $m\in\Z$, $\ep,\et\in \{-,+\}$.
Then, by \eqref{eq:defOmegapmepet} and Lemma \ref{lem:commutationEE*pmepet},
we get ---being careful regarding the domains involved---
\begin{equation} \label{eq:commutation}
\begin{split}
2\,\Om_{p,m+1}^{\ep,\et}\,E^{\ep,\et}_{p,m}
= & \, \bigl[\,(q-q^{-1})^2\,(E^{\ep,\et}_{p,m+1})^* E^{\ep,\et}_{p,m+1}
 - (q^{2m+3} p + q^{-2m-3} p^{-1})\, \Id\,\bigr]\,  \,E^{\ep,\et}_{p,m}
\\  = &\, \bigl[\,(q-q^{-1})^2\, E^{\ep,\et}_{p,m}
\,(E^{\ep,\et}_{p,m})^* + (q-q^{-1})\,(q^{2m+2} p - q^{-2m-2}
p^{-1})\,\Id
\\  & \qquad \qquad \ - (q^{2m+3} p + q^{-2m-3} p^{-1})
\,\text{Id}\,\bigl] \,E^{\ep,\et}_{p,m}
\\  = &\, \bigl[\,(q-q^{-1})^2\,
E^{\ep,\et}_{p,m}\,(E^{\ep,\et}_{p,m})^* -
(q^{2m+1} p + q^{-2m-1} p^{-1})\,\Id\,\bigl]
\,E^{\ep,\et}_{p,m} \\
= &\, E^{\ep,\et}_{p,m}\,\bigl[\,(q-q^{-1})^2\,
(E^{\ep,\et}_{p,m})^* E^{\ep,\et}_{p,m} - (q^{2m+1} p + q^{-2m-1}
p^{-1})\,\Id\,\bigl]\ \\
=&\,  2\,E^{\ep,\et}_{p,m}\,\Om^{\ep,\et}_{p,m} .
\end{split}
\end{equation}
Take the polar decomposition $E^{\ep,\et}_{p,m} = \,\tilde{U}^{\ep,\et}_{p,m}\,|E^{\ep,\et}_{p,m}|$. Since
\[
\begin{split}
|E^{\ep,\et}_{p,m}| = \bigl((E^{\ep,\et}_{p,m})^* E^{\ep,\et}_{p,m} \bigr)^{\frac{1}{2}}&\colon \cK(p,m,\ep,\et) \to \cK(p,m,\ep,\et), \\
\tilde{U}^{\ep,\et}_{p,m}&\colon \cK(p,m,\ep,\et) \to \cK(p,m+1,\ep,\et),
\end{split}
\]
\eqref{eq:defOmegapmepet} implies that $|E^{\ep,\et}_{p,m}|$ and $\Om^{\ep,\et}_{p,m}$ strongly commute. Choose $v \in D(|E^{\ep,\et}_{p,m}|^3\,)$. So $v \in D(\Om^{\ep,\et}_{p,m}\,|E^{\ep,\et}_{p,m}|) \,\cap\, D(|E^{\ep,\et}_{p,m}|\, \Om^{\ep,\et}_{p,m})$ by \eqref{eq:defOmegapmepet}, implying that $\Om^{\ep,\et}_{p,m}(|E^{\ep,\et}_{p,m}|\,v) = |E^{\ep,\et}_{p,m}|(\Om^{\ep,\et}_{p,m}\,v)$. Since $v \in D(E^{\ep,\et}_{p,m}\, \Om^{\ep,\et}_{p,m})$, \eqref{eq:commutation} implies that
$\tilde{U}^{\ep,\et}_{p,m}(|E^{\ep,\et}_{p,m}|\,v) \in D(\Om_{p,m+1}^{\ep,\et})$ and
\[
\Om_{p,m+1}^{\ep,\et}\,\tilde{U}^{\ep,\et}_{p,m}(|E^{\ep,\et}_{p,m}|\,v)
= \tilde{U}^{\ep,\et}_{p,m}\,|E^{\ep,\et}_{p,m}|\,\Om^{\ep,\et}_{p,m}\,v
= \tilde{U}^{\ep,\et}_{p,m}\,\Om^{\ep,\et}_{p,m}\,(|E^{\ep,\et}_{p,m}|\,v) \ .
\]
If $w \in \text{Ker} |E^{\ep,\et}_{p,m}|$, then by \eqref{eq:defOmegapmepet} $\Om^{\ep,\et}_{p,m}\, w = - (q^{2m+1} p + q^{-2m-1} p^{-1})\, w$, thus
$\tilde{U}^{\ep,\et}_{p,m}\,\Om^{\ep,\et}_{p,m}\, w = 0 = \Om^{\ep,\et}_{p,m+1}\,\tilde{U}^{\ep,\et}_{p,m}\,w$.

Now $\text{Ker} |E^{\ep,\et}_{p,m}| +
[\,\text{Im}|E^{\ep,\et}_{p,m}|\,\cap\, D\bigl((E^{\ep,\et}_{p,m})^*
E^{\ep,\et}_{p,m}\bigr)\,]$ is a core for
$(E^{\ep,\et}_{p,m})^* E^{\ep,\et}_{p,m}$ and thus for
$\Om^{\ep,\et}_{p,m}$, as follows by using the spectral decomposition of $|E^{\ep,\et}_{p,m}|$.
Consequently the above results and the
closedness of $\Om^{\ep,\et}_{p,m+1}$ imply that
$\tilde{U}^{\ep,\et}_{p,m}\,\Om^{\ep,\et}_{p,m} \subseteq \Om^{\ep,\et}_{p,m+1}\,\tilde{U}^{\ep,\et}_{p,m}$.
Now $|E|= \bigoplus |E^{\ep,\et}_{p,m}|$, and $\tilde{U}= \bigoplus \tilde{U}^{\ep,\et}_{p,m}$
give the polar decomposition $E\, =\, \tilde{U}\, |E|$, see Appendix \ref{ssecA:summationoperators}, and
we get $\tilde{U}\, \Om \, \subset\,  \Om\, \tilde{U}$, hence $\tilde{U} \, E_\Om(B)\, = \, E_\Om(B)\, \tilde{U}$
for any Borel set $B\subset \R$ by the spectral theorem.
It follows that
$E$ and $\Om$  strongly commute.
\end{proof}

\subsection{Graded commutation relations for the Casimir operator} \label{ssec:gradedcommutation}
This subsection is devoted to the proof of Proposition \ref{prop:decompMintoM+andM-andgradedcommutation}. The first statement of this proposition is an immediate consequence of Proposition \ref{prop:QppngeneratedualM}, which we already proved in Section \ref{ssec:commutantofdualvNalg}. Recall the subspaces $\hat M_+, \hat M_- \subset \hat M$ defined in Definition \ref{def:defKpmandMpm}.
Note that Proposition \ref{prop:QppngeneratedualM} implies that $\hat M_\pm$ is the strong-$\ast$ closure of
\begin{equation} \label{eq:hatM+-closureof}
\Span\{ Q(p_1,p_2,n) \ \mid \ p_1,p_2 \in I_q, n \in \Z \text{ so that } \sgn(p_1p_2)= \pm \}
\end{equation}
Next we investigate the graded commutation relations of the Casimir operator $\Om$ with the elements $Q(p_1,p_2,n)$ generating the von Neumann algebra $\hat{M}$, see Lemma \ref{lem:QppnareallinhatM}, as stated in Proposition \ref{prop:decompMintoM+andM-andgradedcommutation}. The hard computations
are contained in the following lemma, whose proof is postponed to Appendix \ref{app:proofofcommutionQwithCasimir}.

\begin{lemma}\label{lem:sgncommutationQppnwithCasimir0}
For $u,v\in\cK_0$, $p_1,p_2\in I_q$ and $n\in\Z$, we have
\[
\langle Q(p_1,p_2,n)\, u, \Om_0\, v\rangle
= \sgn(p_1p_2)\, \langle Q(p_1,p_2,n)\,\Om_0\, u, v\rangle.
\]
\end{lemma}

\begin{lemma} \label{lem:xOm0subset+-Omx}
Let $x \in \hat M_+$ and $y \in \hat M_-$, then $x\, \Om_0\,  \subset \, \Om\,  x$ and
$y\, \Om_0\,  \subset - \, \Om\, y$.
\end{lemma}

\begin{proof}
Consider $p_1,p_2,p,t \in I_q$, $n,m \in \Z$. From Lemma \ref{lem:sgncommutationQppnwithCasimir0} it follows that the vector $v\, =\, Q(p_1,p_2,n)\, f_{m,p,t}$ belongs to $D(\Om_0^*)$ and
\[
\Om_0^*\, v\, = \, \sgn(p_1p_2)\, Q(p_1,p_2,n)\, \Om_0\, f_{m,p,t}.
\]
By Lemma \ref{lem:explicitactionQppn} the vector $v \in \cK(p,m+n,\ep\sgn(p_1),\et\sgn(p_2))$,
and if $m+n=0$, $\ep\sgn(p_1)=+$, $\et\sgn(p_2)=+$,
there exists by Lemma \ref{lem:behaviourhatJQppnhatJincase++} a function $h:\R_{\geq0} \rightarrow \C$ that is differentiable in $0$ and satisfies $v(x)\, =x^{-1}\, h(x^{-2})$ for all $x \in I_q^+$. From Lemma \ref{lem:vinD(Om)} we now conclude that $v \in D(\Om)$ and that $\Om\, v = \Om_0^*\, v$, hence
\[
\sgn(p_1p_2)\, Q(p_1,p_2,n)\, \Om_0\, \subset\, \Om\, Q(p_1,p_2,n).
\]
Now for $x \in \hat M_+$ and $y \in \hat M_-$ the lemma follows from the closedness of $\Om$ and \eqref{eq:hatM+-closureof}
\end{proof}

We need to improve the commutation relations from Lemma \ref{lem:xOm0subset+-Omx} to come to the second statement of Proposition \ref{prop:decompMintoM+andM-andgradedcommutation}. To do this we need the following lemma.

\begin{lemma}\label{lem:UsterAUisselfadjoint}
Consider a Hilbert space $H$, a self-adjoint operator $A$ in $H$ and a partial isometry $U$ on $H$ for which the final projection  $U U^*$ commutes with $A$. Then $U^* A U$ is self-adjoint.
\end{lemma}

\begin{proof}
First we show that $U^*AU$ is densely defined. Set $P = U^* U$ and $Q = UU^*$. Since $Q A \subset A Q$, we have that $U(U^* D(A)) = Q D(A) \subset D(A)$ implying that $U^* D(A) \subset D(U^* A U)$. Clearly, $(1-P)H \subset D(U^* A U)$ thus $U^* D(A) + (1-P)H \subset D(U^* A U)$ from which it follows that $U^* A U$ is densely defined.

Next we need to verify the self-adjointness. Let $v,w \in D(U^* A U)$, then, since $A$ is self-adjoint,
$$
\langle U^* A U v , w \rangle = \langle A U v , U w \rangle =
\langle U v , A U w \rangle = \langle v , U^* A U w
\rangle\, .
$$
Thus, $U^* A U$ is symmetric. To prove that $U^* A U$ is
self-adjoint, choose $v \in D((U^* A U)^*)$. If $w \in D(A)$, then
$Qw \in D(A)$ and $A(Q w) = Q ( A w)$. Thus,
\begin{equation*}
\begin{split}
\langle U v , A w \rangle & =  \langle v , U^* A w \rangle = \langle v , U^* Q A w \rangle
= \langle v, U^* A Q w \rangle
\\ = & \langle v , (U^* A U) U^* w \rangle =
\langle (U^* A U)^* v , U^* w \rangle = \langle U\,(U^* A U)^* v, w \rangle \, .
\end{split}
\end{equation*}
This implies $U v \in D(A^*)=D(A)$, so that $v \in D(U^* A U)$.
{}From this we conclude that $(U^* A U)^* = U^* A U$.
\end{proof}

We are now in a position to prove the graded commutation relations of the Casimir.

\begin{prop} \label{prop:(anti)commute}
Let $x \in \hat{M}_+$ and $y \in \hat{M}_-$, then $x\,\Om \subset \Om\,x$ and $y\,\Om \subset -\,\Om\,y$.
\end{prop}

We have now collected all the necessary ingredients for the proof of
Proposition \ref{prop:decompMintoM+andM-andgradedcommutation}.

\begin{proof}[Proof of Proposition \ref{prop:decompMintoM+andM-andgradedcommutation}.]
By Proposition \ref{prop:QppngeneratedualM}, already established in
Section \ref{ssec:commutantofdualvNalg}, we obtain the decomposition
$\hat{M}\, =\, \hat{M}_+\oplus \hat{M}_-$. The final statement of
Proposition \ref{prop:decompMintoM+andM-andgradedcommutation} is Proposition
\ref{prop:(anti)commute}.
\end{proof}

\begin{proof}[Proof of Proposition \ref{prop:(anti)commute}.]
First we deal with $\hat{M}_+$. Choose a unitary $u \in \hat{M}_+$. From Lemma \ref{lem:xOm0subset+-Omx} we know that $u\,\Om_0 \,\subseteq \,\Om\, u$, thus $\Om_0\, \subseteq\, u^*\, \Om\, u$. Since $u^*\, \Om\, u$ is a self-adjoint extension of $\Om_0$ that is affiliated with $\hat{M}$, Proposition \ref{prop:OmuniqueextensionofOm0} guarantees that $\Om = u^*\, \Om\, u$, or in other words,
$u\, \Om \, =\, \Om\, u$. Since each element in $\hat{M}_+$
is a linear combination of such unitary elements, we get that $x\,\Om\, \subseteq\, \Om\,x$ for all $x \in \hat{M}_+$, proving the first statement.

Next choose $y \in \hat{M}_-$ and consider the polar decomposition $y = v\,|y|$ of $y$. We are going to show that $v \in \hat M_-$. Since $y^* \in \hat M_-$, the operator $y^* y$ is in the von Neumann algebra $\hat{M}_+$, hence $|y|\, =\, (y^* y)^{\frac{1}{2}} \in \hat{M}_+$. Take $e \in \cK_+$. Since
$|y|\cK_+\subseteq \cK_+$, there exists $e_1 \in \overline{|y| \cK_+}$ and $e_2 \in \cK_+$ with $e_2 \perp |y| \cK_+$ so that $e = e_1 + e_2$. Since also $|y| \cK_-\subseteq \cK_-$, we see that $e_2 \perp |y| \cK$, implying that $v e = v e_1 + v e_2 = v e_1$, since $v$ acts as zero on $(\text{Im}|y|)^\perp$.
Because $y = v\,|y|$ and $y \cK_+ \subseteq \cK_-$, it follows that $v e \in \cK_-$. Similarly, $v \cK_- \subseteq \cK_+$. Hence, $v \in \hat{M}_-$.

It follows that the initial projection $p = v^* v$ and final projection $q = v v^*$ belong to $\hat{M}_+$. This implies that $p\,\Om \subseteq \Om\,p$ and $q\,\Om \subseteq \Om\,q$ by the first part of this proposition.

Because $v \in \hat{M}_-$, we have that $v\, \Om_0 \, \subseteq \, - \Om\, v$ by Lemma \ref{lem:xOm0subset+-Omx}, implying that $p \,\Om_0\, \subseteq\, - v^*\, \Om\, v$. We also have that
$(1-p) \,\Om_0 \, \subseteq \, \Om\, (1-p)$. Thus, $\Om_0 \subseteq - v^* \Om v + \Om (1-p)$.

Since the final projection of $v$ commutes with the self-adjoint operator $\Om$,
the operator $v^*\, \Om\, v$ is also self-adjoint by Lemma \ref{lem:UsterAUisselfadjoint}. Because of the same reason, $\Om (1-p)$ is self-adjoint. Therefore, as the orthogonal sum of self-adjoint operators,
the operator $- v^*\, \Om\, v + \Om\, (1-p)$ is a self-adjoint extension of $\Om_0$.

Since $\Om$ is affiliated to $\hat{M}$ and $v, p \in \hat{M}$, one sees that $- v^* \, \Om\, v + \Om \,(1-p)$ is affiliated to $\hat{M}$. Hence, $\Om = - v^*\, \Om\, v + \Om\, (1-p)$ by Proposition \ref{prop:OmuniqueextensionofOm0}. If $e \in D(\Om)$, this equality
implies that $v e \in D(\Om)$, $(1-p) e \in D(\Om)$
and $\Om\, e = - v^*\, \Om\, v\, e + \Om \,(1-p)\, e$. Thus,
using $v\,p\,=\, v$, $q\,v\,=\, v$,
$$
v\,\Om\, e\, =\, - q\,\Om\, v\, e + v\, p\, \Om \,(1-p)\, e
\,=\, - \Om\, q\, v\, e + v\, \Om\, p\, (1-p)\, e \,=\, - \Om\, q\, v\, e\, =\, - \Om \,v\, e.
$$
Thus, we have proved that $v\, \Om\, \subseteq\, - \Om\, v$.
Since $y = v |y|$ we conclude that $y\, \Om\, \subseteq - \,\Om\, y$.
\end{proof}

\subsection{Spectral decomposition of the Casimir operator}\label{ssec:spectraldecompCasimir}
From (the proof of) Lemma \ref{lem:xOm0subset+-Omx} it follows that $Q(p_1,p_2,n)$ maps eigenvectors for eigenvalue $x$ of $\Om$ in $\cK(p,m,\ep,\et)$ to multiples of eigenvectors of $\Om$
in $\cK(p,m+n,\sgn(p_1)\ep,\sgn(p_2)\et)$ for the eigenvalue $\sgn(p_1p_2)x$ or to zero.
So, it will be convenient to have an alternative description of the GNS-space $\cK$ corresponding to the spectral decomposition of $\Om$. This alternative description has the advantage that the action of the operators $E$ and, of course, $\Om$, is far more transparent. Moreover, it leads to the direct integral decomposition of the left regular corepresentation of $(M,\De)$ into irreducible unitary representations,
see Section \ref{sec:decompleftregularcorep}.

The description of the spectral decomposition of $\Om$ relies on certain special functions which can be written in terms of basic hypergeometric series: the Al-Salam--Chihara polynomials and the little $q$-Jacobi functions. The main properties of these special functions needed in this subsection are given in Appendices \ref{ssecB:AlSCpol} and \ref{ssecB:litteqJacobi}. The spectral decomposition of the Casimir immediately leads to the decomposition of the GNS-space $\cK$ as a $\su$-module. This is done in Section \ref{ssec:decompUq(su11)}.\\

The Casimir operator $\Om$ is a self-adjoint extension of $\Om_0 \in \cL^+(\cK_0)$.
Let $p \in q^\Z$, $m \in \Z$, $\ep,\et \in \{-,+\}$. It follows from
\eqref{eq:actioncasimir} that $\Om_0\vert_{\cK_0(p,m,\ep,\et)}$ is
basically a Jacobi operator, i.e., a tridiagonal operator on $\ell^2(\NN)$ or $\ell^2(\Z)$. The spectral decomposition of these specific Jacobi operators can be described in terms of Al-Salam--Chihara polynomials in case of $\ell^2(\N_0)$, and in terms of little $q$-Jacobi functions in case of $\ell^2(\Z)$. Whether $\cK_0(p,m,\ep,\et)$ can be identified with $\ell^2(\N_0)$ or $\ell^2(\Z)$ depends on the sign of the parameters $\ep$ and $\et$, see the beginning of Section \ref{ssec:decompositionofGNSspace}. We need to distinguish between four different cases.

Let us recall from \eqref{eq:JE0JisminE0dag} that the modular conjugation $J:\cK\to\cK$, defined
by $J\colon f_{m,p,t}\mapsto f_{-m,t,p}$, satisfies that $E_0^\dag J=-E_0J$ and $JK_0=K_0^{-1}J$, and consequently $J \Om_0 = \Om_0 J$. Note that $J\colon \cK(p,m,\ep,\et)\to \cK(p^{-1},-m,\et,\ep)$,
since $Jf_{-m,\ep\et q^mpz,z}=f_{m,\et\ep q^{-m}p^{-1}y,y}$ with $y=\ep\et q^m pz$ and $\sgn(y)=\et$. We will use this to reduce the number of cases that we need to consider.

\subsubsection{The case $\ep=+$ and $\et=-$} \label{sssec:caseep+andet-}
Recall that $\cK(p,m,\ep,\et)\cong \ell^2(J(p,m,\ep,\et))$. In the case under consideration,
\[
J(p,m,+,-)=\{ z\in I_q\mid -q^mpz\in I_q,\ \sgn(z)=+\}
\]
can be labeled by $\NN$ using $n=m+\chi(p)+\chi(z)-1$. Now put $e_n = f_{-m,\ep\et q^mpz,z}$
using this identification, then \eqref{eq:actioncasimir} leads to
\begin{equation}\label{eq:modCasonKpm+-}
\begin{split}
2\Om_0 \, e_n &= \sqrt{(1-q^{2n+2})(1+ p^{-2}q^{2n+2-2m})}
\, e_{n+1} \\
&\qquad + p^{-1}q^{2n+1-2m}\bigl(q^{2m}-1\bigr) \, e_n +
\sqrt{(1-q^{2n})(1+ p^{-2}q^{2n-2m})} \, e_{n-1}.
\end{split}
\end{equation}
Comparing this with the Jacobi operator $J(a,b\mid q)$ for the Al-Salam--Chihara polynomials, \eqref{eq:JacobioperAlSCpol}, see also \eqref{eq:recAlSCpol}, we see that $2\Om_0 = J(q/p,-q^{1-2m}/p\mid q^2)$. By Theorem \ref{thm:spectraldecompASC} and \eqref{eq:measureASC} $2\Om_0$ extends uniquely to a bounded self-adjoint operator on $\cK(p,m,+,-)$, and it has continuous spectrum $[-1,1]$ and discrete spectrum $\si_d(p,m,+,-)=\mu(D(p,m,+,-))$ where
$D(p,m,+,-) = D(q/p,-q^{1-2m}/p|q^2)$, using the notation of \eqref{eq:defIabqforASCpol}.
The multiplicity of the (generalized) eigenspaces is one.

Let $I(p,m,+,-)= I(q/p,-q^{1-2m}/p|q^2)$, see \eqref{eq:defIabqforASCpol}. We define the operator
\begin{equation}\label{eq:defUpsilonpm+-}
\begin{split}
&\Up^{+,-}_{p,m} \colon \cK(p,m,+,-) \to
L^2(I(p,m,+,-)), \\
f_{-m,-q^m pz,z}\mapsto &g_z(\,\cdot\, ;p,m,+,-) =
(-1)^m \, h_{m+\chi(p)+\chi(z)-1}(\,\cdot\,;\frac{q}{p}, -\frac{q^{1-2m}}{p}\mid q^2)
\end{split}
\end{equation}
in terms of Al-Salam--Chihara polynomials using the notation as in
\eqref{eq:deffnabqforASCpol}. Then $\Up^{+,-}_{p,m}$ gives the spectral decomposition of the action
of the Casimir operator on $\cK(p,m,+,-)$, so $\Up^{+,-}_{p,m}$ is a unitary intertwiner of the Casimir operator with the multiplication operator $M(x)$\index{M@$M(x)$} on $L^2(I(p,m,+,-))$. Here, and elsewhere, $M(g)$ denotes the operator of multiplication by the function $g$. The factor $(-1)^m$ in \eqref{eq:defUpsilonpm+-} is not of importance for the spectral decomposition of the Casimir operator, but is inserted in order to avoid signs later on when we decompose $\cK$ as a $\su$-module.

\subsubsection{The case $\ep=-$ and $\et=+$} \label{sssec:caseep-andet+}
Using the modular conjugation $J$, the case $\ep=-$ and $\et=+$ can be reduced to the case
$\ep=+$ and $\et=-$. Define $I(p,m,-,+)=I(p^{-1},-m,+,-)=I(pq,-pq^{1+2m}|q^2)$ using the
notation \eqref{eq:defIabqforASCpol}, then
\begin{equation}\label{eq:defUpsilonpm-+}
\begin{split}
&\Up^{-,+}_{p,m} = (-1)^m\, \Up^{+,-}_{p^{-1},-m}\circ J
\colon \cK(p,m,-,+) \to L^2(I(p,m,-,+)), \\
&f_{-m,-pq^mz,z}\, \mapsto\,  g_z(\,\cdot\, ;p,m,-,+) =
h_{\chi(z)-1}(\,\cdot\, ;pq,-pq^{1+2m}|q^2)
\end{split}
\end{equation}
gives the intertwiner of the action of the Casimir operator
$\Om_0\colon \cK(p,m,-,+)\to\cK(p,m,-,+)$ with $M(x)$. As
before $\Om_0$ has a unique extension to a bounded
self-adjoint operator with multiplicity one for the (generalized)
eigenspaces.

Combining Sections \ref{sssec:caseep+andet-} and \ref{sssec:caseep-andet+}
we see that for $\ep,\et \in \{-,+\}$, $\ep\neq \et$, we have the following description of the discrete spectrum $\mu(D(p,m,\ep,\et))$:
\begin{equation} \label{eq:spectrum+-}
\begin{split}
D(p,m,\ep,\et) =&  \{\, q^{1+2r} p^{-\ep} \mid r \in \N_0, \,q^{1+2r} p^{-\ep} > 1\,
  \} \\
  &\cup\, \{\, - q^{1+2r}
p^{-\ep} \mid r \in \Z,\, r \geq - \ep\,m,\,
\,q^{1+2r} p^{-\ep} > 1\,\} .
\end{split}
\end{equation}

\subsubsection{The case $\ep=-$ and $\et=-$} \label{sssec:caseep-andet-}
In this case the $q$-interval $J(p,m,-,-)=\{ z\in I_q \mid q^mpz\in I_q,\, \sgn(z)=-\}$ can be labeled by $\NN$. If we put $z=-q^{n+1}$, $n\in\NN$, then we need $m+\chi(p)+n\in\NN$ in order to have $q^mpz\in I_q$.
So we have to consider two cases; $m+\chi(p)\geq 0$ and $m+\chi(p)\leq 0$. Since the modular conjugation $J$ changes the sign of $m+\chi(p)$ we can restrict to $m+\chi(p)\geq 0$, and
obtain the other case using $J$.

We assume $m+\chi(p)\geq 0$ and put $z=-q^{n+1}$, $n\in\NN$, so that $n$ labels $J(p,m,-,-)$. Put
$e_n=f_{-m,-pq^{n+m+1},-q^{n+1}}$, then the expression \eqref{eq:actioncasimir} for $\Om_0f_{mpt}$ gives
\begin{equation}\label{eq:modCasonKpm--}
\begin{split}
2(-\Om_0)\, e_n = &\sqrt{(1-q^{2n+2})(1-p^2q^{2m+2n+2})}\, e_{n+1} +
pq^{2n+1}(1+q^{2m})\, e_n \\
&+ \sqrt{(1-q^{2n})(1-p^2q^{2m+2n})}\, e_{n-1},
\end{split}
\end{equation}
which we recognize using \eqref{eq:JacobioperAlSCpol} and
\eqref{eq:recAlSCpol} as the Jacobi operator $J(pq,pq^{1+2m} \mid q^2)$ for the Al-Salam--Chihara polynomials. So $\Om_0$ uniquely extends to a bounded
self-adjoint operator. Put $I(p,m,-,-)= - I(pq,pq^{1+2m}|q^2)$, see
\eqref{eq:defIabqforASCpol}, then
\begin{equation}\label{eq:defUpsilonpm--m+chi(p)geq0}
\begin{split}
&\Up^{-,-}_{p,m} \colon \cK(p,m,-,-) \to
L^2(I(p,m,-,-)), \\
f_{-m,q^m pz,z}\mapsto &g_z(\,\cdot\, ;p,m,-,-) =
(-1)^m \, h_{\chi(z)-1}(-\,\cdot\,;pq, pq^{1+2m}|q^2)
\end{split}
\end{equation}
intertwines the action of the Casimir operator with the
multiplication operator $M(x)$ on $L^2(I(p,m,-,-))$ for
$m+\chi(p)\geq 0$. Note that we take the normalized
Al-Salam--Chihara polynomials with a minus sign in front of the
argument because of the minus sign in front of $\Om_0$ in \eqref{eq:modCasonKpm--}.

In case $m+\chi(p)\leq 0$ we define $I(p,m,-,-)= I(p^{-1},-m,-,-)$ and
\begin{equation}\label{eq:defUpsilonpm--m+chi(p)leq0}
\begin{split}
\Up^{-,-}_{p,m} &= (-1)^{m+\chi(p)} \Up^{-,-}_{p^{-1},-m}\circ J
\colon \cK(p,m,-,-) \to L^2(I(p,m,-,-)), \\
&f_{-m,pq^mz,z} \mapsto g_z(\,\cdot\, ;p,m,-,-) = (-1)^{\chi(p)}
h_{m+\chi(p)+\chi(z)-1}(-\,\cdot\, ;q/p, q^{1-2m}/p\mid q^2).
\end{split}
\end{equation}
This gives two definitions in case $m+\chi(p)=0$ or $q^mp=1$,
and it is straightforward to check that they coincide. Now we have the intertwiner for the action of the Casimir operator with the multiplication operator $M(x)$ on $L^2(I(p,m,-,-))$ for all $m\in \Z$ and $p\in q^\Z$. Let us remark that the discrete spectrum $\mu(D(p,m,-,-))$ is given explicitly by
\begin{equation}\label{eq:spectrum--}
D(p,m,--)  =
\begin{cases}
 \{\, - q^{1+2r} p \mid r \in \N_0, \,q^{1+2r} p > 1\,\} \\
\ \cup\, \{\,  - q^{1+2(r+m)} p \mid r \in \N_0,
\,q^{1+2(r+m)} p > 1\,\}, & p q^m \leq 1,\\
\{\, - q^{1+2r} p^{-1} \mid r \in \N_0,
  \,q^{1+2r} p^{-1} > 1\,\} \\
\  \cup\, \{\,  - q^{1+2(r-m)} p^{-1}
\mid r \in \N_0, \,q^{1+2(r-m)} p^{-1} > 1 \,\}, & p q^m \geq 1.
\end{cases}
\end{equation}
In both cases at most one of the two sets is non-empty.

\subsubsection{The case $\ep=+$ and $\et=+$} \label{sssec:caseep+andet+}
In this case $J(p,m,+,+)$ can be labeled by $\Z$. We put $z=q^n$, $n\in\Z$, and $e_n=f_{-m, pq^{n+m},q^n}$, then \eqref{eq:actioncasimir} gives
\begin{equation}\label{eq:modCasonKpm++}
\begin{split}
2(-\Om_0)\, e_n =& \sqrt{(1+q^{2n})(1+p^2q^{2m+2n})}\, e_{n+1} -
pq^{2n-1}(1+q^{2m})\, e_n \\
&+ \sqrt{(1+q^{2n-2})(1-p^2q^{2m+2n-2})}\, e_{n-1}.
\end{split}
\end{equation}
Comparing this with \eqref{eq:JacobioperatorreclqJacf} we recognize $-2\Om_0$ as the (doubly infinity) Jacobi operator $L(q^{2-2m},q^{1-2m}/p,-q^2 \mid q^2)$ for the little $q$-Jacobi functions. Let us remark that there are other choices for the parameters which, of course, all lead to the same result; we can identify $-2\Om_0$ also with $L(q^{2m+2}, q/p, -q^2\mid q^2)$, $L(q^{2m+2},pq^{2m+1}, -q^{2-2m}/p^2\mid q^2)$ or $L(q^{2-2m},pq, -q^{2-2m}/p^2\mid q^2)$. Because of these symmetries we can obtain the spectral decomposition of a self-adjoint extension of $\Om_0$ in the case $m>0$ from the case $m<0$.

Let us first assume that $m \leq 0$. By Theorem \ref{thm:spectraldecompL} the unbounded operator $\Om_0$ is essentially self-adjoint for $m < 0$, so in this case $\Om_0$ has a unique self-adjoint extension $C$. The spectral decomposition of $C$ is described in Theorem \ref{thm:spectraldecompL}.
For $m=0$ we choose the self-adjoint extension $C$ of
$\Om_0\vert_{\cK(p,0,+,+)}$ with spectral decomposition as
described in Theorem \ref{thm:Jacobioperatorforcasec=q}. The multiplicity of the (generalized) eigenspaces is one. Put $I(p,m,+,+)= -I(q^{2-2m}, p^{-1}q^{1-2m};-q^2|q^2)$ using the
notation as in \eqref{eq:defIabqforlqJacfun}, then for $m\leq 0$
\begin{equation}\label{eq:defUpsilonpm++mleq0}
\begin{split}
&\Up^{+,+}_{p,m} \colon \cK(p,m,+,+) \to
L^2(I(p,m,+,+)), \\
f_{-m,q^m pz,z}\mapsto &g_z(\,\cdot\, ;p,m,+,+) =
(-1)^m \, j_{\chi(z)}(-\,\cdot\,; q^{2-2m}, p^{-1}q^{1-2m};-q^2\mid q^2)
\end{split}
\end{equation}
intertwines the action of the Casimir operator with the multiplication operator $M(x)$ on $L^2(I(p,m,+,+))$, using the notation \eqref{eq:orthonorlqJacfun}. To this end we need to
argue that $C$ agrees with $\Om^{+,+}_{p,m}$, which is clear in the case $m \not=0$. For $m=0$ we recall from the proof of Lemma \ref{lem:extensionsofOmega0} that $\Om^{+,+}_{p,0}$ is the self-adjoint extension of $\Om_0\vert_{\cK(p,0,+,+)}$ described in Theorem \ref{thm:Jacobioperatorforcasec=q}, as is $C$.

Note that we take in \eqref{eq:defUpsilonpm++mleq0} the normalized little $q$-Jacobi
functions with a minus sign in front of the argument because of the minus sign in front of $\Om_0$ in \eqref{eq:modCasonKpm++}.

For $m\geq 0$ define $I(p,m,+,+)= I(p^{-1},-m,+,+)$ and
\begin{equation}\label{eq:defUpsilonpm++mgeq0}
\begin{split}
\Up^{+,+}_{p,m} &= (-1)^m\, \Up^{+,+}_{p^{-1},-m}\circ U
\colon \cK(p,m,+,+) \to L^2(I(p,m,+,+)), \\
f_{m,pq^mz,z} &\mapsto g_z(\,\cdot\, ; p,m,+,+) =
j_{m+\chi(p)+\chi(z)} (-\,\cdot\, ;q^{2+2m},pq^{1+2m};-q^2|q^2).
\end{split}
\end{equation}
Note that this corresponds to the symmetry of the corresponding Jacobi operator,
\[
L(q^{2-2m},q^{1-2m}/p,-q^2 \mid q^2)=L(q^{2+2m},pq^{1+2m},-q^2 \mid q^2),
\]
as we observed earlier. As before, for $m \geq 0$ the operators $\Up_{p,m}^{+,+}$ intertwine the action of the Casimir operator with the multiplication operator $M(x)$ on $L^2(I(p,m,+,+))$.

It may seem that we now have two definitions for $\Up_{p,0}^{+,+}$, but it follows from \eqref{eq:specialsymmetryjk} that they coincide.

Finally, let us give an explicit description of the discrete spectrum $\mu(D(p,m,+,+))$:
\begin{equation}\label{eq:spectrum++}
\begin{split}
D(p,m,+,+)  = &\,\{\, q^{1+2k} p \mid k \in \Z, \,q^{1+2k} p > 1 \,\}\\
&\,\cup\,
 \{\,- q^{1+2r} p \mid r \in \Z, \, r \geq \max\{0,m\}, \,q^{1+2r} p
> 1\,\}
\\ &  \, \cup\, \{\, - q^{1+2r} p^{-1} \mid r \in \Z, \, r
  \geq \max\{0,-m\}, \,q^{1+2r} p^{-1} > 1\,\}.
\end{split}
\end{equation}
The last two sets are finite and at most one of them is non-empty,
while the first set is infinite.

\subsubsection{The spectral decomposition of $\Om$}
Gathering the results from the four different cases $\ep=\pm$ and $\et=\pm$, we obtain the spectral decomposition of the Casimir operator $\Om$.

\begin{thm} \label{thm:spectraldecompositionCasimir}
There exists a unique  unitary operator
\begin{equation} \label{eq:unitaryUpsilon}\index{Y@$\Up^{\ep,\et}_{p,m}$}\index{Y@$\Upsilon$}\index{L@$L^2(I(p,m,\ep,\et)$}
\begin{split}
\Up : \cK &\to \bigoplus_
{\substack{p \in q^\Z, m \in \Z \\ \ep,\et
\in \{-,+\}}} \, L^2\big(I(p,m,\ep,\et)\big), \\
\Up\big(\,f_{-m,\ep\,\et\,q^m\,p\,z,z}\,\big) &=
g_z(\,\cdot\,;p,m,\ep,\et) ,\ \ \forall\,  z \in J(p,m,\ep,\et),
\end{split}
\end{equation}
so that for $p \in q^\Z$, $m \in \Z$ and $\ep,\et \in \{-,+\}$, we have $\Up\big(\cK(p,m,\ep,\et)\big) = L^2\big(I(p,m,\ep,\et)\big)$.
Let $\Up^{\ep,\et}_{p,m} \colon \cK(p,m,\ep,\et) \to L^2(I(p,m,\ep,\et))$ be the restriction of $\Up$ to $\cK(p,m,\ep,\et)$. Then, for $p \in q^\Z$, $m \in \Z$ and $\ep,\et \in \{-,+\}$,
\[
\Up^{\ep,\et}_{p,m} \,\,\Om^{\ep,\et}_{p,m}\,
\bigl(\Up^{\ep,\et}_{p,m}\bigr)^* = M(x)
\qquad \text{on } L^2(I(p,m,\ep,\et))\ .
\]
Here $I(p,m,\ep,\et)\,=\, [-1,1]\,\cup\, \si_d(p,m,\ep,\et)$\index{S@$\si_d(p,m,\ep,\et)$} with
$\si_d(p,m,\ep,\et)\,=\, \mu\bigl(D(p,m,\ep,\et)\bigr)$\index{D@$D(p,m,\ep,\et)$} and with
$D(p,m,\ep,\et)$ given in Section \ref{ssec:spectraldecompCasimir}
for the various choices of $p$, $m$, $\ep$ and $\et$.
\end{thm}

\subsection{The decomposition of the GNS-space $\cK$ as a $\su$-module} \label{ssec:decompUq(su11)}
The (generalized) eigenspaces of the Casimir operator correspond to invariant subspaces under the action of $\su$. In this way, the spectral decomposition of $\Om$ from Section \ref{ssec:spectraldecompCasimir} leads to the decomposition of the GNS-space $\cK$ into irreducible $\ast$-representations of $\su$. Let us first recall these representations of $\su$.

The $\ast$-representations of $\su$ require unbounded operators, and for this we use the theory as developed in \cite[Ch. 8]{Schm}. In particular this means that for such a representation $\pi$ in a Hilbert space $V$ there exists a common dense domain ${\mathcal D}\subset V$,
which is invariant for $\pi(X)$ for all $X\in\su$, such that the relations
of \eqref{eq:defrelUqsu11} remain valid when acting on $v\in {\mathcal D}$.
Moreover, we require $\langle \pi(X)v,w\rangle_V=\langle v, \pi(X^\ast)w\rangle_V$
for all $v,w\in{\mathcal D}$. It follows that each $\pi(X)$, $X\in\su$, is closable.

Admissible representations of $\su$ are $\ast$-representations in a Hilbert space $V$ acting by unbounded operators, such that $V$ decomposes into finite-dimensional eigenspaces for the action of $\KA$, and such that the eigenvalues of $\KA$ are of the form $q^k$, $k\in \hf\Z$. Then the following irreducible admissible representations exhaust the list, see e.g. \cite{BurbK}, \cite{MasuMNNSU}, \cite{VaksK}. In each of these cases the common invariant dense domain is the subspace of finite linear combinations of the basis vectors $e_n$.

Note that each of these admissible irreducible representations
is completely determined by the eigenvalue of the
Casimir operator $\OmA$ on $V$ and the spectrum of $\KA$.

\noindent
\textbf{Positive discrete series.}\index{U@$\su$-representations!positive discrete series}
\quad The representation space is $\ell^2(\NN)$ with orthonormal basis $\{ e_n\}_{n\in\NN}$.
Let $k\in\hf\N$, define the action of the generators by
\begin{equation}\label{eq:posdiscrserrep}
\begin{split}
\KA\cdot e_n&=q^{k+n}\, e_n,\qquad
\KA^{-1}\cdot e_n=q^{-k-n}\, e_n,\\
(q^{-1}-q)\,\,\EA\cdot e_n&=  q^{-\hf-k-n}
\sqrt{(1-q^{2n+2})(1-q^{4k+2n})}\,e_{n+1},\\
(q^{-1}-q)\,\,\FA\cdot e_n&=
q^{\hf-k-n} \sqrt{(1-q^{2n})(1-q^{4k+2n-2})}\,e_{n-1},
\end{split}
\end{equation}
with the convention $e_{-1}=0$. This representation
is denoted by $D_k^+$\index{D@$D_k^+$} and $D_k^+(\OmA)=-\mu(q^{1-2k})$.

\noindent
\textbf{Negative discrete series.}\index{U@$\su$-representations!negative discrete series}
\quad The representation space is
$\ell^2(\NN)$ with orthonormal basis $\{ e_n\}_{n\in\NN}$.
Let $k\in\hf\N$, and define the action of the generators by
\begin{equation}\label{eq:negdiscrserrep}
\begin{split}
\KA\cdot e_n&=q^{-k-n}e_n,\qquad
\KA^{-1}\cdot e_n= q^{k+n}\, e_n,\\
(q^{-1}-q)\,\,\EA\cdot e_n&= q^{\hf-k-n}
\sqrt{(1-q^{2n})(1-q^{4k+2n-2})}\,e_{n-1}, \\
(q^{-1}-q)\,\,\FA\cdot e_n&=
q^{-\hf-k-n} \sqrt{(1-q^{2n+2})(1-q^{4k+2n})}\,e_{n+1},
\end{split}
\end{equation}
with the convention $e_{-1}=0$. This representation is denoted by
$D_k^-$\index{D@$D_k^-$} and $D_k^-(\OmA)=-\mu(q^{1-2k})$.

\noindent
\textbf{Principal series.}\index{U@$\su$-representations!principal series}
\quad The representation space is
${\ell^2}(\Z)$ with orthonormal basis $\{ e_n\}_{n\in\Z}$. Let
$0\leq b\leq-\frac{\pi}{2\ln q}$ and $\ep\in \{0,\hf\}$
and assume  $(b,\ep)\not=(0,\hf)$. The action
of the generators is defined by
\begin{equation}\label{eq:princunitaryserrep}
\begin{split}
\KA\cdot e_n&= q^{n+\ep}\,e_n,\qquad
\KA^{-1}\cdot e_n= q^{-n-\ep}\, e_n,\\
(q^{-1}-q)\,\,\EA\cdot e_n&=
q^{-\hf-n-\ep} \sqrt{(1-q^{2n+1+2\ep+2ib})(1-q^{2n+1+2\ep-2ib})}
\,e_{n+1}, \\
(q^{-1}-q)\,\,\FA\cdot e_n&=
q^{\hf-n-\ep} \sqrt{(1-q^{2n-1+2\ep+2ib})(1-q^{2n-1+2\ep-2ib})} \,e_{n-1}.
\end{split}
\end{equation}
We denote the representation by $\pi_{b,\ep}$.\index{P@$\pi_{b,\ep}$}
In case $(b,\ep)=(0,\hf)$ this still defines an
admissible unitary representation.
It splits as the direct sum $\pi_{-\frac{\pi}{2\ln q},\hf}\cong D_{\hf}^+\oplus
D_{\hf}^-$ of a positive and negative
discrete series representation by restricting to the invariant subspaces
$\text{span}\{ e_n\mid n\geq 0\}$ and
to $\text{span}\{ e_n\mid n <0\}$. We keep this convention for
$\pi_{-\frac{\pi}{2\ln q},\hf}$. Note that
$\pi_{b,\ep}(\OmA)=\mu(q^{2ib}) = \cos(-2b\ln q)$.

\noindent
\textbf{Strange series.}\index{U@$\su$-representations!strange series}
\quad The representation space is
${\ell^2}(\Z)$ with orthonormal basis $\{ e_n\}_{n\in\Z}$. Let $\ep\in
\{0,\hf\}$, and $a>0$. The action of the generators is defined by
\begin{equation}\label{eq:strangeserrep}
\begin{split}
\KA\cdot e_n&= q^{n+\ep}\, e_n,\qquad
\KA^{-1}\cdot e_n= q^{-n-\ep}\, e_n,\\
(q^{-1}-q)\,\,\EA\cdot e_n&= q^{-n-\ep-\hf} \sqrt{(1+q^{2n+2\ep+1+2a})
(1+q^{2n+2\ep-2a+1})}\,e_{n+1},\\
(q^{-1}-q)\,\,\FA\cdot e_n&=
q^{-n-\ep+\hf} \sqrt{(1+q^{2n+2\ep-1+2a}) (1+q^{2n+2\ep-2a-1})}\,e_{n-1}.
\end{split}
\end{equation}
We denote this representation by  $\pi_{a,\ep}^{S}$.\index{P@$\pi_{a,\ep}^{S}$}
Note that $\pi_{a,\ep}^S(\OmA)=\mu(q^{2a})$.

\noindent
\textbf{Complementary series.}\index{U@$\su$-representations!complementary series}
\quad This series of representations
acts in ${\ell^2}(\Z)$. The actions of the generators follow
from the action \eqref{eq:princunitaryserrep} by putting $\ep=0$ and
formally replacing $-\hf+ib$ by $\la$ and taking
$-\hf<\la<0$. This series of representations does not play a role in this paper.

\medskip

We define $\tau(\KA)=\KA^{-1}$, $\tau(\KA^{-1})=\KA$, $\tau(\EA)=-\FA$, $\tau(\FA)=-\EA$. From \eqref{eq:defrelUqsu11} we check that $\tau$ extends to an involutive algebra homomorphism $\tau\colon\su\to\su$. From \eqref{eq:Casimir} it is clear that $\tau(\OmA)=\OmA$. Composing an irreducible admissible representation with the involutive algebra automorphism $\tau\colon\su\to\su$ gives an admissible irreducible representation of $\su$. This easily gives
\begin{equation}\label{eq:tauopirreps}
D^+_k\circ\tau \cong D^-_k, \quad \pi_{b,\ep}\circ\tau \cong \pi_{b,\ep},
\quad \pi^S_{a,\ep}\circ\tau \cong \pi^S_{a,\ep}.
\end{equation}
Denoting the orthonormal bases in the representations on the left hand side of
\eqref{eq:tauopirreps} by $\{e_n^\tau\}$ we can describe the unitary
intertwiners as $e_n^\tau\mapsto (-1)^n e_n$ in the first case and as
$e_n^\tau\mapsto (-1)^n e_{-n-2\ep}$ for the last two cases.

Recall that the modular conjugation $J:\cK\to\cK$ satisfies that $E_0^\dag J=-E_0J$ and $JK_0=K_0^{-1}J$, and consequently $J \Om_0 = \Om_0 J$. This implies that $J$ implements the involutive algebra automorphism $\tau\colon\su\to\su$.

The spectral decomposition of the Casimir operator $\Om$ from Section \ref{ssec:spectraldecompCasimir} gives a decomposition of $\cK$ into invariant subspaces for the action of $\Om$. Let $p \in q^\Z$ and $\ep,\et \in \{-,+\}$. It follows from
\eqref{eq:E0E0dagKoncK0} that the space $\cK_0(p,\ep,\et)=\bigoplus_{m\in\Z} \cK_0(p,m,\ep,\et)$ is invariant for the action of $\su$. We denote by $\pi_\cK(p,\ep,\et)$\index{P@$\pi_\cK(p,\ep,\et)$}
\index{U@$\su$-representations!P@$\pi_\cK(p,\ep,\et)$}
the representation of $\su$ on $\cK(p,\ep,\et)$.\index{K@$\cK(p,\ep,\et)$}
In the following we decompose $\pi_\cK(p,\ep,\et)$ in terms of irreducible admissible $\ast$-representations of $\su$ using the spectral decomposition of  $\Om \colon \cK(p,m,\ep,\et) \to \cK(p,m,\ep,\et)$ from Section \ref{ssec:spectraldecompCasimir}. As before, we have to distinguish four cases depending on the signs of $\ep$ and $\et$. It turns out that the representation label $\ep$ for the principal and strange series representations occurring in the decomposition of $\pi_\cK(p,\ep,\et)$, depends on the parameter $p$. For this reason we define $\epsilon : q^\Z \to \{0,\hf\}$ by
\begin{equation} \label{eq:epsilon(p)}\index{E@$\epsilon(p)$}
\epsilon(p) = \hf \chi(p) \mod 1.
\end{equation}

\subsubsection{The case $\ep=+$ and $\et=-$}
In this case the spectral decomposition of the Casimir operator acting on $\cK(p,m,+,-)$ is determined by \eqref{eq:defUpsilonpm+-}. From the explicit action of $E_0^\dag$ \eqref{eq:dualE0forsu}, \eqref{eq:defUpsilonpm+-} and Lemma \ref{lem:contiguousASC} we obtain
\begin{equation}\label{eq:actionE0+inUpsilonpm+-}
\begin{split}
(q^{-1}-q) \,\,\Up^{+,-}_{p,m-1} \, (E^{+-}_{p,m-1})^*\, (\Up^{+,-}_{p,m})^\ast
&= \, q^{\hf-m}p^{-\hf}
M\big(\sqrt{1+2xq^{2m-1}p+q^{4m-2}p^2}\big)\\
&\, \colon L^2(I(p,m,+,-)) \to L^2(I(p,m-1,+,-)).
\end{split}
\end{equation}
Note that $L^2(I(p,m,+,-)) = L^2(I(p,m-1,+,-))$, unless $q^{1-2m}/p>1$. In this case $I(p,m-1,+,-)=I(p,m,+,-)\backslash\{\mu(-q^{1-2m}/p)\}$, and the multiplication
operator is zero for the point $\mu(-q^{1-2m}/p)$. So the multiplication operator in \eqref{eq:actionE0+inUpsilonpm+-} is well-defined.

{}From \eqref{eq:E0E0dagKoncK0} we have $\Up^{+,-}_{p,m} \, K\,
(\Up^{+,-}_{p,m})^\ast = q^m p^{\hf}\, \Id$, so
\eqref{eq:actionE0+inUpsilonpm+-} and \eqref{eq:Casimir} give
\begin{equation}\label{eq:actionE0inUpsilonpm+-}
\begin{split}
(q^{-1}-q)\,\,\Up^{+,-}_{p,m+1} \,E^{+-}_{p,m}\, (\Up^{+,-}_{p,m})^\ast
&= \, q^{-\hf-m} p^{-\hf} M\big(\sqrt{1+2xq^{2m+1}p+q^{4m+2}p^{2}}\big)\\
&\, \colon L^2(I(p,m,+,-)) \to L^2(I(p,m+1,+,-)).
\end{split}
\end{equation}
This can also be derived directly from a similar identity for the
Al-Salam--Chihara polynomials.

$\cK(p,+,-)$ is not an admissible representation of $\su$, since the
$K$-eigenspaces are not finite dimensional. However,
since the actions of $E$ and $E^\ast$ in \eqref{eq:actionE0+inUpsilonpm+-},
\eqref{eq:actionE0inUpsilonpm+-} match the actions given in the list of irreducible $\ast$-representations for $\su$, we can still determine the decomposition explicitly. The
possible eigenvalues of the Casimir and the eigenvalues of $K$ then
determine the decomposition. In Theorem \ref{thm:decompKp+-asUqmod} we deal with the positive discrete series representations, since $I(p,m-1,+,-) \subset I(p,m,+,-)$ for $m$ large enough implying that $E$ acts as the creation operator. The direct integral and direct sums
of representations of $\su$ by unbounded operators in Theorem \ref{thm:decompKp+-asUqmod} uses the construction of \cite[Ch. 8]{Schm}.

\begin{thm} \label{thm:decompKp+-asUqmod}
The decomposition of $\pi_\cK(p,+,-)$ into irreducible admissible $\ast$-re\-pre\-sen\-ta\-tions is given by
\[
\pi_\cK(p,+,-) \cong \int_0^{-\pi/2\ln q} \pi_{b,\epsilon(p)} \, db
\oplus \bigoplus_{\substack{l\in\Z\\
2l+\chi(p)>1}} D^+_{l+\hf\chi(p)} \oplus
\bigoplus_{\substack{l\in\NN\\
2l-\chi(p)<-1}} \pi^S_{q^{1+2l}/p,\epsilon(p)}.
\]
\end{thm}

\subsubsection{The case $\ep=-$ and $\et=+$}
In Section \ref{sssec:caseep-andet+} we obtained the spectral decomposition of $\Om$ in this case from the case $\ep=+$ and $\et=-$ using the modular conjugation $J$. For the actions of $E$ and $E^\ast$ we can do the same. Using $JE_0=-E_0^\dag J$ we obtain from \eqref{eq:actionE0+inUpsilonpm+-} and
\eqref{eq:actionE0inUpsilonpm+-}
\begin{equation}\label{eq:actionE0E0+inUpsilonpm-+}
\begin{split}
(q^{-1}-q)\,\,\Up^{-,+}_{p,m+1} \, E^{-+}_{p,m}\, (\Up^{-,+}_{p,m})^\ast
&= \, q^{-\hf-m} p^{-\hf}
M(\sqrt{1+2xq^{2m+1}p+q^{4m+2}p^{2}})\\
&\, \colon L^2(I(p,m,-,+)) \to L^2(I(p,m+1,-,+)), \\
(q^{-1}-q)\,\,\Up^{-,+}_{p,m-1} \,  (E^{-+}_{p,m-1})^*\, (\Up^{-,+}_{p,m})^\ast
&= \, q^{\hf-m}p^{-\hf}
M(\sqrt{1+2xq^{2m-1}p+q^{4m-2}p^2})\\
&\, \colon L^2(I(p,m,-,+)) \to L^2(I(p,m-1,-,+)).
\end{split}
\end{equation}
This can also be proved in the same way as \eqref{eq:actionE0+inUpsilonpm+-} and \eqref{eq:actionE0inUpsilonpm+-}.

{}From $\epsilon(p^{-1})=\epsilon(p)$ and \eqref{eq:tauopirreps} we obtain from Theorem
\ref{thm:decompKp+-asUqmod} the following decomposition of $\cK(p,-,+)$ as $\su$-module.

\begin{thm} \label{thm:decompKp-+asUqmod}
Let $\epsilon(p) =\hf\chi(p)\mod 1$. The decomposition of $\pi_\cK(p,-,+)$ into irreducible admissible
$\ast$-representations is given by
\[
\pi_\cK(p,-,+) \cong \int_0^{-\pi/2\ln q} \pi_{b,\epsilon(p)} \,
db \oplus \bigoplus_{\substack{l\in\Z\\
2l>1+\chi(p)}} D^-_{l-\hf\chi(p)} \oplus
\bigoplus_{\substack{l\in\NN \\
2l+\chi(p)<-1}} \pi^S_{pq^{1+2l},\epsilon(p)}.
\]
\end{thm}

\subsubsection{The case $\ep=-$ and $\et=-$}
Similar as in the case $\ep=+$ and $\et=-$ we find
\begin{equation}\label{eq:actionE0E0+inUpsilonpm--}
\begin{split}
(q^{-1}-q)\,\,\Up^{-,-}_{p,m+1} \,  E^{--}_{p,m}\, (\Up^{-,-}_{p,m})^*
= \, q^{-\hf-m} p^{-\hf}& M(\sqrt{1+2xq^{2m+1}p+q^{4m+2}p^{2}})\\
&\, \colon L^2(I(p,m,-,-)) \to L^2(I(p,m+1,-,-)), \\
(q^{-1}-q)\,\,\Up^{-,-}_{p,m-1} \,  (E^{--}_{p,m-1})^*\,
(\Up^{-,-}_{p,m})^* = \, q^{\hf-m}p^{-\hf} &M(\sqrt{1+2xq^{2m-1}p+q^{4m-2}p^2})\\
&\, \colon L^2(I(p,m,-,-)) \to L^2(I(p,m-1,-,-)).
\end{split}
\end{equation}
In the first equation we assume $m+\chi(p)\geq 0$, and in
the second we require $m+\chi(p)>0$.

It is now a matter of bookkeeping to keep track of the discrete
spectrum of $\Om$ in $\cK(p,-,-)$ in order to find the
discrete summands in the decomposition of $\cK(p,-,-)$ as
$\su$-module. Note that for $pq>1$ there is always discrete
spectrum for $m$ large, so that $E$ acts as the creation
operator and hence we have positive discrete series
representations in the decomposition. Similarly, $q>p$ leads
to the occurrence of negative discrete series
representations in the decomposition.

\begin{thm} \label{thm:decompKp--asUqmod}
The decomposition of $\pi_\cK(p,-,-)$
into irreducible admissible representations is given by
\[
\pi_\cK(p,-,-) \cong \int_0^{-\pi/2\ln q} \pi_{b,\epsilon(p)} \, db
\oplus \bigoplus_{\substack{l\in\NN \\
2l+\chi(p)<-1}} D^+_{-\hf\chi(p)-l}
\oplus \bigoplus_{\substack{ l\in\NN\\
2l-\chi(p)<-1}} D^-_{\hf\chi(p)-l}.
\]
\end{thm}

Note that at least one of the direct sums in the
decomposition is empty.

\subsubsection{The case $\ep=+$ and $\et=+$.} In this case the spectral decomposition of the Casimir operator restricted to $\cK(p,m,+,+)$ is described in Section \ref{sssec:caseep+andet+}.
From Lemma \ref{lem:contiguouslittleqJac} we obtain
\begin{equation}\label{eq:actionE0E0+inUpsilonpm++mleq0}
\begin{split}
(q^{-1}-q)\,\,\Up^{+,+}_{p,m+1} \, E^{++}_{p,m}\, (\Up^{+,+}_{p,m})^\ast
= \, q^{-\hf-m}& p^{-\hf} M(\sqrt{1+2xq^{2m+1}p+q^{4m+2}p^{2}})\\
&\, \colon L^2(I(p,m,+,+)) \to L^2(I(p,m+1,+,+)), \\
(q^{-1}-q)\,\,\Up^{+,+}_{p,m-1} \,  (E^{++}_{p,m-1})^*\, (\Up^{+,+}_{p,m})^\ast
= \, q^{\hf-m}
&p^{-\hf} M(\sqrt{1+2xq^{2m-1}p+q^{4m-2}p^2})\\
&\, \colon L^2(I(p,m,+,+)) \to L^2(I(p,m-1,+,+)).
\end{split}
\end{equation}

We have to be a careful in establishing the equality in \eqref{eq:actionE0E0+inUpsilonpm++mleq0} because of the unboundedness of the operators involved. From the way we defined $\Up_{p,m}^{+,+}$ in Section \ref{sssec:caseep+andet+} we conclude that the operators on the left hand side of \eqref{eq:actionE0E0+inUpsilonpm++mleq0} are restrictions of the ones on the right hand side. Let us denote the operator on the left hand side of the first equality in \eqref{eq:actionE0E0+inUpsilonpm++mleq0} by $S$ and the operator on the right hand side of this equality by $T$. So $S \subseteq T$. Then, by \eqref{eq:defOmegapmepet} and the result from Section \ref{sssec:caseep+andet+},
\[
\begin{split}
S^* S &=  2\,\Up^{+,+}_{p,m}\,\Om^{++}_{p,m}\,(\Up^{+,+}_{p,m})^*
+ (q^{2m+1}p + q^{-2m-1}p^{-1})\Id \\
&= 2\,M(x) + (q^{2m+1}p + q^{-2m-1}p^{-1})\Id = T^* T\, ,
\end{split}
\]
implying that $|S| = |T|$, and as a consequence, $D(S) = D(T)$.

It is now a matter of bookkeeping to keep track of
the discrete spectrum of $\Om$ in $\cK(p,+,+)$ in order to find the
discrete summands in the decomposition of
$\cK(p,+,+)$ as $\su$-module. Note that for
$pq>1$ there is always a discrete
spectrum for $m$ large, so that $E_0$ acts as the
creation operator and hence we have positive discrete series
representations in the decomposition.
Similarly, $q>p$ leads to the occurrence of negative discrete series
representations in the decomposition.
These two cases correspond to the (possibly empty) finite sequence of discrete mass
points in the spectral measure of the Casimir operator
\eqref{eq:defUpsilonpm++mleq0}, \eqref{eq:defUpsilonpm++mgeq0}.
The infinite sequence of discrete mass points that
is always present in the spectral decomposition of the Casimir
operator on $\cK(p,m,+,+)$ for all $m\in\Z$
corresponds to strange series representations.

\begin{thm} \label{thm:decompKp++asUqmod}
The decomposition of $\pi_\cK(p,+,+)$ into irreducible admissible representations is given by
\[
\begin{split}
\pi_\cK(p,+,+) \cong &\int_0^{-\pi/2\ln q} \pi_{b,\epsilon(p)} \, db \,
\oplus \bigoplus_{\stackrel{\scriptstyle{l\in\Z}}
{\scriptstyle{2l+\chi(p)<-1}}} \pi^S_{pq^{1+2l},\epsilon(p)} \\
&\, \oplus \bigoplus_{\stackrel{\scriptstyle{l\in\NN}} {\scriptstyle{2l+\chi(p)<-1}}}
D^-_{-\hf\chi(p)-l} \, \oplus
\bigoplus_{\stackrel{\scriptstyle{l\in\NN}} {\scriptstyle{2l-\chi(p)<-1}}}
D^+_{\hf\chi(p)-l}.
\end{split}
\]
\end{thm}

Note that at least one of the finite direct sums in the decomposition is
empty.

\section{Generators of the dual von Neumann algebra $\hat M$}\label{sec:generatorshatM}
By Theorem \ref{thm:Casimirwelldefinedandchar}, $E$ and $K$ strongly commute with the Casimir $\Om$.
Since there are elements in $\hat{M}$ that anti-commute with $\Om$, see Proposition \ref{prop:decompMintoM+andM-andgradedcommutation},
$\hat{M}$ cannot be generated by $E$ and $K$ alone. So we need to find extra operators that, together with $E$ and $K$, generate  the dual von Neumann algebra $\hat{M}$. It is the purpose of this section to describe a generating set for $\hat{M}$, i.e., to prove Theorem \ref{thm:generatorsforhatM}. We do so by establishing a generator $Q(p_1,p_2,n)$ of $\hat{M}$, see \eqref{eq:defQppn} and
Proposition \ref{prop:QppngeneratedualM}, as the composition of a partial isometry and an operator expressed in terms of the Casimir operator. The partial isometries occurring in this way give us the required additional generators for the dual von Neumann algebra $\hat{M}$.

Throughout this section we fix $p_1,p_2 \in I_q$, $p \in q^\Z$, $m,n \in \Z$, $\ep,\et \in \{-,+\}$. Furthermore, we set $m'=m+n$, $\ep'=\ep\,\sgn(p_1)$ and $\et'=\et\,\sgn(p_2)$, and assume $q^{2m}p=q^{-n}|p_2/p_1|$, unless explicitly stated otherwise. In this case the operator $Q(p_1,p_2,n) \colon \cK(p,m,\ep,\et) \to \cK(p,m',\ep',\et')$ is non-zero by Lemma \ref{lem:explicitactionQppn}.

\subsection{A polar-type decomposition for $Q(p_1,p_2,n)$}
In this subsection we establish a polar-type decomposition for the element $Q(p_1,p_2,n)$.
Since operators of this form span $\hat M$ by Proposition \ref{prop:QppngeneratedualM}, we
can obtain the generators of $\hat{M}$.
By Proposition \ref{prop:decompMintoM+andM-andgradedcommutation}
the operator $Q(p_1,p_2,n) \in \hat M$ commutes or anti-commutes with the Casimir operator $\Om$, hence $Q(p_1,p_2,n)$ sends a (generalized) eigenvectors of $\Om$ to another (generalized) eigenvector. In order to avoid working with generalized eigenvectors, we consider an operator $T$, acting on $L^2$-functions on the spectrum of $\Om$, that is unitarily equivalent to $Q(p_1,p_2,n)$. We determine the explicit action of $T$ by investigating how $T$ affects the asymptotic behaviour of certain functions. Having explicitly the action of $T$, we can compute explicitly how $T^*T$ acts, and this leads to the polar decomposition of $T$. This in turn leads to a polar-type decomposition for $Q(p_1,p_2,n)$.

In order the find explicitly the action of the operator $T$ as described above (and defined later on by \eqref{eq:defTp1p2n}), we need a result on the asymptotic behaviour of certain functions.
In order to formulate the result we define the following function:
\[\index{S@$S(t;p_1,p_2,n)$}
\begin{split}
S(t&;p_1,p_2,n) =\\
 & \, \big(\sgn(p_2)\big)^n \, |p_1p_2|\,
c_q^2\, q^n \,
\sqrt{(-\ka(p_1), -\ka(p_2);q^2)_\infty} \\
& \times
\sum_{z \in \sgn(p_1)q^\Z} \big(\sgn(p_1p_2)  t\big)^{\chi(z)}
\frac{1 }{|z|}\, \nu(\frac{p_1}{z}) \, \nu(\frac{p_2 q^n}{z})
\rphis{1}{1}{-q^2/\ka(p_1)}{0}{q^2,q^2\ka(z)} \\
& \qquad \qquad \times   \rphis{1}{1}{-q^2/\ka(p_2)}{0}{q^2,q^2\ka\big(\sgn(p_1p_2) q^{-n}z\big)},
\end{split}
\]
where the sum is absolutely convergent. Clearly, $S(\,\cdot\,;p_1,p_2,n)$ is analytic on $\C\setminus\{0\}$. This function is studied in Appendix \ref{appB:functionS} in some more detail. Two properties of $S$ that we need here are given in the following lemma.

\begin{lemma}\label{lem:propertiesofS}
The analytic function $S(\,\cdot\,;p_1,p_2,n)\colon \C\setminus\{0\} \to \C$
satisfies the following properties:
\begin{enumerate}[(i)]
\item $S(t;p_1,p_2,n) = (-q)^n\, \sgn(p_1)^{\chi(p_1)}
\, \sgn(p_2)^{\chi(p_2)+n}\, \sgn(p_1p_2)\, S(\sgn(p_1p_2)t^{-1};p_1,p_2,-n)$
\item $S(t;p_1,p_2,n)$ is a multiple of a $_2\varphi_1$-function:
\[
\begin{split}
S(t;p_1,p_2,n) =&\, p_2^n q^{\hf n (n-1)} |p_1p_2|\, \nu(p_1)\nu(p_2) c_q^2 \sqrt{ (-\kappa(p_1),-\kappa(p_2);q^2)_\infty }\\
& \times\frac{ (q^2, -q^2/\kappa(p_2), -tq^{3-n}/p_1p_2, -p_1p_2q^{n-1}/t, p_1q^{1-n}/p_2t;q^2)_\infty }{ (-p_1|p_2|q^{-n-1}/t, -tq^{n+3}/p_1|p_2|, |p_1|q^{1+n}/|p_2|t ;q^2)_\infty }\\
& \times (\sgn(p_1p_2)q^{2+2n};q^2)_\infty \rphis{2}{1}{p_2q^{1+n}/p_1t, p_2 t q^{1+n}/p_1}{ \sgn(p_1 p_2) q^{2+2n} }{q^2,-q^2/\kappa(p_2)}.
\end{split}
\]
\end{enumerate}
\end{lemma}

See Appendix \ref{appB:functionS}, and in particular Proposition \ref{prop:S=2phi1} and
Lemma \ref{lem:symmetryS}, for a proof of Lemma \ref{lem:propertiesofS}.

It turns out to be useful to split the function $S$ in a part that is symmetric in $t$ and $t^{-1}$, and a part that is not.
\begin{lemma} \label{lem:S=BN}
For $x=\mu(t)$, define
\[
\begin{split}
B(t;p_1,p_2,n) & =
\begin{cases}
t^n\dfrac{ (-|p_1|q^{1-n}/|p_2|t;q^2)_\infty }{ (-p_1q^{1+n}/p_2t;q^2)_\infty }, & \chi(p_1p_2)+n \text{ even},\\ \\
t^n\dfrac{ (-|p_1|q^{1-n}/|p_2|t;q^2)_\infty }{ (-p_1q^{1+n}/p_2t;q^2)_\infty } \dfrac{ 1-\sgn(p_2)t^{-1} }{1-t^{-1}}, & \chi(p_1p_2)+n \text{ odd},
\end{cases} \\
h(x) & =
\begin{cases}
\sgn(p_2)^{\hf(\chi(p_1p_2)-n+2)}\dfrac{ (qt,q/t;q^2)_\infty }{ (\sgn(p_2)qt,\sgn(p_2)q/t;q^2)_\infty }, &  \chi(p_1p_2)+n \text{ even},\\ \\
\sgn(p_2)^{\hf(\chi(p_1p_2)-n+3)}\dfrac{ (t,t^{-1};q^2)_\infty }{ (\sgn(p_2)t,\sgn(p_2)/t;q^2)_\infty }, &  \chi(p_1p_2)+n \text{ odd},
\end{cases}
\end{split}
\]
and
\[\index{N@$N(x;p_1,p_2,n)$}
\begin{split}
N(x;p_1,p_2,n) =&\,h(x)\ q^{2n} p_2^n q^{\hf n (n-1)} |p_1p_2|^{1-n}\, \nu(p_1)\nu(p_2) c_q^2 \\
& \times (q^2, -q^2/\kappa(p_2),\sgn(p_1p_2)q^{2+2n};q^2)_\infty \sqrt{ (-\kappa(p_1),-\kappa(p_2);q^2)_\infty }\\
& \times \rphis{2}{1}{-|p_2|q^{1+n}/|p_1|t, -|p_2|t q^{1+n}/|p_1|}{ \sgn(p_1 p_2) q^{2+2n} }{q^2,-q^2/\kappa(p_2)},\\
\end{split}
\]
then $S(-\sgn(p_1p_2)t;p_1,p_2,n) = B(t;p_1,p_2,n) N(x;p_1,p_2,n)$.
\end{lemma}
\begin{proof}
According to $\chi(p_1p_2)+n$ being even or odd, we set
\[
\chi(p_1p_2)+n =
\begin{cases}
2-2k, \\
3-2l
\end{cases}
\]
for $k,l \in \Z$. Using the $\te$-product identity \eqref{eq:thetaprodid} we find
\[
\begin{split}
&\frac{ (tq^{3-n}/|p_1p_2|, |p_1p_2|q^{n-1}/t;q^2)_\infty }{ (tq^{3+n}/|p_1|p_2, |p_1|p_2q^{-n-1}/t;q^2)_\infty}  =
\begin{cases}
\sgn(p_2)^{k+n} t^n q^{n(n-1)} q^{2nk} \dfrac{ (qt, q/t;q^2)_\infty }{ (\sgn(p_2)qt, \sgn(p_2)q/t;q^2)_\infty },\\ \\
\sgn(p_2)^{l+n} t^n q^{n(n-1)} q^{2nl} \dfrac{ (t, q^2/t;q^2)_\infty }{ (\sgn(p_2)t, \sgn(p_2)q^2/t;q^2)_\infty },
\end{cases}
\end{split}
\]
then the result follows from Lemma \ref{lem:propertiesofS}(ii). The expression for $N$ is manifestly symmetric in $t$ and $t^{-1}$, so $N$ is indeed a function of $x=\mu(t)$.
\end{proof}
In the following lemmas, and in the rest of this section, we use the notation $f(z)\sim g(z)$ as $z\to 0$, for $\lim_{z\to 0} \bigl( f(z)-g(z)\bigr) = 0$. We are now ready to formulate the asymptotic behaviour we need later on.

\begin{lemma}\label{lem:fundlemasymptoticsQ}
Let $f\colon J(p,m,\ep,\et)\to \C$ be bounded, and consider the function
\begin{equation*}
\begin{split}
g(w) = &(-1)^{m'} (\et')^{\chi(p_1p_2)+m} \,  q^{m'}p_1^2\,
\frac{(\ep'\et')^{\chi(w)}}{|w|} \\& \times
 \sum_{z \in J(p,m,\ep,\et)}\, \frac{f(z)}{|z|}\,\,
a_{p_1}(z,w)\,a_{p_2}(\ep\,\et\,q^m p\, z ,
\ep' \et' q^{m'} p\, w),
\end{split}
\end{equation*}
for $w\in J(p,m',\ep',\et')$.
\begin{enumerate}
\item If $f(z) \sim A  \, t^{-\chi(z)}$ as $z\to 0$
for some $A \in \C$ and $t\in \C$, $|t|>1$, then
\[
g(w)  \sim  A\, t^{-\chi(w)}\, \et^n \, s(\ep,\ep')\,
s(\et,\et')\ S(\ep\et/ t;p_1,p_2,n), \qquad \text{as}\ w\to 0.
\]
\item If $f(z) \sim  \Re ( A e^{-i\psi\chi(z)})$ as $z\to 0$ for some $A \in \C$ and $|\psi| \in (0,\pi)$, then
$$
g(w) \sim \, \et^n \, s(\ep,\ep')\,
s(\et,\et')\ \Re\bigl( A e^{-i\psi\chi(w)} S(\ep\et e^{-i\psi};p_1,p_2,n)\bigr),
\qquad \text{as $w\to 0$.}
$$
\end{enumerate}
\end{lemma}

The proof of Lemma \ref{lem:fundlemasymptoticsQ} is given in Appendix \ref{app:fundlemasymptoticsQ}.

Using the unitary operators $\Up_{p,m}^{\ep,\et}:\cK(p,m,\ep,\et) \to L^2(I(p,m,\ep,\et))$ from Section \ref{ssec:spectraldecompCasimir} we define an action of the generators $Q(p_1,p_2,n)$ of $\hat M$ on the space $L^2(I(p,m,\ep,\et))$ by
\begin{equation} \label{eq:defTp1p2n}\index{T@$T(p_1,p_2,n)$}
T(p_1,p_2,n) = \Up_{p,m'}^{\ep',\et'}\, Q(p_1,p_2,n) (\Up_{p,m}^{\ep,\et})^*:L^2(I(p,m,\ep,\et)) \to L^2(I(p,m',\ep',\et')),
\end{equation}
where $m'=m+n$, $\ep'=\sgn(p_1)\ep$ and $\et'=\sgn(p_2)\et$. Recall that we assume $q^{2m} p = q^{-n}|p_2/p_1|$. If this condition is not satisfied, we see from Lemma \ref{lem:explicitactionQppn} that the operator $T(p_1,p_2,n)$ is trivially zero. Since $Q(p_1,p_2,n)\,  \Om \subseteq \sgn(p_1p_2)\, \Om \,Q(p_1,p_2,n)$,
we have
\[
T(p_1,p_2,n)\, M(x) \subseteq \sgn(p_1p_2)\, M(x)\, T(p_1,p_2,n)
\]
for any $x$ in the spectrum of $\Om$. For $g \in L^2(I(p,m,\ep,\et))$ this implies
\[
\big(T(p_1,p_2,n) g\big)(x) = C(x)\, g\big(\sgn(p_1p_2)x\big),
\]
for a certain bounded measurable function $C:I(p,m',\ep',\et')\to \C$. It follows immediately that $C(x)=0$ if $\sgn(p_1p_2)x \not\in I(p,m,\ep,\et)$, which can only happen in case $x \in \si_d(p,m',\ep',\et')$.

The set
\[
\{ g_z(\cdot;p,m,\ep,\et) \mid z \in J(p,m,\ep,\et)\}
\]
is an orthonormal basis for $L^2(I(p,m,\ep,\et))$. Recall that the functions $g_z(x;p,m,\ep,\et)$ are defined in terms of Al-Salam--Chihara polynomials or little $q$-Jacobi functions, see Section \ref{ssec:spectraldecompCasimir}. From the asymptotic behaviour of these special functions, see \eqref{eq:asymptoticnormalAlSCfcont}, \eqref{eq:asymptoticnormalAlSCfdisc}, \eqref{eq:asymplqJacfunktoinftycont} and \eqref{eq:asymplqJacfunktoinftydisc}, it follows that the functions $g_z(x;p,m,\ep,\et)$ satisfy
\begin{equation} \label{eq:asymptoticsgz}
g_z(\mu(\la);p,m,\ep,\et) \sim
\begin{cases}
A(\la) (-\ep\et \la)^{-\chi(z)}, & \la \in D(p,m,\ep,\et),\\
\Re\big(A(\la)(-\ep\et \la)^{-\chi(z)}\big), & \la \in \T,
\end{cases}
\end{equation}
as $z \to 0$, for a certain $A(\la)=A(\la;p,m,\ep,\et) \in \C$.\index{A@$A(\la;p,m,\ep,\et)$}
In general the functions $A$ are only defined on $\T_0=\T\setminus\{-1,1\}$.\index{T@$\T_0$}
The function $A(\la)$ has an explicit expression in terms of the $c$-functions for the corresponding special functions, for instance
\[
A(e^{i\psi};p,m,+,-) =(-1)^m e^{-i\psi(m+\chi(p)-1)}\sqrt{\frac{2}{\pi|\sin \psi|}}\, \frac{ c(e^{-i\psi};q/p,-q^{1-2m}/p\mid q^2) }{ |c(e^{-i\psi};q/p,-q^{1-2m}/p\mid q^2)|},
\]
for $0<|\psi|<\pi$, which follows from \eqref{eq:defUpsilonpm+-} and \eqref{eq:asymptoticnormalAlSCfcont}, and we have similar expressions in the other cases. For convenience we have written down the explicit formulas for $A$ in Appendix \ref{ssecB:formulasforA}. With the help of the explicit action of $Q(p_1,p_2,n)$ on the basis elements of $\cK(p,m,\ep,\et)$ given in Lemma \ref{lem:explicitactionQppn}, and with Lemma \ref{lem:fundlemasymptoticsQ}, we can now compute explicitly the function $C$. The following notation will be useful: for $n,m \in \Z$, $p \in q^\Z$, $\ep,\et,\si,\tau \in \{-,+\}$, we set
\begin{equation} \label{eq:defsetX}\index{X@$X_n^{\si,\tau}(p,m,\ep,\et)$}
X_n^{\si,\tau}(p,m,\ep,\et) = \T \cup \Big( D(p,m+n,\si\ep,\tau\et) \cap \si\tau D(p,m,\ep,\et)\Big).
\end{equation}

\begin{lemma}\label{lem:actionTC(x)}
Let $g \in L^2(I(p,m,\ep,\et))$, $X = X_n^{\sgn(p_1),\sgn(p_2)}(p,m,\ep,\et)$, then for almost all $x = \mu(\la) \in I(p,m',\ep',\et')$
\begin{equation} \label{eq:actionT=C(x)}
\big(T(p_1,p_2,n) g\big)(x) = C(x) g\big(\sgn(p_1p_2)x\big),
\end{equation}
where $C = C(\,\cdot\,;m,\ep,\et;p_1,p_2,n)$\index{C@$C(x;m,\ep,\et;p_1,p_2,n)$} is given by
\[
C(\mu(\la)) =
 \begin{cases}
     \stackrel{\displaystyle{\ep^{\hf(1-\sgn(p_1))} \et^{\hf(1-\sgn(p_2))+n}\, S(-\sgn(p_1p_2)\la;p_1,p_2,n)}}
              {\displaystyle{\times \ \frac{ A(\la;p,m',\ep',\et')}{A(\sgn(p_1p_2)\la;p,m,\ep,\et)}}}
   & \la \in X ,\\
   \,0, & \text{otherwise}.
 \end{cases}
\]
\end{lemma}

Note that the expression on the right hand side is not obviously symmetric with respect
to interchanging $\la$ and $\la^{-1}$, but it is since the function $C$ only depends
on $x=\mu(\la)$.

\begin{proof}
We assume $\la \in X$. We know that \eqref{eq:actionT=C(x)} is valid for some function $C$ and for all $g \in L^2(I(p,m,\ep,\et))$. We choose $g=g_z=g_z(\cdot;p,m,\ep,\et)$. Since the function $C$ is independent of $z$, we can determine the function $C$ by letting $z \rightarrow 0$.

Using $\Up_{p,m}^{\ep,\et} f_{-m,\ep\et q^m pz,z} = g_z(\cdot;p,m,\ep,\et)$, it follows immediately from Lemma \ref{lem:explicitactionQppn} and \eqref{eq:defTp1p2n} that
\[
\begin{split}
C(\cdot)&g_z(\sgn(p_1p_2)\,\cdot\,) = T(p_1,p_2,n)g_z\big(\sgn(p_1p_2)\,\cdot\,;p,m,\ep,\et\big) \\
=\,& (-1)^{m'} (\et')^{\chi(p_1p_2)+m} q^{n+m}p_1^2/|z| \\
 &\times \sum_{w \in J(p,m',\ep',\et')} \frac{(\ep'\et')^{\chi(w)}}{|w|} a_{p_1}(z,w) a_{p_2}(\ep\et q^m p z, \ep'\et' q^{m'}p w) g_w(\cdot;p,m',\ep',\et'),
\end{split}
\]
as an identity in $L^2(I(p,m',\ep',\et'))$. Since $g_z$ is a real-valued function, we see that the function $C$ is real-valued almost everywhere. From \eqref{eq:asymptoticsgz} it follows that $w \mapsto g_w(x;p,m',\ep',\et')$ is bounded for all $x \in I(p,m',\ep',\et')\setminus\{\pm 1\}$. This implies that the sum
\[
\sum_{w \in J(p,m',\ep',\et')} \frac{(\ep'\et')^{\chi(w)}}{|w|} a_{p_1}(z,w) a_{p_2}(\ep\et q^m p z, \ep'\et' q^{m'}p w) g_w(x;p,m',\ep',\et')
\]
converges for all $x \in I(p,m',\ep',\et')\setminus\{\pm 1\}$. Using symmetry relations \eqref{eq:symmetryforapxy} for the functions $a_p$, we have
\begin{multline*}
C(x)g_z\big(\sgn(p_1p_2)x\big) =
 (-1)^{m'} \sgn(p_1)^{\chi(p_1)} \sgn(p_2)^{\chi(p_2)} \et^{m'+\chi(p_1p_2)} (\ep\et)^{\chi(z)} q^{m'}p_1^2/|z| \\
 \times \sum_{w \in J(p,m',\ep',\et')} \frac{1}{|w|} a_{p_1}(w,z) a_{p_2}(\ep'\et' q^{m'}p w,\ep\et q^m p z) g_w(x;p,m',\ep',\et').
\end{multline*}
Let $z \rightarrow 0$ in this expression using Lemma \ref{lem:fundlemasymptoticsQ} and the asymptotic behaviour \eqref{eq:asymptoticsgz} of $g_z$, then for $\la \in \T_0$,
\[
\begin{split}
C(\mu(\la))\, \Re\Big(A\big(\sgn(p_1p_2)\la\big)\, &\big(-\sgn(p_1p_2)\ep \et \la\big)^{-\chi(z)}\Big) \sim \\
& (-1)^n q^n (\et')^n \sgn(p_1)^{\chi(p_1)} \sgn(p_2)^{\chi(p_2)}s(\ep,\ep') s(\et,\et') \\
& \times \Re\Big((A'(\la) (-\ep'\et'\la)^{-\chi(z)}S(- \la^{-1};p_1,p_2,-n)\Big),
\end{split}
\]
and for $\la \in X_d(p,m',\ep',\et';n)$,
\[
\begin{split}
C(\mu(\la))\, A\big(\sgn(p_1p_2)\la\big)\,& \big(-\sgn(p_1p_2)\ep \et \la\big)^{-\chi(z)} \sim \\
&(-1)^n q^n (\et')^n \sgn(p_1)^{\chi(p_1)} \sgn(p_2)^{\chi(p_2)}s(\ep,\ep') s(\et,\et') \\
& \times A'(\la) (-\ep'\et'\la)^{-\chi(z)} S(-\la^{-1};p_1,p_2,-n),
\end{split}
\]
where we use the shorthand notation $A'(\la)= A(\la;p,m',\ep',\et')$. Applying the first symmetry for the function $S(\cdot;p_1,p_2,n)$ from Lemma \ref{lem:propertiesofS}, using $s(\ep,\ep')=\ep^{\hf(1-\sgn(p_1))}$ and similarly for $s(\et,\et')$, and using the fact that $C$ is real-valued, the result follows.
\end{proof}

\begin{remark}
Lemma \ref{lem:actionTC(x)} immediately gives nontrivial summation formulas for special functions. We work this out in Section \ref{ssec:summationformula}.
\end{remark}

Next we consider the polar decomposition for $Q(p_1,p_2,n)$. We need the following lemma.

\begin{lemma}\label{lem:QadjointisQ}
For $p_1,p_2\in I_q$,
$n\in \Z$, we have
\[
Q(p_1,p_2,n)^\ast = (-q)^n \, \sgn(p_1)^{\chi(p_1)}
\, \sgn(p_2)^{\chi(p_2)}\, Q(p_1,p_2,-n).
\]
\end{lemma}

\begin{proof}
Using the matrix elements \eqref{eq:matrixelementsofQppnandhatJQppnhatJ} and their symmetries following from \eqref{eq:symmetryforapxy},  it is straightforward to check that the matrix elements
\[
\langle  f_{mpt}, Q(p_1,p_2,n)f_{lrs}\rangle \qquad \text{and} \qquad \langle Q(p_1,p_2,-n)\, f_{mpt}, f_{lrs}\rangle
\]
agree up to the factor $(-q)^n\sgn(p_1)^{\chi(p_1)} \, \sgn(p_2)^{\chi(p_2)}$ for all $m,p,t,l,r,s$.
\end{proof}

Alternatively, one can also use Corollary \ref{cor:lemQppnareallinhatM} and
$J\, f_{mpt} \, = \, f_{-m,t,p}$, see Section \ref{sec:quantumSU11description}, to prove Lemma \ref{lem:QadjointisQ} using \eqref{eq:actionunitaryantipode}.

From Lemma \ref{lem:QadjointisQ} it follows that
\begin{equation} \label{eq:TadjointisT}
T(p_1,p_2,n)^* = (-q)^n \sgn(p_1)^{\chi(p_1)} \sgn(p_2)^{\chi(p_2)} T(p_1,p_2,-n).
\end{equation}
Combining this with Lemma \ref{lem:actionTC(x)} we find for $g \in L^2(I(p,m,\ep,\et))$,
\[
\begin{split}
\big(T(p_1,&p_2,n)^* T(p_1,p_2,n) g \big)(\mu(\la))  \\
=\,& (-q)^n \sgn(p_1)^{\chi(p_1)+1} \sgn(p_2)^{\chi(p_2)+n+1} S(-\la;p_1,p_2,n) S(-\sgn(p_1p_2)\la;p_1,p_2,-n)\ g(\mu(\la)), \\
=\,&S(-\la;p_1,p_2,n)S(-\la^{-1};p_1,p_2,n) \ g\big(\mu(\la)\big),
\end{split}
\]
where $\la \in X=X_n^{\sgn(p_1),\sgn(p_2)}(p,m,\ep,\et)$. The last equality follows from a symmetry relation from Lemma \ref{lem:propertiesofS}.  Note that this implies $S(-\la;p_1,p_2,n)S(-\la^{-1};p_1,p_2,n) \geq 0$. Furthermore, for $\la \not\in X$ we have
\[
\big(T(p_1,p_2,n)^* T(p_1,p_2,n) g \big)(\mu(\la))=0.
\]
Now we define for $x =\mu(\la) \in I(p,m,\ep,\et)$,
\[
L\big(x;p_1,p_2,n\big) =
\begin{cases}
\sqrt{S(-\la;p_1,p_2,n)S(-\la^{-1};p_1,p_2,n)}, & \la \in X,\\
\, 0, & \text{otherwise},
\end{cases}
\]
and we define a partial isometry $V(p_1,p_2,n):L^2(I(p,m,\ep,\et)) \to L^2(I(p,m',\ep',\et'))$ by
\begin{equation*}\index{V@$V(p_1,p_2,n)$}
\big(V(p_1,p_2,n)\,g \big)(x) =
\begin{cases}
\dfrac{ C(x)g\big(\sgn(p_1p_2)x\big) }{L(\sgn(p_1 p_2)x)} , & x \in \mu(X),\\[5mm]
\, 0, & \text{otherwise},\end{cases}
\end{equation*}
where $C$ is given in Lemma \ref{lem:actionTC(x)}. We remark that for $\la \in X$ it follows from Lemma \ref{lem:propertiesofS} that $L$ is a multiple of the absolute value of a $_2\varphi_1$-function. Now from Lemma \ref{lem:actionTC(x)} we find the polar decomposition of $T(p_1,p_2,n):L^2(I(p,m,\ep,\et)) \rightarrow L^2(I(p,m',\ep',\et'))$:
\[
T(p_1,p_2,n) = V(p_1,p_2,n) |T(p_1,p_2,n)|,
\]
where
\[
\big(|T(p_1,p_2,n)|\ g\big)(x) =
\begin{cases}
L(x) g(x), & x\in \mu(X),\\
0, & \text{otherwise},
\end{cases}
\]
for $g \in L^2(I(p,m,\ep,\et))$, $x \in I(p,m,\ep,\et)$ and the set $X$ is given by \eqref{eq:defsetX}. Note that $|T(p_1,p_2,n)|=0$ on $L^2(I(p,m,\ep,\et))$ if $p\neq q^{-n-2m}|p_2/p_1|$. We can now describe explicitly the polar decomposition for $Q(p_1,p_2,n)$.

\begin{prop}\label{prop:polardecompQp1p2n}
The operators $U(p_1,p_2,n)$ and $|Q(p_1,p_2,n)|$ in the polar decomposition $Q(p_1,p_2,n)=U(p_1,p_2,n)|Q(p_1,p_2,n)|$ are given by
\[
|Q(p_1,p_2,n)| = L(\Om), \quad \text{and} \quad U(p_1,p_2,n) =\Up^* V(p_1,p_2,n) \Up.
\]
\end{prop}

We are going to define a partial isometry closely related to $V(p_1,p_2,n)$ which is more convenient for us. Let us first have a closer look at the function $\frac{C(x)}{L(\sgn(p_1p_2)x)}$ appearing in the definition of $V(p_1,p_2,n)$. Using Lemma \ref{lem:S=BN} we find (omitting dependence on certain parameters in the notation)
\[
L(x) = \sqrt{ B(\si\tau\la)B(\si\tau\la^{-1}) } |N(\si\tau x)|, \qquad x = \mu(\la),
\]
where $\si=\sgn(p_1)$ and $\tau=\sgn(p_2)$. Here we use that $N(x)$ is symmetric in $\la$ and $\la^{-1}$, hence real-valued, and consequently $B(\la) B(\la^{-1})$ is positive. This shows that $\frac{C(x)}{L(\sgn(p_1p_2)x)}$ can be written as
\[
\ep^{\frac12(1-\si)} \et^{\frac12(1-\tau)+n} \frac{ A'(\la)} {A(\si\tau\la)} \frac{ B(\la) }{ \sqrt{ B(\la) B(\la^{-1})}}\ \sgn\big( N(x) \big).
\]
This expression, in particular the factor $\sgn\big( N(x) \big)$, is not very convenient for us, therefore we are going to consider the partial isometry $\sgn\big( N(\,\cdot\,) \big) V(p_1,p_2,n)$.
Let us introduce the following functions:
\begin{equation} \label{eq:defEGnu}\index{E@$E(\la;p,m)$}\index{G@$G(\la;p,m,\ep,\et)$}\index{N@$\nu_n^\tau(\la;p)$}
\begin{split}
E(\la;p,m) & = \frac{(-q^{1-2m}/p\la;q^2)_\infty }{ \sqrt{(-q^{1-2m}/p\la, -q^{1-2m}\la/p;q^2)_\infty }}, \\
G(\la;p,m,\ep,\et) & = A(\la;p,m,\ep,\et) E(\la;p,m), \\
\nu_n^\tau(\la;p) & = \frac{ (\tau \la)^{\epsilon(p)} }{\la^{\epsilon(p)-n}},
\end{split}
\end{equation}
where $\epsilon(p)$ is defined by \eqref{eq:epsilon(p)}. In particular, for $\epsilon(p)=\hf$ we have $\nu_0^-(\la;p)=\mp i$ for $\la \in \T^{\pm}$. With these functions, we define for $n \in \Z$, $\si,\tau \in \{-,+\}$, a partial isometry
 $V_n^{\si,\tau}\colon L^2(I(p,m,\ep,\et)) \rightarrow L^2(I(p,m+n,\si\et,\tau\et))$ closely related to $V(p_1,p_2,n)$ by
\begin{equation} \label{eq:defVnsitau}\index{V@$V_n^{\si,\tau}$}
\big(V_n^{\si,\tau} g\big)(x) =
\begin{cases}
\ep^{\hf(1-\si)} \et^{\hf(1-\tau)+n} \nu_n^\tau(\la;p)\dfrac{ G(\la;p,m+n,\si\ep,\tau\et)}{G(\si\tau\la;p,m,\ep,\et)} g(\si\tau x), & \la \in X,\\
0, & \text{otherwise}.
\end{cases}
\end{equation}
where $g \in L^2(I(p,m,\ep,\et))$, and $X=X_n^{\si,\tau}(p,m,\ep,\et)$. Let us remark that
\[
\nu_n^{\sgn(p_2)}(\la;p) \frac{ E(\la;p,m+n) }{ E(\sgn(p_1p_2)\la;p,m) }  = \frac{ B(\la;p_1,p_2,n) }{ \sqrt{ B(\la;p_1,p_2,n) B(\la^{-1};p_1,p_2,n) }},
\]
for $p = q^{-n-2m}|p_2/p_1|$, so that $V_n^{\sgn(p_1),\sgn(p_2)} = \sgn\big(N(\,\cdot\,;p_1,p_2,n)\big)V(p_1,p_2,n)$. We also denote
\[
V_n^{\sgn(p_1),\sgn(p_2)} \colon \bigoplus_{\stackrel{\scriptstyle{\ep,\et\in\{\pm\}}}
{\scriptstyle{p\in q^\Z, m\in \Z}}}  L^2(I(p,m,\ep,\et)) \to
\bigoplus_{\stackrel{\scriptstyle{\ep,\et\in\{\pm\}}}
{\scriptstyle{p\in q^\Z, m\in \Z}}}  L^2(I(p,m,\ep,\et))
\]
by summing
$V_n^{\si,\tau}\colon L^2(I(p,m,\ep,\et)) \rightarrow L^2(I(p,m+n,\si\et,\tau\et))$.

We now arrive at the following polar-type decomposition for the operators $Q(p_1,p_2,n)$.

\begin{prop} \label{prop:polardecompQ}
Let $m,n \in \Z$, $p_1,p_2\in I_q$, $\ep,\et \in \{-,+\}$, and assume $p=q^{-n-2m}|p_2/p_1|$. For $\si,\tau \in \{-,+\}$ we define a partial isometry $U_n^{\si,\tau}= \Up^* V_n^{\si,\tau}\Up$, so that
\begin{equation} \label{eq:defUnsitau}
U_n^{\si,\tau}\vert_{\cK(p,m,\ep,\et)}\, =\, (\Up_{p,m+n}^{\si\ep,\tau\et})^* V_n^{\si,\tau} (\Up_{p,m}^{\ep,\et})
\colon \cK(p,m,\ep,\et) \rightarrow \cK(p,m+n,\si\ep,\tau\et)
\end{equation}
Furthermore, we define a continuous function $H=H(\,\cdot\,;p_1,p_2,n)$\index{H@$H(x;p_1,p_2,n)$} by
\[
H(x;p_1,p_2,n) = \frac{1}{\nu_n^{\sgn(p_2)}(\sgn(p_1p_2)\la;p)}\frac{E(\la;p,m) }{ E(\sgn(p_1p_2)\la;p,m+n) } S(-\la;p_1,p_2,n), \quad x=\mu(\la),
\]
and we denote by $P=P(p_1,p_2,n)\in B(\cK)$ the spectral projection of $K$ corresponding to the eigenvalue $\sqrt{q^{-n}\,|p_2/p_1|}$. Then
\[
Q(p_1,p_2,n) = U_n^{\sgn(p_1),\sgn(p_2)}\, H(\Om)\, P.
\]
\end{prop}

Again the right hand side defining $H$ is not obviously symmetric with respect to
interchanging $\la$ and $\la^{-1}$, but it is as can be observed either from the
proof of Proposition \ref{prop:polardecompQ} or by observing that the
$\la$-dependent part in \eqref{eq:CintermsofH} is indeed symmetric with
respect to $\la\,  \leftrightarrow\, \la^{-1}$. Observe that $E(\la;p,m)=E(\la;pq^{2m},0)$ and $pq^{2m}=q^{-n}|\frac{p_2}{p_1}|$. Also, $\nu_n^{\sgn(p_2)}(\sgn(p_1p_2)\la;p)$ 
does not depend on $p$ and $m$ since $\epsilon(p)=\epsilon(q^{-n-2m}|\frac{p_2}{p_1}|)=\epsilon(q^{-n}|\frac{p_2}{p_1}|)$ 
by \eqref{eq:epsilon(p)}, so $H$, as a function of $x$, only depends on the parameters $p_1,p_2$ and $n$.

Note that Proposition \ref{prop:polardecompQ} proves Lemma \ref{lem:polardecompositionQppn}.

\begin{proof}
From \eqref{eq:defEGnu}, Lemma \ref{lem:actionTC(x)} and Theorem \ref{thm:spectraldecompositionCasimir} we find
\[
\begin{split}
\Big(\Up_{p,m'}^{\ep',\et'}&\, Q(p_1,p_2,n)\, (\Up_{p,m}^{\ep,\et})^* g \Big)(x) \\
&= \ep^{\hf(1-\si)} \et^{\hf(1-\tau)+n} \nu_n^\tau(\la;p)\dfrac{G(\la;p,m',\ep',\et') }{ G(\si\tau\la;p,m,\ep,\et)}H(\si\tau x)\, g(\si\tau x) \\
&= \Big((\Up_{p,m'}^{\ep',\et'})^* \, U_n^{\si,\tau} \, H(\Om) \, \Up_{p,m}^{\ep,\et}\,g\Big)(x),
\end{split}
\]
hence $Q(p_1,p_2,n) = U_n^{\si,\tau} \, H(\Om)$ on $\cK(p,m,\ep,\et)$. Now observe that $P$ is the orthogonal projection onto
\[
\bigoplus_{\substack{m \in \Z\\ \ep,\et \in \{-,+\}}} \
\cK(p,m,\ep,\et),\qquad p =q^{-2m} q^{-n}\,|p_2/p_1|,
 \]
then the result follows.
\end{proof}

From this proposition it follows that the function $C(x)$ from Lemma \ref{lem:actionTC(x)} can be written as
\begin{equation} \label{eq:CintermsofH}
C(x) = \ep^{\hf(1-\si)} \et^{\hf(1-\tau)+n}\nu_n^\tau(\la;p)\,\dfrac{  G(\la;p,m',\ep',\et') }{ G(\si\tau \la;p,m,\ep,\et)}H(\si\tau x;p_1,p_2,n),
\end{equation}
for $\si=\sgn(p_1)$ and $\tau = \sgn(p_2)$. Let us give two identities for the function $C$ that will be useful later on. The first identity follows from the structure formula in Proposition \ref{prop:structureconstQs} for the linear basis $\{Q(p_1,p_2,n) \mid p_1,p_2 \in I_q,\ n \in \Z\}$ for the von Neumann algebra $\hat M$. This formula implies a product formula for the function $C$ that is useful later on. The second identity is a consequence of Lemma \ref{lem:QadjointisQ}.

\begin{lemma} \label{lem:identitiesC}
Let $p_1,p_2,r_1,r_2 \in I_q$, $k,m,n \in \Z$, $\ep,\et \in \{-,+\}$ and $y \in [-1,1]$.
\begin{enumerate}[(i)]
\item Assume $|\frac{p_2}{p_1}|=q^m$ and $|\frac{r_1}{r_2}| = q^n$, then the following product formula holds:
\[
\begin{split}
C(y;k+m,\sgn(r_1)\ep,\,&\sgn(r_2)\et;p_1,p_2,n)C(\sgn(p_1p_2)y;k,\ep,\et;r_1,r_2,m) = \\
 &\sum_{\substack{x_1,x_2\in I_q\\ \sgn(x_1)=\sgn(p_1r_1)\\ \sgn(x_2)=\sgn(p_2r_2) \\ |x_1|=|x_2|}} a_{x_1}(r_1,p_1) a_{x_2}(r_2,p_2)C(y;k,\ep,\et;x_1,x_2,m+n).
\end{split}
\]
\item The following symmetry relation holds:
\[
\begin{split}
C(y;&m,\ep,\et;p_1,p_2,n) = \\
&(-q)^n \, \sgn(p_1)^{\chi(p_1)}
\, \sgn(p_2)^{\chi(p_2)}\,  C(\sgn(p_1p_2) y;m+n,\sgn(p_1)\ep,\sgn(p_2)\et;p_1,p_2,-n).
\end{split}
\]
\end{enumerate}
\end{lemma}

\begin{proof}
(i) From \eqref{eq:defTp1p2n} it follows that the operators $T(p_1,p_2,n)$ satisfy the same structure formula as the operators $Q(p_1,p_2,n)$, see Proposition \ref{prop:structureconstQs}. Let $p=q^{-2k-m-n}$, then $p=q^{-2(k+m)-n}|\frac{p_2}{p_1}|= q^{-2k-m}|\frac{r_2}{r_1}|$. Applying the structure formula to a function $g \in L^2(I(p,k,\ep,\et))$ and using the action of $T(p_1,p_2,n)$ as multiplication by the function $C$ from Lemma \ref{lem:actionTC(x)}, we obtain
\[
\begin{split}
C(y;&k+m,\sgn(r_1)\ep,\sgn(r_2)\et;p_1,p_2,n)C(\sgn(p_1p_2)y;k,\ep,\et;r_1,r_2,m) g\big(\sgn(p_1p_2r_1r_2)y\big)\\
& = \sum_{x_1,x_2 \in I_q} a_{x_1}(r_1,p_1) a_{x_2}(r_2,p_2)C(y;k,\ep,\et;x_1,x_2,m+n) g\big(\sgn(x_1x_2) y\big)
\end{split}
\]
Observe that $T(x_1,x_2,n+m)=0$ on $L^2(I(p,k,\ep,\et))$ unless $p=q^{-2k-n-m}|\frac{x_1}{x_2}|$, which implies that the sum is only over $x_1,x_2 \in I_q$ satisfying $|x_1|=|x_2|$. Furthermore, by Definition \ref{def:functionap} we have $a_x(p,r)=0$ if $\sgn(x) \neq \sgn(pr)$, so we may write $\sgn(x_1x_2) = \sgn(p_1p_2r_1r_2)$ in the above sum, since the terms where this is not true do no contribute to the sum. Finally, since $g$ was chosen arbitrarily, the result follows.

(ii) Write out $\langle T(p_1,p_2,n) f, g \rangle = \langle  f,T(p_1,p_2,n)^* g \rangle$ for suitable functions $f$ and $g$, using Lemma \ref{lem:actionTC(x)} and \eqref{eq:TadjointisT}. Using the fact that $f$ and $g$ are chosen arbitrarily and continuity in $y$ of the function $C(y)$, the identity follows. Alternatively, the second identity can also be derived from Lemma \ref{lem:propertiesofS}(i).
\end{proof}

\subsection{Generators of $\hat{M}$}\label{ssec:generatorsofhatM}

The main step towards finding a generating set for $\hat M$ is the polar-type decomposition for $Q(p_1,p_2,n)$ from  Proposition \ref{prop:polardecompQ}. The partial isometries $U^{\si,\tau}_n$, $\si,\tau \in \{-,+\}$, $n \in \Z$, from Proposition \ref{prop:polardecompQ} give us the required extra generators for the dual von Neumann algebra $\hat M$. First we show that the operators $U^{\si,\tau}_n$ belong to the von Neumann algebra $\hat M$.
\begin{prop} \label{prop:UlsitauinhatM}
For $l \in \Z$ and $\si,\tau \in \{-,+\}$, the operator $U_l^{\si,\tau}$ belongs to $\hat M$.
\end{prop}
\begin{proof}
Since $Q(p_1,p_2,n) \in \hat{M}$ by Proposition \ref{prop:QppngeneratedualM}, the polar decomposition $Q(p_1,p_2,n)\, =\ U(p_1,p_2,n) \, |Q(p_1,p_2,n)|$ of
Proposition \ref{prop:polardecompQp1p2n} gives that $U(p_1,p_2,n)\in\hat{M}$, $|Q(p_1,p_2,n)|\in \hat{M}$.
Recall that $U(p_1,p_2,n)\, =\, \Up^\ast\, V(p_1,p_2,n)\, \Up$, and that
\[
 V^{\sgn(p_1),\sgn(p_2)}_n\, =\, \sgn\bigl( N(\cdot;p_1,p_2,n)\bigr)\, V(p_1,p_2,n).
\]
Define the Borel sets $A=\{ x\in \R\mid N(x;p_1,p_2,n) >0\}$,
$B=\{ x\in \R\mid N(x;p_1,p_2,n) < 0\}$, so that
\[
V^{\sgn(p_1),\sgn(p_2)}_n\, =\, M(\chi_A(\cdot))\, V(p_1,p_2,n) \, -\,
M(\chi_B(\cdot))\, V(p_1,p_2,n)
\]
and
\[
\begin{split}
U^{\sgn(p_1),\sgn(p_2)}_n\, &=\, \Up^\ast\, V^{\sgn(p_1),\sgn(p_2)}_n\,\, \Up \, \\ &= \,
\,\Up^\ast\, M(\chi_A)\, \Up\, \Up^\ast\,  V(p_1,p_2,n) \Up\, -\,
\Up^\ast\, M(\chi_B)\, \Up\, \Up^\ast\, V(p_1,p_2,n)\, \Up
\,\\ &=\, E_\Om(A)\, U(p_1,p_2,n) \, - \, E_\Om(B)\, U(p_1,p_2,n)
\end{split}
\]
where $\chi_A$ is the indicator function of the set $A$ and
$E_\Om$ is the spectral decomposition of the Casimir operator using
Theorem \ref{thm:spectraldecompositionCasimir}. Since the Casimir operator $\Om$ is
affiliated to $\hat{M}$ by Theorem \ref{thm:Casimirwelldefinedandchar}, it follows
that the spectral projections $E_\Om(A), E_\Om(B)\in \hat{M}$. Since we already
noted that $U(p_1,p_2,n)\in\hat{M}$, we see that
$U^{\sgn(p_1),\sgn(p_2)}_n\in \hat{M}$.
\end{proof}

We can now show that the partial isometries $U_n^{\si,\tau}$ provide the extra generators for $\hat M$ that we need. The following properties are useful.

\begin{lemma} \label{lem:propertiesofU}
Let $m,n,n'\in \Z$, $p \in q^\Z$ and $\ep,\et,\si,\tau \in \{-,+\}$, then the partial isometries $U_n^{\si,\tau}:\cK(p,m,\ep,\et)\to \cK(p,m+n,\si\ep,\tau\et)$ satisfy the following properties:
\begin{enumerate}[(i)]
\item $U_{n+n'}^{++} = U_n^{++} U_{n'}^{++}$,
\item $U_n^{--} = U_n^{+-} U_0^{-+}$,
\item $U_n^{+-} = U_0^{+-} U_n^{++}$,
\item $U_n^{-+} = U_n^{++} U_0^{-+}$,
\item $(U_n^{\si,\tau})^* = \si^{n+1}\tau^{\epsilon(p)(1-\si) +1 } U_{-n}^{\si,\tau}$.
\end{enumerate}
\end{lemma}

\begin{proof}
This follows directly from the definition of $U_n^{\si,\tau}$, see \eqref{eq:defVnsitau} and \eqref{eq:defUnsitau}. For the computation of $(U_n^{\si,\tau})^*$ it is useful to observe that
$\nu_{-n}^\tau(\la;p)\nu_n^\tau(\si\tau\la;p)$ is equal to $-1$ for $\si=\tau=-$ and $\epsilon(p)=\hf$, and it is equal to $1$ in all other cases.
\end{proof}

Now we can finally show that
the von Neumann algebra $\hat{M}$ is generated by
$K$, $E$, $U^{+-}_0$ and $U^{-+}_0$.

\begin{proof}[Proof of Theorem \ref{thm:generatorsforhatM}.]
From Propositions \ref{prop:polardecompQ}, \ref{prop:UlsitauinhatM} and Lemma \ref{lem:propertiesofU} it follows that $\hat M$ is generated by $K$, $\Om$, $U_1^{++}$, $U_0^{+-}$ and $U_0^{-+}$. Using \eqref{eq:defEGnu} and writing $A(\la)$ explicitly, using the appropriate $c$-functions,  we find for $x=\mu(\la) \in I(p,m,\ep,\et) \cap I(p,m+1,\ep,\et)$
\[
G(\la;p,m+1,\ep,\et) = \eta G(\la;p,m,\ep,\et),
\]
hence $(V_1^{++} g)(x) = g(x)$, so we see from \eqref{eq:actionE0inUpsilonpm+-}, \eqref{eq:actionE0E0+inUpsilonpm-+}, \eqref{eq:actionE0E0+inUpsilonpm--} and \eqref{eq:actionE0E0+inUpsilonpm++mleq0} that $U_1^{++}=\Up^* V_1^{++}\Up$ is the partial isometry in the polar decomposition of $E$. Then, using Definition \ref{def:closedCasimir} for $\Om$, it follows that $\hat M$ is generated by $K$, $E$, $U_0^{+-}$ and $U_0^{-+}$.
\end{proof}

\section{Unitary corepresentations}\label{sec:unitarycoreps}
In this section we need the function $\upsilon:q^\Z \to \Z$ defined by
\begin{equation}\label{eq:defupsilon}\index{U@$\upsilon(t)$}
\begin{split}
\upsilon(t) = \hf\chi(t)+\epsilon(t),\qquad t \in q^\Z.
\end{split}
\end{equation}
So if $t=q^{2k}$ or $t=q^{2k-1}$ for some $k \in \Z$, then $\upsilon(t)=k$.

Recall from Section \ref{sec:decompleftregularcorep} that we assume $p\in q^\Z$, $m\in \Z$ and $\ep,\et \in \{-,+\}$. Let $\cK_d(p,m,\ep,\et)$ denote the closed subspaces of $\cK(p,m,\ep,\et)$ spanned by all the eigenvectors of $\Om$ in $\cK(p,m,\ep,\et)$, and denote its orthogonal complement by $\cK_c(p,m,\ep,\et)$,  so that we have a decomposition $\cK = \cK_c \oplus \cK_d$ corresponding to the continuous and discrete spectrum of $\Om$. The unitary operator $\Up_{p,m}^{\ep,\et}$ restricted to $\cK_d(p,m,\ep,\et)$ or $\cK_c(p,m,\ep,\et)$ is again a unitary operator mapping into $\ell^2(\si_d(p,m,\ep,\et))$, respectively $L^2([-1,1])$.

\subsection{Discrete series}\label{ssec:discreteseries}\index{corepresentation!discrete series}
In this subsection we assume that $x \in \si_d(p,m,\ep,\et)$. For $\ep,\et \in \{-,+\}$, $p \in q^\Z$ and $m \in \Z$, we define an element $e_m^{\ep,\et}(p,x) \in \cK_d(p,m,\ep,\et)$ by
\[
e_m^{\ep,\et}(p,x) = (\Up_{p,m}^{\ep,\et})^* \delta_{\ep\et x} =
 \sum_{z \in J(p,m,\ep,\et)} g_z(\ep\et x;p,m,\ep,\et) f_{-m,\ep\et pq^mz,z}\ .
\]
Since $\{\delta_x \mid x \in \si_d(p,m,\ep,\et)\}$ is an orthonormal basis for $\ell^2(\si_d(p,m,\ep,\et))$, it follows from unitarity of $\Up_{p,m}^{\ep,\et}$ that the set $\{ e_m^{\ep,\et}(p,x) \mid \ep\et x \in \si_d(p,m,\ep,\et)\}$ is an orthonormal basis for $\cK_d(p,m,\ep,\et)$. This can also be seen directly from \eqref{eq:orthonAlScrelationsdisc} and \eqref{eq:orthorefororthonorlqJacfun}. Moreover, from Theorem \ref{thm:spectraldecompositionCasimir} we see that $e_m^{\ep,\et}(p,x)$ is an eigenvector of $\Om$ for eigenvalue $\ep\et x$.

\begin{lemma}\label{lem:actionsofgeneratorsoneigvetsOm}
The actions of the generators of $\hat M$ on $e_m^{\ep,\et}(p,x)$ are given by
\[
\begin{split}
K \, e_m^{\ep,\et}(p,x) &= p^\hf q^m \, e_m^{\ep,\et}(p,x),\\
(q^{-1}-q) E\, e_m^{\ep,\et}(p,x) & = q^{-m-\hf}p^{-\hf} \sqrt{ 1+2\ep\et xq^{2m+1}p + q^{4m+2}p^2} \, e_{m+1}^{\ep,\et}(p,x),\\
U_{0}^{+-}\, e_m^{\ep,\et}(p,x) & = \et\, (-1)^{\upsilon(p)} \, e_m^{\ep,-\et}(p,x),\\
U_{0}^{-+}\, e_m^{\ep,\et}(p,x) & = \ep \et^{\chi(p)}(-1)^m \, e_m^{-\ep,\et}(p,x).
\end{split}
\]
\end{lemma}

\begin{proof} The action of $K$ follows from \eqref{eq:defKanddomain}; the action of $E$ follows from \eqref{eq:actionE0inUpsilonpm+-}, \eqref{eq:actionE0E0+inUpsilonpm-+}, \eqref{eq:actionE0E0+inUpsilonpm--} and \eqref{eq:actionE0E0+inUpsilonpm++mleq0}. To determine the action of $U_0^{+-}$ we observe that
\[
\begin{split}
V_0^{+-} \de_{\ep\et x}(\mu(\ga)) &= \et\, \nu_0^-(\ga;p)\frac{ G(\ga;p,m,\ep,-\et) }{G(-\ga;p,m,\ep,\et)} \, \de_{\ep\et x}(-\mu(\ga)) \\
&=\et\, (-1)^{\upsilon(p)} \, \de_{-\ep\et x}(\mu(\ga)),
\end{split}
\]
by Lemma \ref{lem:G(la)/G(-la)}, see below. Applying $\Up^*$ gives us
\[
U_0^{+-} e_m^{\ep,\et}(p,x) = \et\, (-1)^{\upsilon(p)} \,  e_m^{\ep,-\et}(p,x).
\]
The action of $U_0^{-+}$ is calculated in the same way.
\end{proof}

\begin{lemma} \label{lem:G(la)/G(-la)}
We have
\[
\nu_0^-(\la;p)\frac{ G(\la;p,m,\ep,-\et) }{ G(-\la;p,m,\ep,\et) } = (-1)^{\upsilon(p)}, \qquad
\nu_0^+(\la;p)\frac{ G(\la;p,m,-\ep,\et) }{ G(-\la;p,m,\ep,\et) } = (-1)^m\, \et^{\chi(p)}.
\]
\end{lemma}

\begin{proof}
We treat here the formula for $\la \in \T$, $\ep=\et=+$, and $m \leq 0$ in detail. The formulas corresponding to the other cases are obtained from similar computations. Note that, by construction, all formulas are equal to $\pm 1$

Assume $\la \in \T$, $\ep=\et=+$, and $m \leq 0$. From writing $A(\la;p,m,+,\pm)$ and $E(\la;p,m)$ in terms of $q$-shifted factorials and canceling common factors, we obtain (see \eqref{eq:defEGnu})
\begin{multline*}
\frac{G(\la;p,m,+,-) }{ G(-\la;p,m,+,+) } = \\
 \frac{\la^{1-m-\chi(p)} (q\la/p;q^2)_\infty }{ (pq\la, -q^{3-2m}/p\la, -q^{2m-1}p\la;q^2)_\infty } \sqrt{ \frac{(qp\la^{\pm1}, -q^{3-2m}\la^{\pm1}/p, -q^{2m-1}p\la^{\pm1};q^2)_\infty  }{(q\la^{\pm1}/p;q^2)_\infty }}.
\end{multline*}
Recall here that $(a\la^{\pm 1};q)_\infty = (a\la,a/\la;q)_\infty$, which is strictly positive for $0\neq a \in \R$ and $\la \in \T$. Now assume $\epsilon(p)=\hf$. Using the $\te$-product identity \eqref{eq:thetaprodid} we may write
\[
\begin{split}
(q\la/p;q^2)_\infty &= (-1)^{\upsilon(p)} \la^{\upsilon(p)} q^{-\upsilon(p)(\upsilon(p)-1)} \frac{ (1/\la,q^2\la;q^2)_\infty }{ (pq/\la;q^2)_\infty}, \\
(-q^{3-2m}/p\la, -q^{2m-1}p\la;q^2)_\infty &= \la^{1-m-\upsilon(p)}q^{-(m+\upsilon(p)-1)(m+\upsilon(p)-2)} (-\la,-q^2/\la;q^2)_\infty,
\end{split}
\]
from which it follows that
\[
\frac{G(\la;p,m,+,-) }{ G(-\la;p,m,+,+) } = (-1)^{\upsilon(p)} \frac{ \la\, (1/\la,q^2\la;q^2)_\infty }{ (-\la,-q^2/\la;q^2)_\infty } \sqrt{ \frac{ (-\la^{\pm 1},-q^2\la^{\pm 1};q^2)_\infty }{ (\la^{\pm 1},q^2\la^{\pm 1};q^2)_\infty}}\,.
\]
Using the identity $(aq;q)_\infty = (a;q)_\infty/(1-a)$, and using $(a^{\pm 1};q)>0$ for $a \in \T$, the above expression reduces to
\[
(-1)^{\upsilon(p)} \frac{ 1+\la }{ 1-\la} \sqrt{ \frac{(1-\la)(1-\la^{-1}) }{ (1+\la)(1+\la^{-1}) }}.
\]
From this expression we finally obtain
\[
\frac{G(\la;p,m,+,-) }{ G(-\la;p,m,+,+) }  =
\begin{cases}
i\,(-1)^{\upsilon(p)} & \la \in \T^+,\\
i\,(-1)^{\upsilon(p)+1} & \la \in \T^-.
\end{cases}
\]
Using $\nu_0^-(\la;p) = \mp i$ for $\la \in \T^{\pm}$, the result follows for the case $\epsilon(p)=\hf$.

Next we assume $\epsilon(p)=0$. The $\te$-product identity \eqref{eq:thetaprodid} gives in this case
\[
\begin{split}
(q\la/p;q^2)_\infty &= (-1)^{\upsilon(p)} (\la/q)^{\upsilon(p)} q^{-\upsilon(p)(\upsilon(p)-1)} \frac{ (q\la^{\pm 1};q^2)_\infty }{ (pq/\la;q^2)_\infty}, \\
(-q^{3-2m}/p\la, -q^{2m-1}p\la;q^2)_\infty &= (q\la)^{1-m-\upsilon(p)}q^{-(m+\upsilon(p)-1)(m+\upsilon(p)-2)} (-q\la^{\pm 1};q^2)_\infty.
\end{split}
\]
Now all $q$-shifted factorials become symmetric in $\la$ and $\la^{-1}$, hence positive, and this leads to
\[
\frac{G(\la;p,m,+,-) }{ G(-\la;p,m,+,+) }  = (-1)^{\upsilon(p)} \frac{ (q\la^{\pm 1};q^2)_\infty }{(pq\la^{\pm 1}, -q\la^{\pm 1};q^2)_\infty } \sqrt{ \frac{(pq\la^{\pm 1}, -q\la^{\pm 1};q^2)_\infty^2 }{ (q\la^{\pm 1};q^2)_\infty^2 } } = (-1)^{\upsilon(p)}.
\]
This proves the result in case $\epsilon(p)=0$.
\end{proof}

Notice that the invariance of $\cL_{p,x}$ as defined in
Lemma \ref{lem:closedsubspaceinregularrep} follows from the
fact that for $p_1,p_2 \in I_q$ and $n \in \Z$, the operator
\[
Q(p_1,p_2,n)\colon \cK(p,m,\ep,\et) \to  \cK(p,m+n,\sgn(p_1)\,\ep,\sgn(p_2)\,\et)
\]
and
$\sgn(p_1 p_2)\,\,Q(p_1,p_2,n)\,\Om \subseteq \Om \,Q(p_1,p_2,n)$
for all $p_1,p_2 \in I_q$ and $n \in \Z$, see
Lemma \ref{lem:explicitactionQppn} and
Proposition \ref{prop:QppngeneratedualM}. This proves that
$\cL_{p,x}$ is an invariant subspace for $\hat{M}$, hence it gives
rise to a corepresentation of $(M,\De)$. Since $W$ is the multiplicative
unitary, its restriction $W_{p,x}$ is also unitary. This proves
Lemma \ref{lem:closedsubspaceinregularrep}.

In order to prove Proposition \ref{prop:discretesubrepresinregrepr}
we have to do some bookkeeping, based on the discrete spectrum
of the Casimir operator $\Om$ acting on $\cK(p,m,\ep,\et)$
given as $\si_d(p,m,\ep,\et) = \{ \mu(\la)\mid \la \in D(p,m,\ep,\et)\}$, where the
set  $D(p,m,\ep,\et)$ is given explicitly in \eqref{eq:spectrum+-},
\eqref{eq:spectrum--}, \eqref{eq:spectrum++}. So we have to keep
track which of the eigenvectors $e^{\ep,\et}_m(p,x)$ in $\cL_{p,x}$
correspond to eigenvalues in the spectrum of $\Om$ in $\cK(p,m,\ep,\et)$.

\begin{proof}[Proof of Proposition \ref{prop:discretesubrepresinregrepr}]
Note that $p = q^{-l-j}$ and $|\la| = q^{1+l-j}$.
Since $|\la| > 1$, it follows that $l < j$.
In order to see that $|\la|\in q^{2\Z+1} p$, consider $x = \mu(\la)$ where
$\la \in -q^{-\N} \cup q^{-\N}$. It follows from
\eqref{eq:spectrum+-}, \eqref{eq:spectrum--} and \eqref{eq:spectrum++}
that if there exist $m \in \Z$ and $\ep,\et \in \{-,+\}$
such that $\cK(p,m,\ep,\et)\ni e^{\ep,\et}_m(p,x) \not= 0$, then
$|\lambda|\in q^{2\Z+1} p$.

Assume first that $x > 0$. It follows from \eqref{eq:spectrum++}
that an eigenvector $e^{++}_m(p,x) \in \cK(p,m,+,+)$ is non-zero
if and only if $x=\mu(q^{1+2k}p)$ such that $q^{1+2k}p>1$. So such
an eigenvector exists for all $m\in \Z$. It follows from \eqref{eq:spectrum--}
that such an eigenvector $e^{--}_m(p,x) \in \cK(p,m,-,-)$ does not exist, since
the discrete spectrum of $\Om$ on $\cK(p,m,-,-)$ is always negative.
A check shows that eigenvectors $e^{+-}_m(p,x) \in \cK(p,m,+,-)$
satisfying $\Om\, e^{+-}_m(p,x) = -\mu(q^{1+2k}p)\, e^{+-}_m(p,x)$ exist
precisely when $k\geq -m$. Similarly, eigenvectors $e^{-+}_m(p,x) \in \cK(p,m,-,+)$
satisfying $\Om\, e^{-+}_m(p,x) = -\mu(q^{1-2j}p^{-1})\, e^{-+}_m(p,x)$ exist
precisely when $j\leq m$. This covers the case (i) of Proposition
\ref{prop:discretesubrepresinregrepr}.

For the remainder of the proof we assume $x < 0$.
Since $l < j$, it cannot happen that $j \leq 0$ and $l \geq 0$.
We start by looking at eigenvectors for positive eigenvalues of $\Om$
in $\cK(p,m,\ep,\et)$ for $\ep\not=\et$. From \eqref{eq:spectrum+-}
it follows that such eigenvectors occur for the eigenvalue
$-x=\mu(q^{1+2l}p)$ in $\cK(p,m,-,+)$ precisely when $l\geq 0$ and
that such eigenvectors occur for the eigenvalue
$-x=\mu(q^{1-2j}p^{-1})$ in $\cK(p,m,+,-)$ precisely when $j\leq 0$.
So these cases cannot occur simultaneously, and we consider them
separately.

From \eqref{eq:spectrum++} we find eigenvectors $e^{++}_m(p,x)$ in case $l\geq \max(0,m)$
or $j\leq 0$.
Using \eqref{eq:spectrum--} we see that $e^{--}_m(p,x)$ is an
eigenvector for the eigenvalue $x$ precisely when
$l\geq 0$ or $l \geq m$ for the case $pq^m\leq 1$, i.e. $l+j\leq m$.
In case $pq^m\geq 1$, or $l+j\geq m$, we see that $e^{--}_m(p,x)$ is an
eigenvector for the eigenvalue $x$ precisely when
$j\leq 0$ or $j \leq m$.

Assume $l\geq 0$, and hence $j>0$. Then we find no eigenvectors
of type $e^{+-}_m(p,x)$ and
$e^{-+}_m(p,x)$ for all $m\in \Z$ by \eqref{eq:spectrum+-}. We find
$e^{++}_m(p,x)$ for all $m\in\Z$ with $m\leq l$ by considering
the case $m\leq 0$ and $m\geq 0$ separately in \eqref{eq:spectrum++}.
Consider now \eqref{eq:spectrum--}. In case
$m\leq l+j$ (or $q^mp\geq 1$) we find
eigenvectors $e^{--}_m(p,x)$ for $j\leq 0$, which is excluded
in this case, or $j\leq m$. So in total we get $e^{--}_m(p,x)$
for $j\leq m\leq l+j$. In case $m\geq l+j$ (or $q^mp\leq 1$) we find
eigenvectors $e^{--}_m(p,x)$ for $l\geq m$, which is excluded since
it implies $j\leq 0$, and for $l\geq 0$. So we find $e^{--}_m(p,x)$
for $m\geq l+j$. Combining we find $e^{--}_m(p,x)$ for all $m\in\Z$
with $m\geq j$. This gives case (ii) of Proposition
\ref{prop:discretesubrepresinregrepr}.
Case (iii) is obtained similarly by analyzing $j\leq 0$ and hence
$l<0$.
\end{proof}

Next we show that the corresponding unitary corepresentations
of $(M,\De)$ are irreducible.

\begin{proof}[Proof of Proposition \ref{prop:Wpxgivesirreduciblediscreteseries}]
We have already observed that $\cL_{p,x}$ is invariant.
Consider $\cL_{p,x}$ with the convention $p=q^{-l-j}$ with $l<j$ as
in Proposition \ref{prop:discretesubrepresinregrepr}, and assume for the moment that
$l+1\not=j$, or $l+1<j$. We claim that is possible to choose
$\ep,\et\in \{\pm\}$, $m\in\Z$ so that
\begin{enumerate}
\item  $0\not= e^{\ep,\et}_m(p,x) \in \cL_{p,x}$
\smallskip
\item $e^{s\ep,t\et}_m(p,x)\notin \cL_{p,x}$  for $(s,t) =(-,+),(+,-),(-,-)$.
\end{enumerate}
Take $m \in\Z$ such that $l < m < j$, which is possible by the assumption $l+1<j$.
By Proposition \ref{prop:discretesubrepresinregrepr} we find by inspection that
$e^{++}_m(p,x)$ in case (i), $e^{-+}_m(p,x)$ in case (ii) and
$e^{+-}_m(p,x)$ in case (iii) gives the required choice.

Recall that $K$ is affiliated to $\hat{M}$, see
Proposition \ref{prop:KEaffiliatedtohatM}.
Thus, if $P$ denotes the spectral
projection of $K$ with respect to the eigenvalue
$p^{\frac{1}{2}}\,q^m$, we find $P \in \hat{M}$.
So  $P\vert_{\cL_{p,x}}$ is the orthogonal
projection onto the closed subspace spanned by
$\{ e^{s,t}_m(p,x) \mid s,t \in \{-,+\}\,\}$.
But by our choice of $m$,$\ep$ and $\et$, this implies that
$P\vert_{\cL_{p,x}}$ is the orthogonal
projection onto $e^{\ep,\et}_m(p,x)$.  So we can
look at the invariant subspace of $\cL_{p,x}$
generated by the vector $e^{\ep,\et}_m(p,x)$.

Consider the closure $L$ of
$\{ T\vert_{\cL_{p,x}}\, e^{\ep,\et}_m(p,x) \mid T\in \hat{M}\}$,
so that $L$ is an invariant subspace of $\cL_{p,x}$ and
$L\not= \{0\}$.
Using Lemma \ref{lem:actionsofgeneratorsoneigvetsOm}
the partial isometry $V$ in the polar decomposition
of $E$ maps $e^{\ep,\et}_m(p,x)$ to $e^{\ep,\et}_{m+1}(p,x)$ if this corresponds
to an eigenvector of $\Om$ in $\cK(p,m+1,\ep,\et)$ and to zero if
this is not the case.
Using Lemma \ref{lem:actionsofgeneratorsoneigvetsOm} and
the fact that the partial isometry $V\in \hat{M}$ by Proposition
\ref{prop:EaffiliatedtohatM}, we see that all
other vectors in the three lists for $\cL_{p,x}$
in Proposition \ref{prop:discretesubrepresinregrepr} can be reached
by repeated application of $V$, $V^\ast$, $U^{+-}_0$ and $U^{-+}_0$.
Hence $L=\cL_{p,x}$, and irreducibility follows.

In case $l+1=j$ we cannot establish that $P\vert_{\cL_{p,x}}$ is
the orthogonal projection on a single vector in the lists as
in Proposition \ref{prop:discretesubrepresinregrepr}, but we
can view it, by taking $m=l$, as an orthogonal projection on the subspace
$\C\, e^{++}_l(p,x)\oplus \C e^{-+}_l(p,x)$ in case (i)
of Proposition \ref{prop:discretesubrepresinregrepr}, on
$\C\, e^{-+}_l(p,x)\oplus \C e^{++}_l(p,x)$ in case (ii) and on
$\C\, e^{+-}_l(p,x)\oplus \C e^{--}_l(p,x)$ in case (iii). Now use
the fact that the partial isometry $V\in \hat{M}$ of $E$ kills
the second vector in each of these spaces to see that the range
of the composition $V\vert_{\cL_{p,x}}P\vert_{\cL_{p,x}}$ has dimension $1$ spanned
by $e^{++}_{l+1}(p,x)$ in case (i), by $e^{-+}_{l+1}(p,x)$ in case (ii)
and by $e^{+-}_{l+1}(p,x)$ in case (iii). Now we can argue as in the
case $l+1>j$ above to find that $\cL_{p,x}$ is irreducible.
\end{proof}

\subsection{Principal series}\label{ssec:principalseries}\index{corepresentation!principal series}
We start by recalling the definition of Section \ref{sec:decompleftregularcorep}.
Let $x=\cos\te\in [-1,1]$ and $p \in q^\Z$. We define a Hilbert space $\cL_{p,x}$ by
\[
\cL_{p,x} = \bigoplus_{\ep,\et \in \{-,+\}} \ell^2_{\ep,\et}(p,x),
\]
where each space $\ell^2_{\ep,\et}(p,x)$ denotes a copy of $\ell^2(\Z)$ with standard orthonormal basis $\{ e_m^{\ep,\et}(p,x) \mid m\in \Z\}$. For convenience we recall the definition of the operators $K, E, U_0^{+-}, U_0^{-+}$ on  $\cL_{p,x}$ as given in \eqref{eq:introprincipalseries};
\begin{equation} \label{eq:principalseries}
\begin{split}
K \, e_m^{\ep,\et}(p,x) &= p^\hf q^m \, e_m^{\ep,\et}(p,x),\\
(q^{-1}-q) E\, e_m^{\ep,\et}(p,x) & = q^{-m-\hf}p^{-\hf} |1+\ep\et p q^{2m+1} e^{i\te}| \, e_{m+1}^{\ep,\et}(p,x),\\
U_{0}^{+-}\, e_m^{\ep,\et}(p,x) & = \et\, (-1)^{\upsilon(p)} \, e_m^{\ep,-\et}(p,x),\\
U_{0}^{-+}\, e_m^{\ep,\et}(p,x) & = \ep \et^{\chi(p)}(-1)^m \, e_m^{-\ep,\et}(p,x).
\end{split}
\end{equation}
The operators $E$ and $K$ are unbounded closable operators with dense domain the finite linear combinations of the orthonormal basis vectors $e_m^{\ep,\et}(p,x)$, $m \in \Z$, $\ep,\et \in \{-,+\}$. The operators $U_0^{+-}$ and $U_0^{-+}$ are bounded; they are isometries.

\begin{remark} \label{rem:ell2epetassumodule}
It is useful to observe that each subspace $\ell^2_{\ep,\et}(p,x)$ of $\cL_{p,x}$ is a principal series $\su$-module as defined by \eqref{eq:princunitaryserrep}. The above defined actions of $K$ and $E$ on $e_m^{\ep,\et}(p,x)$ coincide with the actions of $\mathbf K$ and $\mathbf E$ in the principal series representation $\pi_{b,\epsilon(p)}$ on the standard basis vector $e_{m+k}$, where $\mu(q^{2ib}) = -\ep\et x$ and $p=q^{2k+2\epsilon(p)}$. Using $\Om = \hf\big((q^{-1}-q)^2 E^*E-qK^2-q^{-1}K^{-2}\big)$ it can be verified that $\Om\,e_m^{\ep,\et}(p,x) = \ep\et x \, e_m^{\ep,\et}(p,x)$. Furthermore, the discrete series corepresentations from Lemma \ref{lem:actionsofgeneratorsoneigvetsOm} can (formally) be obtained from \eqref{eq:principalseries} by taking $\ep\et x$ in the discrete spectrum of $\Om$.
\end{remark}

The operators \eqref{eq:principalseries} generate a von Neumann algebra $\hat M_{p,x}$. We can construct
the elements $Q_{p,x}(p_1,p_2,n)$,  $p_1,p_2 \in I_q$, $n \in \Z\}$ for $\hat M_{p,x}$, basically by reversing the arguments that led to the proof of Theorem \ref{thm:generatorsforhatM}.
Let us first define operators $U_n^{\si,\tau}:\cL_{p,x} \to \cL_{p,x}$ for $n \in \Z$ and $\si,\tau \in \{-,+\}$ as follows. We set $U_0^{++}=\Id$, and we define $U_1^{++}$ as the partial isometry in the polar decomposition of $E$, i.e., $U_1^{++} e_m^{\ep,\et}(p,x) = e_{m+1}^{\ep,\et}(p,x)$. Now we define $U_n^{\si,\tau}$, $n \in \Z$, $\si,\tau \in \{-,+\}$, recursively by
\begin{align*}
U_n^{++} &= U_{n-1}^{++} U_1^{++}, && n \in \N,\\
U_n^{+-} & = U_0^{+-} U_n^{++}, && n \in \N,\\
U_n^{-+} &= U_n^{++} U_0^{-+}, && n \in \N,\\
U_n^{--} &= U_n^{+-} U_0^{-+}, && n \in \N_0,\\
U_{-n}^{\si\tau} &= \si^{n+1} \tau^{\epsilon(p)(1-\si)+1}(U_n^{\si\tau})^*, && n \in \N.
\end{align*}
From \eqref{eq:principalseries}, Lemma \ref{lem:G(la)/G(-la)} and the identity $G(\la;p,m+1,\ep,\et) = \et G(\la;p,m,\ep,\et)$, see the proof of Theorem \ref{thm:generatorsforhatM}
at the end of Section \ref{ssec:generatorsofhatM}, we find
\begin{equation}\label{eq:Unsitaupx}
U_n^{\si\tau} e_m^{\ep\et}(p,x) = \ep^{\hf(1-\si)} \et^{\hf(1-\tau)+n} \nu_n^\tau(\la;p)\dfrac{ G(\la;p,m+n,\si\ep,\tau\et)}{G(\si\tau\la;p,m,\ep,\et)} \ e_{m+n}^{\si\ep,\tau\et}(p,x),
\end{equation}
where $\mu(\la)=x$. Now for $p_1,p_2 \in I_q$ we set $\si=\sgn(p_1)$, $\tau=\sgn(p_2)$, and we define
\[
Q_{p,x}(p_1,p_2,n) = U_n^{\si,\tau} H(\Om;p_1,p_2,n) P(p_1,p_2,n),
\]
where $P(p_1,p_2,n)$ is the spectral projection of $K$ defined in \eqref{eq:principalseries} corresponding to the eigenvalue $\sqrt{q^{-n} |p_2/p_1|}$, and $H$ is the function defined in Proposition \ref{prop:polardecompQ}.

\begin{lemma} \label{lem:propertiesQpx}
The operators $Q_{p,x}(p_1,p_2,n)$ have the following properties:
\begin{enumerate}[(i)]
\item $Q_{p,x}(p_1,p_2,n)$ acts on the standard basisvectors of $\cL_{p,x}$ by
\[
Q_{p,x}(p_1,p_2,n) e_m^{\ep,\et}(p,x) =
\begin{cases}
0, & q^{2m} \neq q^{-n}\big|\tfrac{p_2}{p_1p}\big|,\\
C(\si\tau \ep\et\, x;m,\ep,\et;p_1,p_2,n)\, e_{m+n}^{\si\ep,\tau\et}(p,x), & q^{2m} = q^{-n}\big|\tfrac{p_2}{p_1p}\big|,
\end{cases}
\]
where $C$ is the function given by \eqref{eq:CintermsofH}.
\item In $\hat M_{p,x}$ we have
\[
Q_{p,x}(p_1,p_2,n) Q_{p,x}(r_1,r_2,m) = 0,
\]
if $|\frac{p_2}{p_1}|\neq q^m$ or $|\frac{r_1}{r_2}|\neq q^n$, and
\[
\begin{split}
Q_{p,x}(p_1,p_2,n)& Q_{p,x}(r_1,r_2,m)  = \sum_{\substack{x_1,x_2\in I_q\\ \sgn(x_1)=\sgn(p_1r_1)\\ \sgn(x_2)=\sgn(p_2r_2)}} a_{x_1}(r_1,p_1) a_{x_2}(r_2,p_2) Q_{p,x}(x_1,x_2,n+m).
\end{split}
\]
if $|\frac{p_2}{p_1}|=q^m$ and $|\frac{r_1}{r_2}|= q^n$.
\item The adjoint of $Q_{p,x}(p_1,p_2,n)$ in $\hat M_{p,x}$ is given by
\[
Q_{p,x}(p_1,p_2,n)^* = (-q)^n \, \sgn(p_1)^{\chi(p_1)}
\, \sgn(p_2)^{\chi(p_2)}\, Q_{p,x}(p_1,p_2,-n).
\]
\end{enumerate}
\end{lemma}
\begin{proof}
(i) First note that $P(p_1,p_2,n)$ is the orthogonal projection onto
\[
\mathrm{Span}\Big\{ e_m^{\ep,\et}(p,x) \mid q^{2m} = q^{-n}\big|\tfrac{p_2}{p_1p}\big|, \ \ep,\et \in \{-,+\}\Big\}.
\]
The explicit action of $Q_{p,x}(p_1,p_2,n)$ on an orthonormal basisvector $e_m^{\ep,\et}(p,x)$ now follows from \eqref{eq:Unsitaupx} and \eqref{eq:CintermsofH}.

(ii) By (i) the product of two $Q_{p,x}$ operators is given by
\[
\begin{split}
Q_{p,x}&(p_1,p_2,n) Q_{p,x}(r_1,r_2,m) e_k^{\ep,\et}(p,x)= \\
&C(y;k+m,\sgn(r_1)\ep,\sgn(r_2)\et;p_1,p_2,n)
C(\sgn(p_1p_2)\, y;k,\ep,\et;r_1,r_2,m)\, e_{k+m+n}^{\ep',\et'}(p,x)
\end{split}
\]
if $q^{2k} = q^{-m}|\frac{r_2}{r_1p}|$ and $q^{2k+2m}=q^{-n}|\frac{p_1}{p_2p}|$, and it is zero otherwise. Here $y =\sgn(p_1p_2r_1r_2)\, \ep\et x$, $\ep'=\sgn(r_1p_1)\ep$ and $\et'=\sgn(r_2p_2)\et$. Now we use the product formula for the function $C$ from Lemma \ref{lem:identitiesC}, then it follows that
\[
\begin{split}
Q_{p,x}(p_1,p_2,n)& Q_{p,x}(r_1,r_2,m) e_k^{\ep,\et}(p,x) =\\
&\sum_{\substack{x_1,x_2\in I_q\\ \sgn(x_1)=\sgn(p_1r_1)\\ \sgn(x_2)=\sgn(p_2r_2) \\ |x_1|=|x_2|}} a_{x_1}(r_1,p_1) a_{x_2}(r_2,p_2)C(y;k,\ep,\et;x_1,x_2,m+n)e_{k}^{\ep,\et}(p,x).
\end{split}
\]
if $|\frac{p_2}{p_1}|=q^m$ and $|\frac{r_1}{r_2}|=q^n$, and the product is zero otherwise. Observe that inside the sum on the right hand side the condition $q^{2k} = q^{-n-m}|\frac{x_1}{x_2p}|$ is satisfied because $|x_1|=|x_2|$, and since $Q_{p,x}(x_1,x_2,n+m)=0$ otherwise, the product formula for two $Q_{p,x}$ operator follows.

(iii) The adjoint of $Q_{p,x}(p_1,p_2,n)$ follows from (i) and the symmetry property for $C$ from Lemma \ref{lem:identitiesC}.
\end{proof}

\begin{prop} \label{prop:principalseriescorepresentation}
Let $p_1,p_2 \in I_q$, $n \in \Z$, and let $\om_{f,g} \in B(\cK)_*$ be the normal functional given by $\om_{f,g}(y) = \langle yf, g\rangle$ for $y \in B(\cK)$, where $f=f_{0,p_1,1}$ and $g=f_{n,p_2,1}$. Then there exists a unique unitary corepresentation $W_{p,x} \in M \otimes B(\cL_{p,x})$ such that
\[
(\om_{f,g} \otimes \Id)(W_{p,x}^*) = Q_{p,x}(p_1,p_2,n).
\]
\end{prop}
The proof of this proposition follows from Lemmas \ref{lem:explicitWpx} - \ref{lem:Wpxunitary}.
\begin{lemma} \label{lem:explicitWpx}
Assume $p_1,t_1 \in I_q$, $m_1,m \in \Z$ and $\ep,\et \in \{-,+\}$. There exists a unique co-isometry $W_{p,x} \in M \otimes B(\cL_{p,x})$ such that
\begin{equation} \label{eq:principalseriesWpx}
\begin{split}
W_{p,x}^*\big(f_{m_1p_1t_1}& \otimes e_m^{\ep,\et}(p,x)\big) =\\
  &\sum_{p_2 \in I_q}  C(\sgn(p_1p_2)\ep\et x;m,\ep,\et;p_1,p_2,\chi(p_2/p_1p)-2m)\\
 & \qquad\times\, f_{\chi(p_2/p_1p)+m_1-2m,p_2,t_1} \otimes e_{\chi(p_2/p_1p)-m}^{\sgn(p_1)\ep, \sgn(p_2)\et}(p,x).
\end{split}
\end{equation}
\end{lemma}
\begin{proof}
We set $W_\Up = (\Id \otimes \Up) W (\Id \otimes \Up^*)$. For $i=1,2$, assume $m_i,n_i \in \Z$, $p_i,t_i \in I_q$, $\ep_i,\et_i, \in \{-,+\}$, $p \in q^\Z$, $g_i \in L^2(I(p,m_i,\ep_i,\et_i))$. Recall from \eqref{eq:generalcasereducestoQppn} that $\bigl(\om_{f_{n_1p_1t_1}, f_{n_2p_2t_2}}\ot\Id\bigr)(W^*) = \de_{t_1t_2}Q(p_1,p_2,n_2-n_1)$, then by Lemma \ref{lem:actionTC(x)}
and \eqref{eq:defTp1p2n} we have
\[
\begin{split}
\big\langle& W_\Up^* (f_{n_1p_1t_1} \otimes g_1),  f_{n_2p_2t_2} \otimes g_2 \big\rangle =
\big\langle \bigl(\om_{f_{n_1p_1t_1}, f_{n_2p_2t_2}}\ot\Id\bigr)(W_{\Up}^*)g_1,g_2\big\rangle \\
&= \de_{t_1t_2} \de_{\sgn(p_1)\ep_1,\ep_2} \de_{\sgn(p_2)\et_1,\et_2} \de_{m_1+n_2-n_1,m_2}\big\langle C(\,\cdot\,;m_1,\ep_1,\et_1;p_1,p_2,n_2-n_1)g_1\big(\sgn(p_1p_2\,\cdot\,\big), g_2 \big\rangle,
\end{split}
\]
if $p=q^{n_1-n_2-2m_1}|\frac{p_2}{p_1}|$, and the expression is equal to zero otherwise. Now it follows that
\[
W_\Up^*(f_{n_1p_1t_1} \otimes g_1) = \sum_{p_2\in I_q} f_{k_1,p_2,t_1} \otimes C(\,\cdot\,;m_1,\ep_1,\et_1;p_1,p_2,k_1-n_1)g_1(\sgn(p_1p_2)\,\cdot\,),
\]
where $k_i=k_i(p_2)\in \Z$, $i=1,2$, is determined by $p=q^{n_i-k_i-2m_i}|\frac{p_2}{p_1}|$.

Let $(\de_n)_{n \in \N}$ be a sequence of nonnegative real-valued continuous functions on $[-1,1]$ that approximate the Dirac $\de$-distribution $\de(\cdot-x)$. In particular, the functions $\de_n$ have the following property:
\[
\lim_{n \to \infty} \int_{-1}^1 \delta_n(u) f(u) du =f(x),
\]
for a continuous function $f$ on $[-1,1]$. We write $\de_n(\,\cdot\,;p,m_i,\ep_i,\et_i)$ for the function $\de_n(\ep_i\et_i\,\cdot\,)$ considered as a function in $L^1(I(p,m_i,\ep_i,\et_i))$. In particular, $\de_n(x;p,m_i,\ep_i,\et_i)=0$ for $x \not\in [-1,1]$. We set $g_i=\sqrt{\de(\,\cdot\,;p,m_i,\ep_i,\et_i)}$, then by unitarity of $W_\Up$,
\[
\begin{split}
\de&_{n_1n_2} \de_{p_1p_2}\de_{t_1t_2} \de_{m_1m_2}\de_{\ep_1\ep_2}\de_{\et_1\et_2} = \big\langle W_\Up^*(f_{n_1p_1t_1}\otimes g_1), W_\Up^*(f_{n_2p_2t_2}\otimes g_2) \big\rangle \\
= &\,\de_{t_1t_2} \de_{n_1-2m_1-\chi(p_1),n_2-2m_2-\chi(p_2)} \de_{m_1-n_1,m_2-n_2} \de_{\sgn(p_1)\ep_1,\sgn(p_2)\ep_2} \de_{\et_1 \et_2} \\
& \times \sum_{p_3\in I_q} \int_{-1}^1  g_1(\sgn(p_1p_3)x)g_2(\sgn(p_2p_3)x)\\
& \qquad \times C(x;m_1,\ep_1,\et_1;p_1,p_3,k_1-n_1)C(x;m_2,\ep_2,\et_2;p_2,p_3,k_2-n_2) dx \\
=&\,\de_{t_1t_2} \de_{n_1-2m_1-\chi(p_1),n_2-2m_2-\chi(p_2)} \de_{m_1-n_1,m_2-n_2} \de_{\sgn(p_1)\ep_1,\sgn(p_2)\ep_2} \de_{\et_1 \et_2} \\
 \times& \sum_{p_3\in I_q} \int_{-1}^1\de_n(\sgn(p_1p_3)\ep_1\et_1 x) C(x;m_1,\ep_1,\et_1;p_1,p_3,k_1-n_1)C(x;m_2,\ep_1,\et_1;p_1,p_3,k_2-n_2) dx.
\end{split}
\]
Here $k_i=k_i(p_3)$. Note that $C$ is real-valued on $[-1,1]$. Since the function $C$ is continuous on $[-1,1]$, we find from letting $n \to \infty$,
\[
\begin{split}
\de_{p_1p_2}  =
\sum_{p_3\in I_q} & C(\sgn(p_3)y;m_1,\ep_1,\et_1;p_1,p_3,\chi(p_3/pp_1)-2m_1)\\
\times\,&  C(\sgn(p_3)y;m_1+\chi(p_1/p_2),\sgn(p_1p_2)\ep_1,\et_1;p_2,p_3,\chi(p_3/pp_1)-2m_1+\chi(p_2/p_1)),
\end{split}
\]
where $y=\sgn(p_1)\ep_1\et_1 x$. Absolute convergence of this sum is obtained from Lemma \ref{lem:asymptoticsforS}, see also the proof of Lemma \ref{lem:biorthogonalityS}. Now it follows that $W_{p,x}^*$ defined by \eqref{eq:principalseriesWpx} is an isometry.

Furthermore, from the explicitly formula for $W_{p,x}^*$ we see that $W_{p,x}^*$ commutes with
$M(\zeta^n) \otimes \Id_{L^2(I_q)}\otimes y \otimes \Id_{\cL_{p,x}}$ for all $n \in \Z$ and $y \in B(L^2(I_q))$. Therefore
\[
W_{p,x}^* \in \Big(L^\infty(\T)\otimes \C\, \Id_{L^2(I_q)} \otimes B(L^2(I_q)) \otimes \C\,\Id_{\cL_{p,x}}\Big)' = L^\infty(\T)\otimes B(L^2(I_q)) \otimes \C\,\Id_{L^2(I_q)} \otimes B(\cL_{p,x}),
\]
so we have indeed $W_{p,x} \in M \otimes B(\cL_{p,x})$, by Remark
\ref{rmk:identifyMinGNSspacevp} and the observation recalled
after Definition \ref{def:vNalgMinBH}.
\end{proof}

\begin{lemma} \label{lem:omWpx*=Qpx}
Let $m_1,m_2 \in \Z$, $p_1,p_2,t_1,t_2 \in I_q$, then
\[
(\om_{f_{m_1p_1t_1},f_{m_2p_2t_2}} \otimes \Id)(W_{p,x}^*) = \de_{t_1t_2} Q_{p,x}(p_1,p_2,m_2-m_1).
\]
\end{lemma}
\begin{proof}
Let $m \in \Z$ and $\ep,\et \in \{-,+\}$. From Lemma \ref{lem:explicitWpx} we find
\[
\begin{split}
&(\om_{f_{m_1p_1t_1},f_{m_2p_2t_2}} \otimes \Id)(W_{p,x}^*) e_m^{\ep,\et}(p,x)  = \\
&\quad \de_{t_1t_2} \de_{\chi(p_2/p_1p)+m_1-2m,m_2} C(\sgn(p_1p_2)\ep\et x;m,\ep,\et;p_1,p_2,m_2-m_1)\, e_{m+m_2-m_1}^{\sgn(p_1)\ep, \sgn(p_2)\et}(p,x).
\end{split}
\]
Compare this with Lemma \ref{lem:propertiesQpx}(i), then the result follows.
\end{proof}

\begin{lemma}
$W_{p,x}$ is a corepresentation of $(M,\De)$, i.e.,
\[
(\De \otimes \Id) (W_{p,x}) = (W_{p,x})_{13} (W_{p,x})_{23}.
\]
\end{lemma}

\begin{proof}
We use the structure formula for the $Q_{p,x}$ operators from Lemma \ref{lem:propertiesQpx}.
For $i=1,2$ let $m_i,n_i \in \Z$ and $p_i,r_i,s_i,t_i \in I_q$. Define for $i=1,2$ the elements $f_i,g_i \in \cK$ by $f_i=f_{n_i,p_i,t_i}$ and $g_i=f_{m_i,r_i,s_i}$. By Lemmas \ref{lem:propertiesQpx} and \ref{lem:omWpx*=Qpx} we have
\[
\begin{split}
\big((\om_{f_1,f_2}& \otimes \Id)(W_{p,x}^*)\big)\big((\om_{g_1,g_2} \otimes \Id)(W_{p,x}^*)\big)\\
 &= \de_{t_1t_2}\de_{s_1s_2} Q_{p,x}(p_1,p_2,n)  Q_{p,x}(r_1,r_2,m) \\
&= \de_{t_1t_2}\de_{s_1s_2} \de_{\chi(p_2/p_1),m} \de_{\chi(r_1/r_2),n}\sum_{\substack{x_1,x_2\in I_q\\ \sgn(x_1)=\sgn(p_1r_1)\\ \sgn(x_2)=\sgn(p_2r_2)}} a_{x_1}(r_1,p_1) a_{x_2}(r_2,p_2) Q_{p,x}(x_1,x_2,n+m),
\end{split}
\]
where $n=n_2-n_1$ and $m=m_2-m_1$. Similar as in the proof of Proposition \ref{prop:structureconstQs}, see \S\ref{ssec:commutantofdualvNalg}, it now follows that
\[
\big((\om_{f_1,f_2} \otimes \Id)(W_{p,x}^*)\big)\big((\om_{g_1,g_2} \otimes \Id)(W_{p,x}^*)\big) = \big(\om_{W(g_1\otimes f_1),W(g_2\otimes f_2)}\otimes \Id\big)(1 \otimes W_{p,x}^*),
\]
where $W \in B(\cK \otimes \cK)$ denotes the multiplicative unitary. We rewrite the right hand side as
\[
(\om_{g_1,g_2} \otimes \om_{f_1,f_2} \otimes \Id)\big((W^* \otimes 1)(1 \otimes W_{p,x}^*)(W\otimes 1)\big),
\]
then we conclude that $W^*_{12} (W_{p,x}^*)_{23} W_{12} = (W_{p,x}^*)_{23} (W_{p,x}^*)_{13}$. Using $\De(y) = W^*(1 \otimes y) W$ for $y \in M$, it follows that $(\De \otimes \Id)(W_{p,x}) = (W_{p,x})_{13} (W_{p,x})_{23}$.
\end{proof}

\begin{lemma} \label{lem:Wpxunitary}
$W_{p,x}$ is unitary.
\end{lemma}
\begin{proof}
For $i=1,2$ let $m_i,n_i \in \Z$, $p_i,t_i \in I_q$, $\ep_i,\et_i \in \{-,+\}$.
Using Lemma \ref{lem:explicitWpx} we find
\[
\begin{split}
W_{p,x}&(f_{m_2,p_2,t_2} \otimes e_{n_2}^{\ep_2,\et_2}(p,x)) = \\
&\sum_{p_1 \in I_q} C(\ep\et x;\chi(p_2/p_1p)-n_2,\sgn(p_1)\ep_2,\sgn(p_2)\et_2;p_1,p_2,2n_2-\chi(p_2/p_1p)) \\
& \quad\times f_{\chi(p_2/p_1p)+m_2-2n_2,p_1,t_2} \otimes e_{\chi(p_2/p_1p)-n_2}^{\sgn(p_1)\ep_2,\sgn(p_2)\et_2}(p,x).
\end{split}
\]
Let $W_\Up$ be defined as in the proof of Lemma \ref{lem:explicitWpx}, and for $i=1,2$ let $g_i \in L^2(I(p,n_i,\ep_i,\et_i))$. Using \eqref{eq:TadjointisT} it follows that
\[
\begin{split}
\big\langle f_{m_1p_1t_1} &\otimes g_1,  W_\Up( f_{m_2p_2t_2} \otimes g_2) \big\rangle = \\
 &(-q)^{m_2-m_1} \sgn(p_1)^{\chi(p_1)} \sgn(p_2)^{\chi(p_2)} \Big\langle  f_{-m_2,p_2,t_2}\otimes g_1,W_\Up^*( f_{-m_1,p_1,t_1}\otimes g_2 )\Big\rangle.
\end{split}
\]
In the same way we find from Lemmas \ref{lem:omWpx*=Qpx} and \ref{lem:propertiesQpx}(iii)
\[
\begin{split}
\Big\langle & f_{m_1,p_1,t_1}\otimes e_{n_1}^{\ep_1,\et_1}(p,x), W_{p,x}\big(f_{m_2,p_2,t_2}\otimes e_{n_2}^{\ep_2,\et_2}(p,x)\big) \Big\rangle = \\
& (-q)^{m_2-m_1} \sgn(p_1)^{\chi(p_1)} \sgn(p_2)^{\chi(p_2)} \Big\langle  f_{-m_2,p_2,t_2}\otimes e_{n_1}^{\ep_1,\et_1}(p,x),W_{p,x}^*\big( f_{-m_1,p_1,t_1}\otimes e_{n_2}^{\ep_2,\et_2}(p,x) \big)\Big\rangle.
\end{split}
\]
It is now straightforward to check that $W_{p,x}$ is obtained from $W_\Up$ in the same as $W_{p,x}^*$ from $W_\Up^*$ in the proof of Lemma \ref{lem:explicitWpx}. Then it follows that $W_{p,x}$ is an isometry, as is $W_{p,x}^*$, hence $W_{p,x}$ is unitary.
\end{proof}

It is a direct consequence of the proof of Lemma \ref{lem:explicitWpx} that the corepresentations $W_{p,x}$ occur as principal series in the left regular corepresentation $W$ of $(M,\De)$ as in Proposition \ref{prop:introdecompWprincipalseries}. Let us give the intertwiner explicitly.

First observe that we have
\[
\cK_c(p) = \bigoplus_{\ep,\et \in \{-,+\}} \cK(p,\ep,\et) \cap \cK_c,
\]
where
\[
\cK(p,\ep,\et) = \bigoplus_{m \in \Z} \cK(p,m,\ep,\et).
\]
We define
\[
\begin{split}
I_p^{\ep,\et} : \cK_c(p,\ep,\et) & \to \int_{-1}^1 \ell^2_{\ep,\et}(p,x)\, dx \\
f_m &\mapsto \int_{-1}^1 \big(\Up_{p,m}^{\ep,\et}f_m\big)(\ep\et x)\, e_m^{\ep,\et}(p,x)\, dx,
\end{split}
\]
where $f_m=f_m(p,\ep,\et) \in \cK_c(p,m,\ep,\et)$. The intertwiner $I_p:\cK(p) \to \int_{-1}^1 \cL_{p,x}\, dx$ which implements the equivalence in Proposition \ref{prop:introdecompWprincipalseries} is given by
\[
I_p=\bigoplus_{\ep,\et \in \{-,+\}} I_p^{\ep,\et}.
\]
\begin{remark}
In Section \ref{ssec:decompUq(su11)} the $\su$-representations $\pi_{\cK}(p,\ep,\et)$ on $\cK(p,\ep,\et)$ are decomposed into irreducible $\ast$-representations. Let $\pi_{\cK_c}(p,\ep,\et)$ be the $\su$-representation $\pi_{\cK}(p,\ep,\et)$ restricted to $\cK_c(p,\ep,\et)$. Using the decompositions from Section \ref{ssec:decompUq(su11)}, see Theorems \ref{thm:decompKp+-asUqmod}, \ref{thm:decompKp-+asUqmod}, \ref{thm:decompKp--asUqmod} and \ref{thm:decompKp++asUqmod}, we see that
\[
\pi_{\cK_c}(p,\ep,\et) \cong \int_{-1}^1 \pi_{b(-\ep\et x),\epsilon(p)}\,dx,
\]
where $b(y)$, $y \in [-1,1]$, is the unique number in $[0,-\frac{\pi}{2\ln q}]$ determined by $\mu(q^{2ib(y)}) = y$.  Here we regard $\ell^2_{\ep,\et}(p,x)$ as a $\su$-module as explained in Remark \ref{rem:ell2epetassumodule}. The operators $I_p^{\ep,\et}$ are the precisely the intertwiners for the above equivalence. Here we regard $\ell^2_{\ep,\et}(p,x)$ as $\su$-modules corresponding to the principal series $\pi_{b(-\ep\et x),\epsilon(p)}$ as explained in Remark \ref{rem:ell2epetassumodule}.
\end{remark}

Next we decompose the principal series corepresentations into irreducible corepresentations.
We need the following closed subspaces of $\cL_{p,x}$:
\begin{equation} \label{eq:subspacesL12}
\begin{split}
\cL_{p,x}^1 &= \overline{\mathrm{Span}}\Big\{ e_m^{++}(p,x) + i^{\chi(p)}e_m^{--}(p,x), e_m^{+-}(p,x) - i^{\chi(p)}e_m^{-+}(p,x) \mid m \in \Z \Big\}, \\
\cL_{p,x}^2 & = \overline{\mathrm{Span}}\Big\{ e_m^{++}(p,x) - i^{\chi(p)}e_m^{--}(p,x), e_m^{+-}(p,x) + i^{\chi(p)}e_m^{-+}(p,x) \mid m \in \Z \Big\}.
\end{split}
\end{equation}

\begin{lemma}
The spaces $\cL_{p,x}^j$, $j=1,2$, are orthogonal $W_{p,x}$-invariant subspaces of $\cL_{p,x}$.
\end{lemma}

\begin{proof}
The orthogonality of $\cL_{p,x}^1$ and $\cL_{p,x}^2$ is immediate from their definitions. Using the actions \eqref{eq:principalseries} of the generators of $\hat M$ it is a straightforward exercise to check that $\cL_{p,x}^j$, $j=1,2$, are $W_{p,x}$-invariant.
\end{proof}

For $j=1,2,$ we denote by $W_{p,x}^j$ the restriction of $W_{p,x}$ to the subspace $\cL^j_{p,x}$.

\begin{prop} \label{prop:Wpxforxneq0}
For $x\neq0$ the corepresentations $W_{p,x}^j$, $j=1,2$, are irreducible.
\end{prop}

\begin{proof}
We prove the irreducibility of $W_{p,x}^1$ in case $x \neq 0$, for $W_{p,x}^2$ the proof is similar. Let $L$ be a nonzero closed $W_{p,x}^1$-invariant subspace of $\cL^1_{p,x}$. We choose a nonzero vector $v \in L$. For $k \in \Z$, let $P_k$ denote spectral projection of $K$ onto the eigenspace corresponding to the eigenvalue $q^kp^\hf$, i.e., $P_k$ is the orthogonal projection onto $\mathrm{Span}\{e_k^{\ep,\et}(p,x) \mid \ep,\et \in \{-,+\}\, \}$, see \eqref{eq:principalseries}. We have $v = \sum_{k \in \Z} P_k v$, and since $v \neq 0$, there exists an $m \in \Z$ such that $P_mv \neq 0$. Since $K$ is affiliated to $\hat M$, the projection $P_m$ belongs to $\hat M$, implying $P_mv \in L$. Now let $C_{\pm x}$ denote the spectral projections of the Casimir $\Om$ onto the eigenspaces corresponding to the eigenvalues $\pm x$, then $P_mv = C_xP_mv + C_{-x}P_m v$, so one of the vectors $C_{\pm x}P_mv$ is nonzero. Let us assume $C_xP_mv$ is nonzero, then it is a nonzero multiple of $e_m^{++}(x,p) + i^{\chi(p)}e_m^{--}(x,p)$, and it belongs to $L$ since $C_x \in \hat M_{p,x}$. Applying $U_0^{+-}$ shows that $e_m^{+-}(x,p) - i^{\chi(p)} e_m^{-+}(x,p) \in L$. Finally, applying the isometries in the polar decompositions of $E$ and $E^*$ repeatedly, we find that the vectors $e_k^{++}(x,p) + i^{\chi(p)}e_k^{--}(x,p)$ and $e_k^{+-}(x,p) - i^{\chi(p)} e_k^{-+}(x,p)$ belong to $L$ for any $k \in \Z$, hence $L = \cL_{p,x}^1$. If $C_xP_mv =0$, then $C_{-x}P_mv \neq 0$, and similar arguments show again that $L = \cL_{p,x}^1$.
\end{proof}

In the proof of Proposition \ref{prop:Wpxforxneq0} we used the Casimir operator $\Om$ to distinguish between the spaces $\overline{\mathrm{Span}}\{e_m^{++}(p,x), e_m^{--}(p,x) \mid m \in \Z\}$ and $\overline{\mathrm{Span}}\{e_m^{+-}(p,x), e_m^{-+}(p,x) \mid m \in \Z\}$. For $x=0$ we can no longer do this, because now the restriction of $\Om$ is the zero operator, so it is possible that there are nontrivial irreducible subspaces inside $\cL_{p,0}^1$ and $\cL_{p,0}^2$. We define the following closed subspaces of $\cL_{p,0}$:
\[
\begin{split}
\cL_{p,0}^{1,1} &= \overline{\mathrm{Span}}\Big\{ e_m^{++}(p,0) + i^{\chi(p)}e_m^{--}(p,0) +i^{\chi(p)+1} e_m^{+-}(p,0) + i (-1)^{\chi(p)+1} e_m^{-+}(p,0) \mid m \in \Z \Big\}, \\
\cL_{p,0}^{1,2} &= \overline{\mathrm{Span}}\Big\{ e_m^{++}(p,0) + i^{\chi(p)}e_m^{--}(p,0) -i^{\chi(p)+1} e_m^{+-}(p,0) + i (-1)^{\chi(p)} e_m^{-+}(p,0) \mid m \in \Z \Big\}, \\
\cL_{p,0}^{2,1} & = \overline{\mathrm{Span}}\Big\{ e_m^{++}(p,0) - i^{\chi(p)}e_m^{--}(p,0) +i^{\chi(p)+1} e_m^{+-}(p,0) + i(-1)^{\chi(p)}e_m^{-+}(p,0) \mid m \in \Z \Big\},\\
\cL_{p,0}^{2,2} & = \overline{\mathrm{Span}}\Big\{ e_m^{++}(p,0) - i^{\chi(p)}e_m^{--}(p,0) -i^{\chi(p)+1} e_m^{+-}(p,0) + i(-1)^{\chi(p)+1}e_m^{-+}(p,0) \mid m \in \Z \Big\}.
\end{split}
\]
Observe that $\cL_{p,0}^j = \cL_{p,0}^{j,1} \oplus \cL_{p,0}^{j,2}$, for $j=1,2$.

\begin{prop}\label{prop:decompprincserforx=0}
Assume $x=0$.
\begin{enumerate}[(i)]
\item For $\chi(p)$ odd, i.e., $\epsilon(p)=\hf$, the corepresentations $W_{p,0}^j$, with $j=1,2$, are irreducible.
\item For $\chi(p)$ even, i.e., $\epsilon(p)=0$, the corepresentations $W_{p,0}^{j,k}  = W_{p,0}\vert_{\cL_{p,0}^{j,k}}$, with $j,k \in \{1,2\}$, are irreducible.
\end{enumerate}
\end{prop}

Before proving Proposition \ref{prop:decompprincserforx=0}, we note that
Propositions \ref{prop:principalseriescorepresentation},  \ref{prop:Wpxforxneq0},   \ref{prop:decompprincserforx=0} prove Proposition \ref{prop:introprincipalseriescorepresentation}.

\begin{proof}
We prove the proposition for $j=1$. The case $j=2$ is proved in the same way.

Let $L$ be a nonzero closed $W_{p,0}^1$-invariant subspace of $\cL^1_{p,0}$. In the same way as in the proof of Proposition \ref{prop:Wpxforxneq0} it follows that the vectors
\[
f_m(c_m) = e_{m}^{++} +i^{\chi(p)}e_m^{--} + c_m e_m^{+-} - c_m i^{\chi(p)} e_m^{-+}, \qquad m \in \Z,
\]
are in $L$ for some (yet to be determined) constant $c_m$, and every vector in $L$ can be expanded in terms of the vectors $f_m(c_m)$. Here we use the shorthand notation $e_m^{\ep\et} = e_m^{\ep\et}(p,0)$, $\ep,\et \in \{-,+\}$. Applying $U_0^{-+}$ we find
\[
U_0^{-+}f_m(c_m) = c_m i^{\chi(p)} (-1)^m\Big[ e_m^{++} + i^{\chi(p)} e_m^{--} + (-1)^{\chi(p)+1} c_m^{-1} e_m^{+-} + (-i)^{\chi(p)} c_m^{-1} e_m^{-+}\Big].
\]
Since $U_0^{-+}f_m(c_m)$ must be in $L$, we see that $c_m$ satisfies $c_m = (-1)^{\chi(p)+1} c_m^{-1}$, so that
\[
c_m = \pm i^{\chi(p)+1}.
\]
Let us write $f_m^1 = f_m(i^{\chi(p)+1})$ and $f_m^2 = f_m(-i^{\chi(p)+1})$, then $\cL_{p,0}^{1,j} = \overline{\mathrm{Span}}\{ f_m^j \mid m \in \Z\}$ for $j=1,2$. Observe that $E f_m^j = d_m f_{m+1}^j$ for some constant $d_m$. Let us assume that $f^1_m \in L$. Applying $U_0^{+-}$ to $f_m^1$ gives us
\[
\begin{split}
U_0^{+-} f_m^1 &= (-1)^{\upsilon(p)} \Big[ e_m^{+-} - i^{\chi(p)} e_m^{-+} - i^{\chi(p)+1} e_m^{++} + (-1)^{\chi(p)+1}i e_m^{--}\Big]\\
& =
\begin{cases}
(-1)^{\upsilon(p)+1} i^{\chi(p)+1}f_m^1, & \chi(p)\ \text{even},\\
(-1)^{\upsilon(p)+1} i^{\chi(p)+1}f_m^2, & \chi(p)\ \text{odd},
\end{cases}
\end{split}
\]
so for $\chi(p)$ even we have $L= \cL_{p,0}^{1,1}$, and for $\chi(p)$ odd we have $L= \cL_{p,0}^{1,1} \oplus \cL_{p,0}^{1,2} = \cL^1_{p,0}$. If we assume $f_m^2 \in L$, we find in the same way that $L= \cL_{p,0}^{1,2}$ for $\chi(p)$ even.
\end{proof}

\subsection{Complementary series} \label{ssec:complementaryseries} \index{corepresentation!complementary series}
Let $p \in q^{2\Z}$, i.e., $\epsilon(p)=0$, and let $x=\pm \mu(\la)$ with $\la \in (q,1)$. Note that $x$ is \emph{not} in the spectrum of the Casimir operator in the left regular corepresentation, which is described in Section \ref{ssec:Casimiroperator}. Let $\cL_{p,x} = \bigoplus_{\ep,\et \in \{-,+\}} \ell^2_{\ep,\et}(p,x)$ with orthonormal basis $\{e_n^{\ep,\et}(p,x) \mid n \in \Z,\, \ep,\et \in \{-,+\} \}$, similar as for the principal series corepresentations. We define a unitary corepresentation $W_{p,x}\in M \otimes B(\cL_{p,x})$ by
\[
\begin{split}
W_{p,x}&(f_{m_2,p_2,t_2} \otimes e_{n_2}^{\ep_2,\et_2}(p,x)) = \\
&\sum_{p_1 \in I_q} C(\ep\et x;\chi(p_2/p_1p)-n_2,\sgn(p_1)\ep_2,\sgn(p_2)\et_2;p_1,p_2,2n_2-\chi(p_2/p_1p)) \\
& \quad\times f_{\chi(p_2/p_1p)+m_2-2n_2,p_1,t_2} \otimes e_{\chi(p_2/p_1p)-n_2}^{\sgn(p_1)\ep_2,\sgn(p_2)\et_2}(p,x).
\end{split}
\]
with the function $C$ from Proposition \ref{lem:actionTC(x)}. Initially, the function $C$, as a function of $x$, is only defined on the spectrum of $\Om$, but using the explicit expressions for $A$ and $S$ (see the definition of $C$ in Proposition \ref{lem:actionTC(x)}), we can also define $C$ for $x=\pm \mu(\la)$ with $\la \in (q,1)$. Observe that the denominator of $A(\la;p,m,\ep,\et)$, contains factors with the square root of $(-q^{1-2n}\la/p, -q^{1-2n}/p\la;q^2)_\infty$ for a certain $n\in \Z$. Assume $-\la \in (q,1)$, then this infinite product is positive for $p \in q^{2\Z}$ , but for $p\in q^{2\Z+1}$ it is not. For this reason we require that $p \in q^{2\Z}$ or $\epsilon(p)=0$, see \eqref{eq:epsilon(p)}. This corresponds nicely with the situation for the
principal unitary and complementary series representations of $SU(1,1)$ and $\su$, see 
Section \ref{ssec:decompUq(su11)}. 
Formally the above defined corepresentation corresponds to the definition of the principal series corepresentation $W_{p,x}$ from Lemma \ref{lem:explicitWpx}. In particular, the actions of the generators of $\hat M$ on the basisvectors $e_m^{\ep,\et}(p,x)$ are given by
\[
\begin{split}
K \, e_m^{\ep,\et}(p,x) &= p^\hf q^m \, e_m^{\ep,\et}(p,x),\\
(q^{-1}-q) E\, e_m^{\ep,\et}(p,x) & = q^{-m-\hf}p^{-\hf} \sqrt{1+2\ep\et x p q^{2m+1} + q^{4m+2}p^2} \, e_{m+1}^{\ep,\et}(p,x),\\
U_{0}^{+-}\, e_m^{\ep,\et}(p,x) & = \et\, (-1)^{\upsilon(p)} \, e_m^{\ep,-\et}(p,x),\\
U_{0}^{-+}\, e_m^{\ep,\et}(p,x) & = \ep \et^{\chi(p)}(-1)^m \, e_m^{-\ep,\et}(p,x).
\end{split}
\]

We call $W_{p,x}$ the complementary series corepresentation of $(M,\De)$. In order to show that this is indeed a unitary operator, we need to find orthogonality relations and dual orthogonality relations for the functions $C$ in case $x=\pm \mu(\la)$ with $\la \in (q,1)$. These relations are obtained in Corollary \ref{cor:orthogonalityC} from the orthogonality relations for the function $C$ by analytic continuation. The fact that $W_{p,x}$ is indeed a corepresentation is proved along the same lines as for the principal series corepresentations. Here we need to show that the product identity from Lemma \ref{lem:identitiesC} remains valid for $x=\pm \mu(\la)$ with $\la \in (q,1)$. This is done in Lemma \ref{lem:analyticcontinuation}.

Finally, in the same way as Proposition \ref{prop:Wpxforxneq0} it can be proved that the subcorepresentations $W_{p,x}^j = W_{p,x}|_{\cL_{p,x}^j}$, with the subspaces $\cL_{p,x}^j$, $j=1,2$, defined as in \eqref{eq:subspacesL12}, are irreducible.


\section{Identities for special functions} \label{sec:specialfunctions}

\subsection{Summation formulas from the action of $Q(p_1,p_2,n)$} \label{ssec:summationformula}
We start by proving the summation formulas in Section \ref{ssec:formulas2phi1}, which
essentially follow by the action of $Q(p_1,p_2,n)$ with respect to the spectral decomposition
of the Casimir operator.

\begin{proof}[Proof of Theorem \ref{thm:introsummationformula1}]
In Lemma \ref{lem:actionTC(x)} we computed how the operator $T(p_1,p_2,n)$ defined by \eqref{eq:defTp1p2n} acts on functions in $L^2(I(p,m,\ep,\et))$. In this computation we actually proved a summation formula involving the functions $a_p(z,w)$ and the orthonormal functions $g_z(x;p,m,\ep,\et)$, which are essentially Al-Salam--Chihara polynomials and little $q^2$-Jacobi functions. Here we write out explicitly, i.e., in terms of basic hypergeometric functions, the summation formula corresponding to the case $\ep=\et=+$; in this case both $g_z$-functions appearing in the formula are little $q$-Jacobi functions, i.e., non-terminating $_2\varphi_1$-functions. This is a rather tedious, but straightforward computation. The second
case follows similarly using $\ep=\et=-$.
\end{proof}

The product formula from Lemma \ref{lem:identitiesC} leads to the summation formula
in Theorem \ref{thm:introsumfromstructureformula} with the same structure as the formula from Theorem \ref{thm:introsummationformula1}.

\begin{proof}[Proof of Theorem \ref{thm:introsumfromstructureformula}]
We first write the formula from Lemma \ref{lem:identitiesC} in terms of the $S$-functions:
\begin{equation} \label{eq:productformulaS}
\begin{split}
\sgn(r_1)^{\hf(1-\sgn(p_1))} \sgn(r_2)^{\hf(1-\sgn(p_2))+n} S(\sgn(r_1r_2)\la;p_1,p_2,n) S(\la;r_1,r_2,m) = \\
\sum_{(x_1,x_2) \in \mathcal A} a_{x_1}(r_1,p_1) a_{x_2}(r_2,p_2) S(\la;x_1,x_2,m+n),
\end{split}
\end{equation}
where we denoted $y=-\sgn(p_1p_2r_1r_2)\mu(\la)$ and canceled common factors. Now the result follows from expressing the $S$-functions as $_2\varphi_1$-functions and the $a_x$-functions as $\Psi$-functions, see Proposition \ref{prop:S=2phi1} and Definition \ref{def:functionap}.
\end{proof}

Observe that the sum in Theorem \ref{thm:introsumfromstructureformula} is actually a single sum. If we denote $x_1=\sgn(p_1r_1)q^k$, then $x_2=\sgn(p_2r_2)q^k$, and we can write the above sum as a sum over $k$ where $k \in \N$ if $\sgn(p_1r_1)=-$ or $\sgn(p_2r_2)=-$, and $k \in \Z$ if $\sgn(p_1r_1)=\sgn(p_2r_2)=+$. Furthermore, if we set $r_1=p_1$, $r_2=p_2$ (this implies $m=-n$), and we use the second symmetry relation from Lemma \ref{lem:symmetryS} for the function $S$, the left hand side in \eqref{eq:productformulaS} contains the product $S(\la^{-1};p_1,p_2,n) S(\la;p_1,p_2,n)$, which is positive (this corresponds to the operator $Q(p_1,p_2,n)^* Q(p_1,p_2,n)$). So we find
\[
(-1)^{\chi(p_2/p_1)} \sgn(p_1)^{\chi(p_1)} \sgn(p_2)^{\chi(p_2)} \sum_{(x_1,x_2) \in \mathcal A} a_{x_1}(p_1,p_1) a_{x_2}(p_2,p_2) S(\la;x_1,x_2,0) > 0.
\]
This leads to Corollary \ref{cor:positivity}.

For the definition of the complementary series corepresentations of $(M,\De)$ in Section \ref{ssec:complementaryseries}, the following lemma is crucial for showing that it is indeed a corepresentation.

\begin{lemma} \label{lem:analyticcontinuation}
The identity from Theorem \ref{thm:introsumfromstructureformula} is also valid for $\la \in (q,1)$. Consequently, if $n+m \in 2\Z$, the product formula for the function $C$ in Lemma \ref{lem:identitiesC}(i) also holds for $y=\mu(\la)$ with $\la \in (q,1)$,
\end{lemma}

\begin{proof}
First we prove that the identity from Theorem \ref{thm:introsumfromstructureformula}, or equivalently \eqref{eq:productformulaS}, is also valid for $\la \in (q,1)$. Actually, we prove a stronger result: the identity \eqref{eq:productformulaS} holds for all $\la \in \C \setminus\{0\}$. Recall that, for $p_1,p_2 \in I_q$ and $n \in \Z$, the function $S(\cdot;p_1,p_2,n)$ is analytic on $\C\setminus \{0\}$. So clearly the left hand side of \eqref{eq:productformulaS} is analytic in $\la$ on $\C\setminus \{0\}$. We show that the right hand side of \eqref{eq:productformulaS} also defines an analytic function, then the result follows from analytic continuation.

Let $r_1,r_2,p_1,p_2 \in I_q$ and $n,m \in \Z$. Assume $\la \in \mathrm{K} \subset \C \setminus\{0\}$ where $\mathrm K$ is a compact set, then there exists a constant $B>0$ such that $|\la| < B$ and $|\la^{-1}| < B$. We define $g:q^\Z\to \mathbb R$ by
\[
g(x) =
\begin{cases}
q^{3\chi(x)}B^{\chi(x)}, & x \leq q,\\
q^{-(n+m)\chi(x)}, & x \geq 1.
\end{cases}
\]
By Lemma \ref{lem:asymptoticsforS} there exists a constant $C>0$ such that
\[
|S(\la;x_1,x_2,n+m)| < C g(|x_1|),
\]
for $(x_1,x_2) \in \mathcal A$. Furthermore, by Lemma \ref{lem:estimate1} there exists a constant $D'>0$ such that
\[
|a_x(r,p)| \leq D' q^{\chi(x)[\chi(r/p)-\frac32]} q^{\hf \chi(x)^2}, \qquad x,r,p \in I_q,
\]
so that
\[
|a_{x_1}(r_1,p_1)a_{\sgn(r_2p_2)|x_1|}(r_2,p_2)|  \leq D q^{\chi(x_1)[\chi(r_1r_2/p_1p_2)-3]}q^{\chi(x_1)^2},
\]
for some constant $D>0$. Because of the factor $q^{\chi(x)^2}$ the sum
\[
\sum_{x} q^{\chi(x)[\chi(r_1r_2/p_1p_2)-3]}q^{\chi(x)^2} g(x)
\]
converges absolutely. Here the sum is over $x\in q^\N$ in case $\sgn(r_1p_1)=-$ or $\sgn(r_2p_2)=-$, and over $x\in q^\Z$ in case $\sgn(r_1p_1)=\sgn(r_2p_2)=+$. It follows that the right hand side of \eqref{eq:productformulaS} converges uniformly on $\mathrm K$, hence it is analytic on $\mathrm K$.

Finally, let  $\pm \la \in (q,1)$. We multiply \eqref{eq:productformulaS} by
\[
\ep^{\hf(1-\sgn(p_1r_1))} \et^{\hf(1-\sgn(p_2r_2))}\frac{ A(-\sgn(p_1r_1p_2r_2)\la;p,k+m+n, \sgn(p_1r_1)\ep, \sgn(p_2r_2)\et) }{ A(-\la;p,k,\ep,\et) }
\]
then we obtain the desired product formula for the function $C$ as in Lemma \ref{lem:identitiesC}(i), with $y=-\sgn(p_1r_1p_2r_2)\mu(\la)$.
\end{proof}
Let us remark that with the same arguments as in the proof of Lemma \ref{lem:analyticcontinuation} it follows that the product formula for the function $C$ holds for all $\la \in (0,1)\setminus q^{\NN}$ if $n+m \in 2\Z$. Note that the points $\la \in q^{\NN}$ correspond to discrete series corepresentations, and at these points the product formula is of course also valid. In this case the functions $A$ are essentially square roots of residues of $c$-functions.

\subsection{Biorthogonality relations}
In the proof of Lemma \ref{lem:explicitWpx} we obtained orthogonality relations for the function $C$ for the case $x=\mu(\la)$ with $\la \in \T$. These relations lead to biorthogonality relations for the $S$ functions, which by analytic continuation hold for all $\la \in \C\setminus\{0\}$. We need these biorthogonality relations for $\pm \la \in (q,1)$ in order to show that the complementary series corepresentations are unitary.
\begin{lemma} \label{lem:biorthogonalityS}
Let $\la \in \C\setminus\{0\}$ and $m \in \Z$. The set
\[
\{p_2\mapsto S(\sgn(p_1)\la;p_1,p_2,\chi(p_1p_2)+m) \mid p_1 \in I_q\}
\]
is basis for $\ell^2(I_q)$ with dual basis
\[
\{p_2\mapsto S(\sgn(p_1)\overline{\la}^{-1};p_1,p_2,\chi(p_1p_2)+m) \mid p_1 \in I_q\}.
\]
Similarly, the set
\[
\{p_1\mapsto S(\sgn(p_1)\la;p_1,p_2,\chi(p_1p_2)+m) \mid p_2 \in I_q\}
\]
is a basis for $\ell^2(I_q)$ with dual basis
\[
\{p_1\mapsto S(\sgn(p_1)\overline{\la}^{-1};p_1,p_2,\chi(p_1p_2)+m) \mid p_2 \in I_q\}.
\]
\end{lemma}
\begin{proof}
First assume $x=\mu(\la)$ with $\la \in \T_0$. From unitarity of $W_{p,x}$ and the explicit expressions for $W_{p,x}^*$ and $W_{p,x}$, we obtain orthogonality and dual orthogonality relations for the matrix elements $C$. Indeed, from writing out $W_{p,x} W_{p,x}^*[f_{n_1p_1t_1} \otimes e_{m-\chi(p_1)}^{\sgn(p_1)\ep,\et}(p,x)]=f_{n_1p_1t_1} \otimes e_{m-\chi(p_1)}^{\sgn(p_1)\ep,\et}(p,x)$ we find, for $p_1'\in I_q$ and $y = \ep\et\,x$,
\[
\begin{split}
\de_{p_1p_1'}  =
\sum_{p_2\in I_q}  & C(\sgn(p_2)y;m-\chi(p_1),\sgn(p_1)\ep,\et;p_1,p_2,\chi(p_2p_1/p)-2m)\\
& \times C(\sgn(p_2)y;m-\chi(p_1'),\sgn(p_1')\ep,\et;p_1',p_2,\chi(p_2p_1'/p)-2m),
\end{split}
\]
and from writing out $W_{p,x}^* W_{p,x}[f_{m_2p_2t_2} \otimes e_{\chi(p_2/p)-m}^{\ep,\sgn(p_2)\et}(p,x)]=f_{m_2p_2t_2} \otimes e_{\chi(p_2/p)-m}^{\ep,\sgn(p_2)\et}(p,x)$ we find, for $p_2' \in I_q$ and $y=\ep\et x$,
\[
\begin{split}
\de_{p_2p_2'}  =
\sum_{p_1\in I_q} &C(\sgn(p_2)y;m-\chi(p_1),\sgn(p_1)\ep,\et;p_1,p_2,\chi(p_1p_2/p)-2m)\\
 & \times C(\sgn(p_2')y;m-\chi(p_1),\sgn(p_1)\ep,\et;p_1,p_2',\chi(p_1p_2'/p)-2m).
\end{split}
\]
Expressing the functions $C$ in terms of the functions $S$, see Lemma \ref{lem:actionTC(x)}, the first orthogonality relation gives, for $\la \in \T$,
\[
\begin{split}
\de_{p_1p_1'} = \sum_{p_2 \in I_q} &\frac{ A(\sgn(p_2)\la;p, \chi(p_2/p)-m, \ep,\sgn(p_2)\et) A(\sgn(p_2)\la^{-1};p, \chi(p_2/p)-m, \ep,\sgn(p_2)\et)}{ A(\sgn(p_1)\la;p, m-\chi(p_1), \sgn(p_1)\ep,\et)A(\sgn(p_1')\la^{-1};p, m-\chi(p_1'), \sgn(p_1')\ep,\et)} \\
& \times S(-\sgn(p_1)\la;p_1,p_2,\chi(p_1p_2/p)-2m) S(-\sgn(p_1')\la^{-1};p_1',p_2,\chi(p_1p_2/p)-2m).
\end{split}
\]
We use $A(\la)A(\la^{-1})=|A(\la)|^2=1$, then we obtain
\[
\begin{split}
\de_{p_1p_1'} = \sum_{p_2 \in I_q}  S(-\sgn(p_1)\la;p_1,p_2,\chi(p_1p_2/p)-2m) S(-\sgn(p_1')\la^{-1};p_1',p_2,\chi(p_1'p_2/p)-2m).
\end{split}
\]
From Lemma \ref{lem:asymptoticsforS}(iii) and (iv) it follows this sum converges uniformly in $\la$ on any compact set of $\C\setminus\{0\}$. Since the function $S$ is analytic for $\la \in \C\setminus\{0\}$, by analytic continuation the orthogonality relations are valid for all $\la \in \C\setminus\{0\}$. In the same way we find from the second orthogonality relations for the functions $C$, for $\la \in \C\setminus \{0\}$,
\[
\de_{p_2p_2'} = \sum_{p_1 \in I_q}  S(-\sgn(p_1)\la;p_1,p_2,\chi(p_1p_2/p)-2m) S(-\sgn(p_1)\la^{-1};p_1,p_2',\chi(p_1p_2'/p)-2m)
\]
In order to show uniform convergence here, we also need the third symmetry relation for $S$ from Lemma \ref{lem:symmetryS}. Now replace $-\la$ by $\la$, and $-\chi(p)-2m$ by $m$, then we have biorthogonality relations in $\ell^2(I_q)$ for the functions $S(\sgn(p_1)\la;p_1,p_2,\chi(p_1p_2)+m)$ with respect to $p_1$ and $p_2$,  which implies that they form a basis for $\ell^2(I_q)$.
\end{proof}
The biorthogonality relations in Theorem \ref{thm:biorthogonality} follow from Lemma \ref{lem:biorthogonalityS} using
\[
s(p_1,p_2;\la,m) = S(\sgn(p_1)\la;p_1,p_2,\chi(p_1p_2)+m).
\]

\begin{remark}
By the third symmetry relation for $S$ from Lemma \ref{lem:symmetryS} the two biorthogonality relations for $S$ from Lemma \ref{lem:biorthogonalityS} are actually equivalent. It is also useful to observe that for $\la \in \T$ the biorthogonality relations are orthogonality relations.
\end{remark}

To prove unitarity for the complementary series corepresentations we need to write the biorthogonality relations from Lemma \ref{lem:biorthogonalityS} in case $\pm \la \in (q,1)$, as orthogonality relations for the functions $C$.
\begin{cor} \label{cor:orthogonalityC}
For $m\in\Z$, $p \in q^{2\Z}$, $\ep,\et\in \{-,+\}$, and $y=\pm \mu(\la)$ with $\la \in (q,1)$, the following orthogonality relations hold:
\[
\begin{split}
\de_{p_1p_1'}  =
\sum_{p_2\in I_q}  & C(\sgn(p_2)y;m-\chi(p_1),\sgn(p_1)\ep,\et;p_1,p_2,\chi(p_2p_1/p)-2m)\\
& \times C(\sgn(p_2)y;m-\chi(p_1'),\sgn(p_1')\ep,\et;p_1',p_2,\chi(p_2p_1'/p)-2m),
\end{split}
\]
\[
\begin{split}
\de_{p_2p_2'}  =
\sum_{p_1\in I_q} &C(\sgn(p_2)y;m-\chi(p_1),\sgn(p_1)\ep,\et;p_1,p_2,\chi(p_1p_2/p)-2m)\\
 & \times C(\sgn(p_2')y;m-\chi(p_1),\sgn(p_1)\ep,\et;p_1,p_2',\chi(p_1p_2'/p)-2m).
\end{split}
\]
\end{cor}
\begin{proof}
This follows from Lemma \ref{lem:biorthogonalityS} and the observations that 
\[
C(\sgn(p_2)y;m-\chi(p_1),\sgn(p_1)\ep,\et;p_1,p_2,\chi(p_2p_1/p)-2m) = g(\la) S(-\la;p_1,p_2,\chi(p_1p_2)+m),
\]
where $g(\la)$ is given by
\[
g(\la) = \sgn(p_1)\ep^{\hf(1-\sgn(p_1)}\et^{\hf(1-\sgn(p_2)+\chi(p_2p_1)}\frac{ A(\sgn(p_2)\la;p, \chi(p_2/p)-m, \ep,\sgn(p_2)\et)}{ A(\sgn(p_1)\la;p, m-\chi(p_1), \sgn(p_1)\ep,\et)}
\]
and from $A(\la)A(\la^{-1})=1$, which follows from the definitions of $A$, see \S\ref{ssecB:formulasforA}.
\end{proof}

\subsection{Proof of the summation and transformation theorems}
In this subsection we prove Theorems \ref{thm:symmetryforsinglesumfourapxys} and  \ref{thm:doublesumfiveapxysequalsthreeapxys}. The theorems are reflections of
the structure constants for the product in $\hat{M}$, see Proposition
\ref{prop:structureconstQs}, and of the coproduct $\hat{\De}$ of the dual
quantum group acting on $Q(p_1,p_2,n)$, see Proposition \ref{prop:hatDeonQppn}.
Inspection of the proofs, see Section \ref{ssec:commutantofdualvNalg}, shows that both results follow from the
pentagonal equation $W_{12}W_{13}W_{23}=W_{23}W_{12}$ for the multiplicative unitary $W$.
However, as remarked in Remark \ref{rmk:doublesumfiveapxysequalsthreeapxys}, the results in
Theorems \ref{thm:symmetryforsinglesumfourapxys} and  \ref{thm:doublesumfiveapxysequalsthreeapxys}
cannot be obtained from each other.

\begin{proof}[Proof of Theorem \ref{thm:symmetryforsinglesumfourapxys}]
We start with the result of Proposition \ref{prop:structureconstQs}
and we next let the corresponding operator identity act on
$f_{-l, \ep\et pq^lz,z}\in \cK(p,l,\ep,\et)$. Lemma \ref{lem:explicitactionQppn}
shows that
\[
Q(p_1,p_2,n)\, Q(r_1,r_2,m) \colon \cK(p,l,\ep,\et)
\to \cK(p,l+m+n, \sgn(r_1p_1)\ep, \sgn(r_2p_2)\et)
\]
is non-zero precisely if $q^{2l}p=q^{-m}|\frac{r_2}{r_1}|$ and
$q^{2l+2m}p=q^{-n}|\frac{p_2}{p_1}|$. In particular,
in case $q^{-n}|\frac{p_2}{p_1}| \not= q^{m}|\frac{r_2}{r_1}|$
we find
$Q(p_1,p_2,n)\, Q(r_1,r_2,m)=0$.

In order to calculate the appropriate matrix coefficient we proceed
for \par\noindent $f_{-l-m-n, \sgn(r_1p_1r_2p_2)\ep\et pq^{l+m+n}w,w}
\in \cK(p,l+m+n, \sgn(r_1p_1)\ep, \sgn(r_2p_2)\et)$ as
\[
\begin{split}
&\langle Q(p_1,p_2,n)\, Q(r_1,r_2,m)\,f_{-l, \ep\et pq^lz,z},
f_{-l-m-n, \sgn(r_1p_1r_2p_2)\ep\et pq^{l+m+n}w,w}\rangle = \\
&\sum_{u\in J(p,l+m,\ep\sgn(r_1),\et\sgn(r_2))}
\langle Q(r_1,r_2,m)\,f_{-l, \ep\et pq^lz,z},
f_{-l-m, \sgn(r_1r_2)\ep\et pq^{l+m}u,u}\rangle \,  \\
&\qquad\qquad
\times \langle Q(p_1,p_2,n)\,f_{-l-m, \sgn(r_1r_2)\ep\et pq^{l+m}u,u},
f_{-l-m-n, \sgn(r_1p_1r_2p_2)\ep\et pq^{l+m+n}w,w}\rangle
\end{split}
\]
using the orthogonal basis for the intermediate space
$\cK(p,l+m,\ep\sgn(r_1),\et\sgn(r_2))$.
In this sum we can use \eqref{eq:matrixelementsofQppnandhatJQppnhatJ} twice,
and using \eqref{eq:defKpmepeta} we find that this equals
\begin{equation}\label{eq:pfthmsymmetryforsinglesumfourapxys1}
\begin{split}
&\de_{|\frac{r_1}{r_2}|p, q^{-2l-m}} \, \de_{|\frac{p_1}{p_2}|p, q^{-2l-2m-n}}
\sum_{\stackrel{\scriptscriptstyle{u\in I_q\text{ so that } \sgn(u)=\sgn(r_1)\ep}}
{\scriptscriptstyle{\text{and }\ep\et \sgn(r_1r_2)pq^{l+m}u\in I_q}}}
\left| \frac{z}{w}\right|\, a_z(r_1,u)\, a_u(p_1,w)\, \\
& \qquad \times \, a_{\ep\et pq^lz}(r_2, \ep\et \sgn(r_1r_2)pq^{l+m}u)
\, a_{\ep\et \sgn(r_1r_2)pq^{l+m}u}(p_2, \sgn(r_1p_1r_2p_2)\ep\et pq^{l+m+n}w)
\end{split}
\end{equation}

Next observe
\[
Q(x_1,x_2,m+n) \colon \cK(p,l,\ep,\et) \to \cK(p,l+m+n, \sgn(x_1)\ep, \sgn(x_2)\et)
\]
is non-zero only if $q^{2l}p = q^{-m-n}|\frac{x_2}{x_1}|$, so that the double sum
in Proposition \ref{prop:structureconstQs} reduces to a single sum. Moreover,
by Definition \ref{def:functionap} shows that in the sum the functions $a_{x_i}(r_i,p_i)$
for $i=1,2$ are non-zero only if $\sgn(x_i)=\sgn(r_ip_i)$ for $i=1,2$. So the
matrix element for the expression on the right hand side is
\[
\sum_{x_1,x_2 \in I_q} a_{x_1}(r_1,p_1) \, a_{x_2}(r_2,p_2)\,
\langle Q(x_1,x_2,m+n)\,f_{-l, \ep\et pq^lz,z},
f_{-l-m-n, \sgn(r_1p_1r_2p_2)\ep\et pq^{l+m+n}w,w}\rangle
\]
and this reduces to a single sum and the summand is evaluated by
\eqref{eq:matrixelementsofQppnandhatJQppnhatJ}. By eliminating $x_2$ and
renaming $x_1$ by $x$ we see that this equals
\begin{equation}\label{eq:pfthmsymmetryforsinglesumfourapxys2}
\begin{split}
&\sum_{\stackrel{\scriptscriptstyle{x\in I_q\text{ so that } \sgn(x)=\sgn(r_1p_1)}}
{\scriptscriptstyle{\text{and }|x|\sgn(r_2p_2)pq^{2l+m+n}\in I_q}}}
\left|\frac{z}{w}\right| a_x(r_1,p_1)\, a_z(x,w)\,
a_{|x|\sgn(r_2p_2)pq^{2l+m+n}}(r_2,p_2)\\
&\qquad\qquad \times\, a_{\ep\et pq^lz}(|x|\sgn(r_2p_2)pq^{2l+m+n},
\sgn(r_1p_1r_2p_2)\ep\et pq^{l+m+n}w)
\end{split}
\end{equation}

Finally, equating \eqref{eq:pfthmsymmetryforsinglesumfourapxys1}
and \eqref{eq:pfthmsymmetryforsinglesumfourapxys2} gives the result, where
the conditions on the parameters in Theorem \ref{thm:symmetryforsinglesumfourapxys}
follows from the fact that the matrix elements are taken with respect to
vectors in the GNS Hilbert space.
\end{proof}

\begin{proof}[Proof of Corollary \ref{cor:thmsymmetryforsinglesumfourapxys}.]
Observe that by Lemma
\ref{lem:QadjointisQ}
\[
Q(r_1,r_2,-m)\, Q(r_1,r_2,m) = (-q)^{-m} \sgn(r_1)^{\chi(r_1)}\,
\sgn(r_2)^{\chi(r_2)}\, Q(r_1,r_2,m)^\ast\, Q(r_1,r_2,m)
\]
so that
\begin{equation}\label{eq:corthmsymmetryforsinglesumfourapxys1}
\begin{split}
& (-q)^{m} \sgn(r_1)^{\chi(r_1)}\,
\sgn(r_2)^{\chi(r_2)}\,  \langle Q(r_1,r_2,-m)\, Q(r_1,r_2,m)\, f_{-l,\ep\et pq^lz,z}, f_{-l,\ep\et pq^lz,z}\rangle\\
=\, &\, \|Q(r_1,r_2,m)\, f_{-l,\ep\et pq^lz,z}\|^2
 =\, \frac{(r_1 r_2)^2}{q^{2l}z^2 p^2}\ \\
&\,  \sum_{w \in J(p,m+l,\ep\sgn(r_1),\et\sgn(r_2))}\,
\frac{1}{|w|^2}\,\bigl(a_{r_1}(z,w)\bigr)^2
\,\bigl(a_{r_2}(\ep\et\,q^l p\,z,\ep\et\sgn(r_1r_2)\,q^{m+l}
p\,w)\bigr)^2
\end{split}
\end{equation}
by Lemma \ref{lem:explicitactionQppn}, see in particular \eqref{eq:defQonbasisvector}.
We are interested in the case $Q(r_1,r_2,m)\, f_{-l,\ep\et pq^lz,z}\not= 0$, so we
assume $q^{2l}p = q^{-m}|\frac{r_2}{r_1}|$. The case that this sum can equal zero,
is already covered by Theorem \ref{thm:symmetryforsinglesumfourapxys}.
Since the right hand side is obviously positive, and the left hand side is (up to
the factor in front) equal to \eqref{eq:pfthmsymmetryforsinglesumfourapxys2}
with $p_1,p_2,n,w$ replaced by $r_1,r_2,-m,z$. Since we assume
$q^{2l}p = q^{-m}|\frac{r_2}{r_1}|$ we replace $p$ by $q^{-m-2l}|\frac{r_2}{r_1}|$, and
moreover, we use the third symmetry of \eqref{eq:symmetryforapxy} twice, to find
\begin{equation}
\begin{split}
&\qquad\qquad (-q)^{m} \sgn(r_1)^{\chi(r_1)}\,
\sgn(r_2)^{\chi(r_2)}\, (-1)^l \, \ep^{\chi(z)} \et^{l+m+\chi(zr_2/r_1)} \frac{q^l}{z^2} \\
&\sum_{\stackrel{\scriptscriptstyle{x\in q^\Z\text{ so that }}}
{\scriptscriptstyle{xq^{-m}|\frac{r_2}{r_1}|\in q^\Z}}} x^2\, a_x(r_1,r_1)\, a_x(z,z)
\, a_{xq^{-m}|\frac{r_2}{r_1}|}(r_2,r_2)\, a_{xq^{-m}|\frac{r_2}{r_1}|}(\ep\et |\frac{r_2}{r_1}|q^{-m-l}z,\ep\et |\frac{r_2}{r_1}|q^{-m-l}z)\\ &\qquad\quad =\, \|Q(r_1,r_2,m)\, f_{-l,\ep\et |\frac{r_2}{r_1}|q^{-m-l}z,z}\|^2 >0
\end{split}
\end{equation}
where the right hand side can be evaluated explicitly as a sum of squares
by \eqref{eq:corthmsymmetryforsinglesumfourapxys1} with $p$ replaced by
$q^{-m-2l}|\frac{r_2}{r_1}|$. This proves the general statement of
Corollary \ref{cor:thmsymmetryforsinglesumfourapxys} since the condition on the summation
parameter $x$ is always satisfied.

For the final statement on $q$-Laguerre polynomials we observe that for $\sgn(p)=+$ and
$\sgn(y)=-$, or $y=-q^{1+k}$, $k\in\NN$, we have from Definition
\ref{def:functionap} and \eqref{eq:defqLaguerrepols}
\[
a_p(-q^{1+k},-q^{1+k}) = c_q\, (-1)^{\chi(p)}\, q^{1+k} \nu(p)\, \sqrt{(-p^2;q^2)_\infty}
\, L^{(0)}_k(q^2p^{-2};q^2).
\]
So we choose $r_1=-q^{1+a}$, $r_2=-q^{1+b}$, $z=-q^{1+c}$,
$\ep\et |\frac{r_2}{r_1}|q^{-m-l}z=-q^{1+d}$
with $a,b,c,d\in\NN$, so we replace $l$ by $c+b-a-m-d$ and take $\ep=-$, $\et=-$.
We replace $m$ by $b-a-e$ with $e\in\Z$, discard the positive $x$-independent terms and find
\[
\begin{split}
\sum_{x\in q^\Z} &x^2\, \nu(x)^2\, \nu(xq^e)^2\, (-x^2, -x^2q^{2e};q^2)_\infty
\, \\& L^{(0)}_a(q^2x^{-2};q^2)\, L^{(0)}_c(q^2x^{-2};q^2)\, L^{(0)}_b(x^{-2}q^{2-2e};q^2)\,
L^{(0)}_d(x^{-2}q^{2-2e};q^2)\,  > 0.
\end{split}
\]
Now putting $x=q^{1-k}$, $k\in\Z$, and using the theta-product identity \eqref{eq:thetaprodid} twice
and not taking into account the $k$-independent positive terms we find
\[
\sum_{k\in \Z} \frac{q^{2k}}{(-q^{2k}, -q^{2k-2e};q^2)_\infty}
L^{(0)}_a(q^{2k};q^2)\, L^{(0)}_c(q^{2k};q^2)\, L^{(0)}_b(q^{2k-2e};q^2)\,
L^{(0)}_d(q^{2k-2e};q^2)\,  > 0.
\]
Relabeling and switching to base $q$ proves the required statement.
\end{proof}

\begin{proof}[Proof of Theorem \ref{thm:doublesumfiveapxysequalsthreeapxys}]
For the proof it is easier to start by conjugating the result of
Proposition \ref{prop:hatDeonQppn} with the flip operator
to obtain
\begin{equation}\label{eq:pfthmdoublesumfiveapxysequalsthreeapxys1}
\sum_{p\in I_q,\, m\in \Z} Q(p_1,p,m)\, \otimes\, Q(p,p_2,n-m) \, =\,
W\, \bigl( Q(p_1,p_2,n)\ot \Id)\, W^\ast,
\end{equation}
which is a consequence of the proof of Proposition \ref{prop:hatDeonQppn}.
We let both sides act on
\[
f_{-m_1,\ep_1\et_1q^{m_1}r_1z_1,z_1}\ot f_{-m_2,\ep_2\et_2q^{m_2}r_2z_2,z_2}
\in \cK(r_1,m_1,\ep_1\et_1)\ot \cK(r_2,m_2,\ep_2\et_2)
\]
and we take inner products with
\[
\begin{split}
& f_{-m_1-M,\si \sgn(p_1)\ep_1\et_1q^{m_1+M}r_1w_1,w_1} \ot
f_{-m_2-n+M,\si \sgn(p_2)\ep_2\et_2q^{m_2+n-M}r_2w_2,w_2} \\ &\qquad\qquad
\in \cK(r_1,m_1+M,\sgn(p_1)\ep_1,\si\et_1)\ot
\cK(r_2,m_2+n-M,\si\ep_2,\sgn(p_2)\et_2).
\end{split}
\]
Then the sum over $I_q$ and $\Z$ reduces to a single term by
a double application of \eqref{eq:matrixelementsofQppnandhatJQppnhatJ}. Indeed, we
find that we need $m=M$ and $\sgn(p)=\si$ for a non-zero contribution, but also
both the conditions $q^{2m_1+M}=|\frac{p}{p_1r_1}|$ and $q^{2m_2+n-M}=|\frac{p_2}{pr_2}|$
need to be satisfied. So for the matrix element of the left hand side of
\eqref{eq:pfthmdoublesumfiveapxysequalsthreeapxys1} to have a single non-zero
term we require $r_1r_2q^{2m_1+2m_2}=q^{-n}|\frac{p_2}{p_1}|$, and in this
case the left hand side equals
\begin{equation}\label{eq:pfthmdoublesumfiveapxysequalsthreeapxys2}
\begin{split}
&\left| \frac{z_1z_2}{w_1w_2}\right| a_{z_1}(p_1,w_1)
\, a_{\ep_1\et_1q^{m_1}r_1z_1}(\si|p_1|r_1q^{2m_1+M},\ep_1\et_1w_1\si\sgn(p_1)r_1q^{m_1+M})
\\ & \times\, a_{z_2}(\si |p_1|r_1q^{2m_1+M},w_2)
\, a_{\ep_2\et_2q^{-2m_1-m_2-n}\frac{|p_2|z_2}{|p_1|r_1}}
(p_2, \ep_2\et_2\si q^{-2m_1-m_2-M}\frac{w_2p_2}{|p_1|r_1})
\end{split}
\end{equation}
where we have chosen to eliminate $r_2$. Here all arguments of the function
$a_p(x,y)$ are indeed elements of $I_q$, except possible $\si |p_1|r_1q^{2m_1+M}$ and
in case $\si |p_1|r_1q^{2m_1+M}\notin I_q$ the expression has to be read as zero.

In order to calculate the same matrix element for the right hand side of
\eqref{eq:pfthmdoublesumfiveapxysequalsthreeapxys1} we rewrite this matrix element as
\begin{equation}\label{eq:pfthmdoublesumfiveapxysequalsthreeapxys3}
\begin{split}
&\Bigl\langle \bigl(Q(p_1,p_2,n)\ot\Id\bigr)\, W^\ast
(f_{-m_1,\ep_1\et_1q^{m_1}r_1z_1,z_1}\ot f_{-m_2,\ep_2\et_2q^{m_2}r_2z_2,z_2}),  \\
&\qquad\qquad \, W^\ast (f_{-m_1-M,\si \sgn(p_1)\ep_1\et_1q^{m_1+M}r_1w_1,w_1} \ot
f_{-m_2-n+M,\si \sgn(p_2)\ep_2\et_2q^{m_2+n-M}r_2w_2,w_2}) \Bigr\rangle.
\end{split}
\end{equation}
In this expression we use \eqref{eq:Wexplicitinapxy} twice, with parameters
$y_1,x_1$ (instead of $y,z$ as in \eqref{eq:Wexplicitinapxy})
for the action of $W^\ast$ in the left leg of the inner product and
with parameters
$y_2,x_2$ for the action of $W^\ast$ in the left leg of the inner product.
The resulting four-fold sum has the advantage that the inner product factorizes, and
we obtain
\begin{equation}\label{eq:pfthmdoublesumfiveapxysequalsthreeapxys4}
\begin{split}
&\sum \left| \frac{z_2w_2}{y_1y_2}\right| \, a_{z_2}(\ep_1\et_1q^{m_1}r_1z_1,y_1)
\, a_{\ep_2\et_2q^{m_2}r_2z_2}(x_1,\ep_1\ep_2\et_1\et_2y_1x_1q^{-m_1-m_2}/r_1z_1) \\
&\qquad\times \, a_{w_2}(\si\sgn(p_1)\ep_1\et_1q^{m_1+M}r_1w_1,y_2) \\
&\qquad\times\, a_{\si \sgn(p_2)\ep_2\et_2q^{m_2+n-M}r_2w_2}(x_2,\ep_1\ep_2\et_1\et_2\sgn(p_1p_2)y_2x_2
q^{-m_1-m_2-n}/r_1w_1) \\
&\qquad \times
\bigl\langle Q(p_1,p_2,n)\, f_{-2m_1-2m_2-\chi(r_1r_2z_1/x_1), x_1,z_1},
f_{-2m_1-2m_2-2n-\chi(r_1r_2w_1/x_2), x_2,w_1}\bigr\rangle \\
&\qquad \times
\bigl\langle f_{m_1+m_2+\chi(r_1r_2z_1/x_1), \ep_1\ep_2\et_1\et_2q^{-m_1-m_2}y_1x_1/r_1z_1,y_1}, \\
&\qquad\qquad \qquad\qquad\qquad f_{m_1+m_2+n+\chi(r_1r_2w_1/x_2), \ep_1\ep_2\et_1\et_2\sgn(p_1p_2)q^{-m_1-m_2-n}y_2x_2/r_1w_1, y_2}
\bigr\rangle
\end{split}
\end{equation}
where the sum is four-fold; $y_1,x_1,y_2,x_2\in I_q$ so that
$\ep_1\ep_2\et_1\et_2q^{-m_1-m_2}y_1x_1/r_1z_1\in I_q$ and
$\ep_1\ep_2\et_1\et_2\sgn(p_1p_2)y_2x_2
q^{-m_1-m_2-n}/r_1w_1\in I_q$.

The final term in the summand \eqref{eq:pfthmdoublesumfiveapxysequalsthreeapxys4} gives three
Kronecker delta's, which lead to the reduction of the four-fold sum to a double(!) sum
since $y_2=y_1$ and $x_2= \sgn(p_1p_2)q^n x_1w_1/z_1$ are required. Substituting this in the
matrix element of $Q(p_1,p_2,n)$ in the summand in
\eqref{eq:pfthmdoublesumfiveapxysequalsthreeapxys4} gives
\[
\begin{split}
\bigl\langle Q(p_1,p_2,n)\, f_{-2m_1-2m_2-\chi(r_1r_2z_1/x_1), x_1,z_1},
f_{-2m_1-2m_2-n-\chi(r_1r_2z_1/x_1), \sgn(p_1p_2)q^nx_1w_1/z_1,w_1}\bigr\rangle
\end{split}
\]
and by \eqref{eq:matrixelementsofQppnandhatJQppnhatJ} this equals zero unless
$r_1r_2 = |\frac{p_2}{p_1}| q^{-2m_1-2m_2-n}$. In case this condition holds we
see that the matrix coefficient of $Q(p_1,p_2,n)$ equals
\[
 \left| \frac{z_1}{w_1}\right| \, a_{z_1}(p_1,w_1)\, a_{x_1}(p_2, \sgn(p_1p_2)q^n \frac{x_1w_1}{z_1}).
\]
Eliminating again $r_2$ and
using this we find that \eqref{eq:pfthmdoublesumfiveapxysequalsthreeapxys4}
equals
\begin{equation}\label{eq:pfthmdoublesumfiveapxysequalsthreeapxys5}
\begin{split}
&\sum_{\stackrel{\scriptscriptstyle{y_1,x_1\in I_q\text{ so that }\sgn(p_1p_2q^nx_1w_1/z_1\in I_q }}
{\scriptscriptstyle{\text{and } \ep_1\ep_2\et_1\et_2q^{-m_1-m_2}y_1x_1/r_1z_1\in I_q}}}
\left| \frac{z_2w_2}{y_1^2}\right|
\, a_{z_2}(\ep_1\et_1q^{m_1}r_1z_1,y_1)
\\ &\qquad\times
\, a_{\ep_2\et_2q^{-2m_1-m_2-n}\frac{z_2|p_2|}{r_1|p_1|}}
(x_1,\ep_1\ep_2\et_1\et_2q^{-m_1-m_2}\frac{y_1x_1}{r_1z_1})
\, a_{w_2}(\si\sgn(p_1)\ep_1\et_1q^{m_1+M}r_1w_1,y_1) \\
&\qquad\times\, a_{\si \sgn(p_2)\ep_2\et_2q^{-2m_1-m_2-M}\frac{w_2|p_2|}{r_1|p_1|}}
(\sgn(p_1p_2)q^n x_1w_1/z_1,\ep_1\ep_2\et_1\et_2q^{-m_1-m_2}\frac{y_1x_1}{r_1z_1}) \\
&\qquad\times\,
 \left| \frac{z_1}{w_1}\right| \, a_{z_1}(p_1,w_1)\, a_{x_1}(p_2, \sgn(p_1p_2)q^n \frac{x_1w_1}{z_1}).
\end{split}
\end{equation}

Equating \eqref{eq:pfthmdoublesumfiveapxysequalsthreeapxys2} and
\eqref{eq:pfthmdoublesumfiveapxysequalsthreeapxys5} and canceling common factors
and relabeling $r_1,x_1,y_1$ by $r,x,y$ then proves Theorem \ref{thm:doublesumfiveapxysequalsthreeapxys}
except for the sign constraint on $y$ in the sum. This follows from
Definition \ref{def:functionap}.
\end{proof}

\appendix

\section{Operators and von Neumann algebras}\label{app:notationterminology}

\subsection{von Neumann algebras}\label{ssecA:vNeumannalgebras}

Let $H$ be a Hilbert space, and $B(H)$ the space of bounded
linear operators equipped with the operator norm
$\| T \| = \sup \{ \|Tx\| \mid \|x\|=1\}$.  Apart from the
topology induced  by the operator norm, there are various
other topologies on $B(H)$.
A
net $\{ T_i\}_{i\in I}$ converges strongly to $T$ if
$\{ T_ix\}_{i\in I}$ converges to $Tx$ for all $x\in H$.\index{strong topology}
A
net $\{ T_i\}_{i\in I}$ converges weakly to $T$ if
$\{ \langle T_ix,y\rangle \}_{i\in I}$ converges to $\langle Tx, y\rangle$
for all $x, y\in H$.\index{weak topology}
A
net $\{ T_i\}_{i\in I}$ converges strongly-$\ast$ to $T$ if
$\{ T_i\}_{i\in I}$ converges strongly to $T$ and
$\{ T_i^\ast\}_{i\in I}$ converges strongly to $T^\ast$.\index{strong $\ast$-topology}

A von Neumann algebra\index{von Neumann algebra} is a unital $\ast$-subalgebra $M$ of $B(H)$
which is closed for the weak topology. A fundamental property
is that $M$ equals its bicommutant $M''$. The elements
of the form $T^\ast T$ form the cone of positive elements,
denoted by $M_+$. A $\ast$-homomorphism is unital when it maps unit to unit.

A linear functional $\om \colon M\to \C$  is normal\index{normal functional} if
$\om\colon M_1 \to \C$ is continuous with respect to the
weak topology, where $M_1$ is the closed unit ball
with respect to the operator norm. The space
of normal functionals form the predual\index{predual $M_\ast$} $M_\ast$ which
is a norm-closed subspace of the dual $M^\ast$. The cone
of positive normal functionals is denoted $M_\ast^+$.
Then $M=(M_\ast)^\ast$ and the $\si$-weak topology\index{sigma-weak@$\si$-weak topology}
on $M$ is the $\si(M,M_\ast)$-topology.
The $\si$-strong-$\ast$ topology\index{sigma-strong@$\si$-strong-$\ast$ topology} is the
locally convex vector topology  induced by the seminorms
$p_\om(T) = \sqrt{\om(T^\ast T)}$,
$p_\om^\ast(T) = \sqrt{\om(TT^\ast)}$ for all $\om\in M_\ast^+$.
A unital $\ast$-homomorphism $\pi \colon M\to N$, $M$ and $N$
von Neumann algebras is normal\index{normal homomorphism} if $\om \pi \in M_\ast$ for all
$\om \in N_\ast$.

The tensor product\index{tensor product} of the von Neumann algebras
$M\subset B(H)$ and $N\subset B(K)$ is the weak
closure $M\ot N$ of the algebraic
tensor product $M\odot N\subset B(H\ot K)$.
For $\om\in M_\ast$, $\et\in N_\ast$ we
have $\om\ot\et\in (M\ot N)_\ast$ as the unique
element extending the algebraic tensor
product $\om\odot\et$.

\subsection{Summation of operators}\label{ssecA:summationoperators}
If we use the symbol $\oplus$ without further mention we mean the completed version. Let $(H_i)_{i \in I}$ be a family
of Hilbert spaces and define the Hilbert space  $H = \oplus_{i \in I} H_i$. Suppose that a permutation $\si : I \to I$
and  for every $i \in I$ a closed, densely defined, linear operator $T_i$ from $H_i$ into $H_{\si(i)}$ is given. Then
$\oplus_{i \in I}\, T_i$ denotes the closed, densely defined, linear operator in $H$ with domain
$$\{\,v \in H \mid v_i \in D(T_i) \text{ for each }  i \in I \text{ and } \sum_{i \in I} \|T_i(v_i)\|^2 < \infty\,\}$$
and so that $(\oplus_{i \in I}\, T_i)(v) = \sum_{i \in I}\,T_i(v_i)$ for all $v \in D(\oplus_{i \in I}\, T_i)$. Also
recall that $T^* = \oplus_{i \in I} T_i^*$. It is also worthwhile to remember that $T^* T = \oplus_{i \in I}\, T_i^*
T_i$ and $|T| = \oplus_{i \in I}\,|T_i|$.

\subsection{Commutation}\label{ssecA:commutationoperators}
Let $H$ be a Hilbert space. Consider two linear operators $S$, $T$ acting in a Hilbert space $H$. We say that $S
\subseteq T$ if $D(S) \subseteq D(T)$ and $Sv = Tv$ for all $v \in D(S)$.

Let $T$ a densely defined, closed, linear (possibly unbounded) operator in $H$. If $S \in B(H)$, we say that $S$ and
$T$ commute\index{commutation of operators} if  $S\,T \subseteq T\,S$. If $N$ is a (possibly unbounded) self-adjoint operator in $H$, we say
that $T$ and $N$ strongly commute\index{strong commutation of operators} if $T$ commutes
with every spectral projection of $N$. If $T$ and $N$ are both (possibly unbounded)
self-adjoint operators, then $T$ and $N$ commute strongly if and
only if their spectral projections commute. This is also known as
\emph{resolvent commuting} self-adjoint operators. In this case
$T+N$ is a closable operator and its closure $\overline{T+N}$ is
self-adjoint.

\subsection{Affiliation and unbounded generators}
\label{ssecA:affiliationandgenerators}
If $M$ is a von Neumann algebra on $H$, then
a densely defined closed linear operator $T$ is affiliated\index{affiliation} to $M$
(in the von Neumann algebraic sense) if
if $TU=UT$ for each unitary $U$ in the commutant $M'$. Then
$T$ is affiliated with $M$ if and only if $T$ commutes with every
element of $M'$. Moreover, if $T$ is affiliated with $M$, then so
are $T^\ast$ and $T^\ast T$. If $T$ is a positive invertible operator
affiliated to $M$, then so is $T^{-1}$.
Also, if $T$ and $N$ are self-adjoint
operators that are affiliated with $M$ and $T$ and $N$
commute strongly, then $\overline{T+N}$ is affiliated with $M$.

For $T_1,\ldots, T_n$ closed, densely defined (possibly unbounded)
linear operators acting on a Hilbert space $H$ we define
the von Neumann algebra
\[
N = \{ x\in B(H) \mid x T_i \subseteq T_i x,\text{ and }
xT_i^\ast \subseteq T_i^\ast x \ \forall \ i\}'.
\]
Then $N$ is the smallest von Neumann algebra so that $T_1, \ldots,T_n$
are affiliated to $N$, and we call $N$ the von Neumann algebra generated
by $T_1,\ldots, T_n$.\index{von Neumann algebra generated by operators}

\section{Special functions}\label{app:specialfunctions}

\subsection{Basic hypergeometric functions}\label{ssecB:BHS}
Here we recall standard notations from the theory of basic hypergeometric functions, see for instance \cite{GaspR}.

We fix a parameter $q \in (0,1)$. The $q$-shifted factorials\index{q-shifted@$q$-shifted factorials} are defined by
\[
(x;q)_\infty = \prod_{k=0}^\infty (1-xq^k), \qquad (x;q)_n = \frac{ (x;q)_\infty }{(xq^n;q)_\infty}, \qquad x \in \mathbb C, \ n \in \Z.
\]
In particular, for $n \in \N$ we have $(x;q)_n=(1-x)(1-qx)\cdots(1-q^{n-1}x)$. Considered as a function of $x$, the $q$-shifted factorial $(x;q)_\infty$ is an entire function. Moreover, $(x;q)_\infty=0$ if and only if $x \in q^{-\N_0}$. For products of $q$-shifted factorials we use the shorthand notation
\[
(x_1, x_2, \ldots, x_k;q)_n = (x_1;q)_n (x_2;q)_n \cdots (x_k;q)_n, \qquad n \in \Z \cup \{\infty\}.
\]
A formula that we frequently use is the $\te$-product identity:
\begin{equation}\label{eq:thetaprodid}
(q^k x, q^{1-k}/x;q)_\infty = (-x)^{-k} q^{-k(k-1)/2}\, (x,q/x;q)_\infty, \qquad x \in \C\setminus\{0\}, \ k \in \Z.
\end{equation}
For $r,s \in \N_0$ the basic hypergeometric series is defined by
\[\index{basic hypergeometric series}\index{F@${}_r\vp_s$-series}
\rphis{r}{s}{x_1,x_2,\ldots, x_r}{y_1,y_2,\ldots,y_s}{q,z} = \sum_{k=0}^\infty \frac{ (x_1,x_2,\ldots,x_r;q)_k }{ (q,y_1,y_2,\ldots,y_s;q)_k } \Big((-1)^k q^{k(k-1)/2} \Big)^{1+s-r} z^k.
\]
Here we assume $x_i \in \C$ for $i=1,\ldots,r$, $y_i \in \C \setminus q^{-\N_0}$ for $i=1,2,\ldots,s$, and $z \in \C$. If $r \leq s$, the series converges absolutely for all $z \in \C$. If $r=s+1$, the series converges absolutely for $|z|<1$. In case $r>s+1$, the definition of the basic hypergeometric series only makes sense if $x_i \in q^{-\N_0}$ for some $i\in \{1,2,\ldots,r\}$, i.e., if the series terminates.

\subsection{The functions $a_p$}\label{ssecB:functionsap}
The functions $a_p(x,y)$ for $x,y,p\in I_q$ have been introduced in
Definition \ref{def:functionap}, and these functions play a crucial
role in the whole construction. We need some more properties of these
functions which are described in this subsection.

We need to study the case $a_p(x,y)$ for $y\in I_q^+=q^\Z$. This is
contained in the following lemma.

\begin{lemma}\label{lemB:apxyforyinqZ}
For $y\in I_q^+$ there exists a differentiable function
$f\colon \R_{\geq 0} \to \R$ such that
$a_p(x,y) = y^{\chi(p/x)}\, f(y^{-2})$.
Moreover, $f(0)=0$ unless $0< x/p\leq 1$, and in that
case $f(0)\not=  0$.
\end{lemma}

\begin{proof}
Assume $y\in I_q^+$, so that $\sgn(y)=+$. So in particular,
$a_p(x,y)=0$ for $\sgn(x)\not=\sgn(p)$ by
Definition \ref{def:functionap} and in this case we can
take $f$ identically equal to zero.

In case $\sgn(x)=\sgn(p)$ we rewrite the $y$-dependent
part in Definition \ref{def:functionap} before the $\Psi$-function,
\[
y\, \nu(py/x)\sqrt{(-y^2;q^2)_\infty} =
\nu(\frac{p}{x})\sqrt{(-1,-q^2;q^2)_\infty} \, \frac{y^{\chi(p/x)}}{\sqrt{(-q^2/y^2;q^2)_\infty}}
\]
using the
theta-product identity \eqref{eq:thetaprodid}. Now using
$s(x,y)=1$ we find
\[
\begin{split}
a_p(x,y)\, &=\, y^{\chi(p/x)} \, f(y^{-2}), \\
f(z)\, &=\, C(p,x)\, \frac{1}{\sqrt{(-q^2z;q^2)_\infty}}
\, \Psis{-q^2z}{q^2\ka(x)z}{q^2, \frac{q^2x^2}{p^2}}, \\
C(p,x)\, &=\, c_q (-1)^{\chi(p)+\chi(x)} \nu(\frac{p}{x})
\sqrt{\frac{(-1,-q^2,-\ka(p);q^2)_\infty}{(-\ka(x);q^2)_\infty}}.
\end{split}
\]
This gives the required differentiable function $f$, which is well-defined on $(-q^{-2},\infty)$ and
even real-analytic.
The value of $f(0)$ is
\[
f(0) = C(p,x)\, \sum_{n=0}^\infty \frac{q^{n(n+1)} }{(q^2;q^2)_n} \left(-\frac{q^2x^2}{p^2}\right)^n =  C(p,x)\,  (q^2x^2/p^2;q^2)_\infty
\]
by \cite[(II.2)]{GaspR}, and this is zero if $x/p>1$ since $x/p\in q^\Z$ and non-zero
otherwise.
\end{proof}

The following contiguous relations are useful.

\begin{lemma}\label{lemB:qcontiguousrelapxy}
Consider $x,y,p \in I_q$. Then
\begin{equation*}
\sqrt{1 + \kappa(q^{-1} x)} \ a_p(q^{-1} x, y) =  (x y/qp)\,\, a_p(x,y) -
\sgn(y)\,q^{-1}\,\sqrt{1+\kappa(y)}\,a_p(x,qy)
\end{equation*}
and
\begin{equation*}
\sqrt{1 + \kappa(x)} \ a_p(q x, y)  =  (x y/ p)\,\, a_p(x,y) -
\sgn(y)\,q\,\sqrt{1+\kappa(q^{-1}y)}\,a_p(x,q^{-1} y)\, .
\end{equation*}
\end{lemma}

\begin{proof}
A proof of the second equality can be found in the second half of
the proof of \cite[Prop.~3.9]{KoelKCMP}, see also
\cite[(6.3)]{KoelKCMP}. If we apply the second contiguous
relation with $x$ and $y$ interchanged, we get
$$
\sqrt{1 + \kappa(y)} \ a_p(q y, x)  =  (x y/ p)\,\, a_p(y,x) -
\sgn(x)\,q\,\sqrt{1+\kappa(q^{-1}x)}\,a_p(y,q^{-1} x)
$$
and the first contiguous relation follows from
the second equality in \eqref{eq:symmetryforapxy}.
\end{proof}

The following identity is essentially the second-order $q$-difference equation for $_1\varphi_1$-functions.

\begin{lemma} \label{lem:2ndorderqdifferenceq}
Consider $x,y,p \in I_q$. Then
\[
(\kappa(p) - \kappa(y) + \frac{y^2 p^2}{x^2})\,a_p(x,y) +
\frac{yp}{x}\,\sqrt{1+ \kappa(q^{-1} p)}\,\,a_{q^{-1} p}(x,y)
+ q\,\frac{yp}{x}\,\sqrt{1+ \kappa(p)}\,a_{qp}(x,y) = 0 \ .
\]
\end{lemma}

\begin{proof}
This equation holds trivially if $p y/x < 0$. From now on we assume that $p y/x > 0$. We know that the $\Psi$-functions satisfy the following $q$-difference equation for $a,b,c,z \in \C$ (see the proof of Lemma 2.1 of \cite{CiccKK}, or take a limit in \cite[Ex.1.13]{GaspR})
\[
(c-az)\,\,\Psi(a;c;q^2,\,q^2 z) + (\,z - (c+q^2)\,)\,\,
\Psi(a;c;q^2,\,z) + q^2\,\,\Psi(a;c;q^2,\,z/q^2) = 0\ .
\]
Hence,
\[
\begin{split}
 (q^2& \kappa(x/y) + q^4 x^2/y^2p^2))\,\, \Psi(-q^2/\kappa(y);\,q^2
\kappa(x/y);q^2,\,q^2 \kappa(x/q^{-1}p))
\\ &   + (-q^2 \kappa(x/y) - q^2 + q^2 \kappa(x/p)\,)\,\,
\Psi(-q^2/\kappa(y);q^2 \kappa(x/y);q^2,\,q^2 \kappa(x/p))
\\ &   + q^2\,\Psi(-q^2/\kappa(y);q^2 \kappa(x/y);q^2,\,q^2 \kappa(x/qp)) = 0
\end{split}
\]
Multiplying this equation with
$y^2 p^2/q^2 x^2\, (-1)^{\chi(p)+1} \,\nu(p y/x)$ and using the fact that
\newline
$\nu(py/x) = q^{-2}\,(py/x)\,\nu(q^{-1} p y/x) = q\,(x/py)\,\nu(q p y/x)$,
we get that
\[
\begin{split}
& (\kappa(p) - \kappa(y) + p^2 y^2/x^2)\, (-1)^{\chi(p)}\,\nu(p y/x)\,\,
\Psi(-q^2/\kappa(y);q^2 \kappa(x/y);q^2,\,q^2 \kappa(x/p))
\\ + &  \, (py/x)\, (1+ \kappa(q^{-1} p))\, (-1)^{\chi(q^{-1} p)}\,
\nu(q^{-1} p y/x)\,\, \Psi(-q^2/\kappa(y);q^2 \kappa(x/y);q^2,\,q^2
\kappa(x/q^{-1}p))
\\ + &  \, q\,(py/x)\, (-1)^{\chi(qp)}\,\nu(q p y/x)\,\,
\Psi(-q^2/\kappa(y);q^2 \kappa(x/y);q^2,\,q^2 \kappa(x/qp)) = 0
\end{split}
\]
Multiplying this with $\sqrt{(\kappa(p);q^2)_\infty}$, it follows that
\begin{equation*}
\begin{split}
 0=&\, \big(\kappa(p) - \kappa(y) + p^2 y^2/x^2\big)\, (-1)^{\chi(p)}\,\nu(p y/x) \\
& \qquad \times \sqrt{(\kappa(p);q^2)_\infty}\,\, \Psi\big(-q^2/\kappa(y);q^2 \kappa(x/y);q^2,\,q^2 \kappa(x/p)\big) \\
+& \, \frac{py}{x}\, \sqrt{1+ \kappa(q^{-1} p)}\, (-1)^{\chi(q^{-1} p)}\, \nu(q^{-1} p y/x) \\
& \qquad \times \sqrt{(\kappa(q^{-1}p);q^2)_\infty}\,\, \Psi\big(-q^2/\kappa(y);q^2 \kappa(x/y);q^2,\,q^2 \kappa(x/q^{-1}p)\big) \\
+&  \, \frac{qpy}{x}\,\sqrt{1+ \kappa( p)}\, (-1)^{\chi(qp)}\,\nu(q p y/x)\\
& \qquad \times \sqrt{(\kappa(q p);q^2)_\infty}\,\, \Psi\big(-q^2/\kappa(y);q^2 \kappa(x/y);q^2,\,q^2 \kappa(x/qp)\big).
\end{split}
\end{equation*}
Now the lemma follows from Definition \ref{def:functionap}.
\end{proof}

We also need a few estimates involving the functions $a_p(x,y)$.

\begin{lemma} \label{lem:estimate1}
Consider $p \in I_q$ and $r,s \in q^\Z$.
Then, there exists a constant $D > 0$ so that
\[
|a_p(x,y)| \leq D\,\nu(p/y)\,|x|^{\chi(p/y)}
\]
for all $x,y \in I_q$ satisfying $|x| \geq r$ and $|y| \leq s$.
\end{lemma}

\begin{proof}
If  $\sgn(x y) = \sgn(p)$ (otherwise $a_p(x,y)=0$), then the symmetry relation \eqref{eq:symmetryforapxy}
and Definition \ref{def:functionap} imply that
\[
\begin{split}
|a_p(x,y)|  & = |a_p(y,x)| \\
  =& \, c_q\,\sqrt{\frac{(-\kappa(p);q^2)_\infty}
{(-\kappa(y);q^2)_\infty}}\ \
\left|\,\Psis{-q^2/\kappa(p)}{q^2 \kappa(y/p)}
{q^2,\,q^2 \kappa(y/x)}\,\right|
 \ |x|\,\nu(p x/y)\,\sqrt{(-\kappa(x);q^2)_\infty},
\end{split}
\]
and $|x|\,\nu(p x/y) = \nu(q x) \,\, \nu(p/y)\,\, |x|^{\chi(p/y)}$ by Definition \ref{defn:chikappanuands}. Now observe that for $x > 0$,
\[
\sqrt{(-\kappa(x);q^2)_\infty} \,\,\nu(q x) =
\frac{\sqrt{2}\,(-q^2;q^2)_\infty}{\sqrt{(-q^2/x^2;q^2)_\infty}}
\]
by the $\theta$-product identity \eqref{eq:thetaprodid}. Furthermore, for $x<0$, the set $\{\,x \in I_q^- \mid |x| \geq r\,\}$ is finite.
Hence, it is clear that there exists a constant $D > 0$ so that $|a_p(x,y)| \leq D \,\, \nu(p/y)\,\, |x|^{\chi(p/y)}$
for all $x,y \in I_q$ satisfying $|x| \geq r$ and $|y| \leq s$.
\end{proof}

\begin{lemma} \label{lem:estimate2}
Consider $p,y \in I_q$, $\al > 0$ and $r \in [1,\infty)$. Then, the family $\bigl(\,|x|^{-\al}\,a_p(x,y)\,\bigr)_{x \in
I_q}$ belongs to $\ell^r(I_q)$.
\end{lemma}

\begin{proof}
Since $|a_p(x,y)| = |a_p(y,x)|$ by \eqref{eq:symmetryforapxy}, Lemma \ref{lem:estimate1}
implies the existence of  a constant $D > 0$ so that
$|x^{-\al}\,a_p(x,y)| \leq D\,\nu(p/x)\,|y|^{\chi(p/x)-\al}$
for all $x \in I_q$ satisfying $|x| \leq q$.

Next we need an estimate for $|x| \geq 1$. If $p/y \geq 1$, Lemma \ref{lem:estimate1} assures the existence of
$E > 0$ so that $|a_p(x,y)| \leq E$ for all $x \in I_q$
satisfying $x \geq 1$. If on the other hand, $p/y < 1$,
Lemma \ref{lem:estimate1} and the fact that
$|a_p(x,y)| = |y/p|\,|a_y(x,p)|$ by \eqref{eq:symmetryforapxy},
guarantee also in this case the existence of  $E > 0$ so that
$|a_p(x,y)| \leq E$ for all $x \in I_q$ satisfying $x \geq 1$.
Hence, the lemma follows.
\end{proof}

\subsection{The function $S(t;p_1,p_2,n)$} \label{appB:functionS}
The following function is defined as an infinite sum of certain limits of the functions $a_p$. Let $p_1,p_2 \in I_q$, $n \in\Z$. The function $S(\,\cdot\,;p_1,p_2,n)\colon\C\setminus\{0\} \rightarrow \C$ is defined by
\begin{equation} \label{eq:S=sum1phi1 1phi1}\index{S@$S(t;p_1,p_2,n)$}
\begin{split}
S(t&;p_1,p_2,n) = \\
 &C\, \sum_{z \in \sgn(p_1)q^\Z} \big(\sgn(p_1p_2)\,  t\big)^{\chi(z)}
\frac{1 }{|z|}\, \nu(\frac{p_1}{z}) \, \nu(\frac{p_2 q^n}{z})
\rphis{1}{1}{-q^2/\ka(p_1)}{0}{q^2,q^2\ka(z)} \\
& \qquad \qquad \times   \rphis{1}{1}{-q^2/\ka(p_2)}{0}{q^2,q^2\ka\big(\sgn(p_1p_2) q^{-n}z\big)},
\end{split}
\end{equation}
where
\[
C=C(p_1,p_2,n) = \big(\sgn(p_2)\big)^n \, |p_1p_2|\, c_q^2\, q^n \, \sqrt{(-\ka(p_1), -\ka(p_2);q^2)_\infty}\,.
\]
The sum is absolutely convergent, so $S(\,\cdot\,;p_1,p_2,n)$ is an analytic function on $\C\setminus\{0\}$. The function $S(t;p_1,p_2,n)$ can be written as a $_2\varphi_1$-function. To see this we need a few lemmas.

In the following lemma the special case $b=q$ is obtained by Koornwinder and Swarttouw as a $q$-analogue of Graf's addition formula for Bessel functions \cite[(4.10)]{KooSwa}. The proof of Lemma \ref{lem:2phi1=sum 1phi1 1phi1} runs along the same lines as the proof used in \cite{KooSwa}.
\begin{lemma} \label{lem:2phi1=sum 1phi1 1phi1}
For $c \in q^\Z$, $|u|<1$, and $|bu/w|<1$,
\[
\begin{split}
\sum_{n=-\infty}^\infty &w^n q^{\hf n (n-1)} \rphis{1}{1}{u}{0}{q;cq^{n}}\rphis{1}{1}{v}{0}{q,bq^n}=\\
&\frac{ (q,u,-w,-q/w,-cu/w,bq/c;q)_\infty }{ (-bu/w,-c/w,-wq/c;q)_\infty }  \rphis{2}{1}{-bv/w, -wq/cu}{bq/c}{q,u}.
\end{split}
\]
\end{lemma}
Other expressions for the sum in the above lemma, for values of $u,w,b$ not satisfying the above conditions, can be obtained using transformation formulas for $_2\varphi_1$-series.
\begin{proof}
Assume $|y|<1$, $|sb/x|<|t|<|y^{-1}|$ and $|y|<|t|$. We write the product of the following $_1\psi_1$-function and $_1\varphi_0$-function as a double series;
\[
\begin{split}
\rpsis{1}{1}{x/sy}{b}{q,yt} \rphis{1}{0}{xs/y}{-}{q,-\frac{y}{t}}
= &\,\sum_{n=-\infty}^\infty \sum_{k=0}^\infty \frac{ (x/sy;q)_n (xs/y;q)_k }{ (b;q)_n (q;q)_k } (-1)^k y^{n+k} t^{n-k}.
\end{split}
\]
Renaming $n=m+k$, the sum over $k$ can be written as a $_2\varphi_1$-series. Using Ramanujan's $_1\psi_1$-summation formula \cite[(II.29)]{GaspR} and the $q$-binomial formula \cite[(II.3)]{GaspR}, we obtain
\begin{equation} \label{eq:sum2phi1}
\begin{split}
\frac{ (q,bsy/x, xt/s, qs/xt;q)_\infty }{ (b, qsy/x, yt, bs/xt;q)_\infty}& \frac{ (-xs/t;q)_\infty}{ (-y/t;q)_\infty} = \\ &\sum_{m=-\infty}^\infty \frac{ (x/sy;q)_m }{ (b;q)_m } (yt)^m \rphis{2}{1}{xs/y, xq^m/sy}{bq^m}{q,-y^2}.
\end{split}
\end{equation}
We consider this formula as the Laurent expansion of the left hand side considered as a function of $t$.

Let us consider two special cases of \eqref{eq:sum2phi1}. Letting $y \rightarrow 0$, we obtain
\[
\begin{split}
\frac{ (q,xt/s, qs/xt, -xs/t;q)_\infty }{ (b, bs/xt;q)_\infty }& = \sum_{m=-\infty}^\infty \left( - \frac{xt}{s} \right)^m q^{\hf m(m-1) } \frac{1}{(b;q)_m} \rphis{0}{1}{-}{bq^m}{q,-q^m x^2} \\
&= \frac{1}{(b;q)_\infty}\sum_{m=-\infty}^\infty \left( - \frac{xt}{s} \right)^m q^{\hf m(m-1) } \rphis{1}{1}{-x^2/b}{0}{q,bq^m}.
\end{split}
\]
In the last line we used the transformations
\begin{equation} \label{eq:0phi1-transform}
\rphis{1}{1}{z}{0}{q,c} =  (c,z;q)_\infty \rphis{2}{1}{0,0}{c}{q,z} = (c;q)_\infty \rphis{0}{1}{-}{c}{q,cz},
\end{equation}
which follow from Heine's $_2\varphi_1$-transformations \cite[(III.1), (III.3)]{GaspR} by letting $a,b\to 0$.

For the second special case we observe that in the above calculations the assumption $|sb/x|<|t|$ was needed for absolute convergence of the bilateral $_1\psi_1$-series. In case $b=q$ this series can be written as a unilateral series, a $_1\varphi_0$-series, and then the assumption $|sb/x|<|t|$ is no longer needed. Now setting $b=q$ and $x=0$, we find
\begin{equation} \label{eq:1phi1transformation}
\begin{split}
\frac{1}{(yt,-y/t;q)_\infty } =& \sum_{m=-\infty}^\infty \frac{ (yt)^{m} }{(q;q)_m} \rphis{2}{1}{0,0}{q^{1+m}}{q,-y^2}\\
=&\frac{1}{(q,-y^2;q)_\infty}\sum_{m=-\infty}^\infty  (yt)^{m} \rphis{1}{1}{-y^2}{0}{q,q^{1+m}}\\
=&\frac{1}{(q,-y^2;q)_\infty}\sum_{m=-\infty}^\infty  \left(-\frac{t}{y}\right)^{m} \rphis{1}{1}{-y^2}{0}{q,q^{1-m}}
\end{split}
\end{equation}
where we used \eqref{eq:0phi1-transform}, and for the last equality we used the $t \leftrightarrow -t^{-1}$ invariance and reversed the sum.

Multiplying our two special cases of \eqref{eq:sum2phi1}, we obtain a second expression for the Laurent expansion of the left hand side of \eqref{eq:sum2phi1} considered as a functions of $t$;
\[
\begin{split}
&\frac{ (q,bsy/x, xt/s, qs/xt,-xs/t;q)_\infty }{ (b, qsy/x, yt, bs/xt,-y/t;q)_\infty}\\
&= \frac{ (bsy/x;q)_\infty }{ (q,-y^2,b,sqy/x;q)_\infty } \sum_{k=-\infty}^\infty  \left(-\frac{t}{y}\right)^{k} \rphis{1}{1}{-y^2}{0}{q;q^{1-k}}\\
 & \phantom{\frac{ (bsy/x;q)_\infty }{ (sqy/x,-y^2;q)_\infty }}\qquad\times \sum_{n=-\infty}^\infty\left(-\frac{xt}{s}\right)^n q^{\hf n (n-1)} \rphis{1}{1}{-x^2/b}{0}{q,bq^n}\\
&= \frac{ (bsy/x;q)_\infty }{ (q,-y^2,b,sqy/x;q)_\infty } \\
&\times\sum_{m=-\infty}^\infty \left(-\frac{t}{y}\right)^{m}
 \sum_{n=-\infty}^\infty \left(\frac{xy}{s}\right)^n q^{\hf n (n-1)} \rphis{1}{1}{-y^2}{0}{q;q^{1-m+n}}\rphis{1}{1}{-x^2/b}{0}{q,bq^n}.
\end{split}
\]
Here we used $n+k=m$. Comparing coefficients of $t$ in \eqref{eq:sum2phi1} and the above formula, and, to get rid of the squares, replacing $(-y^2,-x^2/b, xy/s)$ by $(u,v,w)$, we obtain
\[
\begin{split}
\sum_{n=-\infty}^\infty &w^n q^{\hf n (n-1)} \rphis{1}{1}{u}{0}{q;q^{1+n-m}}\rphis{1}{1}{v}{0}{q,bq^n}=\\
&u^m\frac{ (q,u,-qu/w,-w/u,bq^m;q)_\infty }{ (-bu/w,-wq^m/u;q)_\infty }  \rphis{2}{1}{-bv/w, -wq^m/u}{bq^m}{q,u}.
\end{split}
\]
Observe that by the $\te$-product identity \eqref{eq:thetaprodid},
\[
u^m \frac{ (-qu/w,-w/u;q)_\infty }{(-wq^m/u;q)_\infty } = w^m q^{\hf m(m-1)} (-uq^{1-m}/w;q)_\infty = \frac{ (-w,-q/w,-uq^{1-m}/w;q)_\infty}{(-q^{1-m}/w,-wq^m;q)_\infty},
\]
then the result follows from writing $q^{1-m}=c$.
\end{proof}

\begin{remark}
We can prove a slightly more general result along the same lines as the proof of Lemma \ref{lem:2phi1=sum 1phi1 1phi1}, starting with the product
\[
\rpsis{1}{1}{x/sy}{b}{q,yt}\rpsis{1}{1}{xs/y}{d}{q,-\frac{y}{t}}.
\]
This leads to the identity
\[
\begin{split}
\sum_{k=-\infty}^\infty& \left( \frac{xy}{s} \right)^k q^{\hf k(k-1)} \rphis{1}{1}{-x^2/b}{0}{q,bq^k} \rphis{2}{2}{-y^2,dy/xs}{0,qy/xs}{q,q^{1-m+k}} =\\
&(-y^2)^m \frac{(x/sy;q)_m (d,-y^2,syq/x;q)_\infty}{(b;q)_m (bsy/x;q)_\infty} \rpsis{2}{2}{xq^m/sy, xs/y}{bq^m,d}{q,-y^2}.
\end{split}
\]
For $d=q$ this is equivalent to the result from Lemma \ref{lem:2phi1=sum 1phi1 1phi1}.
\end{remark}

The following lemma shows that the result of Lemma \ref{lem:2phi1=sum 1phi1 1phi1} remains valid for $c\not\in q^{-\Z}$, if we assume $u \in q^{-\N_0}$. The $_2\varphi_1$-series in Lemma \ref{lem:2phi1=sum 1phi1 1phi1} does not converge in this case, but it can be obtained from the $_2\varphi_1$-series in the following Lemma by an application of Heine's transformation \cite[(III.2)]{GaspR}.
\begin{lemma} \label{lem:2phi1=sum 1phi1 1phi1 A}
For $u=q^{-\N_0}$ and $|bu/w|<1$,
\[
\begin{split}
\sum_{n=-\infty}^\infty &w^n q^{\hf n (n-1)} \rphis{1}{1}{u}{0}{q;cq^n}\rphis{1}{1}{v}{0}{q,bq^n}=\\
&\frac{ (q,-w,-q/w, -cu/w;q)_\infty }{(-c/w;q)_\infty} \rphis{2}{1}{-wq/cu, v}{-wq/c}{q,-bu/w}.
\end{split}
\]
\end{lemma}
\begin{proof}
Let us denote the infinite sum on the left hand side by $S$. We write $u=q^{-k}$ with $k \in \N_0$, then by definition of the $_1\varphi_1$-series, we have
\[
S = \sum_{n=-\infty}^\infty \sum_{m=0}^k \sum_{l=0}^\infty \frac{(q^{-k};q)_m (v;q)_l }{(q;q)_m (q;q)_l } q^{\hf m(m-1)} (-c)^m q^{\hf l(l-1)} (-b)^l q^{\hf n(n-1)} (wq^{m+l})^n.
\]
This double sum converges absolutely, so we may first sum over $n$. Using Jacobi's triple product identity \cite[(II.28)]{GaspR} we find
\[
\begin{split}
\sum_{n=-\infty}^\infty q^{\hf n(n-1)} (wq^{m+l})^n &= (q,-wq^{m+l}, -q^{1-m-l}/w;q)_\infty \\
& = w^{-(m+l)} q^{-\hf m(m-1)} q^{-\hf l(l-1)} q^{-lm} (q,-w,-q/w;q)_\infty.
\end{split}
\]
Here the second equality follows from the $\te$-product identity \eqref{eq:thetaprodid}. Now $S$ reduces to
\[
S= (q,-w,-q/w;q)_\infty \sum_{l=0}^\infty \sum_{m=0}^k \frac{ (v;q)_l }{(q;q)_l} (-b/w)^l  \frac{ (q^{-k};q)_m }{(q;q)_m} (-cq^{-l}/w)^m.
\]
The sum over $m$ can be evaluated with the $q$-binomial formula \cite[(II.3)]{GaspR};
\[
\begin{split}
\sum_{m=0}^k \frac{ (q^{-k};q)_m }{(q;q)_m} (-cq^{-l}/w)^m &= \frac{(-cq^{-l-k}/w;q)_\infty}{ (-cq^{-l}/w;q)_\infty} =\frac{ (-cq^{-l-k}/w;q)_l (-cq^{-k}/w;q)_\infty} {(-cq^{-l}/w;q)_l (-c/w;q)_\infty}\\
&= q^{-kl}\frac{(-wq^{1+k}/c;q)_l (-cq^{-k}/w;q)_\infty }{ (-wq/c;q)_l (-c/w;q)_\infty },
\end{split}
\]
using \cite[(I.9)]{GaspR}. We see that $S$ becomes a multiple of a single sum,
\[
S = \frac{ (q,-w,-q/w, -cq^{-k}/w;q)_\infty }{(-c/w;q)_\infty} \sum_{l=0}^\infty \frac{ (-wq^{1+k}/c,v;q)_l }{(q,-wq/c;q)_l} \left(-\frac{ bq^{-k}}{w}\right)^l.
\]
The sum is the $_2\varphi_1$-series in the lemma.
\end{proof}

\begin{remark} \label{rem:symmetry}
In Lemmas \ref{lem:2phi1=sum 1phi1 1phi1} and \ref{lem:2phi1=sum 1phi1 1phi1 A} the sum $\Sigma$ on the left hand side has an obvious symmetry $(u,c)\leftrightarrow (v,b)$. On the right hand side this symmetry is not at all obvious, so there must be a $_2\varphi_1$-transformation behind this symmetry. Let us see how the symmetry follows from known transformation formulas.

Applying the three-term transformation formula \cite[(III.31)]{GaspR} we find
\[
\begin{split}
&(bq/c;q)_\infty \rphis{2}{1}{-bv/w, -wq/cu}{bq/c}{q,u} =\\
 &\ \frac{ (v,bq/c,c/b;q)_\infty }{ (-cuv/wq,-wq/bv;q)_\infty } \rphis{2}{1}{-wq/cv, q/v}{-q^2w/cuv}{q,-wq/bu} \\
 &+ \frac{c}{b} \frac{ (v,cq/b,-wq/cu, -qw/cv, -buv/wq, -q^2w/buv;q)_\infty }{ (u,-wq/bu, -qw/bv, -uvc/wq, -q^2w/uvc;q)_\infty } \rphis{2}{1}{-wq/bv, -cu/w}{cq/b}{q,v},
\end{split}
\]
where we also applied Heine's transformation \cite[(III.3)]{GaspR} for the second $_2\varphi_1$ on the right hand side. Observe that the second $_2\varphi_1$-function on the right hand side is the same as the $_2\varphi_1$-function on the left hand side after the substitutions $(u,v,c,b) \mapsto (v,u,b,c)$, which is exactly the symmetry we are looking for. This shows that the first $_2\varphi_1$-function on the right hand side must vanish, which implies the condition $v \in q^{-\N_0}$ or $c/b \in q^\Z$. Assuming one of these conditions, the symmetry $(u,c) \leftrightarrow (v,b)$ for $\Sigma$ is still not clear at this point, because of all the $q$-shifted factorials in front of the $_2\varphi_1$-function. To take care of these factors we need to apply the $\te$-product identity \eqref{eq:thetaprodid} several times. Let us assume that $v=q^{-k}$, $k \in \N_0$, then
\begin{align*}
\frac{(-cu/w, -wq/cu;q)_\infty}{ (-q^2w/cuv,-vuc/wq;q)_\infty }&= \left( - \frac{wq}{cu}\right)^{k+1} q^{\hf k(k+1)},\\
\frac{ (-buv/wq, -q^2w/buv;q)_\infty }{ (-bu/w, -wq/bu;q)_\infty } &= \left( - \frac{bu}{wq}\right)^{k+1} q^{-\hf k (k+1)},\\
\frac{ (-wq/cv;q)_\infty}{(-c/w,-wq/c;q)_\infty } &= \frac{(-c/wq)^k q^{\hf k(k-1)}}{(-cv/w;q)_\infty },\\
\frac{1}{(-wq/bv;q)_\infty} &= \left( \frac{-wq}{b}\right)^k q^{\hf k(k-1)} \frac{(-bv/w;q)_\infty}{(-b/w,-qw/b;q)_\infty},
\end{align*}
which leads to
\[
\Sigma = \frac{ (q,-w,-q/w,cq/b, -bv/w,v;q)_\infty }{ (-vc/w, -b/w, -wq/b;q)_\infty} \rphis{2}{1}{-wq/bv, -cu/w}{cq/b}{q,v}.
\]
Comparing this with the right hand side in Lemma \ref{lem:2phi1=sum 1phi1 1phi1} the symmetry $(u,c)\leftrightarrow (v,b)$ is now clear. In case $b/c \in q^\Z$ similar computations must be used.

Observe that the conditions $b/c \in q^\Z$ and $v \in q^{-\N_0}$ correspond to Lemmas \ref{lem:2phi1=sum 1phi1 1phi1} and \ref{lem:2phi1=sum 1phi1 1phi1 A}, respectively.
\end{remark}
We are now ready to obtain a $_2\varphi_1$-expression for the function $S(t;p_1,p_2,n)$.

\begin{prop} \label{prop:S=2phi1}
The function $S(t;p_1,p_2,n)$ defined by \eqref{eq:S=sum1phi1 1phi1} can be written as a multiple of a $_2\varphi_1$-function:
\[
\begin{split}
S(t;p_1,p_2,n) =&\, p_2^n q^{\hf n (n-1)} |p_1p_2|\, \nu(p_1)\nu(p_2) c_q^2 \sqrt{ (-\kappa(p_1),-\kappa(p_2);q^2)_\infty }\\
& \times\frac{ (q^2, -q^2/\kappa(p_2), -tq^{3-n}/p_1p_2, -p_1p_2q^{n-1}/t, p_1q^{1-n}/p_2t;q^2)_\infty }{ ( |p_1|q^{1+n}/|p_2|t, -p_1|p_2|q^{-n-1}/t, -tq^{n+3}/p_1|p_2|;q^2)_\infty }\\
& \times (\sgn(p_1p_2)q^{2+2n};q^2)_\infty \rphis{2}{1}{p_2q^{1+n}/p_1t, p_2 t q^{1+n}/p_1}{ \sgn(p_1 p_2) q^{2+2n} }{q^2,-q^2/\kappa(p_2)}.
\end{split}
\]
\end{prop}
\begin{proof}
We substitute $z=\sgn(p_1)q^k$, $k \in \Z$, in \eqref{eq:S=sum1phi1 1phi1}, then
\[
\begin{split}
S(t;p_1,p_2,n)&= K \, \sum_{k=-\infty}^\infty \left(\frac{tq^{3-n}}{p_1p_2}\right)^k q^{k (k-1)}\rphis{1}{1}{-\sgn(p_1)q^2/p_1^2}{0}{q^2, \sgn(p_1)q^{2+2k}} \\
& \qquad \qquad \times \rphis{1}{1}{-\sgn(p_2)q^2/p_2^2}{0}{q^2, \sgn(p_2)q^{2+2k-2n}},\\
K &= q^{\hf n(n-1)} p_2^n |p_1p_2|\, \nu(p_1)\nu(p_2)\, c_q^2 \sqrt{ (-\kappa(p_1), -\kappa(p_2);q^2)}.
\end{split}
\]
Now we apply Lemmas \ref{lem:2phi1=sum 1phi1 1phi1} and \ref{lem:2phi1=sum 1phi1 1phi1 A}, with $q$ replaced by $q^2$, and
\[
w = \frac{tq^{3-n}}{p_1p_2}, \quad u=-\sgn(p_2)\frac{q^2}{p_2^2}, \quad v=-\sgn(p_1)\frac{q^2}{p_1^2}, \quad b=\sgn(p_1)q^{2}, \quad c=\sgn(p_2)q^{2-2n},
\]
to obtain the desired expression.
\end{proof}
The function $S(t;p_1,p_2,n)$ can be written in terms of several other $_2\varphi_1$-functions using the following result.

\begin{lemma} \label{lem:symmetryS}
The function $S(t;p_1,p_2,n)$ satisfies the following symmetry relations:
\[
\begin{split}
S(t;p_1,p_2,n) &= (qt)^n S(t;p_2,p_1,-n)\\
&= (-q)^n\, \sgn(p_1)^{\chi(p_1)}
\, \sgn(p_2)^{\chi(p_2)+n}\, \sgn(p_1p_2)\, S(\sgn(p_1p_2)t^{-1};p_1,p_2,-n),\\
&=(-t)^n \sgn(p_1)^{\chi(p_1)+n} \sgn(p_2)^{\chi(p_2)} \sgn(p_1p_2) S(\sgn(p_1p_t)t^{-1};p_2,p_1,n).
\end{split}
\]
\end{lemma}

\begin{proof}
The first symmetry relation follows from replacing the summation variable $z$ by $zq^{n}$ in definition \eqref{eq:S=sum1phi1 1phi1}.

Comparing coefficients of $t$ in \eqref{eq:1phi1transformation} gives the transformation formula
\[
\rphis{1}{1}{a}{0}{q,q^{1+n}} = a^{-n} \rphis{1}{1}{a}{0}{q,q^{1-n}}, \qquad n \in \Z.
\]
Furthermore, as a special case of \cite[Prop.~6.6]{KoelKCMP} we have
\[
\rphis{1}{1}{q^{-n}}{0}{q,q y} = y^{n} \rphis{1}{1}{q^{-n}}{0}{q,q/y},\qquad n \in \N_0,\ y \in \C\setminus\{0\}.
\]
To both $_1\varphi_1$-functions in \eqref{eq:S=sum1phi1 1phi1} we apply one of the above transformations; the second one in case the $_1\varphi_1$ is a terminating series, the first transformation otherwise. Now we change the summation variable from $z$ to $z^{-1}$ to obtain the second symmetry relation.

The third relation follows from combining the first two relations.
\end{proof}

Proposition \ref{prop:S=2phi1} and the symmetry relations from Lemma \ref{lem:symmetryS} imply transformation formulas between the $_2\varphi_1$-series involved. For instance, the first symmetry relation in Lemma \ref{lem:symmetryS} together with an application of the $\te$-product identity \eqref{eq:thetaprodid}, corresponds to the transformation described in Remark \ref{rem:symmetry}.

We also need the following asymptotic results for the function $S$.

\begin{lemma} \label{lem:asymptoticsforS}
Assume $t \in \C\setminus\{0\}$ and $k,n \in \Z$.
\begin{enumerate}[(i)]
\item
For $k \to -\infty$,
\[
S(t;q^k,q^k,n) = \mathcal O(q^{-nk}).
\]
\item Let $\si,\tau \in \{-,+\}$, then there exist constants $C_1,C_2$ independent of $k$, such that
\[
S(t;\si q^k, \tau q^k,n) = (\si \tau q^3)^k \big(C_1t^{-k} + C_2 t^k\big)\Big(1+\mathcal O(q^{2k})\Big),
\]
for $k \to \infty$.
\item Let $p_1\in I_q$ and $\tau \in \{-,+\}$, then for $k \to \infty$,
\[
S(t;p_1,\tau q^k, k+n) =\mathcal O(q^k).
\]
\item Let $p_1 \in I_q$, then for $k \to -\infty$,
\[
S(t;p_1,q^k,k+n) = \mathcal O\big(q^{\hf k^2} (p_1tq^{n-\hf})^k \big).
\]

\end{enumerate}
\end{lemma}

\begin{proof}
(i) We use Proposition \ref{prop:S=2phi1} to write $S(t;q^k,q^k,n)$ as a multiple of a $_2\varphi_1$-series. Using the $\te$-product identity \eqref{eq:thetaprodid} we find
\[
\frac{ (-tq^{3-n-2k},-q^{n-1+2k}/t;q^2)_\infty }{ (-tq^{n+1-2k}, -q^{2k-n-1}/t;q^2)_\infty} = q^{-2nk} t^n
\]
and
\[
q^{2k}\nu(q^k)^2 (-\kappa(q^k),-q^{2}/\kappa(q^k);q^2)_\infty = q^2 (-1,-q^2;q^2)_\infty,
\]
so that
\[
\begin{split}
S(t;q^k,q^k,n) =&\, c_q^2q^{2-nk} t^n q^{\hf n (n-1)}
\frac{ (-1,-q^2, q^2, q^{1-n}/t;q^2)_\infty }{ (q^{1+n}/t;q^2)_\infty }\\
& \times (q^{2+2n};q^2)_\infty \rphis{2}{1}{q^{1+n}/t, t q^{1+n}}{ q^{2+2n} }{q^2,-q^{2-2k}}.
\end{split}
\]
From this expression it is clear that $S(t;q^k,q^k,n) = \mathcal O(q^{-nk})$ for $k \to -\infty$.

(ii) Write $S(t;\si q^k, \tau q^k,n)$ as a multiple of a $_2\varphi_1$-function using Proposition \ref{prop:S=2phi1}. Using the three-term transformation formula \cite[III.32]{GaspR} and the $\te$-product identity \eqref{eq:thetaprodid} we find
\[
\begin{split}
&S(t;\si q^k, \tau q^k, n) = c_q^2\, (\si\tau q^3)^k \sqrt{ (-\si q^{2k}, -\tau q^{2k};q^2)_\infty } \frac{ (q^2, \si\tau q^{1-n}/t, -\si\tau q^{n-1}/t, -\si\tau tq^{3-n};q^2)_\infty }{ (q^{n+1}/t, -\si q^{-n-1}/t, -\si tq^{3+n};q^2)_\infty} \\
&\quad \times \Bigg\{ t^{-k} \frac{ (\si \tau tq^{1+n}, tq^{1+n}, -\si q^{3+n}/t, -\si tq^{-1-n};q^2)_\infty }{ (t^2, -\tau q^{2k};q^2)_\infty } \rphis{2}{1}{ \si\tau q^{1+n}/t, q^{1-n}/t}{q^2/t^2}{q^2, -\si q^{2k} } \\
& \qquad  + t^k \frac{ (\si \tau q^{1+n}/t, q^{1+n}/t, -\si tq^{3+n}, -\si q^{-1-n}/t;q^2)_\infty }{ (t^{-2}, -\tau q^{2k};q^2)_\infty } \rphis{2}{1}{ \si\tau t q^{1+n}, tq^{1-n}}{q^2t^2}{q^2, -\si q^{2k} } \Bigg\}.
\end{split}
\]
From this expression the result follows.

(iii) By Proposition \ref{prop:S=2phi1} and \cite[(III.4)]{GaspR} there exists a constant $C_1$, which is independent of $k$, such that
\[
\begin{split}
S(t;p_1,\tau q^k,& k+n) = C_1 (\tau q)^k q^{k(k+n)} q^{\hf(k+n)(k+n-1)}
 q^{\hf (k-1)(k-2)} \sqrt{ (-\tau q^{2k};q^2)_\infty } \\
\times &(\tau p_1 q^{1-2k-n}/t, -\tau p_1 q^{2k+n-1}/t, -\tau t q^{3-n-2k}/p_1;q^2)_\infty \\
\times &(\tau \,\sgn(p_1) q^{2+2k+2n};q^2)_\infty \rphis{2}{2}{ \tau q^{1+2k+n}/p_1t, |p_1|q^{1+n}/t }{ \tau\, \sgn(p_1) q^{2+2k+2n}, -q^{3+n}/p_1t}{q^2; -tq^{3+n}/p_1}.
\end{split}
\]
Using the $\te$-product identity \eqref{eq:thetaprodid} twice, we find
\[
\begin{split}
(\tau p_1 q^{1-2k-n}/t,-\tau p_1 q^{2k+n-1}/t, -\tau t q^{3-n-2k}/p_1;q^2)_\infty = C_2 \frac{(-1)^k q^{-2nk} q^{-2k(k-1)}}{ (\tau t q^{1+2k+n}/p_1;q^2)_\infty }.
\end{split}
\]
Now we see that for large $k$ there exists a constant $C_3$, independent of $k$, such that
\[
|S(t;p_1,\tau q^k, k+n)| \leq C_3 q^k.
\]

(iv) Assume $k<0$. By Proposition \ref{prop:S=2phi1} we have
\[
\begin{split}
S&(t;p_1, q^k , -n-k) = C_1\,q^{-k(n+k)} q^{\hf(n+k)(n+k+1)} q^k q^{\hf(k-1)(k-2)} \sqrt{ (-q^{2k};q^2)_\infty} \\
& \times \frac{ (-q^{2-2k}, \sgn(p_1) q^{2-2n-2k};q^2)_\infty }{ (|p_1|q^{1-n-2k}/t, -p_1 q^{2k+n-1}/t, -tq^{3-2k-n}/p_1;q^2)_\infty} \rphis{2}{1}{ q^{1-n}/p_1 t, tq^{1-n}/p_1 }{\sgn(p_1) q^{2-2n-2k} }{q^2; -q^{2-2k} }.
\end{split}
\]
for a certain constant $C_1$ independent of $k$. Using the $\te$-product identity \eqref{eq:thetaprodid} we have
\begin{gather*}
(-p_1 q^{2k+n-1}/t, -tq^{3-2k-n}/p_1;q^2)_\infty = C_2 \Big(\frac{ t}{p_1q^{n-1}}\Big)^k q^{-k(k-1)},\\
(-q^{2k};q^2)_\infty = C_3 \frac{ q^{-k(k-1)} }{(-q^{2-2k};q^2)_\infty },
\end{gather*}
so that, for large $|k|$, there is a constant $C_4$ such that
\[
|S(t;p_1, q^k , -n-k)| \leq  C_4\, q^{\hf k^2} q^{(n-\frac32)k} |p_1/t|^k.
\]
Now the result follows from the second symmetry relation in Lemma \ref{lem:symmetryS}.
\end{proof}

\subsection{Al-Salam--Chihara polynomials} \label{ssecB:AlSCpol}
The spectral analysis of Jacobi operators on $\ell^2(\N_0)$ and $\ell^2(\Z)$ plays an
essential role in this paper. We refer to Berezanski\u\i\ \cite[Ch.7]{Bere}, Pruitt \cite{Prui},
Masson and Repka \cite{MassR}, Kakehi \cite{Kake}, see also \cite[App.~A]{KoelSRIMS}, for general information on Jacobi operators on $\ell^2(\Z)$. We use \cite{KoelLaredo} for general reference. The spectral decomposition of the Jacobi operators we encounter are described with the help of certain special functions, namely the Al-Salam--Chihara polynomials and the little $q$-Jacobi functions. In this subsection we collect some results and notations for the Al-Salam--Chihara polynomials. Results for little $q$-Jacobi functions are given in the next subsection.

The Al-Salam--Chihara polynomials were introduced by Al-Salam and Chihara in \cite{AlSaC} to classify all orthogonal polynomials satisfying a convolution type property. These polynomials also have been studied by Askey and Ismail \cite[\S 3]{AskeI}. The Al-Salam--Chihara polynomials form subfamily of the Askey-Wilson polynomials Askey and Wilson \cite{AskeW}, Gasper and Rahman \cite[\S\S 7.5-7]{GaspR}.

Consider $a,b \in \R \setminus \{0\}$. For $n \in \NN$,
the Al-Salam--Chihara polynomials
$P_n(\cdot\,;a,b\mid q)\colon \C \to \C$ are defined by
\begin{equation}\label{eq:AlSCpol}
\begin{aligned}
P_n(\mu(y);a,b\mid q) &=  a^{-n} (ab;q)_n \, {}_3\vp_2\left(
\begin{array}{c}
q^{-n},\, ay,\, a/y \\ ab,\, 0
\end{array}; q,q\right) \\
&=  (a/y;q)_n\, y^n \, {}_2\vp_1\left(
\begin{array}{c}
q^{-n},\, by \\ q^{1-n}y/a
\end{array};
q, \frac{q}{ya}\right),
\end{aligned}
\end{equation}
for $y \in \C \setminus \{0\}$. The equality in
\eqref{eq:AlSCpol} follows from \cite[(III.7)]{GaspR} and holds if
$q^{1-n} y/a \not\in q^{-\NN}$. We see that for $x=\mu(y) \in \mathbb R$ the
polynomials $P_n(x)=P_n(x;a,b\mid q)$ are real-valued. The Al-Salam--Chihara polynomials satisfy the three-term recurrence relation
\begin{equation}\label{eq:ttrAlSCpol}
2x\, P_n(x) = P_{n+1}(x) + q^n(a+b)\, P_n(x) + (1-q^n)(1-abq^{n-1})\, P_{n-1}(x)
\end{equation}
with initial condition $P_{-1}(x)=0$, $P_0(x)=1$. From this relation we see that the Al-Salam--Chihara polynomials are symmetric in $a$ and $b$. Favard's Theorem gives that these polynomials are orthogonal with respect to a positive measure on the real line for $ab<1$, which from now on we assume to hold. The measure can be determined from the asymptotic behaviour of the Al-Salam--Chihara polynomials as the degree tends to infinity. This behaviour is determined by
\begin{equation}\label{eq:asymptbehaviourAlSCpol}
\begin{split}
&\frac{(abq^n;q)_\infty}{\sqrt{(q,ab;q)_\infty}}
\, P_n(\mu(y);a,b\mid q)  = \\
&\ c(y;a,b\mid q)\, \, y^n
\rphis{2}{1}{ay, by}{qy^2}{q, q^{n+1}} +
c(y^{-1};a,b\mid q)\, \, y^{-n}
\rphis{2}{1}{a/y, b/y}{qy^{-2}}{q, q^{n+1}},
\end{split}
\end{equation}
valid if $y^2 \not \in q^{\Z}$, where
\begin{equation}\label{eq:cfunctionforAlSCpols}
c(y;a,b\mid q) = \frac{(a/y,b/y;q)_\infty}
{(y^{-2};q)_\infty \sqrt{(q,ab;q)_\infty}}.
\end{equation}
We extend the $c$-function $c(\cdot\,;a,b\mid q)$
by continuity to all points of $\C$ where possible.

The asymptotic behaviour can be obtained as a limiting case
($b,c\to 0$) of the asymptotic behaviour of the
Askey-Wilson polynomials \cite[(7.5.9)]{GaspR},
or by using \cite[(3.3.5)]{GaspR} with
$(a,b,c,z)\mapsto(ay,by,qy^2,q^{n+1})$ and next
\cite[(1.4.6)]{GaspR}, \eqref{eq:AlSCpol} and the $\theta$-product
identity \eqref{eq:thetaprodid}.
See also \cite[\S 3.1]{AskeI} for the asymptotic behaviour
using Darboux's method including the cases $x=\pm 1$.

The corresponding orthonormal Al-Salam--Chihara
polynomials $p_n(\cdot\,;a,b\mid q) : \C \to \C$ are defined by
\begin{equation}\label{eq:ortonAlSCpol}
p_n(x;a,b\mid q) = \frac{1}{\sqrt{ (q,ab;q)_n}} \, P_n(x;a,b\mid q)
\end{equation}
for all $x \in \C$. The orthonormal Al-Salam--Chihara
polynomials satisfy the recurrence relation
\begin{equation}\label{eq:recAlSCpol}
\begin{split}
&2x\, p_n(x)\, =
c_n\, p_{n+1}(x) + d_n\, p_n(x) + c_{n-1}\, p_{n-1}(x), \\
&c_n\, = \sqrt{(1-q^{n+1})(1-abq^n)}, \quad d_n = q^n(a+b),
\end{split}
\end{equation}
and initial conditions $p_{-1}(x)=0$, $p_0(x)=1$. Note that the
coefficients $c_n$ and $d_n$ are bounded, since we
assume $0<q<1$. Under our assumption $ab<1$ the
Al-Salam--Chihara polynomials are orthogonal with respect to a
positive measure on $\R$;
\begin{equation}\label{eq:orthoAlSCpol}
\int_\R p_n(x;a,b\mid q)p_m(x;a,b\mid q)\, dm(x;a,b\mid q) = \de_{n,m},
\end{equation}
where the measure $dm(\cdot;a,b\mid q)$ is defined by
\begin{equation} \label{eq:measureASC}
\begin{split}
 \int_\R f(x)\, dm(x;a,b\mid q) =& \frac{(q,ab;q)_\infty}{2\pi} \int_0^\pi f(\cos\psi)
\frac{(e^{2i\psi},e^{-2i\psi};q)_\infty}
{(ae^{i\psi},ae^{-i\psi},be^{i\psi},be^{-i\psi};q)_\infty}\, d\psi \\
& + \sum_{\stackrel{\scriptstyle{r\in\NN}}{\scriptstyle{|aq^r|>1}}}
f\bigl( \mu(aq^r)\bigr) w_r(a;b\mid q) +
\sum_{\stackrel{\scriptstyle{r\in\NN}}{\scriptstyle{|bq^r|>1}}}
f\bigl( \mu(bq^r)\bigr) w_r(b;a\mid q),
\end{split}
\end{equation}
with
\[
w_r(a;b\mid q) = \frac{(a^{-2};q)_\infty (a^2,ab;q)_r (1-a^2q^{2r})} {(b/a;q)_\infty (q,aq/b;q)_r
(1-a^2)} q^{-r^2} a^{-3r}b^{-r}.
\]
Note that the weight function in \eqref{eq:orthoAlSCpol} is
very explicit. It can be rewritten in terms of the $c$-function
\eqref{eq:cfunctionforAlSCpols} as
\begin{equation}\label{eq:orthoAlSCpolRES}
\begin{split}
 \int_\R f(x)\, dm(x;a,b\mid q) &=
\frac{1}{2\pi} \int_0^\pi  f(\cos\psi) \frac{d\psi}{c(e^{i\psi};a,b\mid  q)
c(e^{-i\psi};a,b\mid  q)} \\
& + \sum_{s\in D} f(\mu(s)) \Res{w=s} \frac{1} {w\, c(w;a,b\mid  q)
c(w^{-1};a,b\mid  q)},
\end{split}
\end{equation}
where the set $D$ is given by
\[
D= D(a,b\mid  q) = \{ s\in \C \mid |s|>1,\ c(s;a,b\mid  q)=0\},
\]
and we assume that the zeroes of the $c$-function in $D$ are simple. The two sets of discrete mass points in the measure in \eqref{eq:orthoAlSCpol} are finite. If $ab>0$, at most one of the sets of discrete mass points can occur, since we also assume $ab<1$. If $ab<0$, then both series of discrete mass points can occur.

Consider the corresponding Jacobi operator on
$\ell^2(\NN)$ equipped with the standard orthonormal basis
$\{e_n\}_{n=0}^\infty$,
\begin{equation}\label{eq:JacobioperAlSCpol}
2 J e_n = c_n\, e_{n+1} + d_n \, e_n + c_{n-1}\, e_{n-1},
\end{equation}
with $c_n$ and $d_n$ as in \eqref{eq:recAlSCpol}, initially
defined on the dense domain of finite linear combinations
of the basis vectors. Since the coefficients are bounded,
$J$ extends uniquely to a bounded self-adjoint operator on
$\ell^2(\NN)$. If we need to stress the dependence on the parameters, we write $J=J(a,b\mid q)$. The resolution of the identity for the self-adjoint extension of $J$ can be described with the orthonormal Al-Salam--Chihara polynomials and the corresponding orthogonality measure.
\begin{thm} \label{thm:spectraldecompASC}
The Jacobi operator $J$ extends uniquely to a bounded self-adjoint operator on
$\ell^2(\NN)$. Let $E_J$ be the resolution of the identity for the self-adjoint extension of $J$, then for any Borel set ${\mathcal B}\subset \R$ and
$u=\sum_{n=0}^\infty u_ne_n, v=\sum_{n=0}^\infty v_ne_n\in \ell^2(\NN)$
we have
\begin{equation}\label{eq:residJacopAlScpol}
\langle E_J({\mathcal B})u,v\rangle_{\ell^2(\NN)}
= \int_{\mathcal B} {\mathcal F}_Ju(x) \overline{{\mathcal
F}_Jv(x)}\, dm(x;a,b\mid q), \qquad {\mathcal F}_Ju(x)
= \sum_{n=0}^\infty u_n p_n(x;a,b\mid q).
\end{equation}
\end{thm}

For the purposes in this paper we want to rewrite
the orthogonality relations
\eqref{eq:orthoAlSCpol} for the Al-Salam--Chihara
polynomials as orthogonality relations on
$L^2(I(a,b\mid q))$, where $I(a,b\mid q)$ is the support of
$dm(\cdot\,\, ;a,b\mid q)$, so
\begin{equation}\label{eq:defIabqforASCpol}
\begin{split}
I(a,b\mid q) &= [-1,1] \cup \mu(D(a,b\mid q)), \\
D(a,b\mid q) &= \{\,aq^r\mid r\in\NN, |aq^r|>1 \}
\cup \{\,bq^r\mid r\in\NN, |bq^r|>1 \},
\end{split}
\end{equation}
in accordance with \eqref{eq:orthoAlSCpolRES}.
On $[-1,1]$ we take the Lebesgue measure, and on
the discrete part we take the counting measure. Now define for
$0<|\psi|<\pi$
\begin{equation}\label{eq:deffnabqforASCpol}
\begin{split}
h_n(\cos\psi;a,b\mid q) &= \sqrt{\frac{1}{2\pi |\sin\psi|}}
\ \frac{p_n(\cos\psi;a,b\mid q)}
{|c(e^{i\psi};a, b\mid q)|}, \\
h_n(\mu(eq^r);a,b\mid q) &=
\sqrt{w_r(e;f\mid q)} \, p_n(\mu(eq^r);a,b\mid q),
\end{split}
\end{equation}
where $e$ is either $a$ or $b$, and $f$ is the other parameter,
and $|eq^r|>1$ with $r\in\NN$. So
$\{\, h_n(\cdot\, ;a,b\mid q)\}_{n=0}^\infty$ is
an orthonormal basis for $L^2(I(a,b\mid q))$. It
follows in particular that
\begin{equation}\label{eq:orthonAlScrelationsdisc}
\sum_{n=0}^\infty h_n(\mu(x);a,b\mid q)\,
h_n(\mu(y);a,b\mid q) = \de_{x,y}, \qquad x,y\in D(a,b\mid q),
\end{equation}
so that the functions $h_n(\mu(x);a,b\mid q)$, $n\in \N_0$, have
$\ell^2$-norm $1$ for $x\in D(a,b\mid q)$. The orthogonality relations \eqref{eq:orthonAlScrelationsdisc} can also be proved directly using the $q$-binomial theorem and the $q$-Saalsch\"utz formula  \cite[(II.3),(II.12)]{GaspR}, and it is related to a discrete measure on $q^{-\NN}$ for which only a finite number of moments exist.

The polynomials $p_n(\,\cdot\,;a,b\mid q)$ and the $c$-function are symmetric in $a,b$, which implies the symmetry relation
\begin{equation} \label{eq:symmetryhn1}
h_n(\,\cdot\,;a,b\mid q) = h_n(\,\cdot\,;b,a\mid q).
\end{equation}
Another symmetry that we need is
\begin{equation} \label{eq:symmetryhn2}
h_n(\,\cdot\,;a,b\mid q) = (-1)^n h_n(-\,\cdot\,;-a,-b;q),
\end{equation}
which follows from writing out explicitly $h_n$ as a multiple of a $_2\varphi_1$-function.

The asymptotic behaviour of the orthonormal basis of
$L^2(I(a,b\mid q))$ as the degree $n$ tends to $\infty$ can be
obtained from \eqref{eq:asymptbehaviourAlSCpol}. For $0 < |\psi| < \pi$ we find
\begin{equation}\label{eq:asymptoticnormalAlSCfcont}
h_n(\cos\psi;a,b\mid q) = \sqrt{\frac{2}{\pi | \sin\psi|}}
\frac{\Re \bigl( e^{in\psi} c(e^{i\psi};a,b\mid  q)\bigr)}
{|c(e^{i\psi};a,b\mid  q)|} \bigl( 1+ {\mathcal O}(q^n)\bigr),
\qquad n\to\infty,
\end{equation}
and see \cite[\S 3.1]{AskeI} for the case $x=\pm 1$. Observe that the expression is symmetric with
respect to $\psi \leftrightarrow -\psi$. On the discrete spectrum the zeroes of the $c$-function make the first term on the right hand side of \eqref{eq:asymptbehaviourAlSCpol} vanish, so that the behaviour of $h_n$ is given by
\begin{equation}\label{eq:asymptoticnormalAlSCfdisc}
h_n(\mu(aq^r);a,b\mid q) =
(aq^r)^{-n} \sqrt{w_r(a;b\mid q)}\, c(1/aq^r;a,b\mid q) \,
\bigl( 1+ {\mathcal O}(q^n)\bigr), \qquad n\to \infty.
\end{equation}
This implies $h_\cdot(x;a,b\mid q) \in \ell^2(\N_0)$ for $x$ in the discrete spectrum.
The expression for $h_n(\mu(bq^k);a,b\mid q)$ follows from \eqref{eq:asymptoticnormalAlSCfdisc} by interchanging $a$ and $b$ in the right hand side. We can also reformulate
\eqref{eq:asymptoticnormalAlSCfdisc} as
\begin{equation}\label{eq:asymptoticnormalAlSCfdiscRES}
h_n(\mu(s);a,b\mid q) =
s^{-n} \sqrt{\Res{w=s} \frac{c(w^{-1};a,b\mid q)}
{w \, c(w;a,b\mid q)}} \,
\bigl( 1+ {\mathcal O}(q^n)\bigr), \qquad n\to \infty,
\end{equation}
for $s\in D(a,b\mid  q)$, assuming such zeroes of the $c$-function are simple.

In this paper we need a certain contiguous relations for the Al-Salam--Chihara polynomials. The contiguous relation can be looked upon as an operator that can be used for a Darboux factorization of the Jacobi operator $J$.
\begin{lemma} \label{lem:contiguousASC}
The orthonormal basis functions $h_n(x;a,b\mid q)$ satisfy
\[
\sqrt{1-2bx+b^2}\, h_n(x;a,b\mid q) =
\sqrt{1-abq^n}\, h_n(x;a,bq\mid q) - b \sqrt{1-q^n}\, h_{n-1}(x;a,bq\mid q),
\]
for $x \in I(a,b\mid q)$.
\end{lemma}
\begin{proof}
From the connection coefficient formula \cite[\S 6]{AskeW}, \cite[\S 7.6]{GaspR} it follows that
\begin{equation}\label{eq:contigrelAlScpols}
P_n(x;a,b\mid q) = P_n(x;a,bq\mid q) - b (1-q^n)\, P_{n-1}(x;a,bq\mid q).
\end{equation}
This can also be obtained directly from the second explicit expression of $P_n$ in \eqref{eq:AlSCpol} by writing out the $_2\varphi_1$-function as a sum, and using the identity $(by;q)_k = (bqy;q)_k -by(1-q^k)(bqy;q)_k$.  Rewriting \eqref{eq:contigrelAlScpols} for the orthonormal basis $h_n(x;a,b\mid q)$, $x\in I(a,b\mid q)$, gives the desired relation. For $x=\cos\psi$ this follows directly from \eqref{eq:ortonAlSCpol}, \eqref{eq:deffnabqforASCpol}, and for $x$ in the
discrete spectrum  this is a consequence of
\begin{equation*}
\frac{w_r(a;bq\mid q)}{w_r(a;b\mid q)} = \frac{(1-abq^r)(1-q^{-r}b/a)}{1-ab}, \qquad
\frac{w_{r-1}(bq;a\mid q)}{w_r(b;a\mid q)} = \frac{(1-b^2q^r)(1-q^{-r})}{1-ab}.
\end{equation*}
Here we use the convention that $h_n(x;a,b\mid q)=0$ for $x\not\in I(a,b\mid q)$.
\end{proof}

\subsection{Little $q$-Jacobi functions} \label{ssecB:litteqJacobi}
In this subsection we collect the results and notations for the
little $q$-Jacobi functions needed in this paper. The little
$q$-Jacobi functions are the kernel of an explicit transform pair that is
related to the spectral analysis of the hypergeometric $q$-difference equation, and they arise as matrix elements for the quantum $SU(1,1)$ group, see \cite{MasuMNNSU}.  References for this subsection are Kakehi \cite{Kake}, Kakehi et al. \cite{KakeMU}, and also \cite[App. A]{KoelSRIMS}, \cite{KoelLaredo}, \cite{KoelSTempe}.

The hypergeometric $q$-difference equation, see \cite[Exerc. 1.13]{GaspR},
can be rewritten as
\begin{equation}\label{eq:hypqdiffeq}
(c-abz)\, u(qz) + \bigl( (a+b)z-c-q\bigr)\, u(z) + (q-z)\, u(z/q)=0
\end{equation}
for a function $u(z)$ and one explicit solution of
\eqref{eq:hypqdiffeq} is $u(z)= {}_2\vp_1(a,b;c;q,z)$.

Using the hypergeometric $q$-difference we find solutions to
\begin{equation}\label{eq:2ndordqdiffeq}
2x\, f_k(x) = (1-\frac{q^{1+k}}{z})\, f_{k+1}(x)
+ q^{k}\frac{c+q}{dz}\, f_k(x) + (1-\frac{cq^{k}}{d^2z})\, f_{k-1}(x),
\end{equation}
where we assume from now on that $z<0$, $c>0$, and
$d\in\R\backslash\{0\}$. For more general sets of parameters, see
\cite[App. A]{KoelSRIMS}.
Indeed, we find the solution,
\begin{equation}\label{eq:2ndordqdiffeqBHS}
f_k(\mu(y)) = (c,z,q/z;q)_\infty  d^{-k}\,
{}_2\vp_1\left( \begin{array}{c}
dy,\, d/y \\ c
\end{array};q,zq^{-k}\right),
\end{equation}
where we from now on assume $0<c<1$ in order to avoid
complications for $c\in q^{-\NN}$. We use the notation
$f_k(x)=f_k(x;c,d;z\mid q)$ if we want to stress the dependence
on the parameters. Note that replacing $c$ and $d$ by
$q^2/c$ and $qd/c$ leaves \eqref{eq:2ndordqdiffeq} invariant,
hence $f_k(x;q^2/c,qd/c;z\mid q)$ is also a solution to
\eqref{eq:2ndordqdiffeq}, as can also be checked directly from
\eqref{eq:hypqdiffeq}. These solutions are linearly
independent for $c\not=q$.

The equation \eqref{eq:2ndordqdiffeq} can also be viewed for
$k\geq 0$ as the recurrence relation for the (suitably renormalized) associated
Al-Salam--Chihara polynomials, and the description of the solution space matches Gupta,
Ismail and Masson \cite{GuptIM}.

Next we define
\begin{equation}\label{eq:2ndordqdiffeqBHS3}
F_k(y) = y^k \, {}_2\vp_1\left( \begin{array}{c} dy,\, qdy/c
\\qy^2 \end{array}; q, \frac{q^{1+k}c}{d^2z}\right),
\qquad y^2\notin q^{-\N},
\end{equation}
then, for $y \neq \pm 1$, $F_k(y)$ and $F_k(y^{-1})$ define two linearly independent
solutions to \eqref{eq:2ndordqdiffeqBHS} as follows easily from
\eqref{eq:hypqdiffeq}. We use the notation $F_k(y^{\pm 1}) =
F_k(y^{\pm 1};c,d;z\mid q)$ if we want to stress the
dependence on the parameters. Note that $F_k(y^{\pm 1})$ are
invariant under replacing $c$ and $d$ by $q^2/c$ and $qd/c$.
Since the solution space to \eqref{eq:2ndordqdiffeq} is
two-dimensional there are relations between the solutions; in particular,
\begin{equation}\label{eq:cfunctions}
f_k(\mu(y)) = c(y)\, F_k(y) + c(y^{-1})\,
F_k(y^{-1}), \quad c(y) = \frac{(c/dy,d/y,dzy,q/dzy;q)_\infty}
{(y^{-2};q)_\infty}
\end{equation}
which follows from \cite[(4.3.2)]{GaspR} for $y^2\notin q^\Z$.
As in the previous subsection we extend this
$c$-function by continuity to all points of $\C$ where possible.
We use the notation $c(y;c,d;z\mid q)$ if we want to
stress the dependence on the parameters. Note that this
$c$-function is different from the one for the Al-Salam--Chihara polynomials in Section
\ref{ssecB:AlSCpol}. In this subsection $c(y)$ is defined by \eqref{eq:cfunctions}.

The corresponding orthonormal recurrence relation, i.e., the
normalization which makes the corresponding Jacobi operator
symmetric, is
\begin{equation}\label{eq:reclqJacf}
\begin{split}
&2x\, u_k(x) \, = a_k\, u_{k+1}(x) + b_k\, u_k(x) + a_{k-1}\, u_{k-1}(x), \\
&a_k \, = \sqrt{\left(1-\frac{q^{k+1}}{z}\right)\left(1-\frac{cq^{k+1}}{d^2z}\right)},
\qquad b_k =\frac{q^{k}(c+q)}{dz}.
\end{split}
\end{equation}
Note that we assume $z<0$, $0<c<1$,
$d\in\R\setminus\{0\}$, so that the square root is well-defined. We put
\begin{equation}\label{eq:deflqJacfweight}
\rho_k^2 = \frac{(cq^{1+k}/d^2z;q)_\infty} {(q^{1+k}/z;q)_\infty}
= \left( \frac{c}{d^2}\right)^{-k} \frac{(zq^{-k},
d^2z/c, cq/d^2z;q)_\infty} {(d^2zq^{-k}/c, z, q/z;q)_\infty},
\end{equation}
where the second expression follows from the $\te$-product identity \eqref{eq:thetaprodid},
then $u_k(z)=\rho_k f_k(z)$ satisfies \eqref{eq:reclqJacf} if and only if $f_k(z)$ satisfies \eqref{eq:2ndordqdiffeq}. We use the notation $\rho_k(c,d;z\mid q)$ if we want to stress
the dependence on the parameters. Now the following orthogonality relations hold;
\begin{equation}\label{eq:ortholqJacf}
\int_\R \rho_kf_k(x) \rho_lf_l(x)\, d\nu(x;c,d;z\mid q) = \de_{k,l},
\end{equation}
where the measure $d\nu$ is defined by
\[
\begin{split}
\int_\R g(x)\, d\nu(x;c,d;z\mid q) =& \frac{1}{2\pi}
\int_0^\pi g(\cos\psi) \frac{d\psi}{|c(e^{i\psi})|^2} +
\sum_{\stackrel{\scriptstyle{r\in\Z}}{\scriptstyle{|q^{1-r}/dz|>1}}}
g(\mu(q^{1-r}/dz)) v_r  \\
& + \sum_{\stackrel{\scriptstyle{r\in\NN}}{\scriptstyle{|cq^r/d|>1}}}
g(\mu(cq^r/d)) w_r +
\sum_{\stackrel{\scriptstyle{r\in\NN}}{\scriptstyle{|dq^r|>1}}}
g(\mu(dq^r)) w^\prime_r,
\end{split}
\]
with
\[
\begin{split}
& c(y) = c(y;c,d;z\mid q), \\
& v_r = \frac{-(1-q^{2-2r}/d^2z^2)\, (dz)^{2(1-r)} q^{-(r-2)(r-1)}}
{(q,q,cq^{1-r}/d^2z,q^{1-r}/z,czq^{r-1},d^2zq^{r-1};q)_\infty}, \\
& w_r = \frac{(d^2/c^2;q)_\infty}
{(q,c,d^2/c,cz,d^2z/c,cq/d^2z,q/zc;q)_\infty}
\frac{(1-c^2q^{2r}/d^2)}{(1-c^2/d^2)}
\frac{(c^2/d^2,c;q)_r}{(q,cq/d^2;q)_r}\, c^{-r},
\\
& w^\prime_r = \frac{(d^{-2};q)_\infty}
{(q,c,c/d^2,d^2z,z,q/z,q/d^2z;q)_\infty} \frac{(1-d^2q^{2r})}{(1-d^2)}
\frac{(d^2,c;q)_r}{(q,qd^2/c;q)_r}\,c^{-r}.
\end{split}
\]
If we want to stress the dependence on the parameters we use the notation $w_r(c,d;z\mid q)$, $w^\prime_r(c,d;z\mid q)$ and $v_r(c,d;z\mid q)$ for the weights in \eqref{eq:ortholqJacf}.
Note that at most one of the last two sets of discrete mass points can occur, since we assume $0<c<1$. The first set of discrete mass points always occurs.
The orthogonality measure \eqref{eq:ortholqJacf} can be rewritten in terms of the $c$-function;
\begin{equation}\label{eq:ortholqJacfRES}
\int_\R g(x)\, d\nu(x;c,d;z\mid q) =  \frac{1}{2\pi}
\int_0^\pi g(\cos\psi) \frac{d\psi}{|c(e^{i\psi})|^2}
+ \sum_{s\in D} g(\mu(s)) \Res{w=s} \frac{1}{w\, c(w)c(w^{-1})},
\end{equation}
where we assume that the zeroes of the $c$-function are simple, and where the set $D$ is defined by
\[
D=D(c,d;z\mid q) = \{ s\in \C \mid |s|>1,\, c(s)=0\}.
\]
See Kakehi \cite{Kake}, and \cite[App. A]{KoelSRIMS} for a
bit more general situation, \cite{KoelLaredo} for an introduction,
and \cite{KoelSTempe} for a general scheme of function
transforms with basic hypergeometric kernel of which
\eqref{eq:ortholqJacf} is part.

Denote by $L$ the corresponding (doubly infinite)
Jacobi operator on $\ell^2(\Z)$ with orthonormal basis
$\{ e_k\}_{k\in\Z}$, i.e.,
\begin{equation}\label{eq:JacobioperatorreclqJacf}
2L\, e_k \, = a_k\, e_{k+1} + b_k\, e_k + a_{k-1}\, e_{k-1},
\end{equation}
with $a_k$ and $b_k$ defined as in \eqref{eq:reclqJacf}, and $L$ initially defined on the dense domain of finite linear combinations of the basis vectors. We write $L=L(c,d,z\mid q)$\index{L@$L(c,d,z\mid q)$}
if we need to stress the dependence on the parameters. The operator $L$ is unbounded, because the coefficients tend to $\pm\infty$ as $k\to-\infty$. Its adjoint is given by the same formula
\eqref{eq:JacobioperatorreclqJacf} with its maximal domain, i.e. ${\mathcal D}^\ast = \{ v=\sum_k v_ke_k\in\ell^2(\Z) \mid \sum_k (a_kv_{k+1} + b_k v_k + a_{k-1}v_{k-1})e_k\in \ell^2(\Z)\}$. From Section  4.5 of \cite{KoelLaredo} we have the following result. Note that we need to switch from the basis
$e_k$ to $e_{-k}$ of $\ell^2(\Z)$ for the correspondence with \cite{KoelLaredo}.
\begin{thm} \label{thm:spectraldecompL}
The operator $L$ is essentially self-adjoint for $0<c \leq  q^2$. In this case the resolution of the identity $E_L$ for the unique self-adjoint extension of $L$ is given by
\[
\langle E_L({\mathcal B})u,v\rangle_{\ell^2(\Z)} =
\int_{\mathcal B} {\mathcal F}_Lu(x) \overline{{\mathcal F}_Lv(x)}\,
d\nu(x;c,d;z\mid q), \qquad {\mathcal F}_Lu(x) =
\sum_{k=-\infty}^\infty u_k \rho_k f_k(x),
\]
for any Borel set ${\mathcal B}\subset \R$
and any $u=\sum_k u_ke_k, v=\sum_k v_ke_k\in \ell^2(\Z)$.
\end{thm}
In \cite[Prop.4.5.3]{KoelLaredo} it is also proved that $L$ has deficiency indices $(1,1)$ in case $q^2 < c < 1$, $c \neq q$, hence $L$ has self-adjoint extensions. In the proof linear independence of certain functions $wf(z)$ and $wg(z)$ (see \cite{KoelLaredo}) is used, which is no longer valid in case $c=q$. The special case $c=q$ is also needed in this paper, and we treat this case in Appendix \ref{app:jacobi}.

In this paper it is convenient to rewrite the orthogonality relations \eqref{eq:ortholqJacf} as orthogonality relations on $L^2(I(c,d;z\mid q))$, where $I(c,d;z\mid q)$ is the support of $d\nu(\cdot\, ;c,d;z\mid q)$. So
\begin{equation}\label{eq:defIabqforlqJacfun}
\begin{split}
I(c,d;z\mid q) = &\, [-1,1] \cup \mu\big(D(c,d;z\mid q)\big), \\
D(c,d;z\mid q) = &\, \Big\{ dq^r\mid r\in\NN, |aq^r|>1 \Big\}
\cup \Big\{ \frac{c}{d}q^r\mid r\in\NN, |\frac{c}{d}q^r|>1 \Big\} \\
&\, \cup \Big\{ \frac{q^{1-r}}{dz} \mid r\in\Z, |\frac{q^{1-r}}{dz}|>1 \Big\}
\end{split}
\end{equation}
in accordance with \eqref{eq:ortholqJacfRES}. On $[-1,1]$ we take the Lebesgue measure, and on
the discrete part we take the counting measure. We now define the
function $j_k(x;c,d;z\mid q)\in L^2(I(c,d;z\mid q))$ by
\begin{equation}\label{eq:orthonorlqJacfun}
\begin{split}
j_k(\cos\psi;c,d;z\mid q) &=
\frac{\rho_k(c,d;z\mid q)\, f_k(\cos\psi;c,d;z\mid q)}
{\sqrt{2\pi|\sin\psi|}\, |c(e^{i\psi};c,d;z\mid q)|},
\qquad 0<|\psi|<\pi, \\
j_k(\mu(q^{1-r}/dz);c,d;z\mid q) &= \sqrt{v_r(c,d;z\mid q)}\,
\rho_k(c,d;z\mid q)\,
f_k(\mu(q^{1-r}/dz);c,d;z\mid q), \\
j_k(\mu(cq^r/d);c,d;z\mid q) &= \sqrt{w_r(c,d;z\mid q)}\, \rho_k(c,d;z\mid q)\,
f_k(\mu(cq^r/d);c,d;z\mid q), \\
j_k(\mu(dq^r);c,d;z\mid q) &=
\sqrt{w^\prime_r(c,d;z\mid q)}\, \rho_k(c,d;z\mid q)\, f_k(\mu(dq^r);c,d;z\mid q),
\end{split}
\end{equation}
so that $\{ j_k(\cdot\,  ;c,d;z\mid q)\}_{k\in\Z}$ yields an orthonormal basis for $L^2(I(c,d;z\mid q))$. We use the convention that $j_k(x;c,d;z\mid q)=0$ for
$x\notin I(c,d;z\mid q)$. In particular this implies that
\begin{equation}\label{eq:orthorefororthonorlqJacfun}
\sum_{k\in\Z} j_k(\mu(x);c,d;z\mid q) j_k(\mu(y);c,d;z\mid q)  =
\de_{x,y}, \qquad x,y \in D(c,d;z\mid q),
\end{equation}
so that $\{ j_k(\mu(x);c,d;z\mid q)\}_{k\in\Z}$ has $\ell^2$-norm $1$ for $x\in D(c,d;z\mid q)$.

The asymptotic behaviour of $j_k(x;c,d;z\mid q)$ as $k\to -\infty$ follows from
\begin{equation}\label{eq:asymplqJacfunkto-infty}
\rho_k\, f_k(x) = (\sgn(d)\sqrt{c})^k \,
(c,d^2z/c,cq/d^2z;q)_\infty \bigl( 1+ {\mathcal O}(q^{-k})\bigr), \qquad
x\in\C,
\end{equation}
which is an immediate consequence of \eqref{eq:2ndordqdiffeqBHS} and
\eqref{eq:deflqJacfweight}. For the asymptotic behaviour as $k\to\infty$ we use \eqref{eq:cfunctions}, \eqref{eq:2ndordqdiffeqBHS3}, \eqref{eq:deflqJacfweight}, and we proceed
analogously as in the derivation of \eqref{eq:asymptoticnormalAlSCfcont}.
This gives
\begin{equation}\label{eq:asymplqJacfunktoinftycont}
j_k(\cos\psi ;c,d;z\mid q) = \sqrt{\frac{2}{\pi |\sin\psi|}}
\frac{\Re\bigl( c(e^{i\psi};c,d;z\mid q)\, e^{ik\psi}\bigr)}
{|c(e^{i\psi};c,d;z\mid q)|}
\bigl( 1+ {\mathcal O}(q^k)\bigr)\,\bigr), \qquad k\to \infty,
\end{equation}
for $0<|\psi|<\pi$. Note that the expression is symmetric with respect to $\psi\leftrightarrow -\psi$.
The asymptotic behaviour in the
discrete mass points as $k\to\infty$ follows
similarly as \eqref{eq:asymptoticnormalAlSCfdisc}. The behaviour is
$\ell^2$, and for $k\to\infty$ we have
\begin{equation}\label{eq:asymplqJacfunktoinftydisc}
\begin{split}
j_k(\mu(q^{1-r}/dz);c,d;z\mid q) &=
\bigl(q^{1-r}/dz\bigr)^{-k} \sqrt{v_r(c,d;z\mid q)}\,
c(q^{r-1}dz;c,d;z\mid q)\, \bigl( 1+ {\mathcal O}(q^k)\bigr),  \\
j_k(\mu(cq^r/d);c,d;z\mid q) &= \bigl(cq^r/d)^{-k} \sqrt{w_r(c,d;z\mid q)}\,
c(dq^{-r}/c;c,d;z\mid q)\, \bigl( 1+ {\mathcal O}(q^k)\bigr),  \\
j_k(\mu(dq^r);c,d;z\mid q) &=
\bigl( dq^r \bigr)^{-k} \sqrt{w^\prime_r(c,d;z\mid q)}\, c(q^{-r}/d;c,d;z\mid q)\,
\bigl(1+ {\mathcal O}(q^k)\bigr).
\end{split}
\end{equation}
We can rewrite \eqref{eq:asymplqJacfunktoinftydisc},
cf.~\eqref{eq:asymptoticnormalAlSCfdiscRES},
\begin{equation}\label{eq:asymplqJacfunktoinftydiscRES}
j_k(\mu(s);c,d;z\mid q)  =
s^{-k} \sqrt{\Res{w=s} \frac{c(w^{-1};c,d;z\mid q)}
{w \, c(w;c,d;z\mid q)}} \,
\bigl( 1+ {\mathcal O}(q^k)\bigr), \qquad k\to \infty,
\end{equation}
for $s\in D(c,d;z\mid q)$ assuming the zeroes of the $c$-function are simple.

We will need a contiguous relation for the normalized
little $q$-Jacobi functions, which can be obtained from
the $q$-derivative of the ${}_2\vp_1$-series.
\begin{lemma} \label{lem:contiguouslittleqJac}
The orthonormal basis functions $j_k(x ;c,d;z\mid q)$ satisfy
\[
\begin{split}
\sqrt{1-2x/d+d^{-2}}\,& j_k(x;qc,qd;z\mid q) = \\
&\frac{1}{d}\sqrt{1-\frac{q^k}{z}}\, j_{k-1}(x;c,d;z\mid q)
-\sqrt{1-\frac{cq^k}{d^2z}}\, j_k(x;c,d;z\mid q),
\end{split}
\]
for $x \in I(c,d;z\mid q)$.
\end{lemma}
\begin{proof}
A direct calculation, or see \cite[Exerc. 1.12]{GaspR}, shows that
\begin{equation}\label{eq:qdifffor2vp1}
f_k(x;c,d;z\mid q) - \frac{1}{d}\, f_{k+1}(x;c,d;z\mid q) =
z (1-2dx+d^2)\, f_k(x;qc,qd;z\mid q).
\end{equation}
Rewriting \eqref{eq:qdifffor2vp1} for the orthonormal basis $j_k(x ;c,d;z\mid q)$ then gives the desired contiguous relation. For $x=\cos\psi$ this is immediate from \eqref{eq:orthonorlqJacfun},
\eqref{eq:deflqJacfweight} and \eqref{eq:ortholqJacf}. For $x$ in the discrete spectrum it follows from
\[
\begin{split}
\frac{w_r(qc,qd;z\mid q)}{w_r(c,d;z\mid q)} &= d^2z^2 (1-cq^r)(1-d^2q^{-r}/c), \\
\frac{w^\prime_{r-1}(qc,qd;z\mid q)}{w^\prime_r(c,d;z\mid q)} &=
d^2z^2 (1- d^2q^r)(1-q^{-r}),\\
 \frac{v_{r-1}(qc,qd;z\mid q)}{v_r(c,d;z\mid q)} &=
d^2z^2 (1-d^2zq^{r-1})(1-q^{1-r}/z). \qedhere
\end{split}
\]
\end{proof}

Yet another result for the little $q$-Jacobi functions needed
in this paper is related to a symmetry property that
follows from Heine's transformation \cite[(1.4.6)]{GaspR} and
analytic continuation;
\begin{equation}\label{eq:HeineforlqJf}
{}_2\vp_1\left( \begin{array}{c} dy,\, d/y \\ c \end{array};q,zq^{-k}\right) =
\frac{(zd^2q^{-k}/c;q)_\infty}{(zq^{-k};q)_\infty}
{}_2\vp_1\left( \begin{array}{c} cy/d,\, c/dy \\ c
\end{array};q,q^{-k}\frac{zd^2}{c}\right).
\end{equation}
Together with \eqref{eq:2ndordqdiffeqBHS} and \eqref{eq:deflqJacfweight}
this implies the symmetry
\begin{equation}\label{eq:symmetryforlqJacf}
\rho_k(c,d;z\mid q)\, f_k(x;c,d;z\mid q) =
\rho_k(c,\frac{c}{d};\frac{zd^2}{c}\mid q)\,
f_k(x;c,\frac{c}{d};\frac{zd^2}{c}\mid q).
\end{equation}
The action on the parameters is an involution, and
$I(x;c,d;z\mid q) =I(c,c/d;zd^2/c\mid q)$. Moreover, we have
\[
\begin{split}
c(y;c,d;z\mid q) &=c(y;c,\frac{c}{d};\frac{zd^2}{c}\mid q), \\
v_k(c,d;z\mid q) &=v_k(c,\frac{c}{d};\frac{zd^2}{c}\mid q), \\
\ w_k(c,d;z\mid q) &=w^\prime_k(c,\frac{c}{d};\frac{zd^2}{c}\mid q),
\end{split}
\]
which implies
\begin{equation} \label{eq:symmetryjk1}
j_k(x;c,d;z\mid q) = j_k(x;c,c/d;zd^2/c\mid q).
\end{equation}
This shows that in the special case $d^2/c\in q^\Z$, we can transfer the multiplication by a power of $q$ in $z$ to a shift in the index $k$. Using
\eqref{eq:thetaprodid} we obtain for $p\in\Z$
\begin{equation*}
\begin{split}
\rho_k(c,d;zq^{-p}\mid q)\, f_k(x;c,d;zq^{-p}\mid q) &= (-dz)^p q^{-\hf p(p+1)} \,
\rho_{k+p}(c,d;z\mid q)\, f_{k+p}(x;c,d;z\mid q), \\
c(y;c,d;zq^{-p}\mid q) &= (-dzy)^p q^{-\hf p(p+1)} \, c(y;c,d;z\mid q),
\\
v_r(c,d;zq^{-p}\mid q) &= (dz)^{-2p} q^{p(p+1)} \, v_{r-p}(c,d;z\mid q),
\\
w_r(c,d;zq^{-p}\mid q) &= (dz)^{-2p} q^{p(p+1)} \, w_r(c,d;z\mid q),
\\
w^\prime_r(c,d;zq^{-p}\mid q) &= (dz)^{-2p} q^{p(p+1)} \, w^\prime_r(c,d;z\mid q).
\end{split}
\end{equation*}
Moreover, $I(c,d;zq^{-p}\mid q) =I(c,d;z\mid q)$ and so
\begin{equation} \label{eq:symmetryjk2}
j_k(x;c,d;zq^{-p}\mid q) =(\sgn(d))^p\,j_{k+p}(x;c,d;z\mid q).
\end{equation}
Combining gives the following special case
\begin{equation}\label{eq:specialsymmetryjk}
j_k(x;q,q^{\hf(1-p)};z\mid q) = j_{k+p}(x;q,q^{\hf(1+p)};z\mid q),
\end{equation}
for all $x\in I(q,q^{\hf(1-p)};z\mid q) = I(q,q^{\hf(1+p)};z\mid q)$.

\subsection{Explicit formulas for the function $A$} \label{ssecB:formulasforA}
Here we write out explicitly the functions $A=A(\,\cdot\,;p,m,\ep,\et)$, $p \in q^\Z$, $m \in \Z$ and $\ep,\et \in \{-,+\}$. These functions are used in \S\ref{ssec:generatorsofhatM} for the description of the polar decomposition of the elements $Q(p_1,p_2,n) \in \hat M$, and they are used later on in \S\ref{ssec:discreteseries} and \S\ref{ssec:principalseries} to describe explicitly the actions of the generators of $\hat M$ on $\cL_{p,x}$ in the discrete series and principal series corepresentations. The functions $A$ are essentially special cases of the $c$-functions for Al-Salam--Chihara polynomials and little $q$-Jacobi functions, divided by their absolute value. We only give the formulas for $A(\la)$ with $\la=e^{i\psi} \in \T_0$.

For $\ep=+$, $\et=-$,
\[
\begin{split}
A(\la;&p,m,+,-) = \\
&(-1)^m \la^{1-m-\chi(p)}\sqrt{\frac{2}{\pi|\sin \psi|}}\, \frac{ (q\la/p, -q^{1-2m}\la/p;q^2)_\infty }{ (\la^2;q^2)_\infty } \Bigg( \frac{ (\la^{\pm 2};q^2)_\infty }{ (q\la^{\pm 1}/p, -q^{1-2m}\la^{\pm 1}/p;q^2)_\infty } \Bigg)^\hf,
\end{split}
\]
and for $\ep=-$, $\et=+$,
\[
\begin{split}
A(\la;p,m,-,+) =
\la \sqrt{\frac{2}{\pi|\sin \psi|}}\, \frac{ (pq\la, -pq^{1+2m}\la ;q^2)_\infty }{ (\la^2;q^2)_\infty } \Bigg( \frac{ (\la^{\pm 2};q^2)_\infty }{ (pq\la^{\pm 1}, -pq^{1+2m}\la^{\pm 1};q^2)_\infty } \Bigg)^\hf.
\end{split}
\]
For $\ep=\et=-$,
\[
\begin{split}
A(\la;&p,m,-,-) =\\
& (-1)^{m+1} \la \sqrt{\frac{2}{\pi|\sin \psi|}}\, \frac{ (-pq\la, -pq^{1+2m}\la;q^2)_\infty }{ (\la^2;q^2)_\infty } \Bigg( \frac{ (\la^{\pm 2};q^2)_\infty }{ (-pq\la^{\pm 1}, -pq^{1+2m}\la^{\pm 1};q^2)_\infty } \Bigg)^\hf,
\end{split}
\]
for $\chi(p)+m \geq 0$, and for $\chi(p)+m<0$,
\[
\begin{split}
A&(\la;p,m,-,-) = \\
&(-1)^{m+1} \la^{1-m-\chi(p)}\sqrt{\frac{2}{\pi|\sin \psi|}}\, \frac{ (-q\la/p, -q^{1-2m}\la/p;q^2)_\infty }{ (\la^2;q^2)_\infty } \Bigg( \frac{ (\la^{\pm 2};q^2)_\infty }{ (-q\la^{\pm 1}/p, -q^{1-2m}\la^{\pm 1}/p;q^2)_\infty } \Bigg)^\hf.
\end{split}
\]
For $\ep=\et=+$,
\begin{align*}
A(\la;p,m,+,+) =\, &(-1)^{m+\chi(p)} \la^{-m-\chi(p)} \frac{ (-q\la/p, -pq^{1+2m}\la, pq^{3+2m}/\la, q^{-1-2m}\la/p;q^2)_\infty }{ (\la^2;q^2)_\infty } \\
& \times\sqrt{\frac{2}{\pi|\sin \psi|}}\, \Bigg( \frac{ (\la^{\pm 2};q^2)_\infty }{ (-q\la^{\pm 1}/p, -pq^{1+2m}\la^{\pm 1}, pq^{3+2m}\la^{\pm 1}, q^{-1-2m}\la^{\pm 1}/p;q^2)_\infty } \Bigg)^\hf,
\intertext{for $m \geq 0$, and for $m <0$,}
A(\la;p,m,+,+) =\, &(-1)^{m}\sqrt{\frac{2}{\pi|\sin \psi|}}\, \frac{ (-pq\la, -q^{1-2m}\la/p, q^{3-2m}/\la p, pq^{-1+2m}\la;q^2)_\infty }{ (\la^2;q^2)_\infty } \\
& \times \Bigg( \frac{ (\la^{\pm 2};q^2)_\infty }{ (-pq\la^{\pm 1}, -q^{1-2m}\la^{\pm 1}/p, q^{3-2m}\la^{\pm 1}/p, pq^{-1+2m}\la^{\pm 1};q^2)_\infty } \Bigg)^\hf.
\end{align*}

\section{Special case of a Jacobi operator} \label{app:jacobi}
In this section we study the special case $c=q$ of the Jacobi operator $L=L_c=L(c,d,z\mid q)$ defined by \eqref{eq:JacobioperatorreclqJacf}. For special choices of $c$, $d$ and $z$, the operator $L$ is a certain  restriction of $E_0^\dag E_0$ or the Casimir operator (see Section \ref{ssec:spectraldecompCasimir}). The operator $L(q,d,z\mid q)$ that we consider in this subsection corresponds to the case $\ep=\et=+$, $m=0$.

Let $\cF(\Z)$ be the space of complex-valued functions on $\Z$. We study the linear operator
$\tilde{L}_c \colon \cF(\Z) \rightarrow \cF(\Z)$, given by
\[
2\,(\tilde{L}_c u)_k = a_{k-1}(c)\,u_{k-1} + b_k(c)\, u_k + a_k(c)\, u_{k+1}
\]
for all $u \in \cF(\Z)$ and $k \in \Z$. The coefficients $a_k(c)$ and $b_k(c)$ are given by \eqref{eq:reclqJacf}, and we write $a_k(c), b_k(c)$ instead of $a_k, b_k$ to stress the dependence on the parameter $c$. Recall from Section  \ref{ssecB:litteqJacobi} that $d \in \R \setminus \{0\}$ and $z \in (-\infty,0)$, so that both terms in the square root are positive, and $a_k>0$ and $b_k \in \R$. We define the linear operator $L\colon \cK(\Z) \rightarrow \cK(\Z)$ as the restriction of $\tilde{L}_c$ to $\cK(\Z)$, the linear subspace of finite linear combinations of basis vectors, i.e., the subspace of compactly supported functions in $\cF(\Z)$. Then $(L_c, \cK(\Z))$ is an unbounded symmetric operator on the Hilbert space $\ell^2(\Z)$. Moreover, the unboundedness occurs as $k\to-\infty$, since in this case the coefficients $a_k(c)$ and $b_k(c)$ grow exponentially. Note that for $k\to\infty$ the coefficients $a_k(c)$ and $b_k(c)$ remain bounded.

In this subsection we need the Wronskian associated to the Jacobi operator $L$;
\begin{equation} \label{eq:Wronskian}
[u,v]_k = a_k \bigl( u_{k+1}v_k - u_k v_{k+1}\bigr),
\end{equation}
see \cite[(4.2.3)]{KoelLaredo}. Two eigenfunctions $u,v$ of $L$ are linearly independent if and only if $[u,v]\neq 0$.

The remainder of this subsection furnishes the proof the following result.
\begin{thm} \label{thm:Jacobioperatorforcasec=q}
Consider $u \in \ell^2(\Z)$ so that $\tilde{L}_q(u) \in \ell^2(\Z)$
and so that there exists a function $f\colon \R_{\geq 0} \rightarrow \C$
that is differentiable in $0$ and satisfies $f(0) \not= 0$ and
$u_{-k} = q^{\frac{k}{2}}\,f(q^k)$ for all $k \in \N$.
Then there exists a unique self-adjoint extension $T$ of $L_q$
so that $u \in D(T)$. Moreover, if $v \in \ell^2(\Z)$, $\tilde{L}_q(v) \in \ell^2(\Z)$
and if there exists a function $g \colon \R_{\geq 0} \rightarrow \C$ that is differentiable in $0$ and satisfies $v_{-k} = q^{\frac{k}{2}}\,g(q^k)$ for all $k \in \N$, then $v \in D(T)$ as
well.

The resolution of the identity $E_{T}$ for the self-adjoint extension $T$ of $L_q$ is given by
\[
\langle E_{T}({\mathcal B})u,v\rangle_{\ell^2(\Z)} =
\int_{\mathcal B} {\mathcal F}_Tu(x) \overline{{\mathcal F}_Tv(x)}\,
d\nu(x;q,d;z\mid q), \qquad {\mathcal F}_Tu(x) =
\sum_{k=-\infty}^\infty u_k \rho_k f_k(x),
\]
for any Borel set ${\mathcal B}\subset \R$
and any $u=\sum_k u_ke_k, v=\sum_k v_ke_k\in \ell^2(\Z)$.
\end{thm}
Observe that the resolution of the identity is the same as in Theorem \ref{thm:spectraldecompL} with $c=q$.\\

For the proof we need the eigenfunctions of the operator $L_c$. For $c \in (0,1)$ and $y \in \C \setminus \{0\}$, let us denote
\begin{equation} \label{eq:tilde f}
\begin{split}
\tilde{f}(c,y)_k &= \rho_k(c,d;z\mid q) f_k(\mu(y);c,d;z\mid q),\\
\tilde{g}(c,y)_k &= \rho_k(c,d;z\mid q) f_k(\mu(y);q^2/c,qd/c;z\mid q),
\end{split}
\end{equation}
where $f_k$ and $\rho_k$ are defined by \eqref{eq:2ndordqdiffeqBHS} and \eqref{eq:deflqJacfweight}, respectively. From Section \ref{ssecB:litteqJacobi} we know that $\tilde f$ and $\tilde g$ are both solutions of the eigenvalue equation $L_cu = \mu(y) u$. Another solution is the function
\[
\tilde{F}(c,y)_k = \rho_k(c,d;z\mid q) F_k(y;c,d;z\mid q),
\]
see \eqref{eq:2ndordqdiffeqBHS3} for the definition of $F_k$.

In \cite[Section 4.5]{KoelLaredo} it is shown that the operator $L_c$ has deficiency indices $(1,1)$ in case $q^2<c<1$, $q\neq c$. The proof of this fact relies on the fact that the functions $\tilde f(c,y)$ and $\tilde g(c,y)$ are both in the space $\{ u \mid L^* u = zu, \ \sum_{k=-\infty}^0 |u_k|^2<\infty\}$ for $z=\mu(y) \in \C\setminus\R$. In case $c=q$, we have $\tilde{f}(q,y) = \tilde{g}(q,y)$, so we must provide another eigenvector for $\tilde{L}_q$.

\begin{defn}\label{def:eigenfunction h}
Let $y \in \C \setminus \{0\}$. We define $\tilde{h}(y) \in \cF(\Z)$
\begin{equation*}
\tilde{h}(y)_k = \lim_{c \rightarrow q} \, \frac{ \tilde{f}(c,y)_k -
\tilde{g}(c,y)_k }{c - q},
\hspace{4ex} \text{ for all } k \in \Z.
\end{equation*}
\end{defn}

For $c \in (q^2,1)$, we have
\[
\tilde{L}_c \left(\, \frac{\tilde{f}(c,y) -
\tilde{g}(c,y)}{c-q}\,\right) =
\mu(y)\,\,\frac{\tilde{f}(c,y) -
\tilde{g}(c,y)}{c-q} \ .
\]
Since the coefficients $a_k(c), b_k(c)$ of $\tilde{L}_c$
depend continuously on $c$, and $\mu(y)$ is independent of $c$, the
above equality together with Definition \ref{def:eigenfunction h}
imply that $\tilde{L}_q \,\tilde{h}(y) = \mu(y)\, \tilde{h}(y)$.

Let us establish the asymptotics of $\tilde{h}(y)_k$ as $k \rightarrow -\infty$.
\begin{lemma} \label{lem:h=r+f}
Consider $y \in \C \setminus \{0\}$. Then there exists a convergent sequence $(r_k)_{k=1}^\infty$ in $\C$ and a differentiable function $f : \R^+ \rightarrow \C$ such that $f(0) \not= 0$ and
$\tilde{h}(y)_{-k} = q^{\frac{k}{2}}\,(r_k + k\,f(q^k))$ for all $k \in \N$.
\end{lemma}

\begin{proof}
Define the $C^\infty$-functions
$B,C \colon (q^2,1) \times [0,\infty) \rightarrow \C$ such that
\[
B(c,x) = \,\rphis{2}{1}{dy, d/y}{c}{q , z x} \hspace{5ex}
\text{and} \hspace{5ex} C(c,x) =
\,\rphis{2}{1}{qdy/c, qd/yc}{q^2/c}{q ,z x}
\]
for all $c \in (q^2,1)$, $x \in \R^+$.
We have for $c \in (q^2,1)$, $k\in\Z$,  that
\[
\tilde{f}(c,y)_{-k} - \tilde{g}(c,y)_{-k} = w_{-k}(c)\,(\,B(c,q^k) -
(q/c)^k\, C(c,q^k)\,),
\]
where
\[
w_k(c) = (c,z,q/z;q)_\infty d^{-k} \rho_k(c,d;z\mid q).
\]
Therefore,
\begin{equation*}
\tilde{h}(y)_{-k} =  w_{-k}(q)\,\bigl(\,(\partial_1 B)(q,q^k) -
(\partial_1 C)(q,q^k) + (k/q)\, C(q,q^k)\,\bigr)
\end{equation*}
Now define the $C^\infty$-function
$D \colon (q^2,1) \times [0,\infty) \rightarrow \C$ such that
\[
D(c,x) =(c;q)_\infty \sqrt{\frac{(z x;q)_\infty \,(d^2 z/c, qc/d^2 z,z,q/z;q)_\infty}
{(d^2 z x/c;q)_\infty}} \hspace{5ex} \text{ for
all } c \in (q^2,1),\, x \in \R^+ \ .
\]
Now \eqref{eq:deflqJacfweight} shows that $w_{-k}(c) = c^{\frac{k}{2}}\,D(c,q^k)$ for
all $c \in (q^2,1)$. Thus,
\begin{equation*}
\tilde{h}(y)_{-k}  =  q^{\frac{k}{2}}\, D(q,q^k)\,\big(
(\partial_1 B)(q,q^k) - (\partial_1 C)(q,q^k) \big)\,
 +\, q^{-1}\, k\,q^{\frac{k}{2}}\,\,D(q,q^k)\, C(q,q^k)\, \, .
\end{equation*}
Note that
$\displaystyle{q^{-1}\, D(q,0) C(q,0) = q^{-1}\, D(q,0) =
q^{-1}(c;q)_\infty\,\sqrt{(d^2 z/q, q^2/d^2z,z,q/z;q)_\infty}
> 0}$. So the lemma follows.
\end{proof}

\begin{lemma} \label{lem:F=r+h}
Let $y \in \C \setminus \R$, $|y| < 1$. Then
$\tilde{F}(q,y)$ belongs to $\ell^2(\Z)$ and there exists a
convergent sequence $(r_k)_{k=1}^\infty$ in $\C$ and a
differentiable function $h : \R^+ \rightarrow \C$ so that
$h(0) \not= 0$ and
$\tilde{F}(q,y)_{-k} = q^{\frac{k}{2}}\,(r_k + k\,h(q^k))$
for all $k \in \N$.
\end{lemma}

\begin{proof}
Definition \eqref{eq:2ndordqdiffeqBHS} and \eqref{eq:deflqJacfweight} imply
that $\tilde{f}(q,y)_{-k}/q^{\frac{k}{2}}$ converges as
$k \rightarrow \infty$. Since $\tilde{f}(q,y)_{-k}/k q^{\frac{k}{2}}$
converges to $0$ as $k \rightarrow 0$ and, by Lemma \ref{lem:h=r+f},
$\tilde{h}(y)_{-k}/k q^{\frac{k}{2}}$ converges to a non-zero number
as $k \rightarrow 0$, we conclude
that $\tilde{f}(q,y)$ and $\tilde{h}(y)$ are linearly independent.

Because $\tilde{f}(q,y)$, $\tilde{h}(y)$ and
$\tilde{F}(q,y)$ belong to the eigenspace of
$\tilde{L}_q$ for the eigenvalue $\mu(y)$, and since such an eigenspace is always two-dimensional,
there exist complex numbers $\la$ and $\nu$ so that
$\tilde{F}(q,y) = \la\, \tilde{f}(q,y) + \nu \,\tilde{h}(y)$.
Clearly, this gives $[\tilde{f}(q,y),\tilde{F}(q,y)] = \nu\,[\tilde{f}(q,y),\tilde{h}(y)]$, see \eqref{eq:Wronskian}. By \cite[last Eq.~of (4.5.4)]{KoelLaredo} we know that
$[\tilde{f}(q,y),\tilde{F}(q,y)] \not= 0$, implying that
$\nu \not= 0$. Hence, Lemma \ref{lem:h=r+f} and the remarks
in the beginning of this proof guarantee the existence of a
convergent sequence $(r_k)_{k=1}^\infty$ in $\C$
and a differentiable function $h : \R^+ \rightarrow \C$
so that $h(0) \not= 0$ and
$\tilde{F}(q,y)_{-k} = q^{\frac{k}{2}}\,(r_k + k\,h(q^k))$
for all $k \in \N$. So we immediately get that
$\tilde{F}(q,y)_k$ is $\ell^2$ as $k \rightarrow -\infty$.
Definition \eqref{eq:2ndordqdiffeqBHS3} and \eqref{eq:deflqJacfweight} imply that
$\tilde{F}(q,y)_k$ is $\ell^2$ as $k \rightarrow \infty$, since $|y|<1$.
So we conclude that $\tilde{F}(q,y) \in \ell^2(\Z)$.
\end{proof}

Note that Lemma \ref{lem:F=r+h} applies to
$y=(1-\sqrt{2}\,)\,i$, so $\mu(y) = i$. Since $\tilde{F}(q,y)$
belongs to $\ell^2(\Z)$, the vector $\tilde{F}(q,y)$ belongs to
$D(L_q^*)$ and $L_q^*(\tilde{F}(q,y)) = i\,\tilde{F}(q,y)$.
This implies that $L_q$ is not essentially self-adjoint.

\begin{lemma} \label{lem:r(fg-gf)}
Let $f,g : \R \rightarrow \C$ be functions that are
differentiable in $0$,  $(r_k)_{k=1}^\infty$ a sequence in $\R$
such that $(r_k q^k)_{k=1}^\infty$ converges to $0$.
Then
$\bigl(\,r_k\,(f(q^{k-1})\,g(q^k)- f(q^k)\,g(q^{k-1}))\,\bigr)_{k=1}^\infty$
converges to $0$.
\end{lemma}

\begin{proof}
For $k \in \N$, write
\[
\begin{split}
r_k\,(f(q^{k-1})&\,g(q^k)- f(q^k)\,g(q^{k-1})) = \\
&(r_k q^k)\,\,\left(\ \frac{f(q^{k-1}) - f(q^k)}{q^k}\,\,g(q^k)
+ f(q^k)\,\,\frac{g(q^{k}) - g(q^{k-1})}{q^k}\ \right)
\end{split}
\]
and observe that
\[
\left(\,\frac{f(q^{k-1}) - f(q^k)}{q^k}\,\right)_{k=1}^\infty
\hspace{10ex} \text{and} \hspace{10ex} \left(\,\frac{g(q^{k}) -
g(q^{k-1})}{q^k}\,\right)_{k=1}^\infty
\]
are bounded because $f$ and $g$ are differentiable in $0$.
\end{proof}

We are now ready to prove Theorem \ref{thm:Jacobioperatorforcasec=q}.

\begin{proof}[Proof of Theorem \ref{thm:Jacobioperatorforcasec=q}]
We set $y = (1-\sqrt{2}\,)\,i$, then $\mu(y) = i$. Consider $\la \in \T$.
We define a linear operator $T_\la$ in $\ell^2(\Z)$ such that
$$
D(T_\la) = \Big\{\,w \in \ell^2(\Z)  \mid \tilde{L}_q(w)
\in \ell^2(\Z) \text{ and } \lim_{k \rightarrow
\infty}\,[w,\la\,\tilde{F}(q,y) +
\bar{\la}\,\tilde{F}(q,\bar{y})]_{-k} = 0 \Big\}
$$
and $T_\la$ is the restriction of $\tilde{L}_q$ to $D(T_\la)$. Here we use the Wronskian $[\cdot,\cdot]$ defined by \eqref{eq:Wronskian}. We know by
\cite[Lemma~(4.2.3)]{KoelLaredo} that $T_\la$ is
a self-adjoint extension of $L_q$ and that every self-adjoint
extension arises in this way.

By Lemma \ref{lem:F=r+h} there exists a convergent sequence
$(r_k)_{k=1}^\infty$ in $\C$, a differentiable function
$h \colon \R^+ \rightarrow \C$ such that $h(0) \not= 0$
and $\tilde{F}(q,y)_{-k} = q^{\frac{k}{2}}\,(r_k + k\,h(q^k))$ for
all $k \in \N$.
Take $v \in \ell^2(\Z)$ such that $\tilde{L}_q(v) \in \ell^2(\Z)$
and such that there exists a function
$g \colon \R^+ \rightarrow \C$ that is differentiable in 0 and
satisfies  $v_{-k} = q^{\frac{k}{2}}\,g(q^k)$ for all $k \in \N$. Let
us calculate $\lim_{k \rightarrow \infty}\,[v,\tilde{F}(q,y)]_{-k}$.

For $k \in \N$,
\[
\begin{split}
 q^{-k+\frac{1}{2}}&\,\bigl(v_{-k+1} \,
\tilde{F}(q,y)_{-k} - v_{-k}\, \tilde{F}(q,y)_{-k+1}\bigr)
\\
&  = g(q^{k-1})\,(r_k + k\, h(q^k)) -
g(q^k)\,(r_{k-1} +  (k-1)\, h(q^{k-1}))
\\
& = (\,g(q^{k-1})\,r_k -
g(q^k)\,r_{k-1}\,) +  k\, (\,g(q^{k-1})\,h(q^k) -
g(q^k)\,h(q^{k-1})\,) +  g(q^k)\,h(q^{k-1}).
\end{split}
\]
The  first term converges to $0$, since $g$ is continuous
and $\{ r_k\}_k$ is convergent.
Since $(k\,q^k)_{k=1}^{\infty}$ converges to $0$, Lemma
\ref{lem:r(fg-gf)} implies that the second term of the above sum
converges to $0$ as $k \rightarrow \infty$. Therefore
the above expression converges to
$g(0)\,h(0)$ as $k \rightarrow \infty$.
Since $a_{-k}(c) = q^{1-k}\frac{\sqrt{c}}{d|z|}(1+{\mathcal O}(q^k))$,
this implies that
$\lim_{k \rightarrow \infty} [v,\tilde{F}(q,y)]_{-k}
= \frac{q}{d |z|}\,g(0)\,h(0)$. Since
$\overline{\tilde{F}(q,y)_k} = \tilde{F}(q,\bar{y})_k$
for all $k \in \Z$ by the assumptions $z<0$ and
$d\in\R\backslash\{ 0\}$,
we see that
\begin{equation} \label{eq:[v,laF+laF]}
\begin{split}
\lim_{k \rightarrow \infty} [v,\la \tilde{F}(q,y)
+ \bar{\la} \tilde{F}(q,\bar{y})]_{-k} &=
\frac{q}{d|z|}\,g(0)\,(\la\,h(0) +
\overline{\la\,h(0)}\,) \\
&= \frac{2\,q}{d|z|}\,g(0)\,\Re(\la\,h(0))\ .
\end{split}
\end{equation}

If we use this equality for $v=u$ and $g=f$, we see that $u$ belongs to the domain of $T_\la$ if and only if $\Re(\la\,h(0)) = 0$. Notice that such a $\la$ clearly exists and is determined up to a sign, but that $T_\la = T_{-\la}$. So we have proved the existence and uniqueness of the self-adjoint extension $T$. Equation \eqref{eq:[v,laF+laF]} also guarantees that an element $v$ satisfying the properties described in the lemma belongs to $D(T)$.

The spectral decomposition of a self-adjoint extension $T$ of the Jacobi operator $L_c$, for $0<c\leq q^2$, is determined in \cite[\S4.5]{KoelLaredo} from eigenfunctions $\Phi_y$ and $\phi_y$ for eigenvalue $\mu(y)$, $0<|y|<1$, such that $\Phi(y) \in \ell^2(\N)$ and $\phi(y) \in \ell^2(-\N)$, see \cite[\S4.3.2]{KoelLaredo}. Here $\phi(y)$, extended to $\ell^2(\Z)$ by setting $\phi(y)_k=0$ for $k\geq 0$, must be an element of the domain of $T$. In case $0<c \leq q^2$ we have $\Phi(y)=\tilde F(c,y)$ and $\phi(y) = \tilde f(c,y)$, and these functions determine the spectral decomposition of $L_c$ from Theorem \ref{thm:spectraldecompL}. In order to find the spectral decomposition of $T_\lambda$ we need to find the right choices of $\Phi(y)$ and $\phi(y)$ in this case. Note that there is only one eigenfunction of $L_q$ for eigenvalue $\mu(y)$ in $\ell^2(\N)$, namely $\tilde F(q,y)$, so $\Phi(y) = \tilde F(q,y)$. There are two eigenfunctions in $\ell^2(-\N)$, namely $\tilde f(q,y)$ and $\tilde h(q,y)$, so $\phi(y)$ is a linear combination of these two functions. We show that $\phi(y) = \tilde f(q,y)$ is the right choice for $\phi(y)$ here. This implies that the spectral decomposition of $T_\lambda$ is the same as the spectral decomposition of $L$ from Theorem \ref{thm:spectraldecompL} (with $c=q$, of course). We only need to show that $\phi(y) \in D(T_\lambda)$, so it suffices to show that there exists a function $g:\R_{\geq 0}\rightarrow \C$, differentiable in $0$, such that $\tilde f(q,y)_{-k} = q^{k/2} g(q^k)$ for all $k \in \N$. But this follows directly from the definition of $\tilde f(q,y)$, see \eqref{eq:tilde f}, \eqref{eq:2ndordqdiffeqBHS} and \eqref{eq:deflqJacfweight}.
\end{proof}

\section{Proofs of some lemmas}\label{app:proofsoflemmas}

\subsection{Proof of Lemma \ref{lem:firststepEaffiliatedhatM}}
\label{ssecC:pflemfirststepEaffiliatedhatM}
We prove the following result. Let $p_1,p_2 \in I_q$ and $n \in \Z$. Then
\[
\langle \, \hat{J}\,Q(p_1,p_2,n) \hat{J} v , E_0^\dag \, w \rangle =
\langle \, \hat{J}\,Q(p_1,p_2,n) \hat{J} E_0 \, v ,  w \rangle, \qquad
\forall\, v,w \in \cK_0.
\]

\begin{proof}
Assume first that $v = f_{mpt}$ and $w = f_{lrs}$
for  $m,l \in \Z$ and $p,t,r,s \in I_q$. Then
\eqref{eq:dualE0forsu}, \eqref{eq:matrixelementsofQppnandhatJQppnhatJ}
and the last symmetry of \eqref{eq:symmetryforapxy}
imply
\begin{equation*}
\begin{split}
&(q - q^{-1}) \, \langle \, \hat{J}\,Q(p_1,p_2,n) \hat{J} v ,
E_0^\dag \, w \rangle \\
= &\  \de_{\chi(p_1 p/p_2 t),-m-n}\,
\de_{m+n-1,l}\,\de_{\sgn(pt) (p_2/p_1) q^{-m+1} s,r}
\ |p_1 p_2/p|\, (-1)^m \, \sgn(p)^{\chi(p)}\,\sgn(t)^{\chi(t)}
\\ &    \ \ \ \, \times\,
\Bigl[\,\sgn(s)\,q^{-\frac{l+1}{2}}\,|r/s|^{\frac{1}{2}}
\,\sqrt{1+\kappa(s)}\, |q
s|^{-1}\,a_{p_1}(t,qs)\,a_{p_2}(p,r)
\\ &    \ \ \ \ \ \ \ \ \ -
\,\sgn(r)\,q^{\frac{l+1}{2}}\,|s/r|^{\frac{1}{2}}\,
\sqrt{1+\kappa(q^{-1} r)}\,|s|^{-1}\,a_{p_1}(t,s)
\,a_{p_2}(p,q^{-1} r)
\,\Bigr].
\end{split}
\end{equation*}
Because of the presence of the three Kronecker deltas,
we can replace $|r/s|q^{-l-1}$ by $|p/t|q^{1-m}$. This
gives
\begin{equation}\label{eq:1pflemfirststepEaffiliatedhatM}
\begin{split}
&(q - q^{-1}) \, \langle \, \hat{J}\,Q(p_1,p_2,n) \hat{J} v ,
E_0^\dag \, w \rangle \\
= &\  \de_{\chi(p_1 p/p_2 t),-m-n}\,
\de_{m+n-1,l}\,\de_{\sgn(pt) (p_2/p_1) q^{-m+1} s,r}
\ |p_1 p_2/p|\, (-1)^m \, \sgn(p)^{\chi(p)}\,\sgn(t)^{\chi(t)}
\\ &    \ \ \ \, \times\,
\Bigl[\,\sgn(s)\,q^{-\frac{m-1}{2}}\,|p/t|^{\frac{1}{2}}
\,\sqrt{1+\kappa(s)}\, |q
s|^{-1}\,a_{p_1}(t,qs)\,a_{p_2}(p,r)
\\ &    \ \ \ \ \ \ \ \ \ -
\,\sgn(r)\,q^{\frac{m-1}{2}}\,|t/p|^{\frac{1}{2}}\,
\sqrt{1+\kappa(q^{-1} r)}\,|s|^{-1}\,a_{p_1}(t,s)
\,a_{p_2}(p,q^{-1} r)
\,\Bigr].
\end{split}
\end{equation}

For the other side of the required equation we similarly
derive from \eqref{eq:defE0forsu}, \eqref{eq:matrixelementsofQppnandhatJQppnhatJ}
and the last symmetry of \eqref{eq:symmetryforapxy}
that
\[
\begin{split}
&(q - q^{-1}) \, \langle \, \hat{J}\,Q(p_1,p_2,n) \hat{J} E_0\,v , w \rangle
 \\
=\, &\,  \de_{\chi(p_1 p/p_2 t),-m-n}\,\,
\de_{m+n-1,l}\,\de_{\sgn(pt) (p_2/p_1) q^{-m+1}
s,r}\,  \ |p_1 p_2/p|\, (-1)^m \,
\sgn(p)^{\chi(p)}\,\sgn(t)^{\chi(t)}
\\ &  \ \ \ \, \times\ \Bigl[\,
-q^{-\frac{m-1}{2}}\,|p/t|^{\frac{1}{2}}\,\sqrt{1+\kappa(q^{-1} t)}\,
|s|^{-1}\,a_{p_1}(q^{-1} t,s)\,a_{p_2}(p,r)
\\ &   \  \ \ \ \ \ \ \ \ \ + \,
q^{\frac{m-1}{2}}\,|t/p|^{\frac{1}{2}}\,\sqrt{1+\kappa(p)}\,
|qs|^{-1}\,a_{p_1}(t,s)\,a_{p_2}(qp,r)
\,\Bigr].
\end{split}
\]
Comparing this expression with
\eqref{eq:1pflemfirststepEaffiliatedhatM} we see that
we need the $q$-contiguous relations of Lemma
\ref{lemB:qcontiguousrelapxy}. Using the
first equality of Lemma \ref{lemB:qcontiguousrelapxy}
for $a_{p_1}(q^{-1} t,s)$ and the second equality of
Lemma \ref{lemB:qcontiguousrelapxy}
for $a_{p_2}(qp,r)$ gives
\begin{equation*}
\begin{split}
&(q - q^{-1}) \, \langle \, \hat{J}\,
Q(p_1,p_2,n) \hat{J} E_0\,v , w \rangle
\\ &  =  \de_{\chi(p_1 p/p_2t),-m-n}\,\,
\de_{m+n-1,l}\,\,\de_{\sgn(pt)
(p_2/p_1) q^{-m+1} s,r} \
 |p_1 p_2/p|\, (-1)^m \, \sgn(p)^{\chi(p)}\,\sgn(t)^{\chi(t)}
\\ &   \ \ \ \times \   \Bigl[\,
\sgn(s)\, q^{-\frac{m-1}{2}}\,|p/t|^{\frac{1}{2}}\,\sqrt{1+\kappa(s)}\,
|qs|^{-1}\,a_{p_1}(t,qs)\,a_{p_2}(p,r)
\\ &    \ \ \ \ \ \ \ \ \ - \,
(st/qp_1)\, q^{-\frac{m-1}{2}}\,|p/t|^{\frac{1}{2}}
\,|s|^{-1}\, a_{p_1}(t,s)\,a_{p_2}(p,r)
\\ &    \ \ \ \ \ \ \ \ \ - \,\sgn(r)\,
q^{\frac{m-1}{2}}\,|t/p|^{\frac{1}{2}}\,
\sqrt{1+\kappa(q^{-1}r)}\,|s|^{-1}\,a_{p_1}(t,s)
\,a_{p_2}(p,q^{-1}r)
\\ &    \ \ \ \ \ \ \ \ \ + \,(pr/q p_2)
\,q^{\frac{m-1}{2}}\,|t/p|^{\frac{1}{2}}\,|s|^{-1}
\,a_{p_1}(t,s)\,a_{p_2}(p,r)\,\Bigr]\ .
\end{split}
\end{equation*}
Comparing this expression with \eqref{eq:1pflemfirststepEaffiliatedhatM}
we see that
\begin{equation*}
\begin{split}
& (q - q^{-1}) \, \langle \, \hat{J}\,Q(p_1,p_2,n) \hat{J} E_0\,v , w \rangle
 = (q - q^{-1}) \, \langle \, \hat{J}\,Q(p_1,p_2,n) \hat{J} v , E_0^\dag \, w
\rangle
\\ &   +\, \de_{\chi(p_1 p/p_2 t),-m-n}\,\,
\de_{m+n-1,l}\,\,\de_{\sgn(pt) (p_2/p_1) q^{-m+1}
s,r} \  |p_1 p_2/p|\, (-1)^m \, \sgn(p)^{\chi(p)}\,\sgn(t)^{\chi(t)}
\, |pt|^{\frac{1}{2}}\,
\\ &  \ \ \ \ \ \ \times \ |qs|^{-1}\, a_{p_1}(t,s)\,a_{p_2}(p,r)\
\Bigl[\,- \sgn(t)\, (s/p_1)\, q^{-\frac{m-1}{2}} + \sgn(p)\,(r/p_2)
\,q^{\frac{m-1}{2}}\,\Bigr]
\end{split}
\end{equation*}
If the Kronecker $\delta$-function $\de_{\sgn(pt) (p_2/p_1) q^{-m+1} s,r}$ is non-zero, then the term in square brackets equals $0$, thus
$\langle \, \hat{J}\,Q(p_1,p_2,n) \hat{J} E_0\,v , w \rangle =
\langle \, \hat{J}\,Q(p_1,p_2,n) \hat{J} v , E_0^\dag \,
w \rangle$ for $v = f_{mpt}$ and $w = f_{lrs}$. By linearity the lemma holds for all $v,w \in \cK_0$.
\end{proof}

\subsection{Proof of Lemma \ref{lem:sgncommutationQppnwithCasimir0}} \label{app:proofofcommutionQwithCasimir}
Here we prove the following result: For $u,v\in\cK_0$, $p_1,p_2\in I_q$ and $n\in\Z$, we have
\begin{equation} \label{eq:proofoflemma}
\langle Q(p_1,p_2,n)\, u, \Om_0\, v\rangle
= \sgn(p_1p_2)\, \langle Q(p_1,p_2,n)\,\Om_0\, u, v\rangle.
\end{equation}

The proof depends on properties of the functions $a_p(\cdot,\cdot)$. One of the properties is the second-order $q$-difference equation from Lemma \ref{lem:2ndorderqdifferenceq}. The other properties we need are essentially the contiguous relations from Lemma \ref{lemB:qcontiguousrelapxy}. We state these relations in the following lemma.

\begin{lemma} \label{lem:contiguous}
Consider $x,y, p \in I_q$, then
\[
\sqrt{1+\kappa(y/q)}\,\, a_p(x,y/q) =
\frac{py}{qx}\,\,a_p(x,y) + \sqrt{1+\kappa(p)}\,\,a_{qp}(x,y),
\]
and
\[
\sqrt{1+\kappa(y)}\,\, a_p(x,q y) = \frac{py}{x}\,\,a_p(x,y) +
\sqrt{1+\kappa(p/q)}\,\,a_{p/q}(x,y) \ .
\]
\end{lemma}
\begin{proof}
One uses the last equation of \eqref{eq:symmetryforapxy} to write $a_p(x,q^{-1} y)$ in terms of $a_x(p,q^{-1} y)$. Then apply the second relation of Lemma \ref{lemB:qcontiguousrelapxy} and use \eqref{eq:symmetryforapxy} again to obtain the first equality.
The second equality is proved in the same way using the first relation of Lemma \ref{lemB:qcontiguousrelapxy}.
\end{proof}

\begin{proof}[Proof of \eqref{eq:proofoflemma}]
Let $l,m,n \in \Z$ and $p_1,p_2,p,r,\si,\tau \in I_q$. We will establish
\begin{equation} \label{eq:crucialpartincommutationlemma}
\langle  Q(p_1,p_2,n)\,f_{m,p,t} , \Om_0\,f_{l,r,s}\rangle =
\sgn(p_1 p_2)\,\, \langle\, Q(p_1,p_2,n)\,\Om_0\,f_{m,p,t} , f_{l,r,s} \rangle\ .
\end{equation}
by writing out both sides of this identity in terms of matrix coefficients \eqref{eq:matrixelementsofQppnandhatJQppnhatJ} of $Q(p_1,p_2,n)$.

Let us first consider the left hand side, which we call $S_L$ for convenience, of \eqref{eq:crucialpartincommutationlemma}. From the explicit action \eqref{eq:actioncasimir} of $\Om_0$ on $f_{mpt}$ we find
\[
\begin{split}
2S_L =&\ (q^{l-1} r\,|s| + q^{-l-1} s\,|r|\,) \,
\langle  Q(p_1,p_2,n)\,f_{m,p,t} , f_{l,r,s}\rangle\\
 &   - \sgn(rs)\,\,\sqrt{(1+\kappa(r))(1+\kappa(s))}\,\,
\langle  Q(p_1,p_2,n)\,f_{m,p,t} , f_{l,qr,qs}\rangle\\
&   - \sgn(rs)\,\,\sqrt{(1+\kappa(q^{-1}r))(1+\kappa(q^{-1}s))}\,\,
\langle  Q(p_1,p_2,n)\,f_{m,p,t} , f_{l,q^{-1}r,q^{-1}s}\rangle.
\end{split}
\]
In terms of the matrix coefficients \eqref{eq:matrixelementsofQppnandhatJQppnhatJ} of $Q(p_1,p_2,n)$, we have
\[
\begin{split}
2S_L =&\  \de_{|p_1 p/p_2 t|,q^{m-n}}\,\de_{m-n,l}\, \de_{\sgn(pt)(p_2/p_1)q^m s ,r}
\\  & \times \Big[\, (q^{l-1} r\,|s| + q^{-l-1} s\,|r|\,) \,\,
\left|\frac{t}{s}\right|\,a_t(p_1,s)\,a_p(p_2,r) \\
&  \quad - \sgn(rs)\,\sqrt{(1+\kappa(r)) (1+\kappa(s))} \,\,\left|\frac{t}{qs}\right|\,\,a_t(p_1,qs)\,a_p(p_2,qr) \\
& \quad - \sgn(rs)\,\sqrt{(1+\kappa(r/q))(1+\kappa(s/q))}\,\,
\left|\frac{tq}{s}\right|\,\,a_t(p_1,s/q)\,a_p(p_2,r/q)\,\Big].
\end{split}
\]
From the $q$-contiguous relations of Lemma \ref{lem:contiguous} it follows that
\[
\begin{split}
&\sqrt{(1+\kappa(r)) (1+\kappa(s))}\,a_t(p_1,qs)\,a_p(p_2,qr) = \\
&\quad\Big(\frac{ts}{p_1}\,a_t(p_1,s) + \sqrt{1+\kappa(t/q)}\,a_{t/q}(p_1,s)\Big)
\Big(\frac{pr}{p_2}\,a_p(p_2,r) + \sqrt{1+ \kappa(p/q)}\,a_{p/q}(p_2,r)\Big)
\end{split}
\]
and
\[
\begin{split}
&\sqrt{(1+\kappa(r/q))(1+\kappa(s/q))}\, a_t(p_1,s/q)\,a_p(p_2,r/q) = \\
&\quad\Big(\frac{ts}{q p_1}\,a_t(p_1,s) + \sqrt{1+\kappa(t)}\,a_{qt}(p_1,s)\Big)
\Big(\frac{pr}{q p_2}\,a_p(p_2,r) + \sqrt{1+ \kappa(p)}\,a_{q p}(p_2,r)\Big),
\end{split}
\]
which implies
\[
\begin{split}
2S_L =& \ \de_{|p_1 p/p_2 t|,q^{m-n}}\,\de_{m-n,l}\, \de_{\sgn(pt)(p_2/p_1)q^m s,r} \\
&  \Bigg[\, (q^{l-1} r\,|s| + q^{-l-1} s\,|r|\,) \, \left|\frac{t}{s}\right|\,a_t(p_1,s)\,a_p(p_2,r) \\
&  -\sgn(rs)\,\left|\frac{t}{s}\right|\,q^{-1}\, \sqrt{(1+\kappa(t/q))(1+\kappa(p/q))} \,\,a_{t/q}(p_1,s)\,a_{p/q}(p_2,r) \\
& -\sgn(rs)\,\left|\frac{t}{s}\right| q\,\sqrt{(1+ \kappa(t))(1+\kappa(p))} \,a_{qp}(p_1,s)\,a_{qp}(p_2,r)\\
&  -\sgn(rs)\,\left|\frac{t}{s}\right|\Big( \frac{pr}{qp_2} A_t(p_1,s) a_p(p_2,r)+ \frac{ts}{qp_1} A_p(p_2,r) a_t(p_1,s) \Big)\ \Bigg],
\end{split}
\]
where
\[
\begin{split}
A_x&(y,z)=
\frac{xz}{y}\,a_x(y,z) + \sqrt{1+\kappa(x/q)}\,\,a_{x/q}(y,z) + q\,\sqrt{1+\kappa(x)} \,a_{qx}(y,z).
\end{split}
\]
The expression of $A_x(y,z)$ simplifies by Lemma \ref{lemB:qcontiguousrelapxy} to
\[
\frac{xz}{y}A_x(y,z) = \big(\kappa(z)-\kappa(x)\big) a_x(y,z).
\]
Since $\de_{|p_1 p/p_2 t|,q^{m-n}}\,\de_{m-n,l}\, \de_{\sgn(pt)(p_2/p_1)q^m s,r} = 0$
unless $p r = \sgn(p_1 p_2)\, q^{m+l} st p_2^2/p_1^2$, we now get
\[
\begin{split}
 2 S_L &= \de_{|p_1 p/p_2 t|,q^{m-n}}\,\de_{m-n,l}\, \de_{\sgn(pt)(p_2/p_1)q^m s,r}\\
&  \Big[ (q^{l-1} r\,|s| +
q^{-l-1} s\,|r|\,) \,\,\left|\frac{t}{s}\right|\,a_t(p_1,s)\,a_p(p_2,r) \\
&  - \sgn(rs)\,\left|\frac{t}{s}\right|\,q^{-1}\,\sqrt{(1+\kappa(t/q))
  (1+\kappa(p/q))}\,\,a_{t/q}(p_1,s)\,a_{p/q}(p_2,r)\\
& - \sgn(rs)\,\left|\frac{t}{s}\right|\, q\,\sqrt{(1+ \kappa(t))(1+\kappa(p))} \,a_{qp}(p_1,s)\,a_{qp}(p_2,r) \\
& - \sgn(rs)\,q^{m+l-1}\, \left|\frac{tp_2}{sp_1}\right|
\,(\kappa(s) - \kappa(t))\,a_{t}(p_1,s)\, a_p(p_2,r) \\
& - \sgn(rs)\,q^{-m-l-1}\, \left|\frac{tp_1}{sp_2}\right|\,
(\kappa(r) - \kappa(p))\,a_p(p_2,r) \, a_t(p_1,s) \ \Big].
\end{split}
\]
Unless
$\sgn(p_1 p_2) = \sgn(r s) \sgn(p t)$, $|p_2/p_1| q^m = |r/s|$
and $q^l |p_2/p_1| = |p/t|$, the above expression is zero. Thus,
\begin{equation} \label{eq:lefthandside}
\begin{split}
2S_L  = & \ \de_{|p_1 p/p_2 t|,q^{m-n}}\,\de_{m-n,l}\,
\de_{\sgn(pt)(p_2/p_1)q^m s,r}\,\sgn(p_1 p_2)\, \left|\frac{t}{s}\right| \\
&  \Big[\, (q^{m-1} p\,|t| + q^{-m-1} t\,|p|) \, a_t(p_1,s)\,a_p(p_2,r) \\
&  - \sgn(pt)\,q^{-1}\, \sqrt{(1+ \kappa(t/q))(1+ \kappa(p/q))}\,\, a_{t/q}(p_1,s)\,
a_{p/q}(p_2,r) \\
& - \sgn(pt)\,q\, \sqrt{(1+ \kappa(t)) (1+ \kappa(p))}\,\, a_{q p}(p_1,s)\,
a_{q p}(p_2,r)  \Big] .
\end{split}
\end{equation}

Next we write out the right hand side $S_R$ of \eqref{eq:crucialpartincommutationlemma}. Using the action \eqref{eq:actioncasimir} of $\Om_0$ again, we see that
\[
\begin{split}
2S_R = & \ (q^{m-1} p \,|t| + q^{-m-1} t\,|p|)\,
\langle\, Q(p_1,p_2,n)\,f_{m,p,t} , f_{l,r,s} \rangle \\
& - \sgn(pt)\,\sqrt{(1+\kappa(p))(1+\kappa(t))}\,
\langle  Q(p_1,p_2,n)\,f_{m,qp,qt} , f_{l,r,s} \rangle \\
& - \sgn(pt)\,\sqrt{(1+\kappa(q^{-1}p))(1+\kappa(q^{-1}t))}\,
\langle  Q(p_1,p_2,n)\,f_{m,q^{-1}p,q^{-1}t} , f_{l,r,s} \rangle.
\end{split}
\]
Writing this out in terms of the matrix coefficients of $Q(p_1,p_2,n)$, see \eqref{eq:matrixelementsofQppnandhatJQppnhatJ}, we obtain
\[
\begin{split}
2S_R =& \  \de_{|p_1 p/p_2 t|,q^{m-n}}\,\de_{m-n,l}\, \de_{\sgn(pt)(p_2/p_1)q^m s,r}\, \left|\frac{t}{s}\right| \\
& \Big[ (q^{m-1} p\,|t| + q^{-m-1} p\,|t|) \, a_t(p_1,s)\,a_p(p_2,r) \\
& - \sgn(pt)\,q^{-1}\, \sqrt{(1+ \kappa(q^{-1} t))(1+ \kappa(q^{-1}p))}\, a_{q^{-1} t}(p_1,s)\,
a_{q^{-1} p}(p_2,r) \\
&- \sgn(pt)\,q\, \sqrt{(1+ \kappa(t))(1+ \kappa(p))}\,\, a_{q t}(p_1,s)\, a_{q p}(p_2,r) \ \bigr].
\end{split}
\]
Comparing this with \eqref{eq:lefthandside} we see that $S_L=S_R$, hence \eqref{eq:crucialpartincommutationlemma} holds.
\end{proof}

\subsection{Proof of Lemma \ref{lem:fundlemasymptoticsQ}} \label{app:fundlemasymptoticsQ}
We prove the following result:
Let $f\colon J(p,m,\ep,\et)\to \C$ be bounded, and consider the function
\begin{equation*}
\begin{split}
g(w) = &(-1)^{m'} (\et')^{\chi(p_1p_2)+m} \,  q^{n+m}p_1^2\,
\frac{(\ep'\et')^{\chi(w)}}{|w|} \\& \times
 \sum_{z \in J(p,m,\ep,\et)}\, \frac{f(z)}{|z|}\,\,
a_{p_1}(z,w)\,a_{p_2}(\ep\,\et\,q^m p\, z ,
\ep' \et' q^{m'} p\, w),
\end{split}
\end{equation*}
for $w\in J(p,m',\ep',\et')$.
\begin{enumerate}
\item If $f(z) \sim A  \, t^{-\chi(z)}$ as $z\to 0$
for some $A \in \C$ and $t\in \C$, $|t|>1$, then
\[
g(w)  \sim  A\, t^{-\chi(w)}\, \et^n \, s(\ep,\ep')\,
s(\et,\et')\ S(\ep\et/ t;p_1,p_2,n), \qquad \text{as}\ w\to 0.
\]
\item If $f(z) \sim  \Re ( A e^{-i\psi\chi(z)})$ as $z\to 0$ for some $A \in \C$ and $\psi \in \mathbb R$, then
\[
g(w) \sim \, \et^n \, s(\ep,\ep')\,
s(\et,\et')\ \Re\bigl( A e^{-i\psi\chi(w)} S(\ep\et e^{-i\psi};p_1,p_2,n)\bigr),
\qquad \text{as $w\to 0$.}
\]
\end{enumerate}
Here we use the notation $f(z)\sim g(z)$ as $z\to 0$, for $\lim_{z\to 0} \bigl( f(z)-g(z)\bigr) = 0$. The function $S(\cdot;p_1,p_2,n)$ is defined by \eqref{eq:S=sum1phi1 1phi1}.
\begin{proof}
The proof is based on splitting the sum in $g(w)$, and taking limits in both parts of the sum using Tannery's theorem, i.e., the dominated convergence theorem for infinite sums.

First of all, the boundedness of $f$ together with Lemma \ref{lem:estimate2} implies that the sum by which $g(w)$ is defined is absolutely convergent. Let us denote $\theta = \ep\et\,q^m p$, $\theta' = \ep'\et'\,q^{m'} p$ and $r = \min\{q,q/|\theta|\}$. Now we split the sum for $g$ into a part with $|z| >r$ and a part with $|z|\leq r$. First we consider the part with $|z|>r$. We define, for $y \in J(p,m',\ep',\et')$, \[
B(y) = \frac{(\ep'\et')^{-\chi(y)}}{|y|}\,\,
\sum_{\substack{z \in J(p,m,\ep,\et)\\|z| > r  }}
\frac{1}{|z|}\,a_{p_1}(z,y)\,
a_{p_2}(\theta z, \theta' y)\,f(z) \ .
\]
By Lemma \ref{lem:estimate1} there exists a constant
$D > 0$ so that
\begin{equation} \label{eq:inequality1}
|a_{p_1}(z,y)\,a_{p_2}(\theta z, \theta' y)|
\leq D\,\nu(p_1/y)\,\nu(p_2/\theta' y)\, |z|^{\chi(p_1/y)}\,|\theta
z|^{\chi(p_2/\theta' y)}
\end{equation}
for all $z \in J(p,m,\ep,\et)$ and $y \in J(p,m',\ep',\et')$ satisfying $|z| > r$ and $|y| < r$.
Since, by assumption, $f$ is bounded, inequality \eqref{eq:inequality1} and Tannery's theorem imply that $B(y)\to 0$
as $y \to 0$.

Next consider the remaining sum over $z\in J(p,m,\ep,\et)$, $|z|\leq r$, for all $y \in J(p,m',\ep',\et')$. We go over to a new summation parameter $x=z/y$, so that it follows from $\sgn(z)=\ep$ that $\sgn(x)=\sgn(p_1)$. This gives
\begin{equation}\label{eq:newsummationparameter}
\begin{split}
&\frac{(\ep' \et')^{\chi(y)}}{|y|}\,\,
\sum_{\substack{x \in J(p,m,\ep,\et)\\ |z|\leq r }} \
\frac{1}{|z|}\,a_{p_1}(z,y)\,a_{p_2}
(\theta z, \theta' y)\,f(z) \\
&   = \frac{(\ep'\et')^{\chi(y)}}{|y|}\,\,
\sum_{\substack{x \in \sgn(p_1) q^\Z \\ |x| \leq r/|y| }}
\  \frac{1}{|y x|}\,a_{p_1}(y x,y)\,
a_{p_2}(\theta y x, \theta' y)\,f(y x).
\end{split}
\end{equation}
Let $F:J(p,m,\ep,\et)\to \C$ be a bounded function such that
\[
f(y) =
\begin{cases}
t^{-\chi(y)}F(y),& \text{if}\ f(w) \sim  A t^{-\chi(w)},\ \text{as } w \to 0,\\
\Re\big(e^{-i\psi \chi(y)} F(y)\big), & \text{if}\ f(w) \sim \Re(A e^{-i\psi \chi(w)}),\ \text{as } w \to 0.
\end{cases}
\]
Observe that this implies $\lim_{y \to 0}F(y)=A$. Now for $y \in J(p,m',\ep',\et')$ and $|t|\geq 1$, we define
\begin{equation}\label{eq:explexprsCy}
C(y;t)  = \frac{(\ep'\et')^{\chi(y)}}{|y|}\,\,
\sum_{\substack{x \in \sgn(p_1)q^\Z \\ |x| \leq \frac{r}{|y|} }}  \
\frac{1}{|y x|}\,a_{p_1}(y x,y)\, a_{p_2}(\theta y x, \theta' y)
 \, t^{-\chi(x)} F(yx)\, .
\end{equation}
We now consider the asymptotic behaviour of $C(y;t)$ as $y \to 0$.

Let us first see that we can take termwise limits in \eqref{eq:explexprsCy}. For $x \in \sgn(p_1) q^\Z$ satisfying $|x| \leq \frac{r}{|y|}$, we have by Definition \ref{def:functionap},
\begin{equation} \label{eq:explexprsCy2}
\begin{split}
&  \frac{(\ep'\et')^{\chi(y)}}{|y^2x|}\, a_{p_1}(y x,y)\, a_{p_2}(\theta y x, \theta' y) \,
t^{-\chi(x)}\,F(yx)\\
&   \ =  \frac{(\ep'\et')^{\chi(y)}}{|y^2x|} \,c_q^2\, s(\ep,\ep')\, s(\et,\et') \,
(-1)^{\chi(p_1 p_2)} \, (-\ep')^{\chi(yx)}\,(-\et')^{\chi(\theta y
x)}\, |y|^2\,|\theta'|\,\, t^{-\chi(x)}\,F(yx) \\
&  \qquad \times\ \nu(p_1/x) \, \nu(p_2 q^{n}/x)
\sqrt{(-\kappa(p_1),-\kappa(p_2);q^2)_\infty} \,\,
\sqrt{\frac{(-\kappa(y),-\kappa(\theta' y);q^2)_\infty}
{(-\kappa(y x),-\kappa(\theta\, y x);q^2)_\infty}} \\
&  \qquad \times \ \Psis{-q^2/\kappa(p_1)}{q^2 \kappa(yx/p_1)}{q^2,\,q^2
\kappa(x)}\ \Psis{-q^2/\kappa(p_2)}{q^2 \kappa(\theta yx/p_2)}
{q^2,\,q^2 \kappa\big(\sgn(p_1 p_2) q^{-n}x\big)} \\ \\
&   \ = (-1)^{\chi(p_1 p_2 p)+m}\,\,s(\ep,\ep')\, s(\et,\et')\, (\et')^{\chi(p)+m}\,\,q^m\,p \,\,c_q^2\,
q^{n}\,\sqrt{(-\kappa(p_1),-\kappa(p_2);q^2)_\infty} \, \\
& \qquad \times \  (\ep'\et'/t)^{\chi(x)}\, F(yx)
\,   \nu(p_1/x)\,\nu(p_2 q^{n}/x)\,\, \sqrt{\frac{(-\kappa(y),-\kappa(\theta' y);q^2)_\infty}
{(-\kappa(y x),-\kappa(\theta\, y x);q^2)_\infty}} \\
& \qquad \times \ |x|^{-1}\,(q^2\kappa(yx/p_1),q^2 \kappa(\theta yx/p_2) ;q^2)_\infty\\
&  \qquad \times \  \rphis{1}{1}{-q^2/\kappa(p_1)} {q^2 \kappa(yx/p_1)}{q^2,\,q^2 \kappa(x)}\
\rphis{1}{1}{-q^2/\kappa(p_2)}{q^2 \kappa(\theta yx/p_2)} {q^2,\,q^2 \kappa\big(\sgn(p_1 p_2) q^{-n}x\big)}\,.
\end{split}
\end{equation}
Assuming for the moment that we can apply Tannery's Theorem, we see from the last expression that
$C(y;t)$ converges to
\[
A\,\,(-1)^{\chi(p_1 p_2 p)+m} \,\frac{q^m\,p}{|p_1 p_2|} \,\,(\et')^{\chi(p)+m}\,\sgn(p_2)^n \, s(\ep',\ep)\,
s(\et', \et)\,   S(\ep\et/t;p_1,p_2,n)
\]
as $y\to 0$, using \eqref{eq:S=sum1phi1 1phi1}. This proves the lemma.

In order to be able to apply Tannery's Theorem, we need to estimate the summand by a term independent
of $y$. For small $x$ such an estimate follows from \eqref{eq:explexprsCy2}, since $F(y)\to A$ as $y\to 0$
and the functions $\nu$ are small. It remains to give an estimate for large $x$ uniformly for $|y|\leq q^l$ for some $l\in\Z$. By \eqref{eq:symmetryforapxy}
we have
\[
\begin{split}
 \frac{1}{|y^2x|}&\,|a_{p_1}(y x,y)|\,
|a_{p_2}( \theta y x, \theta' y)| \, |t^{-\chi(x)}\,F(yx)| \\
=&\, \frac{1}{|y^2x|}\,|a_{p_1}(y,y x)|\, |a_{p_2}( \theta' y, \theta yx)| \,
|t^{-\chi(x)}\,F(yx)| \\
=&\,  c_q^2\,|\theta x t^{-\chi(x)}F(yx)|\, \sqrt{(-\kappa(p_1),-\kappa(p_2);q^2)_\infty}\,
 \nu(p_1x)\,\nu(p_2\,q^{-n}\,x)\,  \\
 & \quad \times \sqrt{\frac{(-\kappa(y x),-\kappa(\theta\,y x);q^2)_\infty}
{(-\kappa(y),-\kappa(\theta' y);q^2)_\infty}}\, \left| \Psis{-q^2/\kappa(p_1)}{q^2 \kappa(y/p_1)}{q^2,\,q^2 /\kappa(x)}\right|\ \\
&  \quad \times \left|
\Psis{-q^2/\kappa(p_2)}{q^2 \kappa(\theta' y/p_2)} {q^2,\,q^2 \kappa(\sgn(p_1 p_2)\,q^{n }/x)}\right| \, .
\end{split}
\]
The $\Psi$-functions are bounded for $|x|$ large and
$|y|\leq q^l$. Put $|x|=q^{-k}$, then using the boundedness of $F$ and the $\theta$-product identity \eqref{eq:thetaprodid}, we find
\[
\begin{split}
|xt^{-\chi(x)}\,F(yx)&|\,
\nu(p_1x)\,\nu(p_2\,q^{-n}\,x)\,
\sqrt{(-\kappa(y x),-\kappa(\theta\,y x);q^2)_\infty} \\
&\leq D_1 |xt^{-\chi(x)}|\,
\nu(p_1x)\,\nu(p_2\,q^{-n}\,x)\,
\sqrt{(-q^{2l-2k},-|\theta|\,q^{2l-2k});q^2)_\infty}  \\
&= D_2 |t q^{n+1+l}/p_1p_2|^k |\theta|^{\hf k} \sqrt{ (-q^{2-l}, -q^{2-l}/|\theta|;q^2)_k }\\
& \leq D_3 |t q^{n+1+l}/p_1p_2|^k |\theta|^{\hf k},
\end{split}
\]
where the constants $D_i$ are independent of $x$. We see that for $l$ large enough this gives us the desired estimate.
\end{proof}


\printindex


\end{document}